%% file: main.tex
\documentclass[a4paper]{amsart}

\usepackage{hyperref} \hypersetup{%
  colorlinks,%
  linkcolor={red!80!black},%
  citecolor={blue!80!black},%
  urlcolor={blue!80!black}%
}

\usepackage{%
  amsthm,%
  amssymb,%
  amsmath,%
  mathtools,%
  stmaryrd,%
  enumitem,%
  eucal,%
  macros,%
  environments,%
  mathscinet,%
}

\usetikzlibrary{matrix,arrows,calc}

\newcounter{rpage}
\newcommand{\rpage}{\value{rpage}} 		
\newcommand{\rescale}{0.3}			
\newcommand{\margen}{0.1}			
\newcommand{\abajo}{5}				
\newcommand{\arriba}{5}
\newcommand{\derecha}{5} 			

\newcounter{trunco}

\DeclareSymbolFont{sfoperators}{OT1}{cmss}{m}{n}%
\DeclareSymbolFontAlphabet{\mathsf}{sfoperators}%
\makeatletter%
\def\operator@font{\mathgroup\symsfoperators}%
\makeatother

\let\oldpentagon\pentagon%
\renewcommand{\pentagon}{{\normalfont\oldpentagon}}

\title{The Derived Auslander--Iyama Correspondence}



\author[G.~Jasso]{Gustavo Jasso}%
\address[G.~Jasso]{%
  Lund University, %
  Centre for Mathematical Sciences, %
  Slvegatan 18A, %
  221 00 Lund, %
  Sweden%
}
\curraddr{Mathematisches Institut, %
  Universität zu Köln, %
  Weyertal 86-90, %
  50931 Köln, %
  Germany}
\email{gjasso@math.uni-koeln.de}
\urladdr{https://gustavo.jasso.info}

\address[B.~Keller]{%
  Université Paris Cité, %
  UFR de Mathématiques, Case 7012, %
  Bâtiment Sophie Germain,         %
  8 place Aurélie Nemours,%
  75013 Paris Cedex 13, %
  France
}                                   %
\email{bernhard.keller@imj-prg.fr }%
\urladdr{https://webusers.imj-prg.fr/~bernhard.keller/}

\author[F.~Muro]{Fernando Muro}%
\address[F.~Muro]{%
  Universidad de Sevilla, %
  Facultad de Matemáticas, %
  Departamento de Álgebra, %
  Calle Tarfia s/n, %
  41012 Sevilla, %
  Spain%
} \email{fmuro@us.es} \urladdr{https://personal.us.es/fmuro/}

\begin{document}

\sloppy

\dedicatory{With an appendix by Bernhard Keller}

\maketitle

\begin{abstract}
  \input{abstract}
\end{abstract}

\setcounter{tocdepth}{1}
\tableofcontents

\crefname{question}{Question}{Questions}
\Crefname{question}{Question}{Questions}
\crefname{notation}{Notation}{Notations}
\Crefname{notation}{Notation}{Notations}
\crefname{construction}{Construction}{Constructions}
\Crefname{construction}{Construction}{Constructions}
\crefname{setting}{Setting}{Settings}
\Crefname{setting}{Setting}{Settings}

\section*{Introduction}

We work over a field $\kk$. Triangulated categories, introduced by Verdier in
the late 1960s~\cite{Ver96} (see also Puppe's~\cite{Pup62}) are central objects
of study in homological algebra~\cite{HJR10}. Typical examples of triangulated
categories include derived categories of rings and schemes, stable categories of
representations of finite groups, and the stable homotopy category (of spectra).
The main result in this article is the following classification theorem, whose
name is motivated by the parallel with the classical Auslander
Correspondence~\cite{Aus71}, a seminal result in the representation theory of
Artin algebras, and Iyama's higher-dimensional generalisation
thereof~\cite{Iya07a}.

\begin{theorem intro}[Derived Auslander--Iyama Correspondence]
  \label{thm:dZ-Auslander_correspondence}
  Suppose that $\kk$ is a perfect field and $d\geq1$ an integer. There are
  bijective correspondences between the following:
  \begin{enumerate}
  \item\label{it:dZ-CT_DGAs} Quasi-isomorphism classes of DG algebras $A$ that
    satisfy the following:
    \begin{enumerate}
    \item The ordinary algebra $\dgH[0]{A}$ is a basic finite-dimensional algebra.
    \item The free DG $A$-module $A\in\DerCat[c]{A}$ is a $d\ZZ$-cluster tilting
      object, that is $A$ is a $d$-cluster tilting object that satisfies $A\cong A[d]$.
    \end{enumerate}
  \item\label{it:Tc} Equivalence classes of pairs $(\T,c)$ consisting of
    \begin{enumerate}
    \item an algebraic triangulated category $\T$ with finite-dimensional
      morphism spaces and split idempotents and
    \item a basic $d\ZZ$-cluster tilting object $c\in\T$, that is $c$ is a $d$-cluster
      tilting object that satisfies $c\cong c[d]$.
    \end{enumerate}
  \item\label{it:pairsLambdaI} Equivalence classes of pairs $(\Lambda,I)$
    consisting of
    \begin{enumerate}
    \item a basic Frobenius algebra $\Lambda$ and
    \item an invertible $\Lambda$-bimodule $I$ such that there exists an isomorphism
      $\Omega_{\Lambda^e}^{d+2}(\Lambda)\cong I$ in the stable category of
      $\Lambda$-bimodules, where $\Lambda^e\coloneqq\Lambda\otimes\Lambda^\op$.
    \end{enumerate}
  \end{enumerate}
  The correspondences are given by
  \[
    A\longmapsto (\DerCat[c]{A},A)\qquad\text{and}\qquad(\T,c)\longmapsto (\T(c,c),\T(c,c[-d])).
  \]
  Moreover, the algebraic triangulated
  categories in \eqref{it:Tc} admit a unique DG enhancement in the sense of~\cite{BK90}.
\end{theorem intro}

For the definition of $d\ZZ$-cluster tilting
object, see \Cref{def:d-CT_dZ-CT} (and compare with \Cref{rmk:dZ}). Here we only
mention that such objects are in particular $d$-step generators of the ambient
category in the sense of Rouquier~\cite{Rou08}, but they generate in a stronger
sense (see~\Cref{thm:IY-Bel-d-CT}).

In the context of \Cref{thm:dZ-Auslander_correspondence}, two pairs $(\T,c)$ and $(\T',c')$ as in
\eqref{it:Tc} are \emph{equivalent} if
there exists an equivalence of triangulated categories
\[
  F\colon\T\stackrel{\sim}{\longrightarrow}\T'
\]
such that $F(c)\cong c'$. Similarly, two pairs $(\Lambda,I)$ and $(\Lambda',I')$
as in~\eqref{it:pairsLambdaI} are equivalent if there exists an algebra
isomorphism $\varphi\colon\Lambda\stackrel{\sim}{\to}\Lambda'$ such that
$\varphi^*(I')\cong I$. Also, a devisage argument shows that
$\DerCat[c]{A}$ has finite-dimensional morphism spaces if and only if
for all $i\in\ZZ$ the vector space $\dgH[i]{A}\cong\Hom[\DerCat{A}]{A}{A[i]}$ is
finite dimensional, which is the case for the DG algebras considered in the
theorem.

A $1\ZZ$-cluster tilting object is simply an additive generator of
the ambient triangulated category. Hence, for $d=1$,
\Cref{thm:dZ-Auslander_correspondence} specialises to the main result in the
second author's article~\cite{Mur22}. Similar to~\emph{op.~cit.},
our proof of \Cref{thm:dZ-Auslander_correspondence} relies crucially on
the obstruction theory developed by the second-named author in~\cite{Mur20a}. The central problem is
that of interpreting the property of an object being $d\ZZ$-cluster tilting in
terms of the Hochschild(--Tate) cohomology of its (graded) endomorphism algebra;
the precise statement is given~\Cref{coro:pre-d+2-ang=unit}.

\Cref{thm:dZ-Auslander_correspondence} shows a deep connection between
$d\ZZ$-cluster tilting objects and twisted $(d+2)$-periodic algebras. The fact that the endomorphism algebra of a $d\ZZ$-cluster tilting object in a
triangulated category is twisted $(d+2)$-periodic was observed first
by Dugas~\cite{Dug12} in the algebraic case and in general by Chan, Darp{\"o}, Iyama and
Marczinzik in~\cite{CDIM25}, where they also investigate the
connection between twisted periodicity and higher Auslander--Reiten theory
focusing on the fundamental class of trivial extensions of finite-dimensional
algebras. From the viewpoint promoted in this article, the periodicity of such
endomorphism algebras is a consequence of
the relationship between $d\ZZ$-cluster tilting objects and $(d+2)$-angulated
categories (in the sense of Gei\ss, Keller and Oppermann~\cite{GKO13}) and of a general characterisation of
twisted periodic algebras~\cite{GSS03,Han20}, see~\Cref{prop:d+2-angulated_twisted_periodic}. For more
information on (twisted) periodic algebras and their importance in
representation theory and elsewhere in mathematics, we refer the reader
to~\cite{ES08,CDIM25} and the references therein. From this
perspective, \Cref{thm:dZ-Auslander_correspondence} can be regarded as a
contribution to the emerging subject of higher-dimensional homological algebra,
where $(d+2)$-angulated categories and their exact/abelian variants \cite{Jas16}
play a central role; see~\Cref{thm:projLambda_existence_and_uniqueness} for a
precise statement.

It is worth highlighting that \Cref{thm:dZ-Auslander_correspondence} cannot be generalised to
triangulated categories with $d$-cluster tilting objects that are not
$d\ZZ$-cluster tilting, at least not without further assumptions. On the one
hand, for $d>2$, Ladkani~\cite{Lad22} has shown that \emph{every} basic
finite-dimensional algebra over an algebraically closed field arises as the
endomorphism algebra of a $d$-cluster tilting object in an
algebraic triangulated category. Moreover, still for $d>2$, Ladkani also
has shown that there are \emph{non-equivalent} algebraic triangulated categories
with $d$-cluster tilting objects whose endomorphism algebras are isomorphic.
On the other hand, there are classification results under additional
Calabi--Yau assumptions~\cite{KL23a,KL23c,Han22}.

The proof of \Cref{thm:dZ-Auslander_correspondence} relies heavily on the
well-known relationship between DG algebras and
$A_\infty$-algebras~\cite{Kad82}. Consider a DG algebra $\mathbf{\Lambda}$
whose cohomology is \emph{$d$-sparse}, which is to say that it is concentrated
in degrees $d\ZZ$ (for example, the derived endomorphism algebra of a
$d\ZZ$-cluster tilting object). Then, its cohomology $\dgH{\mathbf{\Lambda}}$
admits a minimal $A_\infty$-algebra structure whose first possibly non-vanishing
higher operation is $m_{d+2}$ for degree reasons. This operation is a cocycle in the Hochschild
cohomology of the graded algebra $\dgH{\mathbf{\Lambda}}$ and the
Gerstenhaber square of its cohomology class vanishes: $\Sq(\class{m_{d+2}})=0$.\footnote{
When the characteristic of $\kk$ is different from $2$, there is an equality
\[\Sq(\class{m_{d+2}})=\frac{1}{2}[\class{m_{d+2}},\class{m_{d+2}}].\]} The class
$\class{m_{d+2}}$ is an invariant of the quasi-isomorphism class of
$\mathbf{\Lambda}$ that we call the \emph{universal Massey product of
  length $d+2$}. This invariant, which plays a fundamental role in this article,
was investigated in the case $d=1$ in~\cite{BKS04}, with its topolgical analogue~\cite{Sag08}
having roots in~\cite{Bau97,BD89}.

More generally, given a \emph{$d$-sparse Massey
  algebra}, that is a pair $(A,m)$ consisting $d$-sparse graded algebra $A$ and
a class $m\in\HH[d+2][-d]{A}$ whose Gerstenhaber square $\Sq(m)$ vanishes, there is a bigraded cochain
complex $\SHC{A}[m]$ whose vector spaces of cochains are given by the Hochschild
cohomology $\HH{A}$ of the graded algebra $A$ and whose differential of bidegree
$(d+1,-d)$ is given by
\[
  \partial\colon x\longmapsto\pm[m,x]
\]
almost everywhere (\Cref{def:Massey_algebra}). The \emph{Hochschild--Massey cohomology} of the
$d$-sparse Massey algebra $(A,m)$ is the cohomology of this complex, denoted by
$\SHH{A}[m]$. The following result is the second main theorem in this article:

\begin{theorem intro}
  \label{thm:secondary_formality}
  Let $(A,m)$ be a $d$-sparse Massey algebra $A$. Suppose that
  \[
    \SHH[p+2][-p]{A}[m]=0,\qquad p>d.
  \]
  If $B$ and $C$ are minimal $A_\infty$-algebras with $\dgH{B}=\dgH{C}=A$ as
  graded algebras such that
  \[
    \class{m_{d+2}^B}=\class{m_{d+2}^C}=m\in\HH[d+2][-d]{A},
  \]
  then $B$ and $C$ are $A_\infty$-isomorphic through an $A_\infty$-isomorphism with identity linear part.
\end{theorem intro}

\Cref{thm:secondary_formality} can be
regarded as a vast generalisation of Kadeishvili's Intrinsic Formality
Theorem~\cite{Kad88}, as we show in~\Cref{coro:Kadeishvili}. Although \Cref{thm:secondary_formality} likely has a direct proof, we deduce it
from the vanishing of certain terms in an extended spectral sequence of Bousfield--Kan
type~\cite{BK72} that is the main technical tool in the obstruction theory developed in
\cite{Mur20b}. Results similar to \Cref{thm:secondary_formality} have appeared
in other contexts; for example, in symplectic geometry, Lekili and Perutz have shown that the
$A_\infty$-structures on certain graded algebras arising from the Fukaya category of the
punctured torus are controlled by a \emph{pair} of Hochschild cohomology
classes, see~\cite{LP11} for details.

\subsection*{Applications}

Although \Cref{thm:dZ-Auslander_correspondence} has a theoretical character, it
has interesting applications, as we now explain.

\subsubsection*{(Non-)uniquness of DG enhancements}

Despite their ubiquity and usefulness, it has long been recognised that
triangulated categories lack a number of desirable formal/categorical
properties, see for example the discussion in~\cite[Sec.~2.2]{Toe11}. In the seminal article~\cite{BK90}, it is proposed to address the above
issues by endowing the triangulated categories that appear in practice (at
least in algebra and geometry) with what nowadays is called a `differential
graded enhancement.' Recall that a differential graded (=DG)
category~\cite{Kel65} is a category in which the morphisms
between two given objects form a cochain complex, and whose composition law is
compatible with the differentials in a suitable sense. A DG category $\A$ has an
associated homotopy category $\dgH[0]{\A}$, obtained by passing to the $0$-th
cohomology of its morphism complexes; when the DG category $\A$ has the property
of being \emph{pre-triangulated} (\Cref{def:pre-trianguled_DG}), the category
$\dgH[0]{\A}$ is canonically a triangulated category. Given a triangulated
category $\T$, we say that a pre-triangulated DG category $\A$ is a \emph{DG
  enhancement of $\T$} if there is an equivalence of triangulated categories
$\T\simeq\dgH[0]{\A}$. Interpreting these in the appropriate manner,
pre-triangulated DG categories have all the expected formal properties that
triangulated categories lack and, moreover, all triangulated categories that
arise naturally in algebra and geometry (but, in general, not in topology) admit
a DG enhancement, see~\cite{Kel06} for details.

Given a triangulated category $\T$, two natural questions arise: Does $\T$ admit
a DG enhancement? and if it does, is the enhancement unique up to the
appropriate notion of equivalence? In general, the first question has a negative
answer, for the stable homotopy category does not admit a DG
enhancement~\cite{Sch10a}, and there are even triangulated categories that are
linear over a field and do not admit a DG enhancement~\cite{RVdB20} (there also exist
triangulated categories that do not even admit `topological' enhancements
\cite{MSS07}). The second question is also rather delicate. Indeed, there exist
triangulated categories that admit inequivalent DG
enhancements~\cite{DS07,Sch02}, and there are even such triangulated categories
that are linear over a field~\cite{Kaj13,RVdB19}. Notwithstanding, there are
several results that show that important classes of triangulated categories,
such as derived categories of modules or of quasi-coherent sheaves, admit a
unique DG enhancement~\cite{LO10,CS18,CNS22}. We note that there are also
several important results on the uniqueness of `topological'
enhancements~\cite{Ant18,Sch01,Sch07,SS02}.

In this context, the relevance of \Cref{thm:dZ-Auslander_correspondence} is that
it provides a large class of triangulated categories with unique DG enhancements
and, moreover, these are classified by a minimal amount of algebraic data
(essentially that of a twisted periodic algebra). Another interesting feature of
\Cref{thm:dZ-Auslander_correspondence} is that the triangulated categories we
consider in general do not admit non-trivial $t$-structures; this is in stark
contrast with most uniqueness-of-enhancements results available in the
literature.

Our approach to the proof of
\Cref{thm:dZ-Auslander_correspondence} also yields a
more precise statement concerning the equivalence classes of strong enhancements
of the (algebraic) $(d+2)$-angulated categories we consider,
see~\Cref{thm:strong_enhancements}). As a consequence we obtain in \Cref{coro:non_strong_examples}, to our
knowledge, the first examples of triangulated categories that admit a unique
enhancement but not a unique strong enhancement (that is, an enhancement that is
not `strongly unique' in the sense of \cite{LO10}).

\subsubsection*{Recognition theorems}

As for further applications, the bijectivity of the correspondence in
\Cref{thm:dZ-Auslander_correspondence} yields recognition theorems for
interesting classes of algebraic triangulated categories that admit
$d\ZZ$-cluster tilting objects, from which we single out the following (we refer
the reader to \Cref{sec:recognition_thms} for details and definitions). Below,
we restrict to connected non-separable algebras for simplicity.

Following Iyama and Oppermann \cite{IO11}, we say that a finite-dimensional
algebra $A$ is \emph{($d$-hereditary) $d$-representation finite} if $A$ has
global dimension at most $d$ and there exists a $d$-cluster tilting $A$-module
(which turns out to be unique up to multiplicity of its indecomposable direct
summands \cite{Iya11}). For example, if $d=1$, then $A$ is a hereditary algebra
of finite representation type and, if the ground field is moreover algebraically
closed, then $A$ is Morita equivalent to the path algebra of a Dynkin quiver.
Recall that the the category of finitely-generated projective
modules over the Gelfand--Ponomarev preprojective algebra~\cite{GP79} of a Dynkin quiver
admits an algebraic (1-Calabi--Yau) triangulated structure~\cite{AR96} that is essentially unique by the main result in \cite{Mur22}.
\Cref{thm:dZ-Auslander_correspondence} yields the
following analogous recognition theorem.

\begin{theorem*}[{\Cref{thm:AGK_uniqueness-dRF}}]
  Suppose that $\kk$ is a perfect field. The $d$-Calabi--Yau Amiot--Guo--Keller
  cluster category of a connected, non-separable $d$-representation finite
  algebra has a unique enhancement and is uniquely characterised among algebraic
  triangulated categories by the existence of a $d\ZZ$-cluster tilting object
  whose endomorphism algebra is isomorphic to the $(d+1)$-preprojective algebra
  $\Pi_{d+1}(A)$ of $A$ in the sense of~\cite{IO13}.
\end{theorem*}

In the case $d=1$ of the above theorem, the algebra $\Pi_{d+1}(A)=\Pi_2(A)$ is
the preprojective algebra of $A$ as defined in~\cite{BGL87}.

As a further application of \Cref{thm:dZ-Auslander_correspondence}, we obtain
the following recognition theorem for the Amiot cluster category of a
self-injective quiver with potential. This result complements the Recognition
Theorem for acyclic cluster categories of Keller and Reiten \cite{KR08}, which
deals with $2$-Calabi--Yau categories with a cluster tilting object whose
endomorphism algebra has acyclic Gabriel quiver.

\begin{theorem*}[{\Cref{thm:recognition_cluster_cat_self-inj_QP}}]
  Let $\kk$ be an arbitrary field. The $2$-Calabi--Yau Amiot cluster category of
  a connected, non-separable self-injective quiver with potential $(Q,W)$ has a
  unique enhancement and is uniquely characterised among algebraic triangulated
  categories by the existence of a $2\ZZ$-cluster tilting object whose
  endomorphism algebra is isomorphic to the (completed) Jacobian algebra of
  $(Q,W)$.
\end{theorem*}

Both of the above results are special cases of a more general recognition
theorem for the Amiot--Guo--Keller cluster category of a (pseudo-compact)
homologically smooth DG algebra that satisfies a number of technical conditions
that guarantee, in particular, the existence of a $d\ZZ$-cluster tilting object,
see \Cref{thm:AGK_uniqueness}.

\subsubsection*{Proof of the Donovan--Wemyss Conjecture}

As a final application, in the appendix, B.~Keller explains how---combined with
crucial results of August~\cite{Aug20a} and Hua--Keller~\cite{HK24}---our main
result yields the last key ingredient to prove the Donovan--Wemyss Conjecture
(\cite[Conj.~1.4]{DW16} and \cite[Conj.~1.3]{Aug20a}) in the
context of the Homological Minimal Model Program for threefolds~\cite{Wem18}.

Let $X=\operatorname{Spec}R$ be the formal neighbourhood of an isolated compound
Du Val (=cDV) singularity over the field of complex numbers, that is $R$ is a
complete local three-dimensional hypersurface with an isolated singularity whose
generic hyperplane section is a Du Val/Kleinian singularity. The class of cDV
singularities was introduced by Reid~\cite{Reid81} who was motivated by their
natural occurrence in the minimal model program for three-folds~\cite{KM98}.
Suppose now that $X$ admits a crepant resolution (=smooth minimal model)
$p\colon\tilde{X}\to X$. To this geometric setup, Donovan and Wemyss~\cite{DW16}
associate a basic finite-dimensional algebra $\Lambda=\Lambda(p)$ -- the
contraction algebra -- that controls
the non-commutative deformations of the exceptional fibre of $p$,
see~\cite{Wem23} for a survey of the state of the art. Donovan and Wemyss made
the following remarkable conjecture~\cite{DW16,Aug20a}:
\begin{quote}
  Let $R_1$ and $R_2$ be isolated cDV singularities that admit crepant
  resolutions $p_1\colon\tilde{X}_1\to X_1$ and $p_2\colon\tilde{X}_2\to X_2$
  with corresponding contraction algebras $\Lambda(p_1)$ and $\Lambda(p_2)$.
  Then, the $\CC$-algebras $R_1$ and $R_2$ are isomorphic if and only if the
  contraction algebras $\Lambda(p_1)$ and $\Lambda(p_2)$ are derived equivalent.
\end{quote}
The Donovan--Wemyss Conjecture is considered the most important open problem in
the context of the Homological Minimal Model Program~\cite{Wem18}. It is known
that derived contraction algebras of a given isolated compound Du Val
singularity are derived equivalent~\cite{Wem18,Dug10}. In the appendix, Keller
explains how to reduce the remaining implication of the Donovan--Wemyss
Conjecture to the validity of the following statement:
\begin{quote}
  Let $R$ be an isolated cDV singularity that admits a crepant resolution
  $p\colon\tilde{X}\to X$ and
  $\sg(R)=\DerCat[b]{\mmod{R}}/\operatorname{K}^{\mathrm{b}}(\proj{R})$ its
  singularity category with its canonical DG enhancement $\sg_{dg}(R)$. Then, up
  to quasi-equivalence of DG categories, $\sg_{dg}(R)$ is uniquely determined by
  the following properties:
  \begin{itemize}
  \item The algebraic triangulated category $\dgH[0]{\sg_{dg}(R)}=\sg(R)$ has
    finite-dimensional morphism spaces and split idempotents.
  \item There exists a $2\ZZ$-cluster tilting object in $\sg(R)$ whose graded
    endomorphism algebra is isomorphic to the algebra of Laurent polynomials
    $\Lambda(p)[\imath^{\pm1}]$, $|\imath|=-2$.
  \end{itemize}
\end{quote}
The above statement is a special case of \Cref{thm:dZ-Auslander_correspondence},
and hence the remaining implication of the Dononvan--Wemyss Conjecture follows,
see the appendix for details and also our survey article~\cite{JKM24}. We also
mention that, two years after the first version of this article was posted in
the arXiv, an independent proof of the Donovan--Wemyss Conjecture was obtained
by Karmazyn, Lepri and Wemyss~\cite{KLW24}. Their proof also leverages the
relationship between DG and $A_\infty$-algebras and is more direct; however, it does not yield the kind of uniqueness results that follow from
\Cref{thm:dZ-Auslander_correspondence} (such as~\cite[Thm.~5.2.4]{JKM24}) for
it occurs in a different algebraic context more closely related to the geometric setup described above.

\subsection*{Outline of the proof of \Cref{thm:dZ-Auslander_correspondence} and
  structure of the article}

The proof of \Cref{thm:dZ-Auslander_correspondence} is involved and occupies the
vast majority of the article. For this reason, we provide a rough outline of the
argument that we hope will help the reader have a clear overview of the proof
and of the article as a whole.

In \Cref{sec:CT} we recall the definition and basic aspects related to
$d$-cluster tilting subcategories. Most of the material in this section is well
known to specialists, but \Cref{thm:dZ-ct_characterisation} (a characterisation
of $d\ZZ$-cluster tilting subcategories) is new.

In \Cref{sec:d+2} we sum up the essentials of the theory of $(d+2)$-angulated
categories. Again, most of the material in this section is known to experts, but
\Cref{thm:standard_d+2-angulated_characterisation} (essentially a reformulation
of \Cref{thm:dZ-ct_characterisation} in terms of $(d+2)$-angulated structures)
is new. In particular, \Cref{prop:dZ_CT_twisted_periodic} shows that the
correspondence in \Cref{thm:dZ-Auslander_correspondence} is well defined.
Moreover, we recall the construction of $(d+2)$-angulated structures on the
category of finitely-generated projective modules over a twisted
$(d+2)$-periodic algebra $\Lambda$ given by Amiot \cite{Ami07} in the case $d=1$
and Lin \cite{Lin19} in the general case $d\geq1$. In a nutshell, a choice of an
exact sequence of $\Lambda$-bimodules
\[
  \eta\colon\quad 0\to\twBim{\Lambda}[\sigma]\to P_{d+1}\to\cdots\to P_1\to
  P_0\to\Lambda\to0
\]
with projective-injective middle terms exhibiting the twisted periodicity of
$\Lambda$ defines a class of $(d+2)$-angles $\pentagon_\eta$ on
$\proj*{\Lambda}$ that, up to equivalence, depends only on the class of $\sigma$
in the outer automorphism group of $\Lambda$ and not on the choice of $\eta$
(see \Cref{prop:Amiot-Lin_independence_of_res}, which is also new). We call
these structures Amiot--Lin (AL) $(d+2)$-angulations. In \Cref{sec:dg_d+2} we
review the theory of DG categories and their derived categories, and introduce
DG enhancements of $(d+2)$-angulated categories
(\Cref{def:DG_enhancement_d+2-angulated}) as straightforward generalisations of
DG enhancements of triangulated categories.

We prove \Cref{thm:dZ-Auslander_correspondence} by providing an inverse to the
map $(\T,c)\mapsto(\Lambda,\sigma)$ that associates to an algebraic triangulated category
(with the relevant finiteness assumptions) with a $d\ZZ$-cluster tilting object
$c\in\T$ the corresponding twisted $(d+2)$-periodic algebra $\Lambda=\T(c,c)$
together with the algebra automorphism $\sigma$ induced by the action of the
$d$-fold shift of $\T$. Thus, given a pair $(\Lambda,\sigma)$ consisting of a
twisted $(d+2)$-periodic algebra with respect to the algebra automorphism
$\sigma$, we need to construct an \emph{algebraic} triangulated category $\T$
with a $d\ZZ$-cluster tilting object $c\in\T$ together with a fully faithful functor
\begin{align*}
  \proj*{\Lambda}&\longrightarrow\T\\
  \Lambda&\longmapsto c
\end{align*}
that is compatible with the actions of the functors
$-\otimes_\Lambda\twBim[\sigma]{\Lambda}$ and $[d]$ on the source and the target
category, respectively. For this, we observe that in fact $\T$ must be
equivalent to the perfect derived category $\DerCat[c]{A}$ of a DG algebra $A$ whose
cohomology is (isomorphic to) the $d$-sparse graded algebra
\[
  \gLambda=\bigoplus_{di\in\ZZ}\twBim[\sigma^i]{\Lambda},\qquad a\cdot b=\sigma^j(a)b,\quad |b|=dj.
\]
Thus, we need to construct a DG algebra $A$ such that $\dgH{A}\cong\gLambda$ and
${A\in\DerCat[c]{A}}$ is a $d\ZZ$-cluster tilting object. Equivalently, by
Kadeishvili's Homotopy Transfer Theorem, we need to endow $\gLambda$ with a minimal
$A_\infty$-algebra structure
\[
  A=(\gLambda,m_{d+2},m_{2d+2},m_{3d+2},\cdots),
\]
where $m_{i+2}=0$ if $i\not\in d\ZZ$ due to the fact that $\gLambda$ is
$d$-sparse (\Cref{prop:vanishing_d-sparse}), in such a way that
$A\in\DerCat[c]{A}$ is a $d\ZZ$-cluster tilting object. Moreover, in order to
obtain an inverse to the correspondence in
\Cref{thm:dZ-Auslander_correspondence}, such a minimal $A_\infty$-algebra $A$
should be unique up to $A_\infty$-quasi-isomorphism. In particular, we need to
translate the property `$A\in\DerCat[c]{A}$ is a $d\ZZ$-cluster tilting object'
into a property of the minimal $A_\infty$-algebra that can be leveraged using
techniques of homological and homotopical algebra.

Suppose first that $A$ is a minimal $A_\infty$-algebra with $\dgH{A}=\gLambda$
and observe that, since the latter graded algebra is $d$-sparse, the first
non-trivial higher operation
\[
  m_{d+2}\in\HC[d+2][-d]{\gLambda}[\gLambda]
\]
is in fact a Hochschild \emph{cocycle} (\Cref{prop:universal_Massey_welldef}).
Thus, associated to $A$ there is a well-defined class in Hochschild cohomology
\[
  \class{m_{d+2}}\in\HH[d+2][-d]{\gLambda}[\gLambda]
\]
that we call its \emph{universal Massey product of length $d+2$}. As it turns
out, most relevant to our purpose is the so-called \emph{restricted universal
  Massey product}
\[
  j^*\class{m_{d+2}}\in\HH[d+2][-d]{\Lambda}[\gLambda]\cong\Ext[\Lambda^e]{\Lambda}{\twBim{\Lambda}[\sigma]}[d+2]
\]
obtained by restricting $\class{m_{d+2}}$ along the morphism
\[
  j^*\colon\HH{\gLambda}[\gLambda]\longrightarrow\HH{\Lambda}[\gLambda]
\]
induced by the inclusion $j\colon\Lambda\to\gLambda$ of the degree $0$ part.
\Cref{coro:pre-d+2-ang=unit} then establishes the following remarkable fact:
$A\in\DerCat[c]{A}$ is a $d\ZZ$-cluster tilting object if and only if the
restricted universal Massey product $j^*\class{m_{d+2}}$ is represented by an
extension all of whose middle terms are projective-injective
$\Lambda$-bimodules. Furthermore, the latter property is equivalent to the class
$j^*\class{m_{d+2}}$ being a \emph{unit} in the Hochschild--Tate cohomology 
$\TateHH{\Lambda}[\gLambda]$, see \Cref{rmk:TateUnits}.

The proof of \Cref{coro:pre-d+2-ang=unit}, which is at the core of the argument,
relies on a careful analysis of the relationship between Toda brackets, Massey
products, minimal $A_\infty$-structures, Hochschild cohomology and
$(d+2)$-angulations; we study these relationships in \Cref{sec:higher_stuff}, the
results of which can be regarded as $(d+2)$-angulated analogues of results in
\cite{Mur20a,Mur22} and of classical results for triangulated categories. To
understand why this technical excursion is essential in our approach, notice
that if the restricted universal Massey product
$j^*\class{m_{d+2}}\in\Ext[\Lambda^e]{\Lambda}{\twBim{\Lambda}[\sigma]}[d+2]$ is
represented by an extension with projective-injective middle terms, then there
is a well-defined AL $(d+2)$-angulation $\pentagon_{j^*\class{m_{d+2}}}$ on
$\proj*{\Lambda}$. On the other hand, a general theorem of Geiss, Keller and
Oppermann (\Cref{thm:GKO-standard}) shows that
\[
  \proj*{\Lambda}\simeq\add*{A}\subseteq\DerCat[c]{A}
\]
is endowed with the structure of $(d+2)$-angulated category---which we call
standard---that is induced by the canonical triangulation of $\DerCat[c]{A}$,
provided that $A\in\DerCat[c]{A}$ is a $d\ZZ$-cluster tilting object. The proof
of \Cref{coro:pre-d+2-ang=unit} relies on the fact that the AL
$(d+2)$-angulation and the standard $(d+2)$-angulation of $\proj*{\Lambda}$
coincide up to sign; we demonstrate this in \Cref{thm:GKO_AL_Massey}. Toda brackets enter
the proof of \Cref{thm:GKO_AL_Massey} as these can be used to detect standard
$(d+2)$-angles (\Cref{thm:TodaBracket_d+2-angles}), and a certain agreement
between Toda brackets and Massey products on $\dgH{A}=\gLambda$
(\Cref{thm:TodaMassey}) turns out to be essential for computations. Due to its
technical nature, the reader may want to skip \Cref{sec:higher_stuff} on a first
reading and refer to it as necessary throughout their reading of
\Cref{sec:existence_and_uniqueness} (familiarity with
\Cref{def:universal_Massey,def:restricted_universal_Massey} will largely
suffice).

Having established all the necessary technical preliminaries, in
\Cref{sec:existence_and_uniqueness} we prove
\Cref{thm:dZ-Auslander_correspondence} and \Cref{thm:secondary_formality}. We
deduce \Cref{thm:dZ-Auslander_correspondence} from a more precise statement,
\Cref{thm:A-Ainfty-version}, which allows us to construct the inverse of the
correspondence: There exists an essentially unique minimal $A_\infty$-algebra
structure $A$ on $\gLambda$ whose restricted universal Massey product is a unit
in the corresponding Hochschild--Tate cohomology, that is such that
$A\in\DerCat[c]{A}$ is a $d\ZZ$-cluster tilting object. The proofs of
\Cref{thm:secondary_formality} and \Cref{thm:A-Ainfty-version} utilise the same
homotopical techniques (which we discuss in the relevant subsections) leveraged
in \cite{Mur22} to establish \Cref{thm:dZ-Auslander_correspondence} in the case
$d=1$ of triangulated categories of finite type; in particular the obstruction
theory developed in \cite{Mur20b} by the second-named author plays a crucial
role in this part of the article. In \Cref{sec:comms_non_uniqueness} we
illustrate the necessity for the hypotheses in
\Cref{thm:dZ-Auslander_correspondence} and
\Cref{thm:projLambda_existence_and_uniqueness} by means of examples. Finally,
\Cref{sec:recognition_thms} includes several recognition theorems that we deduce
from \Cref{thm:dZ-Auslander_correspondence}.

\subsection*{Conventions}

We fix a positive integer $d\geq1$ once and for all (we occasionally let $d=-1,0$,
but this is always clear from the context). All algebras and categories in this
article are assumed to be linear over a ground field $\kk$, except for a few obvious examples; unless explicitly
noted otherwise, all modules are right modules. The Jacobson radical of an
algebra $\Lambda$ is denoted by $J_\Lambda$. Our main results require $\kk$ to
be a perfect field, and this assumption is indicated whenever relevant. Unless
noted otherwise, all (ordinary) categories considered in this article are
additive; accordingly, all subcategories are assumed to be (strictly) full and
closed under the formation of finite direct sums in the ambient category. Given
a collection $X$ of objects in a category $\C$ with split idempotents
(=Karoubian), we denote by $\add*{X}$ the smallest (additive) subcategory of
$\C$ that contains $X$ and is closed under direct summands; if $X=\set{x}$
consists of a single object, we write $\add*{x}=\add*{X}$. We say that an object
$c$ in a category in which the Krull--Remak--Schmidt Theorem holds (for example,
in an additive category with finite-dimensional morphism spaces and split
idempotents \cite[Cor.~4.4]{Kra15}) is \emph{basic} if in any decomposition
$c=c_1\oplus c_2\oplus\cdots\oplus c_n$ into indecomposable objects, there is an
isomorphism $c_i\cong c_j$ if and only if $i=j$. A finite-dimensional algebra
$\Lambda$ is \emph{basic} if the regular representation of $\Lambda$ is basic as an object
of its category of finite-dimensional (projective) modules. We usually denote
the shift/suspension functor in a triangulated category by $[1]$, although the
notation $\Sigma$ is used when dealing with general $(d+2)$-angulated
categories. If $\T$ is a triangulated category with split idempotents and $X$ is
a class of objects in $\T$, we denote by $\thick*{X}$ the smallest triangulated
subcategory of $\T$ that contains $X$ and is closed under direct summands; if
$X=\set{x}$ consists of a single object, we write $\thick*{x}=\thick*{X}$.
Finally, we compose morphisms in a category as functions: the composite of
morphisms $f\colon x\to y$ and $g\colon y\to z$ is the morphism $gf=g\circ
f\colon x\to z$.

\subsection*{Acknowledgements}

The authors thank Yifei Cheng, Carlo Klapproth, Peter J{\o}rgensen, Zhengfang Wang and the Uppsala
Representation Theory Group for comments on a previous versions of
this article, as well as Yank{\i} Lekili for bringing references
\cite{Sei15,LP11} to our attention. The authors also wish to thank A.~Lorenzin for a question that lead to the inclusion of
\Cref{subsec:strong} and Martin Herschend for
discussions in relation to extending
the scope of our main results, and for suggesting to
include~\Cref{rmk:skew-group-algebras}. The authors are particularly grateful to Julian
K{\"u}lshammer, Alexandra Zvonareva and an anonymous referee for their numerous
comments and suggestions that have greatly improved the presentation of this
article.

\subsection*{Financial support}

G.~J.~was partially supported by the Deutsche Forschungsgemeinschaft (German
Research Foundation) under Germany's Excellence Strategy – GZ 2047/1, Projekt-
ID 390685813 as well by the Swedish Research Council (Vetenskapsrådet) Research Project
Grant 2022-03748 `Higher structures in higher-dimensional homological algebra.' F.~M.~was partially supported by grants PID2020-117971GB-C21
funded by MCIN/AEI/10.13039/501100011033, US-1263032 (US/JUNTA/FEDER, UE),
P20\_01109 (JUNTA/FEDER, UE), and PID2024-157173NB-I00 funded by MCIN/AEI/10.13039/501100011033 and by FEDER, UE.

\section{Cluster tilting subcategories}
\label{sec:CT}

\begin{setting}
  \label{setting:ct_angulated}
  We fix a triangulated category $\T$ with finite-dimensional morphism spaces
  and split idempotents as well as a subcategory $\C\subseteq\T$ that satisfies
  ${\C=\add*{\C}}$.
\end{setting}

\subsection{Functorially finite subcategories}

We begin by recalling a classical definition due to Auslander and Smal{\o}, and
independently to Enochs.

\begin{definition}[{\cite{AS80,AS81} and \cite{Eno81}}]
  The subcategory $\C\subseteq\T$ is \emph{contravariantly finite in $\T$} if
  the following condition is satisfied:
  \begin{enumerate}
  \item For each object $x\in\T$ there exist an object $c\in\C$ and a morphism
    $f\colon c\to x$ with the following property: For each morphism $g\colon
    c'\to x$ with $c'\in\C$, there exists a (not necessarily unique) morphism
    $h\colon c'\to c$ such that $g=f\circ h$.
    \begin{center}
      \begin{tikzcd}
        &c'\dar{\forall g}\dlar[swap,dotted]{\exists h}\\
        c\rar[swap]{f}&x
      \end{tikzcd}
    \end{center}
    Equivalently, the induced linear map
    \[
      \T(c',c)\longrightarrow\T(c',x),\qquad h\longmapsto f\circ h,
    \]
    is surjective (but not necessarily injective). Such a morphism $f$ is called
    a \emph{right $\C$-approximation of $x$}.
  \end{enumerate}
  Dually, the subcategory $\C\subseteq\T$ is \emph{covariantly finite in $\T$}
  if the following condition is satisfied:
  \begin{enumerate}
  \item For each object $x\in\T$ there exist an object $c\in\C$ and a morphism
    $f\colon x\to c$ with the following property: For each morphism $g\colon
    x\to c'$ with $c'\in\C$, there exists a (not necessarily unique) morphism
    $h\colon c\to c'$ such that $g=h\circ f$.
    \begin{center}
      \begin{tikzcd}
        x\rar{f}\drar[swap]{\forall g}&c\dar[dotted]{\exists h}\\ &c'
      \end{tikzcd}
    \end{center}
    Equivalently, the induced linear map
    \[
      \T(c,c')\longrightarrow\T(x,c'),\qquad h\longmapsto h\circ f,
    \]
    is surjective (but not necessarily injective). Such a morphism $f$ is called
    a \emph{left $\C$-approximation of $x$}.
  \end{enumerate}
  Finally, the subcategory $\C\subseteq\T$ is \emph{functorially finite in $\T$}
  if it is both contravariantly finite and covariantly finite in $\T$.
\end{definition}

The following basic example of a functorially finite subcategory is the most
relevant for the purposes of this article.

\begin{example}
  Let $c\in\T$ and $\C=\add*{c}$. Then, the subcategory $\C$ is functorially
  finite in $\T$. Indeed, $\C$-approximations of an object $x\in\T$ are easily
  constructed from a decomposition $c=c_1\oplus\cdots\oplus c_k$ into
  indecomposable objects and bases of the (finite-dimensional) vector spaces
  $\T(c_i,x)$ and $\T(x,c_i)$, $1\leq i\leq k$.
\end{example}

\begin{remark}
  Suppose that $\C$ is contravariantly finite in $\T$. For each object $x\in\T$
  there exists a commutative diagram, constructed inductively, of the form
  \begin{center}
    \begin{tikzcd}[column sep=small]
      \cdots&c_2\ar{rr}\drar&&c_1\drar\ar{rr}&&c_0\drar&\\
      &\cdots&x_2\urar&&x_1\urar\ar{ll}[description]{+1}&&x_0=x\ar{ll}[description]{+1}
    \end{tikzcd}
  \end{center}
  in which each of the morphisms $c_i\to x_i$ is a right $\C$-approximation and,
  as is customary, the triangles
  \begin{center}
    \begin{tikzcd}[column sep=small]
      &c_i\drar\\
      x_{i+1}\urar&&x_i\ar{ll}[description]{+1}
    \end{tikzcd}
  \end{center}
  represent exact triangles
  \[
    x_{i+1}\longrightarrow c_i\longrightarrow x_i\longrightarrow x_{i+1}[1].
  \]
  Dually, if $\C$ is covariantly finite in $\T$, then for each object $x$ of
  $\T$ there exists a commutative diagram of the form
  \begin{center}
    \begin{tikzcd}[column sep=small]
      &c^0\drar\ar{rr}&&c^1\drar\ar{rr}&&c^2&\cdots\\
      x=x^0\urar&&x^1\ar{ll}[description]{+1}\urar&&x^2\ar{ll}[description]{+1}\urar&\cdots
    \end{tikzcd}
  \end{center}
  in which each of the morphisms $x^i\to c^i$ is a left $\C$-approximation and
  with corresponding exact triangles
  \[
    x^i\longrightarrow c^i\longrightarrow x^{i+1}\longrightarrow x^i[1].
  \]
  From the point of view of this article, the existence of such
  `$\C$-(co)resolutions' is one of the main motivations for considering
  functorially finite subcategories.
\end{remark}

\subsection{The Verdier product}

The following elementary operation on subcategories of a triangulated category
is needed in the proof of \Cref{thm:dZ-ct_characterisation} below. 

\begin{definition}
  Let $\X,\Y\subseteq\T$ be additive subcategories closed under direct summands. The \emph{Verdier product of $\X$
    with $\Y$} is the full additive subcategory $\X*\Y$ of $\T$ spanned by the objects
  $t\in\T$ such that there exists an exact triangle
  \[
    x\longrightarrow t\longrightarrow y\longrightarrow x[1]
  \]
  with $x\in\X$ and $y\in\Y$. \Cref{lemma:Verdier_product} \eqref{it:verdier_product_summands} gives a sufficient condition for $\X*\Y$ to be closed under direct summands.
\end{definition}

\begin{notation}
  Given additive subcategories $\X,\Y\subseteq\T$, we let $\X\vee\Y$ be the smallest
  additive subcategory of $\T$ containing $\X$ and $\Y$. Obviously,
  $\X\vee\Y\subseteq\X*\Y$ and $\X\vee\Y=\Y\vee\X$.
\end{notation}

\begin{lemma}
  \label{lemma:Verdier_product}
  Let $\X,\Y,\Z\subseteq\T$ be additive subcategories cloised under direct summands. The following statements hold:
  \begin{enumerate}
  \item\label{it:verdier_product_assoc} The Verdier product is associative:
    $(\X*\Y)*\Z=\X*(\Y*\Z)$.
  \item\label{it:verdier_product_split} Suppose that
    \[
      \forall x\in\X,\ \forall y\in\Y,\qquad \T(y,x[1])=0.
    \]
    Then,
    $\X*\Y=\X\vee\Y\subseteq\Y*\X$. 
    \item\label{it:verdier_product_summands} If
    \[
      \forall x\in\X,\ \forall y\in\Y,\qquad \T(x,y)=0,
    \]
    then $\X*\Y$ is closed under direct summands.
  \end{enumerate}
\end{lemma}
\begin{proof}
  \eqref{it:verdier_product_assoc} This is an immediate consequence of the
  octahedral axiom.

  \eqref{it:verdier_product_split} The equality $\X*\Y=\X\vee\Y$ follows from
  the fact that every triangle
  \[
    x\to z\to y\to x[1]
  \]
  with $x\in\X$ and $y\in\Y$ splits, for the morphism $y\to x[1]$ vanishes by
  assumption. The inclusion $\X\vee\Y=\Y\vee\X\subseteq\Y*\X$ is obvious.

  \eqref{it:verdier_product_summands} See~\cite[Prop.~2.1(1)]{IY08}.
\end{proof}

\begin{example}\label{ex:nested-triangles}
  Let $\C\subseteq\T$ be an additive subcategory closed under direct summands. The reader can easily verify, using the
  rotation axiom, that the Verdier product $\C*\C[1]*\cdots*\C[d-1]$ coincides
  with the full subcategory of $\T$ spanned by those objects $x\in\T$ for which
  there exists a commutative diagram of the form
  \begin{center}
    \begin{tikzcd}[column sep=small]
      &c_{d-2}\drar&\cdots&&&c_1\drar\ar{rr}&&c_0\drar[end anchor=north west]\\
      c_{d-1}\urar&&x_{d-2}\ar{ll}[description]{+1}&\cdots&x_2\urar&&x_1\urar\ar{ll}[description]{+1}&&x_0=x\ar{ll}[description]{+1}
    \end{tikzcd}
  \end{center}
  in which $c_i\in\C$, $0\leq i<d$. In other words, $\C*\C[1]*\cdots*\C[d-1]$ is
  the full subcategory of $\T$ spanned by those objects that admit a
  $\C$-resolution of length (at most) $d$.
\end{example}

\subsection{Cluster tilting subcategories}

The following remarkable class of subcategories of a triangulated category,
introduced by Iyama and Yoshino (see also~\cite{Iya07a,Iya07,IJ17} for the case
of abelian categories), is the main focus of this article.

\begin{definition}[{\cite[Sec.~3]{IY08}}]
  \label{def:d-CT_dZ-CT}
  The subcategory $\C$ is \emph{$d$-cluster tilting} in $\T$ if the following
  three conditions are satisfied:
  \begin{enumerate}
  \item The subcategory $\C$ is functorially finite in $\T$.
  \item An object $x\in\T$ lies in $\C$ if and only if
    \[
      \forall c\in\C,\ \forall 0<i<d,\qquad \T(x,c[i])=0.
    \]
  \item An object $x\in\T$ lies in $\C$ if and only if
    \[
      \forall c\in\C,\ \forall 0<i<d,\qquad \T(c,x[i])=0.
    \]
  \end{enumerate}
  The subcategory $\C$ is \emph{$d\ZZ$-cluster tilting in $\T$} if, in addition
  to the three conditions above, $\C$ satisfies the following condition:
  \begin{enumerate}
    \setcounter{enumi}{3}
  \item The equality $\C=\C[d]$ holds.
  \end{enumerate}
  We say that an object $c\in\T$ is \emph{$d$- or $d\ZZ$-cluster-tilting in
    $\T$}, if the subcategory $\C=\add*{c}$ is $d$- or $d\ZZ$-cluster tilting in
  $\T$, respectively.
\end{definition}

\begin{remark}
  Let $d=1$. By vacuity, a subcategory $\C\subseteq\T$ is $1$-cluster tilting if
  and only if $\C=\T$. Consequently, the notions of $1$-cluster tilting and
  $1\ZZ$-cluster tilting coincide. In particular, $c\in\T$ is a $1$-cluster
  tilting object (=$1\ZZ$-cluster tilting object) if and only if $\add*{c}=\T$,
  that is $c$ generates $\T$ under finite direct sums and direct summands.
\end{remark}

\begin{definition}
  A subcategory $\C\subseteq\T$ is \emph{$d$-rigid} if
  \[
    \forall c,c'\in\C,\ \forall 0<i<d,\qquad \T(c,c'[i])=0.
  \]
  We say that $\C$ is \emph{$d\ZZ$-rigid} if, moreover, for each $c,c'\in\C$ and
  each $i\in\ZZ$,
  \[
    i\not\in d\ZZ\qquad\Longrightarrow\qquad\T(c,c'[i])=0.
  \]
\end{definition}

\begin{remark}
  By definition, $d$-cluster tilting subcategories are $d$-rigid.
\end{remark}

\begin{remark}
  \label{rmk:dZ}
  Let $\C\subseteq\T$ be a $d$-cluster tilting subcategory. It follows
  immediately from the definition of $d$-cluster tilting subcategory that the
  following conditions are equivalent:
  \begin{enumerate}
  \item The equality $\C=\C[d]$ holds.
  \item The subcategory $\C$ is $d\ZZ$-rigid.
  \end{enumerate}
  In other words, the $d\ZZ$-cluster tilting subcategories are precisely the
  $d$-cluster tilting subcategories that are $d\ZZ$-rigid.
\end{remark}

\begin{lemma}
  \label{lemma:commutativity_lemma}
  Let $\C\subseteq\T$ be a $d$-rigid subcategory. For all $0\leq j<i<d$ such
  that $(j,i)\neq(0,d-1)$, the inclusion
  \[
    \C[i]*\C[j]\subseteq\C[j]*\C[i]
  \]
  holds.
\end{lemma}
\begin{proof}
  Given that $\C\subseteq\T$ is a $d$-rigid subcategory, the assumption on $i$
  and $j$ implies that
  \[
    \forall c,c'\in\C,\qquad\T(c,c'[i-j+1])=0.
  \]
  Thus, by \Cref{lemma:Verdier_product}\eqref{it:verdier_product_split} the
  inclusion $\C[i-j]*\C\subseteq\C*\C[i-j]$ holds. Since
  \[
    \C[i]*\C[j]\subseteq\C[j]*\C[i]\qquad\Longleftrightarrow\qquad\C[i-j]*\C\subseteq\C*\C[i-j],
  \]
  the claim follows.
\end{proof}

The class of $d$-cluster tilting subcategories admits the following useful
characterisation, due to Iyama and Yoshino
(\eqref{it:IY-Bel-d-CT-a}$\Rightarrow$\eqref{it:IY-Bel-d-CT-b}) and to
Beligiannis (\eqref{it:IY-Bel-d-CT-b}$\Rightarrow$\eqref{it:IY-Bel-d-CT-a}).

\begin{theorem}[{\cite[Thm.~3.1]{IY08},~\cite[Thm.~5.3]{Bel15}}]
  \label{thm:IY-Bel-d-CT}
  Let $\C\subseteq\T$ be a subcategory. The following
  statements are equivalent:
  \begin{enumerate}
  \item\label{it:IY-Bel-d-CT-a} The subcategory $\C$ is $d$-cluster tilting in
    $\T$.
  \item\label{it:IY-Bel-d-CT-b} The following conditions are satisfied:
    \begin{enumerate}
    \item The subcategory $\C$ is $d$-rigid.
    \item The equality $\T=\C*\C[1]*\cdots*\C[d-1]$ holds.
    \end{enumerate}
  \end{enumerate}
\end{theorem}
\begin{proof}
  The implication \eqref{it:IY-Bel-d-CT-a}$\Rightarrow$\eqref{it:IY-Bel-d-CT-b}
  is proven in~\cite[Thm.~3.1]{IY08} while the converse implication
  \eqref{it:IY-Bel-d-CT-b}$\Rightarrow$\eqref{it:IY-Bel-d-CT-a} is part
  of~\cite[Thm.~5.3]{Bel15}, which contains many more equivalent conditions and
  requires the subcategory $\C\subseteq\T$ to be contravariantly finite. We
  prove that the latter property in fact follows from the assumptions in part
  \eqref{it:IY-Bel-d-CT-b} of the statement theorem. The argument is well known
  to experts (compare with~\cite[Lemma~4.3]{GKO13}); we include it for the sake
  of completeness.

  Suppose that the conditions in part \eqref{it:IY-Bel-d-CT-b} of the statement
  of the theorem are satisfied. If $d=1$, then $\C=\T$ and there is nothing to
  prove; thus, we may assume that $d\geq2$. By \Cref{ex:nested-triangles}, for every object
  $x\in\T$ there exists a commutative diagram of the form
  \begin{center}
    \begin{tikzcd}[column sep=small]
      &c_{d-2}\drar&\cdots&&&c_1\drar\ar{rr}&&c_0\drar[end anchor=north west]\\
      c_{d-1}\urar&&x_{d-2}\ar{ll}[description]{+1}&\cdots&x_2\urar&&x_1\urar\ar{ll}[description]{+1}&&x=x_0\ar{ll}[description]{+1}
    \end{tikzcd}
  \end{center}
  in which $c_i\in\C$, $0\leq i<d$. We claim that
  \[
    \forall c\in\C,\ \forall 0<i\leq j\leq d-2,\qquad \T(c,x_j[i])=0.
  \]
  We prove the claim by reverse induction on $j$. If $j=d-2$, then for each
  $0<i\leq d-2$ and each $c\in\C$ there is an induced exact sequence
  \[
    0=\T(c,c_{d-2}[i])\longrightarrow\T(c,x_{d-2}[i])\longrightarrow
    \T(c,c_{d-1}[i+1])=0
  \]
  in which the outer terms vanish since the subcategory $\C$ is $d$-rigid; this
  proves the claim in this case. Suppose now that the claim holds for some
  $1<j\leq d-2$. Then, for each $0<i\leq j-1$ and each $c\in\C$ there is an
  induced exact sequence
  \[
    0=\T(c,c_{j-1}[i])\longrightarrow\T(c,x_{j-1}[i])\longrightarrow\T(c,x_j[i+1])=0
  \]
  in which the right-most term vanishes by the inductive hypothesis and the
  left-most terms since the subcategory $\C$ is $d$-rigid; this finishes the
  induction step. We claim now that the morphism $c_0\to x_0$ is in fact a right
  $\C$-approximation of $x=x_0$. Indeed, there is an induced exact sequence
  \[
    \T(c,c_0)\longrightarrow\T(c,x_0)\longrightarrow\T(c,x_1[1])=0
  \]
  in which the term on the right vanishes, as shown above. This proves the
  claim. Since the object $x\in\T$ was chosen arbitrarily, this shows that $\C$
  is contravariantly finite in $\T$, which is what we needed to prove.
\end{proof}

\begin{remark}
  Let $\C\subseteq\T$ be a $d$-cluster tilting subcategory. Then, $\C$ generates
  $\T$ as a triangulated category: indeed, $\T=\C*\C[1]*\cdots*\C[d-1]$ and the
  right-hand side is of course contained in the triangulated closure of $\C$.
  Thus, in particular, $\C$ is a $d$-step generator in the sense
  of~\cite{Rou08}. However, it is remarkable that every object in $\T$ admits a
  $\C$-resolution of length at most $d$ (not involving any shifts).
\end{remark}

The following characterisation of $d\ZZ$-cluster tilting subcategories is quite
natural from the point of view of $(d+2)$-angulated categories, whose definition
we recall in \Cref{def:d+2-angulated_cat}.

\begin{theorem}
  \label{thm:dZ-ct_characterisation}
  Let $\C\subseteq\T$ be a subcategory of $\T$ that satisfies $\C=\add*{\C}$.
  The following statements are equivalent:
  \begin{enumerate}
  \item\label{it:dZ-CT} The subcategory $\C$ is $d\ZZ$-cluster tilting in $\T$.
  \item\label{it:weak_ang_subcat} The following conditions are satisfied:
    \begin{enumerate}
    \item\label{itit:d-rigid} The subcategory $\C$ is $d$-rigid.
    \item\label{itit:d-stable} The equality $\C=\C[d]$ holds.
    \item\label{itit:new_inclusion} The inclusion
      $(\C[-1]*\C)\subseteq(\C*\C[1]*\cdots*\C[d-1])$ holds.
    \item\label{itit:thick} The subcategory $\C$ satisfies $\thick*{\C}=\T$.
    \end{enumerate}
  \end{enumerate}
\end{theorem}
\begin{proof}
  The implication \eqref{it:dZ-CT}$\Rightarrow$\eqref{it:weak_ang_subcat}
  follows immediately from the definition of $d\ZZ$-cluster tilting subcategory
  and \Cref{thm:IY-Bel-d-CT}.

  \eqref{it:weak_ang_subcat}$\Rightarrow$\eqref{it:dZ-CT} Notice that the case
  $d=1$ is trivial, for assumptions \eqref{itit:d-stable} and
  \eqref{itit:new_inclusion} imply that $\C$ is closed under (co)cones and, in
  view of assumption \eqref{itit:thick}, we must have $\C=\add*{\C}=\T$
  (assumption \eqref{itit:d-rigid} is vacuous in this case).

  Suppose now that $d\geq2$. Since $\C=\C[d]$ by assumption, we only need to
  prove that $\C$ is a $d$-cluster tilting subcategory of $\T$; for this, we
  shall verify the remaining condition in
  \Cref{thm:IY-Bel-d-CT}\eqref{it:IY-Bel-d-CT-b}, namely that
  \[
    \T=\C*\C[1]*\cdots*\C[d-1].
  \]
  Let $\thick*{\C}$ be the smallest thick subcategory of $\T$ containing $\C$.
  By~\cite[Lemma~2.15(a)]{AI12}
  \[
    \T=\thick*{\C}=\bigcup_{\ell>0,\
      n_1,\dots,n_\ell\in\ZZ}\smd*{\C[n_1]*\cdots*\C[n_\ell]},
  \]
  where $\X\mapsto\smd*{\X}$ denotes the passage to the closure under direct
  summands (recall that $\thick*{\C}=\T$ by assumption \eqref{itit:thick}). We
  make the following observations:
  \begin{itemize}
  \item Since $\C=\C[d]$, for each $n\in\ZZ$ there exists an integer $0\leq i<d$
    such that $\C[n]=\C[i]$. Consequently,
    \[
      \T=\thick*{\C}=\bigcup_{\ell>0,\ 0\leq
        n_1,\dots,n_\ell<d}\smd*{\C[n_1]*\cdots*\C[n_\ell]}.
    \]
  \item Since $\C$ is a $d$-rigid subcategory of $\T$,
    \Cref{lemma:Verdier_product}\eqref{it:verdier_product_summands} readily
    implies that
    \[
      \C*\C[1]*\cdots*\C[d-1]=\smd*{\C*\C[1]*\cdots*\C[d-1]}.
    \]
  \end{itemize}
  Thus, it is enough to prove that,
  given $\ell>0$ and $0\leq n_1,\dots,n_\ell<d$, the inclusion
  \[
    \C[n_1]*\cdots*\C[n_\ell]\subseteq\C*\C[1]*\cdots*\C[d-1]
  \]
  holds. We prove this claim by induction on $\ell$. For $\ell=1$ the claim is obvious,
  so suppose that the claim holds for some $\ell>0$. Let $0\leq
  n_1,\dots,n_\ell,n<d$ and notice that, by the inductive hypothesis, the
  inclusion
  \[
    (\C[n_1]*\cdots*\C[n_\ell])*\C[n]\subseteq(\C*\C[1]*\cdots\C[d-1])*\C[n].
  \]
  holds. We distinguish two cases:

  \begin{description}
  \item[$n=0$] Since $\C[d-1]=(\C[d])[-1]=\C[-1]$, the equality
    \[
      \C[d-1]*\C[n]=\C[-1]*\C
    \]
    holds. Moreover, by assumption,
    \[
      (\C[-1]*\C)\subseteq(\C*\C[1]*\cdots*\C[d-1]).
    \]
    Consequently,
    \begin{align*}
      (\C[n_1]*\cdots*\C[n_\ell])*\C[n] & \subseteq(\C*\C[1]*\cdots\C[d-1])*\C[n]               \\
                                        & =(\C*\cdots*\C[d-2])*(\C[d-1]*\C[n])                  \\
                                        & =(\C*\cdots*\C[d-2])*(\C[-1]*\C)                      \\
                                        & \subseteq(\C*\cdots\C[d-2])*(\C*\C[1]*\cdots*\C[d-1]) \\
                                        & \subseteq(\C*\C[1]*\cdots*\C[d-2])*\C[d-1],
    \end{align*}
    where the last inclusion is obtained from the inclusions
    \[
      \C[i]*\C[j]\subseteq\C[j]*\C[i]
    \]
    for $0\leq j<i<d$ with $(j,i)\neq(0,d-1)$ (see
    \Cref{lemma:commutativity_lemma}) and the equality
    \[
      \C[i]*\C[i]=\C[i]\vee\C[i]=\C[i],
    \]
    (see \Cref{lemma:Verdier_product}\eqref{it:verdier_product_split}) which
    stems from the fact that there are no non-trivial self-extensions between
    the objects of $\C[i]$ (recall that $d\geq2$ by assumption). This proves the
    claim in this case.
  \item[$n\neq0$] A simpler version of the argument used in the case $n=0$ shows
    that there are inclusions
    \begin{align*}
      \C[n_1]*\cdots*\C[n_\ell]*\C[n] & \subseteq(\C*\C[1]*\cdots*\C[d-1])*\C[n] \\
                                      & \subseteq\C*\C[1]*\cdots*\C[d-1],
    \end{align*}
    where in the last inclusion we use \Cref{lemma:commutativity_lemma} and
    \Cref{lemma:Verdier_product}\eqref{it:verdier_product_split} and the
    assumption that $n\neq0$.
  \end{description}
  In view of \Cref{thm:IY-Bel-d-CT}, this finishes the proof of the theorem.
\end{proof}

\begin{remark}
  \label{rmk:T_thick}
  In the context of \Cref{thm:dZ-ct_characterisation}, condition
  \eqref{itit:thick} is not essential: We may replace $\T$ by its full
  subcategory $\thick*\C$, noticing that conditions
  \eqref{itit:d-rigid}--\eqref{itit:new_inclusion} only depend on the latter
  subcategory.
\end{remark}

\section{$(d+2)$-angulated categories}
\label{sec:d+2}

\subsection{Definition and basic properties}

Although we do not use most axioms in the definition (at least not explicitly),
we include the complete definition of a $(d+2)$-angulated category since it
plays an important conceptual role in our main results. Recall that $d\geq1$ is
a fixed integer.

\begin{definition}[{\cite[Def.~2.1]{GKO13}}]
  \label{def:d+2-angulated_cat}
  Let $(\F,\Sigma)$ be a pair consisting of an additive category $\F$ and an
  automorphism\footnote{As in the case of triangulated categories, one may
    assume instead that $\Sigma$ is merely an auto-\emph{equivalence}, but every
    autoequivalence can be strictified by replacing $\F$ by an equivalent
    category, see for example~\cite[Sec.~2]{KV87}. We mostly ignore this
    distinction as it is not crucial to our results.} $\Sigma\colon\F\simto\F$.
  A \emph{$(d+2)$-angulation} of $(\F,\Sigma)$ is a class $\pentagon$ of
  sequences
  \[
    x_1\to x_2\to\cdots\to x_{d+2}\to\Sigma(x_1)
  \]
  that satisfy the following axioms:
  \begin{enumerate}[label={($d$-TR\arabic*)}, ref={$d$-TR\arabic*},
    leftmargin=*]
  \item\label{dTR1} \begin{enumerate}
    \item The class $\pentagon$ is closed under direct sums and direct summands.
    \item\label{dTR1b} For each object $x\in\F$, the sequence
      \[
        x\xrightarrow{1}x\to0\to\cdots\to0\to\Sigma(x)
      \]
      ($d$ consecutive zeroes) lies in $\pentagon$.
    \item Every morphism $x\to y$ in $\F$ is the \emph{leftmost} morphism in a
      sequence that lies in $\pentagon$.
    \end{enumerate}
  \item\label{dTR2} A sequence
    \[
      x_1\xrightarrow{f} x_2\to\cdots x_{d+1}\to\Sigma(x_1)
    \]
    lies in $\pentagon$ if and only if its \emph{left rotation}
    \[
      x_2\to x_3\to\cdots\to
      x_{d+2}\to\Sigma(x_1)\xrightarrow{(-1)^d\Sigma(f)}\Sigma(x_2)
    \]
    lies in $\pentagon$ (notice that the sign in the rotation depends on $d$).
  \item\label{dTR3} Given a solid diagram in $\F$ of the form
    \begin{center}
      \begin{tikzcd}[column sep=small,row sep=small]
        x_{1}\dar{u}\rar&x_2\rar\dar&x_{3}\rar\dar[dotted]&\cdots\rar&x_{d+1}\rar\dar[dotted]&x_{d+2}\rar\dar[dotted]&\Sigma(x_{1})\dar{\Sigma(u)}\\
        y_{1}\rar&y_2\rar&y_{3}\rar&\cdots\rar&y_{d+1}\rar&y_{d+2}\rar&\Sigma(y_{1}),
      \end{tikzcd}
    \end{center}
    in which the leftmost square commutes and both rows lie in $\pentagon$,
    there exist morphisms $x_i\to y_i$, $3\leq i\leq d+2$, rendering the diagram
    commutative.
  \item\label{dTR4} In the situation of \eqref{dTR3}, the morphisms $x_i\to
    y_i$, $3\leq i\leq d+2$, can be chosen such that the usual mapping cone
    (see~\cite[2.1]{GKO13} for the precise form of the morphisms)
    \begin{equation*}
      x_2\oplus y_1 \to{x_3\oplus y_2}\to\cdots \to
      {\Sigma x_{1}\oplus y_{d+2}}\to \Sigma(x_2\oplus y_1)
    \end{equation*}
    lies in $\pentagon$.
  \end{enumerate}
  In this case, the triple $(\F,\Sigma,\pentagon)$ is called a
  \emph{$(d+2)$-angulated category}. Similarly, we say that $\pentagon$ is a
  \emph{pre-$(d+2)$-angulation} if it satisfies axioms
  \eqref{dTR1}--\eqref{dTR3}, in which case the triple $(\F,\Sigma,\pentagon)$
  is a \emph{pre-$(d+2)$-angulated category}.
\end{definition}

\Cref{def:d+2-angulated_cat} \eqref{dTR4} is a higher analogue of Neeman's mapping cone axiom for
triangulated categories~\cite[Def.~1.3.13]{Nee01} which, together with axioms
\eqref{dTR1}--\eqref{dTR3} for $d=3$, is equivalent to the octahedral axiom. In particular, a $3$-angulated category is nothing else but a triangulated category. Bergh
and Thaule~\cite{BT13} have given a suitable analogue of octahedral
axiom in the context of $(d+2)$-angulated categories \cite[Thm.~4.2]{BT13}.

\begin{notation}
  Let $(\F,\Sigma,\pentagon)$ be a (pre-)$(d+2)$-angulated category. If there is
  no risk of confusion, we often abuse the notation and identify the triple
  $(\F,\Sigma,\pentagon)$ with its underlying additive category $\F$.
\end{notation}

\begin{remark}
  \label{rmk:several_angulations}
  Given a $(d+2)$-angulated category $(\F,\Sigma,\pentagon)$ and a unit
  $u\in\kk^\times$ in the ground field, we can form a new $(d+2)$-angulation
  $u\cdot\pentagon$ of $(\F,\Sigma)$ as follows:
  \[
    x_1\xrightarrow{f} x_2\to\cdots x_{d+1}\to\Sigma(x_1)
  \]
  belongs to $\pentagon$ if and only if
  \[
    x_1\xrightarrow{u^{-1}\cdot f} x_2\to\cdots x_{d+1}\to\Sigma(x_1)
  \]
  belongs to $u\cdot\pentagon$ (see
  \cite{balmer_2002_triangulated_categories_several} for the case $d=1$ of
  ordinary triangulations).
\end{remark}

For the purposes of this article, it is enough to consider only
$(d+2)$-angulated categories whose underlying additive category has split
idempotents. This is not a serious restriction, since $(d+2)$-angulations can
be extended to idempotent completions in an essentially unique way, see
~\cite{BS01} for the case $d=1$ of triangulated categories and~\cite{Lin21}
for the general case.

The following result is a $(d+2)$-angulated analogue of a classical result of
Freyd~\cite{Fre66} and of Heller's parametrisation of the pre-triangulations on
a pre-triangulated category~\cite{Hel68}.

\begin{proposition}[{\cite[Prop.~2.5 and Prop.~3.4]{GKO13}}]
  \label{prop:GKO-Freyd-Heller}
  Let $(\F,\Sigma)$ be a pair consisting of an additive category and an
  automorphism $\Sigma\colon\F\simto\F$. Suppose that $(\F,\Sigma)$ admits a
  pre-$(d+2)$-angulation. Then, the following statements hold:
  \begin{enumerate}
  \item\label{it:modF-Frobenius} The category $\mmod*\F$ of finitely-presented
    functors $\F^\op\to\Mod\kk$ to vector spaces is a Frobenius abelian
    category. Moreover, if idempotents split in $\F$, then the Yoneda embedding
    $x\mapsto\F(-,x)$ induces an equivalence between $\F$ and the full
    subcategory of $\mmod*\F$ spanned by the projective-injective objects.
  \item The automorphism of $\Sigma\colon\mmod*\F\simto\mmod*\F$ induced by
    $\Sigma\colon\F\simto\F$ is an exact functor. In particular, there is an
    induced natural isomorphism
    \[
      \sigma\colon\Sigma\Omega_\F^{-1}\stackrel{\sim}{\Longrightarrow}\Omega_\F^{-1}\Sigma,
    \]
    where $\Omega_\F^{-1}\colon\smmod*{\F}\simto\smmod*\F$ is a choice of
    cosyzygy functor on the stable category, that makes the pair
    $(\Sigma,\sigma)$ into an exact functor on $\smmod*\F$.
  \item\label{it:Sigma-d+2-syzygy} There is a (canonical) bijective map from the
    class of pre-$(d+2)$-angulations on $(\F,\Sigma)$ to the class of
    isomorphisms of exact functors
    \[
      (\Sigma,\sigma)\stackrel{\sim}{\Longrightarrow}(\Omega_\F^{-(d+2)},(-1)^d\id[\Omega_\F^{-(d+2)-1}])
    \]
    on $\smmod*\F$. In particular, if the pair $(\F,\Sigma)$ admits a
    pre-$(d+2)$-angulation, then there are natural isomorphisms
    \[
      \Sigma\cong\Omega_\F^{-(d+2)}\qquad\text{and}\qquad\Sigma^{-1}\cong\Omega_\F^{d+2}
    \]
    of exact functors on $\smmod*{\F}$. \end {enumerate}
\end{proposition}

We record the apparent notion of equivalence between (pre-)$(d+2)$-angulated
categories for later use. For this, we also introduce the following auxiliary
definition.

\begin{definition}
  \label{def:functor_of_pairs}
  Let $(\C,\Phi_\C)$ and $(\D,\Phi_\D)$ be pairs consisting of categories $\C$
  and $\D$ and automorphisms $\Phi_\C\colon\C\simto\C$ and
  $\Phi_\D\colon\D\simto\D$. A \emph{functor of pairs
    $(\C,\Phi_\C)\to(\D,\Phi_\D)$} is a pair $(F,F^\sharp)$ consisting of a
  functor $F\colon\C\to\D$ and a natural isomorphism $F^\sharp\colon
  F\Phi_\C\Rightarrow \Phi_\D F$. A functor of pairs is an \emph{equivalence} if
  its underlying functor is an equivalence of categories.
\end{definition}

\begin{remark}
  Observe that the collection of pairs (as in \Cref{def:functor_of_pairs}) form
  a $2$-category in a natural way. In particular, the notion of equivalence of
  pairs is simply the corresponding notion of equivalence in this $2$-category.
\end{remark}

\begin{notation}
  In the context of \Cref{def:functor_of_pairs}, we always identify a functor of
  pairs with the functor between the corresponding underlying categories.
\end{notation}

\begin{definition}
  \label{def:d+2-angulated_cat-equivalence}
  Let $(\F,\Sigma,\pentagon)$ and $(\F',\Sigma',\pentagon')$ be two
  pre-$(d+2)$-angulated categories. A \emph{pre-$(d+2)$-angulated equivalence}
  between $(\F,\Sigma,\pentagon)$ and $(\F',\Sigma',\pentagon')$ is an equivalence of
  pairs
  \[
    F\colon(\F,\Sigma)\stackrel{\sim}{\longrightarrow}(\F',\Sigma')
  \]
  such that, for each
  $(d+2)$-angle
  \[
    x_1\to x_2\to\cdots\to x_{d+2}\xrightarrow{\eta}\Sigma(x_1)
  \]
  in $\pentagon$, the sequence
  \[
    F(x_1)\to F(x_2)\to\cdots\to F(x_{d+2})\xrightarrow{\eta'}\Sigma(F(x_1))
  \]
  lies in $\pentagon'$, where $\eta'=F^\sharp_{x_1}F(\eta)$. If $\pentagon$ and
  $\pentagon'$ are in fact $(d+2)$-angulations, we say that $F$ is a
  \emph{$(d+2)$-angulated equivalence} or an \emph{equivalence of
    $(d+2)$-angulated categories}.
\end{definition}

\subsection{Constructions of $(d+2)$-angulated categories}

We recall the two constructions of $(d+2)$-angulated categories that are most
relevant for the purposes of this article.

\subsubsection{Standard $(d+2)$-angulated categories}

The following theorem of Geiss, Keller and Oppermann is one of the main sources
of $(d+2)$-angulated categories; in fact, this theorem was the main motivation
for introducing the notion of $(d+2)$-angulated category, and it also motivated
the definition of a $d\ZZ$-cluster tilting subcategory (\Cref{def:d-CT_dZ-CT}).

\begin{theorem}[{\cite[Thm.~1]{GKO13}}]
  \label{thm:GKO-standard}
  Let $\C\subseteq\T$ be a $d\ZZ$-cluster tilting subcategory of $\T$. Then, the
  pair $(\C,[d])$ is endowed with a $(d+2)$-angulation whose $(d+2)$-angles are
  those sequences
  \[
    c_{d+2}\to c_{d+1}\to\cdots\to c_2\to c_1\to c_{d+2}[d]
  \]
  in $\C$ which fit as the spine of a commutative diagram of the form
  \begin{center}
    \begin{tikzcd}[column sep=small]
      &c_{d+1}\drar&\cdots&&&c_3\drar\ar{rr}&&c_2\drar\\
      c_{d+2}\urar&&x_{d.5}\ar{ll}[description]{+1}&\cdots&x_{3.5}\urar&&x_{2.5}\urar\ar{ll}[description]{+1}&&c_1\ar{ll}[description]{+1}
    \end{tikzcd}
  \end{center}
  in which the oriented triangles are exact triangles in $\T$ and such that the
  connecting morphism $c_1\to c_{d+2}[d]$ is given by the obvious (shifted)
  composite along the bottom row of the diagram.
\end{theorem}

\begin{definition}
  \label{def:stanard_angulated_cat}
  We call a $(d+2)$-angulated category $(\F,\Sigma,\pentagon_\F)$
  \emph{standard} if there exists a fully faithful functor of pairs
  $\iota\colon(\F,\Sigma)\hookrightarrow(\T,[d])$ with the following properties:
  \begin{enumerate}
  \item The category $\T$ is triangulated with suspension functor $[1]$.
  \item The essential image $\iota(\F)$ of $\iota$ is a $d$-rigid subcategory of
    $\T$.
  \item The image of the $(d+2)$-angulation of $\F$ under $\iota$ coincides with
    the class of sequences considered in \Cref{thm:GKO-standard}. More
    precisely, the closure under isomorphisms of the class of sequences
    \[
      \iota(x_1)\to \iota(x_2)\to\cdots\to \iota(x_{d+2})\xrightarrow{\iota_{x_1}^\sharp
        \iota(\eta)}\Sigma(\iota(x_1))
    \]
    in $\iota(\F)$ that are the image of a $(d+2)$-angle
    \[
      x_1\to x_2\to\cdots\to x_{d+2}\xrightarrow{\eta}\Sigma(x_1)
    \]
    that lies in the class $\pentagon_\F$ coincides with the class of sequences
    considered in \Cref{thm:GKO-standard}.
  \end{enumerate}
  In this case, we say that $\iota(\F)\subseteq\T$ is a \emph{standard
    $(d+2)$-angulated subcategory} and call the class $\pentagon_\F$ a
  \emph{standard $(d+2)$-angulation}.
\end{definition}

\begin{remark}
  \label{rem:dZ-rigid}
  Let $(\F,\Sigma,\pentagon)$ be a standard $(d+2)$-angulated category with
  respect to a fully faithful functor $\iota\colon\F\hookrightarrow\T$ into a
  triangulated category $\T$. Since $\iota$ is a functor of pairs
  $\iota(\F)[d]=\iota(\F)$ and since moreover $\iota(\F)$ is $d$-rigid then
  $\iota(\F)$ is actually a $d\ZZ$-rigid subcategory of $\T$.
\end{remark}

\Cref{thm:dZ-ct_characterisation,thm:GKO-standard} combine into the following
result.

\begin{theorem}
  \label{thm:standard_d+2-angulated_characterisation}
  Suppose that $\T$ has finite-dimensional morphism spaces and split
  idempotents. Let $\C\subseteq\T$ be a subcategory that satisfies
  $\add*{\C}=\C$ and $\iota\colon\C\hookrightarrow\T$ the inclusion functor.
  Then, the following statements are equivalent:
  \begin{enumerate}
  \item\label{it:dZ-CT-standard} The subcategory $\C$ is a $d\ZZ$-cluster
    tilting subcategory of $\T$.
  \item\label{it:d+2-ang-standard} $\thick*{\C}=\T$ and the functor $\iota$
    exhibits $\C$ as a standard $(d+2)$-angulated category; in other words, the
    pair $(\C,[d])$ admits a $(d+2)$-angulation whose $(d+2)$-angles are given
    as in \Cref{thm:GKO-standard}.
  \end{enumerate}
\end{theorem}
\begin{proof}
  The implication
  \eqref{it:dZ-CT-standard}$\Rightarrow$\eqref{it:d+2-ang-standard} is precisely
  \Cref{thm:GKO-standard}. We prove the converse implication
  \eqref{it:d+2-ang-standard}$\Rightarrow$\eqref{it:dZ-CT-standard}. By
  \Cref{rem:dZ-rigid}, $\C$ is a $d\ZZ$-rigid subcategory with $\C=\C[d]$, and by assumption $\thick*{\C}=\T$. In view of \Cref{thm:dZ-ct_characterisation}, it is enough to
  prove that the inclusion
  \[
    (\C[-1]*\C)\subseteq(\C*\C[1]*\cdots*\C[d-1])
  \]
  holds. Indeed, let $x\in\C[-1]*\C$ and choose an exact triangle in $\T$ of the
  form
  \[
    x\to c_2\xrightarrow{f} c_1\to x[1]
  \]
  with $c_1,c_2\in\C$. By axioms \eqref{dTR1} and \eqref{dTR2}, there exists a
  $(d+2)$-angle in $\C$ of the form
  \[
    c_{d+2}\to c_{d+1}\to\cdots\to c_2\xrightarrow{f}c_1\to c_{d+2}[d].
  \]
  Thus, since $\iota$ exhibits $\C$ as a standard $(d+2)$-angulated category,
  there must exist a commutative diagram of the form
  \begin{center}
    \begin{tikzcd}[column sep=small]
      &c_{d+1}\drar&\cdots&&&c_3\drar\ar{rr}&&c_2\drar\\
      c_{d+2}\urar&&x_{d.5}\ar{ll}[description]{+1}&\cdots&x_{3.5}\urar&&x_{2.5}\urar\ar{ll}[description]{+1}&&c_1\ar{ll}[description]{+1}
    \end{tikzcd}
  \end{center}
  where $x=x_{2.5}$. This shows that $x\in\C*\C[1]*\cdots*\C[d-1]$, as required,
  see \Cref{ex:nested-triangles}. This finishes the proof of the theorem.
\end{proof}

\begin{remark}
  In the context of \Cref{thm:standard_d+2-angulated_characterisation}, the
  condition $$\thick*{\C}=\T$$ is not essential: We may replace $\T$ by its full
  subcategory $\thick*\C$, noticing that statement \eqref{it:d+2-ang-standard}
  only depends on the latter subcategory (compare with \Cref{rmk:T_thick}).
\end{remark}

\subsubsection{Amiot--Lin $(d+2)$-angulations}
\label{subsec:AL_angulations}

In addition to \Cref{thm:GKO-standard}, an important source of $(d+2)$-angulated
categories is the following `higher-dimensional' version of a theorem of Amiot,
who proved the result in the case $d=1$. We need some preliminaries to state the
result.

Given an algebra $\Lambda$, we let $\Lambda^e=\Lambda\otimes\Lambda^\op$ be its
enveloping algebra and identify the category of $\Lambda$-bimodules with
that of (right) $\Lambda^e$-modules in the canonical way. Recall that a
finite-dimensional self-injective algebra $\Lambda$ is \emph{twisted
  $n$-periodic} if there exist an algebra automorphism $\sigma\colon
\Lambda\simto \Lambda$ and an isomorphism
\[
  \Omega_{\Lambda^e}^n(\Lambda)\cong \twBim{\Lambda}[\sigma]
\]
in the stable category of $\Lambda$-bimodules; here, $\twBim{\Lambda}[\sigma]$
denotes the diagonal $\Lambda$-bimodule with right action twisted by $\sigma$.
Equivalently, $\Lambda$ is twisted $n$-periodic if there exists an exact
sequence of $\Lambda$-bimodules
\[
  0\to\twBim{\Lambda}[\sigma]\to P_{n}\to\cdots\to P_2\to P_1\to\Lambda\to0
\]
whose middle terms are projective $\Lambda$-bimodules. Similarly, $\Lambda$ is
\emph{$n$-periodic} if it is twisted $n$-periodic with $\sigma=\id[\Lambda]$.

On the other hand, let $\Lambda$ be a basic self-injective algebra. By
\cite[Prop.~3.8]{Bol84}, the map
\[
  \Out*{\Lambda}\longrightarrow\Pic*{\Lambda},\qquad
  [\sigma]\longmapsto[\twBim{\Lambda}[\sigma]],
\]
yields an isomorphism between the outer automorphism group $\Out*{\Lambda}$ of
$\Lambda$ and the Picard group $\Pic*{\Lambda}$ of invertible
$\Lambda$-bimodules. In particular, every autoequivalence of the additive
category $\proj*{\Lambda}$ is of the form
\[
  -\otimes_\Lambda\twBim{\Lambda}[\sigma]\colon\proj*{\Lambda}\stackrel{\sim}{\longrightarrow}\proj*{\Lambda}
\]
for some automorphism $\sigma$ of $\Lambda$; notice also that the above
autoequivalence is isomorphic to the restriction of scalars
\[
  (-)_{\sigma}\colon P\longmapsto P_{\sigma}
\]
along $\sigma$. In particular, $\Lambda$ is twisted $n$-periodic if and only if
the $\Lambda$-bimodule $\Omega_{\Lambda^e}^n(\Lambda)$ is isomorphic to
an invertible $\Lambda$-bimodule in the stable category of $\Lambda$-bimodules.

We record the following observation for later use. In the context of \Cref{prop:unique_suspension}, the syzygy
$\Omega(M)$ of a module $M$ is taken to be the kernel of its projective cover.

\begin{proposition}
  \label{prop:unique_suspension} Let $\Lambda$ be a finite-dimensional
  self-injective algebra that is twisted $(d+2)$-periodic. Suppose that
  $\Lambda$ is connected and non-separable. Then, the following statements hold:
  \begin{enumerate}
  \item Every invertible $\Lambda$-bimodule $I$ that is isomorphic to
    $\Omega_{\Lambda^e}^{d+2}(\Lambda)$ in $\smmod*{\Lambda^e}$ is already
    isomorphic to $\Omega_{\Lambda^e}^{d+2}(\Lambda)$ as a $\Lambda$-bimodule.
  \item The $(d+2)$-syzygy $\Omega_{\Lambda^e}^{d+2}(\Lambda)$ is an invertible
    $\Lambda$-bimodule.
  \end{enumerate}
\end{proposition}
\begin{proof}
  The proof of \cite[Prop.~9.8]{Mur22} applies verbatim.
\end{proof}

Recall that a finite-dimensional algebra $\Lambda$ is \emph{separable} if it is
isomorphic to a finite product of matrix algebras over skew fields, whose centres are
separable field extensions of the ground field $\kk$; in particular, if $\kk$ is
perfect, then $\Lambda$ is separable if and only if it is semisimple (see for
example \cite[Sec.~9.2.1]{Wei94}). The following theorem of Green, Snashall and
Solberg (with some simplifications due to Hanihara) gives a convenient
characterisation of the class of twisted periodic algebras, see also the
comments in the proof of~\cite[Prop.~3.5]{CDIM25}.

\begin{theorem}[{\cite[Thm.~1.4]{GSS03} and~\cite[Cor.~2.2]{Han20}}]
  \label{thm:GSS_twisted_periodicity}
  Let $\Lambda$ be a finite-dimensional self-injective algebra and $n\geq1$ an
  integer; suppose, moreover, that $\Lambda/J_\Lambda$ is separable. The
  following statements are equivalent:
  \begin{enumerate}
  \item The algebra $\Lambda$ is twisted $n$-periodic.
  \item There exist an algebra automorphism $\sigma\colon\Lambda\simto\Lambda$
    and a natural isomorphism
    \[
      \Omega_\Lambda^n\cong(-)_\sigma
    \]
    of exact functors $\smmod*{\Lambda}\simto\smmod*{\Lambda}$.
  \end{enumerate}
\end{theorem}

\begin{remark}
  Suppose that the equivalent conditions in~\Cref{thm:GSS_twisted_periodicity}
  are satisfied and that $\Lambda$ is connected and non-separable.
  \Cref{prop:unique_suspension} shows that $\Omega_{\Lambda^e}^{d+2}(\Lambda)$ is an invertible
  $\Lambda$-bimodule but it need not be isomorphic to ${}_1\Lambda_{\sigma}$.
  Although these bimodules induce isomorphic functors on $\smmod{\Lambda}$, this
  only permits us to conclude that
  $\Omega_{\Lambda^e}^{d+2}(\Lambda)\cong{}_1\Lambda_{\sigma\gamma}$ where $\gamma$
  is a \emph{stable} inner automorphism, that is an algebra automorphism such that
  the restriction of scalars along $\gamma$ is isomorphic to the identity
  functor of the stable category $\smmod{\Lambda}$ (compare
  with~\cite[Thm.~1.8]{IV14} for example).
\end{remark}

Our interest in twisted periodic algebras in the context of $(d+2)$-angulated
categories stems from the following observation, which is essentially a
reformulation of \Cref{prop:GKO-Freyd-Heller}.

\begin{proposition}
  \label{prop:d+2-angulated_twisted_periodic}
  Let $(\F,\Sigma,\pentagon)$ be a pre $(d+2)$-angulated category with
  finite-dimensional morphism spaces and split idempotents. Suppose that there
  exists a basic object $x\in\F$ such that $\add*{x}=\F$ and set
  $\Lambda\coloneqq\F(x,x)$. If $\Lambda/J_\Lambda$ is separable, then $\Lambda$
  is self-injective and twisted $(d+2)$-periodic.
\end{proposition}
\begin{proof}
  The proof is an adaptation of the proof of~\cite[Prop.~3.1]{Han20}, which
  treats the case $d=1$ of triangulated categories. By assumption, the functor
  \[
    \F(x,-)\colon\F\longrightarrow\proj*{\Lambda},\qquad y\longmapsto\F(x,y)
  \]
  is an equivalence of categories; hence the finite-dimensional algebra
  $\Lambda$ is self-injective by
  \Cref{prop:GKO-Freyd-Heller}\eqref{it:modF-Frobenius} (notice that this
  proposition only requires $\pentagon$ to be a pre-$(d+2)$-angulation). In particular, there
  are commutative diagrams of equivalences of categories
  \begin{center}
    \begin{tikzcd}
      \F\rar{\sim}\dar{\Sigma^{-1}}&\proj*{\Lambda}\dar{(-)_\sigma}\\
      \F\rar{\sim}&\proj*{\Lambda}
    \end{tikzcd}
    \qquad\text{and}\qquad
    \begin{tikzcd}
      \smmod*{\F}\rar{\sim}\dar{\Sigma^{-1}}&\smmod*{\Lambda}\dar{(-)_\sigma}\\
      \smmod*{\F}\rar{\sim}&\smmod*{\Lambda},
    \end{tikzcd}
  \end{center}
  where $\sigma\colon\Lambda\stackrel{\sim}{\to}\Lambda$ is an
  algebra automorphism (compare \cite[Prop.~3.8]{Bol84}) and the functors in the diagram on the right are exact
  with respect to the induced triangulated structures.
  \Cref{prop:GKO-Freyd-Heller}\eqref{it:Sigma-d+2-syzygy} yields the existence
  of a natural isomorphism $\Sigma^{-1}\cong\Omega_\F^{d+2}$ of exact functors
  on $\smmod*{\F}$ and, consequently, there is also a natural isomorphism
  $(-)_{\sigma}\cong\Omega_\Lambda^{d+2}$ of exact functors on
  $\smmod*{\Lambda}$. \Cref{thm:GSS_twisted_periodicity} then implies that
  $\Lambda$ is twisted $(d+2)$-periodic, which is what we needed to prove.
\end{proof}

Combining \Cref{thm:GKO-standard} and
\Cref{prop:d+2-angulated_twisted_periodic}, we obtain an alternative proof of
the following result that emphasises the role of $(d+2)$-angulations in this
context.

\begin{proposition}[{\cite{Dug12}, \cite[Prop.~8.5(a)]{CDIM25}}]
  \label{prop:dZ_CT_twisted_periodic}
  Let $\T$ be a triangulated category with finite-dimensional morphism spaces
  and split idempotents and such that there exists a $d\ZZ$-cluster tilting
  object $c\in\T$. Set $\Lambda\coloneqq\T(c,c)$. If $\Lambda/J_\Lambda$ is
  separable, then $\Lambda$ is self-injective and twisted $(d+2)$-periodic.
\end{proposition}
\begin{proof}
  Indeed, \Cref{thm:GKO-standard} shows that the additive category
  $\C\coloneqq\add(c)\subseteq\T$ is endowed with the structure of a
  $(d+2)$-angulated category with suspension functor $[d]\colon\C\simto\C$. Both
  claims then follows from \Cref{prop:d+2-angulated_twisted_periodic}.
\end{proof}

\begin{remark}
  In \cite[Prop.~8.5(a)]{CDIM25}, the authors assume that the ambient
  triangulated category admits a Serre functor in the sense of \cite{BK89}, but
  this fact is not used in their proof.
\end{remark}

\begin{construction}
  \label{const:AL_angulation}
  Let $\Lambda$ be a basic finite-dimensional self-injective algebra and
  $\sigma$ an automorphism. Notice that $\Lambda$ and $\Lambda^\op$ are
  Frobenius algebras~\cite[Prop.~4.5.7]{Zim14} and therefore so is their tensor
  product $\Lambda^e$; in particular the projective $\Lambda$-bimodules are also
  injective $\Lambda$-bimodules. Consider an exact sequence of
  $\Lambda$-bimodules
  \[
    \eta\colon\quad0\to\twBim{\Lambda}[\sigma]\to P_{d+1}\to\cdots\to P_1\to
    P_0\to\Lambda\to0
  \]
  with $P_i$ projective-injective as $\Lambda$-bimodules for $0\leq i<d+1$. In
  particular $\eta$, regarded as a complex of left $\Lambda$-modules, is
  contractible\footnote{It suffices to assume that $P_i$ is left projective for
    $0\leq i<d+1$ (as a consequence $P_{d+1}$ too), so that the tensored
    sequence in \eqref{it:tensored_extension} remains exact. This hypothesis is
    missing in \cite[Prop.~5.6 and Rmk.~5.7]{Mur20a}. Nevertheless, any
    extension has representatives satisfying this assumption, such as the
    representative considered in the proof of \cite[Prop.~5.6]{Mur20a}.}. We will be mainly interested in the case that $\Lambda$ is
  twisted $(d+2)$-periodic with respect to $\sigma$ and $P_{d+1}$ is also a
  projective-injective $\Lambda$-bimodule, so that $\eta$ exhibits the existence
  of an isomorphism
  $\Omega_{\Lambda^e}^{d+2}(\Lambda)\cong\twBim{\Lambda}[\sigma]$ in
  $\smmod*{\Lambda^e}$. Let
  \[
    \Sigma\coloneqq-\otimes_\Lambda\twBim[\sigma]{\Lambda}\colon\proj*{\Lambda}\stackrel{\sim}{\longrightarrow}\proj*{\Lambda}
  \]
  and
  \[
    \Sigma^{-1}\coloneqq-\otimes_\Lambda\twBim{\Lambda}[\sigma]\colon\proj*{\Lambda}\stackrel{\sim}{\longrightarrow}\proj*{\Lambda}
  \]
  its quasi-inverse. 
  These functors extend to quasi-inverse exact autoequivalences of
  $\mmod*{\Lambda}$, denoted in the same way and defined by the very same tensor products.
  We define a class $\pentagon_\eta$ of \emph{exact
    $(d+2)$-angles} in $\proj*{\Lambda}$ as follows: A complex of finitely
  generated projective-injective $\Lambda$-modules
  \[
    Q_{d+2}\xrightarrow{f} Q_{d+1}\to\cdots\to Q_1\xrightarrow{g}\Sigma Q_{d+2}
  \]
  lies in $\pentagon_\eta$ if it has the following two properties:
  \begin{enumerate}
  \item The extended complex
    \[
      Q_{d+2}\xrightarrow{f} Q_{d+1}\to\cdots\to Q_1\xrightarrow{g}\Sigma
      Q_{d+2}\xrightarrow{\Sigma f}\Sigma Q_{d+1}
    \]
    is exact (notice that this condition only depends on the pair
    $(\proj*{\Lambda},\Sigma)$).
  \item\label{it:tensored_extension} Denote the $\Lambda$-module $\coker{g}$
    by $N$. By construction, there is an exact sequence
    \[
      0\to\Sigma^{-1}N\xrightarrow{i} Q_{d+1}\to\cdots\to
      Q_1\xrightarrow{g}\Sigma Q_{d+2}\xrightarrow{p} N\to 0
    \]
    such that $f=i\circ\Sigma^{-1}p$ and $\Sigma^{-1}N=N\otimes_\Lambda\twBim{\Lambda}[\sigma]$. We require the above exact sequence to be
    equivalent to the exact sequence
    \[
      N\otimes_\Lambda(0\to\twBim{\Lambda}[\sigma]\to P_{d+1}\to\cdots\to P_1\to
      P_0\to\Lambda\to0)
    \]
    in the extension space $\Ext[\Lambda]{N}{\Sigma^{-1}N}[d+2]$. The latter
    sequence is exact because, as we have pointed out above, $\eta$ is
    contractible as a complex of left $\Lambda$-modules.
  \end{enumerate}
\end{construction}

\begin{remark}
  \label{rmk:several_angulations_and_extensions}
  In connection with \Cref{rmk:several_angulations}, given an exact sequence of
  $\Lambda$-bimodules
  \[
    \eta\colon\quad0\to\twBim{\Lambda}[\sigma]\stackrel{\iota}{\to}
    P_{d+1}\to\cdots\to P_1\to P_0\to\Lambda\to0
  \]
  and a unit $u\in\kk^\times$ in the ground field the class
  \[
    u\cdot [\eta]\in\Ext[\Lambda^e]{\Lambda}{\twBim{\Lambda}[\sigma]}[d+2]
  \]
  is represented by the exact sequence of $\Lambda$-bimodules
  \[
    u\cdot\eta\colon\quad0\to\twBim{\Lambda}[\sigma]\xrightarrow{u^{-1}\cdot \iota}
    P_{d+1}\to\cdots\to P_1\to P_0\to\Lambda\to0.
  \]
  Therefore, the class $\pentagon_{u\cdot\eta}$ given by
  \Cref{const:AL_angulation} coincides with $u\cdot\pentagon_{\eta}$.
\end{remark}

We are ready to state the theorem(s) of Amiot and Lin.

\begin{theorem}[{\cite[Thm.~8.1]{Ami07} and~\cite[Thm.~1.3]{Lin19}}]
  \label{thm:Amiot-Lin}
  Let $\Lambda$ be a basic finite-dimensional algebra that is twisted
  $(d+2)$-periodic with respect to an algebra automorphism $\sigma$ and let
  \[
    \Sigma\coloneqq-\otimes_\Lambda\twBim[\sigma]{\Lambda}\colon\proj*{\Lambda}\stackrel{\sim}{\longrightarrow}\proj*{\Lambda}.
  \]
  Then, the class $\pentagon_\eta$ of exact $(d+2)$-angles from
  \Cref{const:AL_angulation} associated to an extension 
    \[
    \eta\colon\quad0\to\twBim{\Lambda}[\sigma]\to
    P_{d+1}\to\cdots\to P_1\to P_0\to\Lambda\to0
  \]
  with
  projective-injective middle $\Lambda$-bimodules $P_0,\dots, P_d,P_{d+1}$ endows the pair
  $(\proj*{\Lambda},\Sigma)$ with the structure of a $(d+2)$-angulated category.
\end{theorem}

\begin{remark}
  Notice that \Cref{thm:Amiot-Lin} is \emph{a priori} quite different from
  \Cref{thm:GKO-standard}. Indeed, the former result does not require the
  existence of an ambient triangulated category in which to embed
  $\proj*{\Lambda}$. However, \Cref{thm:dZ-Auslander_correspondence} and
  \Cref{thm:GKO_AL_Massey} will show that, if $\kk$ is perfect, $\proj*{\Lambda}$
  embeds as a $d\ZZ$-cluster tilting subcategory in an essentially unique
  algebraic triangulated category in such a way that the $(d+2)$-angulated
  structure from \Cref{thm:Amiot-Lin} coincides with that from
  \Cref{thm:GKO-standard}.
\end{remark}

\begin{definition}
  \label{def:AL_angulation}
  In the context of \Cref{thm:Amiot-Lin}, we call the class $\pentagon_\eta$ an
  \emph{Amiot--Lin (AL) $(d+2)$-angulation} of the pair
  $(\proj*{\Lambda},\Sigma)$.
\end{definition}

The following result shows that the $(d+2)$-angulated structure from
\Cref{thm:Amiot-Lin} is independent of the choice of (truncated) projective
resolution of the diagonal bimodule.

\begin{proposition}
  \label{prop:Amiot-Lin_independence_of_res}
  Let $\Lambda$ be a basic finite-dimensional algebra that is twisted
  $(d+2)$-periodic with respect to an algebra automorphism $\sigma$ and let
  \[
    \Sigma\coloneqq-\otimes_\Lambda\twBim[\sigma]{\Lambda}\colon\proj*{\Lambda}\stackrel{\sim}{\longrightarrow}\proj*{\Lambda}.
  \]
  Choose exact sequences
  \[
    \eta\colon\quad 0\to\twBim{\Lambda}[\sigma]\to P_{d+1}\to\cdots\to P_1\to
    P_0\to\Lambda\to0
  \]
  and
  \[
    \eta'\colon\quad 0\to\twBim{\Lambda}[\sigma]\to Q_{d+1}\to\cdots\to Q_1\to
    Q_0\to\Lambda\to0
  \]
  of $\Lambda$-bimodules with projective-injective middle terms. Then, there
  exists a morphism of exact sequences
  \begin{center}
    \begin{tikzcd}[column sep=small,row sep=small]
      \eta\colon&0\rar&\twBim{\Lambda}[\sigma]\rar\dar{u}&P_{d+1}\rar\dar&\cdots\rar&P_1\rar\dar&P_0\dar\rar&\Lambda\dar[equals]\rar&0\\
      \eta'\colon&0\rar&\twBim{\Lambda}[\sigma]\rar&Q_{d+1}\rar&\cdots\rar&Q_1\rar&Q_0\rar&\Lambda\rar&0
    \end{tikzcd}
  \end{center}
  where $u\in Z(\Lambda)^\times$ is a unit in the centre of $\Lambda$. In
  particular, the $(d+2)$-angulations $\pentagon_{\eta}$ and $\pentagon_{\eta'}$
  of $(\proj*{\Lambda},\Sigma)$ are equivalent in the sense of
  \Cref{def:d+2-angulated_cat-equivalence}.
\end{proposition}
\begin{proof}
  The Comparison Lemma for projective resolutions yields a morphism of complexes
  \begin{center}
    \begin{tikzcd}[column sep=small,row sep=small]
      \eta\colon&0\rar&\twBim{\Lambda}[\sigma]\rar\dar{\varphi}&P_{d+1}\rar\dar&\cdots\rar&P_1\rar\dar&P_0\dar\rar&\Lambda\dar[equals]\rar&0\\
      \eta'\colon&0\rar&\twBim{\Lambda}[\sigma]\rar&Q_{d+1}\rar&\cdots\rar&Q_1\rar&Q_0\rar&\Lambda\rar&0
    \end{tikzcd}
  \end{center}
  in which the morphism $\varphi$ is an isomorphism in the stable category of
  $\Lambda$-bimodules. Since the algebra $\Lambda$ is finite-dimensional, the
  morphism $\varphi$ can be replaced by an isomorphism in the category of
  $\Lambda$-bimodules without modifying the rows and the other vertical
  morphisms in the diagram (except for possibly $P_{d+1}\to Q_{d+1}$), see~\cite[Cor.~2.3]{Che21}. Also, it is easily
  verified that there are isomorphisms of algebras
  \[
    \End[\Lambda^e]{\twBim{\Lambda}[\sigma]}=\End[\Lambda^e]{\Lambda}\cong
    Z(\Lambda),
  \]
  where the rightmost isomorphism is classical. In particular, the isomorphism
  $\varphi$ is given by right multiplication by a unit $u\in Z(\Lambda)^\times$, which
  we interpret as a natural isomorphism $u\colon\Sigma\Rightarrow\Sigma$ in the
  usual way. It readily follows that the pair $(\id[\proj*{\Lambda}],u)$ yields
  an equivalence between the $(d+2)$-angulations $\pentagon_{\eta'}$ and
  $\pentagon_{\eta}$ on the pair $(\proj*{\Lambda},\Sigma)$.
\end{proof}

The following theorem is a $(d+2)$-angulated analogue of~\cite[Thm.~1.2]{Han20},
which treats the case $d=1$.

\begin{theorem}
  \label{thm:d-Hanihara}
  Let $\Lambda$ be a basic finite-dimensional algebra such that
  $\Lambda/J_\Lambda$ is separable. The following statements are equivalent:
  \begin{enumerate}
  \item The algebra $\Lambda$ is twisted $(d+2)$-periodic.
  \item The additive category $\proj*{\Lambda}$ admits the structure of a
    $(d+2)$-angulated category.
  \end{enumerate}
\end{theorem}
\begin{proof}
  The proof is a straightforward adaptation of the proof of
  ~\cite[Thm.~1.2]{Han20}. If $\Lambda$ is twisted $(d+2)$-periodic, then
  \Cref{thm:Amiot-Lin} shows that the additive category $\proj*{\Lambda}$ admits
  the structure of a $(d+2)$-angulated category. The converse is precisely
  \Cref{prop:d+2-angulated_twisted_periodic}.
\end{proof}

\section{Enhanced $(d+2)$-angulated categories}
\label{sec:dg_d+2}

\subsection{Pre-triangulated differential graded categories}

We begin by recalling the notion of a pre-triangulated differential graded
category in the sense of \cite{BK90}. We need a few preliminaries.

\subsubsection{Reminder on differential graded categories}

We recall basic aspects of the theory of differential graded categories that are
needed in the sequel; we refer the reader to~\cite{Kel06} and the references
therein for details.

A \emph{differential graded (=DG) category} is a category enriched in the
symmetric monoidal category $\Ch{\kk}=\Ch{\Mod*{\kk}}$ of cochain complexes of
$\kk$-modules (endowed with the usual tensor product of cochain complexes).
Thus, the morphisms between two objects in a DG category form a cochain complex
and the composition law satisfies the graded Leibniz rule. A \emph{DG functor}
is simply an enriched functor between DG categories. A \emph{graded category} is
a DG category whose cochain complexes of morphisms have vanishing differential,
and a \emph{graded functor} is a DG functor between graded categories. For
example, given a ($\kk$-linear) additive category $\C$, there is a DG category
$\dgCh{\C}$ whose objects are the cochain complexes in $\C$ and for
$X,Y\in\dgCh{\C}$ we let $ \dgHom[\C]{X}{Y}[\bullet]$ be the complex whose
component of degree $i$ is the $\kk$-module
\[
  \dgHom[\C]{X}{Y}[i]\coloneqq\prod_{k\in\ZZ}\C(X^k,Y^{k+i})
\]
of degree $i$ morphisms of graded objects $X\to Y$ and differential
\[
  \partial\colon f\longmapsto d_Yf-(-1)^ifd_X,\qquad f\in\dgHom[\C]{X}{Y}[i]. \]
Note also that the usual shift of cochain complexes
\[
  (X,d_X)\longmapsto X[1]\coloneqq(X(1),-d_X),
\]
where $X\mapsto X(1)$ denotes the shift of graded objects, can be promoted to a
DG functor
\[
  [1]\colon\dgCh{\C}\longrightarrow\dgCh{\C},\qquad X\longmapsto X[1]
\]
whose action on morphisms is given by
\[
  [1]\colon f\longmapsto (-1)^if,\qquad f\in\dgHom[\C]{X}{Y}[i].
\]

A DG category $\A$ has two associated categories whose objects are the same as
those of $\A$: the \emph{underlying category $\dgZ[0]{\A}$} and the \emph{$0$-th
  cohomology category $\dgH[0]{\A}$}. As the notation suggests, the morphisms in
$\dgZ[0]{\A}$ are the $\kk$-modules of cocycles
\[
  \dgZ[0]{\A}(x,y)\coloneqq\dgZ[0]{\A(x,y)},\qquad x,y\in\dgZ[0]{\A},
\]
while the morphisms in $\dgH[0]{\A}$ are the cohomology $\kk$-modules
\[
  \dgH[0]{\A}(x,y)\coloneqq\dgH[0]{\A(x,y)},\qquad x,y\in\dgH[0]{\A}.
\]
Similarly, $\A$ has an associated graded category $\dgH{\A}$ whose objects are
the same as those of $\A$ and with graded $\kk$-modules of morphisms
\[
  \dgH{\A}(x,y)\coloneqq\dgH{\A(x,y)},\qquad x,y\in\dgH{\A}.
\]
A DG functor $F\colon\A\to\B$ induces apparent functors
\[
  \dgZ[0]{F}\colon\dgZ[0]{\A}\longrightarrow\dgZ[0]{\B}\qquad\text{and}\qquad\dgH[0]{F}\colon\dgH[0]{\A}\longrightarrow\dgH[0]{\B}
\]
as well as a graded functor
\[
  \dgH{F}\colon\dgH{\A}\longrightarrow\dgH{\B}.
\]
A DG functor $F$ is a \emph{quasi-equivalence} if the induced functor $\dgH{F}$
is an equivalence of graded categories.

For example, if $\C$ is an additive category then
\[
  \dgZ[0]{\dgCh{\C}}=\Ch{\C}
\]
while
\[
  \dgH[0]{\dgCh{\C}}=\KCh{\C}
\]
is the homotopy category of cochain complexes in $\C$. Similarly, the graded
category $\dgH{\dgCh{\C}}$ has graded $\kk$-modules of morphisms
\[
  \dgH[i]{\dgHom[\C]{X}{Y}[\bullet]}=\KCh{\C}(X,Y[i]),\qquad X,Y\in\Ch{\C},\
  i\in\ZZ.
\]
Under these identifications, the shift (DG) functor on $\dgCh{\C}$ induces the
usual shift functors on $\Ch{\C}$ and $\KCh{\C}$.

\subsubsection{The homotopy category of small DG categories }
\label{subsubsec:hocat}

As with quasi-isomorphisms between cochain complexes, quasi-equivalences between
DG categories need not admit a quasi-inverse given by a DG functor. However, the
category $\dgcat$ of small DG categories admits a cofibrantly generated Quillen
model category structure, called the \emph{Tabuada model structure}, whose weak
equivalences are the quasi-equivalences (we do not need to recall here what are
the classes of (co)fibrations)~\cite{Tab05}. We denote the homotopy category of
the Tabuada model structure by $\Ho{\dgcat}$, which is, by the general theory of
Quillen model categories, equivalent to the localisation of $\dgcat$ at the
class of quasi-equivalences. In particular, two small DG categories $\A$ and
$\B$ are isomorphic in $\Ho{\dgcat}$ if and only if they are connected by a
finite zig-zag of quasi-equivalences. Furthermore, the set of morphisms
$\A\to\B$ in $\Ho{\dgcat}$ is in bijection with the set of quasi-isomorphism
classes of DG $\A$-$\B$-bimodules $X$ such that, for each $a\in\A$, the DG
$\B$-module $X(-,a)$ is quasi-isomorphic to a representable DG $\B$-module,
see~\cite[Cor.~4.8]{Toe11}.

\subsubsection{The derived category of a DG category}

We recall the construction of the derived category of a (small) DG category. We
refer the reader~\cite{Kel94} for details. Let $\A$ be a small DG category, that
is the objects in $\A$ form a set. The \emph{opposite DG category $A^\op$ of
  $\A$} is the DG category with the same objects as $\A$ and graded morphism
spaces
\[
  \A^\op(x,y)\coloneqq\A(y,x),\qquad x,y\in\A^\op;
\]
the composition law in $\A^\op$ is given in terms of the composition law in $\A$
by the formula
\[
  g\circ^\op f\coloneqq (-1)^{ij}f\circ g
\]
whenever $f$ and $g$ are homogeneous of degree $i$ and $j$, respectively.
The \emph{DG category of (right) DG
  $\A$-modules} is the DG category
\[
  \dgModdg*{\A}\coloneqq\dgFun{\A^\op}{\dgCh{\kk}}
\]
of DG functors $\A^\op\to\dgCh{\kk}=\dgCh{\Mod\kk}$. We denote the morphism complexes in
$\dgModdg*{\A}$ by
\[
  \dgHom[\A]{M}{N},\qquad M,N\in\dgModdg*{\A}.
\]
The DG Yoneda embedding
\[
  \Yoneda\colon\A\longrightarrow\dgModdg*{\A},\qquad x\longmapsto\A(-,x).
\]
identifies $\A$ with full DG subcategory of $\dgModdg*{\A}$.

The underlying category
\[
  \dgMod*{\A}\coloneqq\dgZ[0]{\dgModdg*{\A}}
\]
is a Grothendieck abelian category. Moreover, $\dgMod*{\A}$ is a Frobenius exact
category in which a short exact sequence
\[
  0\to L\stackrel{i}{\to}M\to N\to 0
\]
of DG $\A$-modules is admissible (=a conflation) if there exists a morphism
$p\colon M\to L$ of graded (!) $\A$-modules such that $pi=\id[L]$; a DG
$\A$-module is projective-injective in this exact structure if and only if it is
\emph{contractible}, that is if it is a zero object in the homotopy category
$\dgH[0]{\dgModdg*{\A}}$. In particular, the stable category of the Frobenius
exact category $\dgMod*{\A}$ equals $\dgH[0]{\dgModdg*{\A}}$ and is a
triangulated category with suspension functor induced by the shift functor on
$\dgCh{\kk}$ by means of the formula
\[
  [1]\colon M\longmapsto M[1]\coloneqq[1]\circ M,\qquad M\in\dgModdg*{\A},
\]
keeping in mind that a DG $\A$-module is a DG functor $\A^\op\to\dgCh{\kk}$. By
construction, there are canonical isomorphisms
\[
  \dgH[i]{\dgHom[\A]{M}{N}}\cong\dgH[0]{\dgModdg*{\A}}(M,N[i]),\qquad
  M,N\in\dgMod*{\A}
\]
and, as a consequence of the DG Yoneda Lemma,
\[
  \dgH[i]{\dgHom[\A]{\Yoneda[x]}{M}}\cong\dgH[i]{M_x},\qquad x\in\A,\
  M\in\dgMod*{\A}.
\]

A DG $\A$-module $N$ is \emph{acyclic} if for each $x\in\A$ the cochain complex
of $\kk$-modules $N_x$ is acyclic. The \emph{derived category $\DerCat{\A}$ of
  $\A$} is the full subcategory of the triangulated category
$\dgH[0]{\dgModdg*{\A}}$ spanned by the DG $\A$-modules that are \emph{DG
  projective}, that is the DG $\A$-modules $P$ such that every epimorphism $M\to
P$ with acyclic kernel is split. The derived category of $\A$ is a triangulated
subcategory of $\dgH[0]{\dgModdg*{\A}}$ that is closed under small coproducts
and retracts and is compactly generated by the \emph{free DG $\A$-modules}
\[
  \Yoneda[x]=\A(-,x)\colon y\longmapsto\A(y,x),\qquad x\in\A.
\]
The \emph{perfect derived category $\DerCat[c]{\A}$ of $\A$} is the full
subcategory of $\DerCat{\A}$ spanned by the compact objects; equivalently,
$\DerCat[c]{\A}$ is the thick subcategory of $\DerCat{\A}$ generated by the free
DG $\A$-modules~\cite[Thm.~5.3]{Kel94}. In particular, $\DerCat[c]{\A}$ is a
triangulated category in which idempotents split.

Let $F\colon\A\to\B$ be a DG functor. The restriction of scalars
\[
  F^*\colon\dgModdg*{\B}\longrightarrow\dgModdg*{\A}
\]
induces an exact functor
\[
  F^*\colon\DerCat{\B}\longrightarrow\DerCat{\A}
\]
that admits an exact left adjoint
\[
  \mathbb{L}F_!\colon\DerCat{\A}\longrightarrow\DerCat{\B}
\]
that preserves small co-products (constructed explicitly as a left derived
functor). Moreover, the functor $\mathbb{L}F_!$ preserves compact objects (since
it preserves the free DG modules) and hence restricts to an exact functor
\[
  \mathbb{L}F_!\colon\DerCat[c]{\A}\longrightarrow\DerCat[c]{\B}
\]
between the corresponding perfect derived categories.

\subsubsection{Morita equivalences}\label{subsubsec:Morita_equivalences}

The category of small DG categories and DG functors between them admits a
cofibrantly generated Quillen model category structure, called the Morita model
structure~\cite[Thm.~5.3]{Tab05}, see also
\cite{tabuada_2006_addendum_invariants_additifs,tabuada_2007_corrections_invariants_additifs},
whose weak equivalences are the \emph{Morita equivalences}, that is the DG
functors $F\colon\A\to\B$ such that the induced exact functor
\[
  \mathbb{L}F_!\colon\DerCat[c]{\A}\longrightarrow\DerCat[c]{\B}
\]
is an equivalence of triangulated categories. The DG Yoneda embedding
\[
  \Yoneda\colon\A\longrightarrow\DerCat[c][dg]{\A}
\]
is a fibrant replacement (that is~a Morita equivalence with fibrant target) for
every small DG category $\A$, where $\DerCat[c][dg]{\A}$ is the \emph{perfect
  derived DG category of $\A$}, that is the full DG subcategory of
$\dgModdg*{\A}$ spanned by the DG $\A$-modules that are DG projective and
compact in $\DerCat{\A}$. Clearly $\dgH[0]{\DerCat[c][dg]{\A}}\simeq\DerCat[c]{\A}$. We use this equivalence as an identification. The \emph{Morita category of small DG categories} is
the homotopy category $\Hmo$ associated to the Morita model structure. Two DG
categories are \emph{Morita equivalent} if they are isomorphic in $\Hmo$. In
particular, every small DG category is Morita equivalent to its perfect derived
DG category.

\begin{remark}
  If $\A$ and $\B$ are small DG categories that are Morita equivalent, then the
  perfect derived categories $\DerCat[c]{\A}$ and $\DerCat[c]{\B}$ are
  equivalent as triangulated categories. The converse, however, is
  false~\cite{DS07,Sch02,Kaj13,RVdB19}.
\end{remark}

\begin{remark}
  \label{rmk:quasi-equivalence_vs_Morita_equivalence}
  Every quasi-equivalence between small DG categories is also a Morita
  equivalence.
\end{remark}

\subsubsection{Pre-triangulated DG categories}

The following definitions motivate our definition of DG enhancement of a
$(d+2)$-angulated category (see
\Cref{def:pre-d+2-anguled_DG,def:DG_enhancement_d+2-angulated}).

\begin{definition}[{\cite{BK90}}]
  \label{def:pre-trianguled_DG}
  A small DG category $\A$ is \emph{pre-triangulated} if the fully faithful
  functor
  \[
    \dgH[0]{\Yoneda}\colon\dgH[0]{\A}\longrightarrow\DerCat[c]{\A}
  \]
  identifies $\dgH[0]{\A}$ with a triangulated subcategory of $\DerCat[c]{\A}$.
  Similarly, $\A$ is \emph{Karoubian pre-triangulated} if the above functor is
  an equivalence of categories.

  A pre-triangulated DG category $\A$ is Karoubian if and only if idempotents split in $\dgH[0]{\A}$, compare \cite[Cor.~3.7]{Kel06}.
\end{definition}

\begin{remark}
  \label{rem:Morita_fibrant_pretriangulated}
  It is proven in \cite{tabuada_2007_corrections_invariants_additifs} that a DG category is fibrant in the Morita model structure if and only if it is Karoubian pre-triangulated. Moreover, the homotopy category $\Hmo$ associated to the Morita model structure is equivalent to
  the full subcategory of $\Ho{\dgcat}$ spanned by the Karoubian
  pre-triangulated DG categories. The equivalence sends a Karoubian
  pre-triangulated DG category in $\Ho{\dgcat}$ to itself viewed as an object of $\Hmo$, and the same for morphisms.
  In particular, two Karoubian pre-triangulated
  DG categories are quasi-equivalent if and only if they are Morita equivalent.
\end{remark}

\begin{remark}
  The term \emph{pre-triangulated} is also used for triangulated categories
  which need not satisfy the octahedral axiom \cite[Def.~1.1.2]{Nee01}, see also
  \Cref{def:d+2-angulated_cat}.
  Despite the clash of
  terminology, it is always clear from the context which notion is in use. In
  particular, notice that the $0$-th cohomology category of a pre-triangulated
  DG category is a triangulated category (and therefore satisfies the octahedral
  axiom).
\end{remark}

\begin{definition}[{\cite{BK90}}]
  \label{def:DG_enhancement}
  Let $\T$ be a triangulated category with split idempotents.
  \begin{enumerate}
  \item A \emph{(DG) enhancement of $\T$} is a (necessarily Karoubian)          pre-triangulated DG category $\A$
    such that there exists an equivalence of triangulated categories
    \[
      \T\simeq\dgH[0]{\A}
    \]
    (notice that we only require the existence of such an equivalence, and not
    the datum of a preferred such).
  \item An \emph{equivalence} of enhancements is a quasi-equivalence of DG
    categories (that is an isomorphism in $\Ho{\dgcat}$).
  \item We say that \emph{$\T$ admits a unique (DG) enhancement} if it admits an
    enhancement and any two enhancements of $\T$ are equivalent.
  \end{enumerate}
\end{definition}

\begin{remark}
  \label{def:DG_enhancement_equivalent}
  We could have alternatively defined DG enhancements and their equivalences in the following way, which yields the same set of equivalence classes of enhancements, in particular the same notion of uniqueness of enhancements.

  Let $\T$ be a triangulated category with split idempotents.
  \begin{enumerate}
  \item A \emph{(DG) enhancement of $\T$} is a DG category $\A$
    such that there exists an equivalence of triangulated categories
    \[
      \T\simeq\DerCat[c]{\A}.
    \]
  \item An \emph{equivalence} of enhancements is a Morita equivalence of DG
    categories (that is an isomorphism in $\Hmo$).
  \item We say that \emph{$\T$ admits a unique (DG) enhancement} if it admits an
    enhancement and any two enhancements of $\T$ are equivalent.
  \end{enumerate}

  An enhancement $\A$ in the sense of \Cref{def:DG_enhancement} is clearly an
  enhancement in this sense since $\dgH[0]{\A}\simeq\DerCat[c]{\A}$ in this case. Moreover, two enhancements in the sense of
  \Cref{def:DG_enhancement} are equivalent if and only if they are equivalent in
  this sense, see \Cref{rem:Morita_fibrant_pretriangulated}. Furthermore, any enhancement $\A$ in this sense is equivalent to an enhancement in the sense of \Cref{def:DG_enhancement}, namely to its perfect derived DG category, via the DG Yoneda embedding $\Yoneda\colon\A\rightarrow\DerCat[c][dg]{\A}$, which is a Morita equivalence, see \Cref{subsubsec:Morita_equivalences}. Therefore, as claimed, both definitions yield the same set of equivalence classes of enhancements. In particular, $\T$ admits a unique enhancement in the sense of \Cref{def:DG_enhancement} if and only if it admits a unique enhancement in this sense.
\end{remark}

\begin{theorem}[{\cite[Thm.~3.8]{Kel06}}]
  \label{thm:Keller-perf}
  Let $\X$ be a Karoubian pre-triangulated DG category and $\A\subseteq\X$ a
  full DG subcategory such that $\thick*{\dgH[0]{\A}}=\dgH[0]{\X}$. Then, the
  inclusion functor $\A\hookrightarrow\X$ is a Morita equivalence and, in
  particular, there are canonical equivalences of triangulated categories
  \[
    \DerCat[c]{\A}\stackrel{\sim}{\longrightarrow}\DerCat[c]{\X}\stackrel{\sim}{\longleftarrow}\dgH[0]{\X}.
  \]
\end{theorem}

\begin{example}
  \label{ex:alebra_unique_enhancement}
  Given a DG algebra $A$, there are canonical morphisms of DG algebras (=DG
  functors between DG categories with a single object)
  \[
    \dgH[0]{A}\longleftarrow\tau^{\leq0}A\longrightarrow A,
  \]
  where
  \[
    (\tau^{\leq0}A)^i\coloneqq\begin{cases}A^i&i<0,\\\dgZ[0]{A}&i=0,\\0&i>0.\end{cases}
  \]
  By construction, the map $\tau^{\leq0}A\to A$ is a quasi-isomorphism if and
  only if the cohomology of $A$ is concentrated in non-positive degrees. In particular,
  if $\dgH{A}=\dgH[0]{A}$ is concentrated in degree $0$, then both of the above
  maps are quasi-isomorphisms of DG algebras and thus also Morita equivalences.
  Combining this observation with \Cref{thm:Keller-perf}, we conclude that the
  (compact) derived category of an ordinary algebra admits a unique enhancement
  in the sense of \Cref{def:DG_enhancement}.
\end{example}

\begin{remark}
  In contrast to the situation described in \Cref{ex:alebra_unique_enhancement}
  for the derived category of an ordinary algebra, establishing the uniqueness
  of enhancements for more general triangulated categories is a subtle
  endeavour. Moreover, the majority of results of this kind that are available
  rely on the existence of $t$-structures in a crucial way, see for example
  \cite{LO10,CS18,CNS22}. Similar to the situation in \cite{Mur22}, in this
  article we consider triangulated categories that need not admit any
  non-trivial $t$-structures (such as indecomposable $2$-Calabi--Yau triangulated categories
  with a $2$-cluster tilting object, see \Cref{thm:AGK_uniqueness} and compare
  with \cite[Thm.~4.1]{ZZ14}).
\end{remark}

\subsection{Pre-$(d+2)$-angulated differential graded categories}

The following are straightforward extensions of
\Cref{def:pre-trianguled_DG,def:DG_enhancement} to the context of
$(d+2)$-angulated categories.

\begin{definition}
  \label{def:pre-d+2-anguled_DG}
  A small DG category $\A$ is \emph{pre-$(d+2)$-angulated} if the fully faithful
  functor
  \[
    \dgH[0]{\Yoneda}\colon\dgH[0]{\A}\longrightarrow\DerCat[c]{\A}
  \]
  identifies identifies $\dgH[0]{\A}$ with a standard $(d+2)$-angulated
  subcategory of $\DerCat[c]{\A}$ (see \Cref{def:stanard_angulated_cat}).
  Similarly, $\A$ is \emph{Karoubian pre-$(d+2)$-angulated} if it is
  pre-$(d+2)$-angulated and idempotents split in $\dgH[0]{\A}$.
\end{definition}

\begin{remark}
  A (Karoubian) pre-3-angulated DG category is precisely a (Karoubian)
  pre-triangulated DG category.
\end{remark}

\begin{remark}
  \label{rmk:sparseness}
  Let $\A$ be a pre-$(d+2)$-angulated DG category. Since, by definition,
  $\dgH[0]{\A}$ is equivalent to a $d$-rigid subcategory of $\DerCat[c]{\A}$
  closed under the action of the $d$-fold suspension and its inverse, the
  standard isomorphisms
  \[
    \Hom[\DerCat{\A}]{\Yoneda[x]}{\Yoneda[y]{[i]}}\cong\dgH[i]{\A(x,y)},\qquad
    x,y\in\A,
  \]
  imply that
  \[
    \forall i\not\in d\ZZ,\qquad \dgH[i]{\A(x,y)}=0.
  \]
  This is a strong restriction on the cohomology of $\A$ (see also
  \Cref{def:sparse}).
\end{remark}

\begin{proposition}
  \label{prop:d+2=dZ-CT}
  Let $\A$ be a small DG category such that $\dgH[0]{\A}$ is an additive
  category with finite-dimensional morphism spaces and split idempotents. The
  following statements are equivalent:
  \begin{enumerate}
  \item The DG category $\A$ is Karoubian pre-$(d+2)$-angulated.
  \item The fully faithful functor
    \[
      \dgH[0]{\Yoneda}\colon\dgH[0]{\A}\longrightarrow\DerCat[c]{\A}
    \]
    identifies $\dgH[0]{\A}$ with a $d\ZZ$-cluster tilting
    subcategory of $\DerCat[c]{\A}$.
  \end{enumerate}
\end{proposition}
\begin{proof}
  Notice that, since $\DerCat[c]{\A}$ is generated as a triangulated category
  with split idempotents by the free DG $\A$-modules, the finiteness assumption
  on $\dgH[0]{\A}$ implies that $\DerCat[c]{\A}$ has finite-dimensional morphism
  spaces. The equivalence between the two conditions then follows from
  \Cref{thm:standard_d+2-angulated_characterisation}.
\end{proof}

\begin{definition}
  \label{def:DG_enhancement_d+2-angulated}
  Let $\F$ be a $(d+2)$-angulated category with split idempotents.
  \begin{enumerate}
  \item A \emph{(DG) enhancement of $\F$} is a (necessarily Karoubian) pre-$(d+2)$-angulated DG category
    $\A$ such that there exists an equivalence of $(d+2)$-angulated categories
    \[
      \F\simeq\dgH[0]{\A},
    \]
    where the $(d+2)$-angulation on the right-hand side is induced by the
    triangulation on $\DerCat[c]{\A}$ as in \Cref{thm:GKO-standard}. If $\F$ has an enhancement we say that it is \emph{algebraic.}
  \item An \emph{equivalence} of enhancements is a quasi-equivalence of DG
    categories (that is an isomorphism in $\Ho{\dgcat}$).
  \item\label{it:DG_enhancement_pre-d+2-ang_equivalence} We say that \emph{$\F$
      admits a unique (DG) enhancement} if it admits an enhancement and any two
    enhancements of $\F$ are equivalent.
  \end{enumerate}
\end{definition}

\begin{remark}
  Algebraic $(d+2)$-angulated categories are defined in~\cite[Def.~5.12]{Jas16}
  in terms of Frobenius $d$-exact categories. The agreement between \Cref{def:DG_enhancement_d+2-angulated}
  and this alternative definition goes back to~\cite{Kel94} in the case $d=1$,
  while for $d\geq1$ it has been shown recently by Kvamme in~\cite[Prop.~7.7]{Kva25}.
\end{remark}

The following two questions are natural.

\begin{question}
  \label{question:one}
  Let $\F$ be a $(d+2)$-angulated category with split idempotents. Does $\F$ admit an enhancement? If
  it does, under which conditions is the enhancement unique?
\end{question}

\begin{definition}
  \label{def:enhanced_angulated_structure}
  Let $(\C,\Sigma)$ be a pair consisting of an additive category $\C$ with split idempotents and an
  autoequivalence $\Sigma\colon\C\simto\C$. We say that a pre-$(d+2)$-angulated
  DG category $\A$ \emph{induces a (DG) enhanced $(d+2)$-angulated structure on
  $(\C,\Sigma)$} if there exists an equivalence of pairs $\varphi\colon\dgH[0]{\A}\simto\C$, 
  i.e.~such that the diagram
  \begin{center}
    \begin{tikzcd}
      \dgH[0]{\A}\rar{\varphi}\dar[swap]{[d]}&\C\dar{\Sigma}\\
      \dgH[0]{\A}\rar{\varphi}&\C
    \end{tikzcd}
  \end{center}
  commutes up to natural isomorphism. In this case, we say that the pair
  $(\C,\Sigma)$ \emph{admits a (DG) enhanced $(d+2)$-angulated structure}.
\end{definition}

\begin{question}
  \label{question:two}
  Let $(\C,\Sigma)$ be a pair consisting of an additive category $\C$ with split idempotents and an
  automorphism $\Sigma\colon\C\simto\C$. Does the pair admit an enhanced
  $(d+2)$-angulated structure? If it does, is the enhancement of the underlying
  $(d+2)$-angulated category unique?
\end{question}

\begin{remark}
  \Cref{thm:GKO_AL_Massey,thm:projLambda_existence_and_uniqueness} provide a complete answer to
  \Cref{question:one} and \Cref{question:two} in the setting of
  \Cref{thm:Amiot-Lin,thm:d-Hanihara}.
\end{remark}

\section{Higher structures and $(d+2)$-angulations}
\label{sec:higher_stuff}

\subsection{Toda brackets and standard $(d+2)$-angulations}

\begin{setting}
  We fix a triangulated category $\T$ with shift functor $[1]$ in this
  subsection.
\end{setting}

We aim to characterise standard $(d+2)$-angles
(\Cref{def:stanard_angulated_cat}) in terms of higher Toda brackets. We begin by
recalling the necessary definitions.

\begin{definition}
  \label{def:Toda_family}
  Consider a sequence of morphisms in $\T$ of the form
  \[
    x_3\xrightarrow{f_3}x_2\xrightarrow{f_2}x_1\xrightarrow{f_1}x_0.
  \]
  The \emph{Toda family $\TodaFam{f_1,f_2,f_3}$} is the set of pairs
  $(\beta,\alpha)$ such that there exists a commutative diagram of the form
  \begin{equation}
    \label{eq:TodaFam}
    \begin{tikzcd}
      x_3\rar{f_3}\dar[swap]{\alpha}&x_2\rar{f_2}\dar[equals]&x_1\rar{f_1}\dar[equals]&x_0\\
      c\rar{u}&x_2\rar{f_2}&x_1\rar{v}&c[1]\uar[swap]{\beta[1]}
    \end{tikzcd}
  \end{equation}
  in which the bottom horizontal row is an exact triangle in $\T$.
\end{definition}

\begin{remark}
  In the setting of \Cref{def:Toda_family}, notice that the Toda family
  $\TodaFam{f_1,f_2,f_3}$ is non-empty if and only if $f_1f_2=0$ and $f_2f_3=0$.
\end{remark}

\begin{remark}
  \label{rmk:TodaBracket_betaalpha}
  An alternative (perhaps more suggestive) way to depict the condition for the
  pair $(\beta,\alpha)$ to lie in $\TodaFam{f_1,f_2,f_3}$ is as a commutative
  diagram
  \begin{center}
    \begin{tikzcd}[column sep=small]
      &x_2\ar{rr}{f_2}&&x_1\drar{f_1}\ar{dl}[description]{+1}\\
      x_3\urar{f_3}\ar{rr}{\alpha}&&%
      c\ar{rr}[near start]{\beta}\ular\ar{rr}[description]{-1}&&x_0
    \end{tikzcd}
  \end{center}
  in which $\beta$ is a morphism of degree $-1$. Here, the oriented triangle is
  an exact triangle in $\T$ and the un-oriented triangles commute. In
  particular, the pair $(\beta,\alpha)$ yields the well-defined composite
  \begin{center}
    \begin{tikzcd}
      x_3[1]\rar{\alpha[1]}&c[1]\rar{\beta[1]}&x_0.
    \end{tikzcd}
  \end{center}
\end{remark}

\begin{definition}
  \label{def:Toda_bracket}
  Consider a sequence of morphisms in $\T$ of the form
  \[
    x_{d+2}\xrightarrow{f_{d+2}}x_{d+1}\xrightarrow{f_{d+1}}\cdots\xrightarrow{f_2}x_1\xrightarrow{f_1}x_0.
  \]
  The \emph{Toda bracket $\TodaBracket{f_1,\dots,f_{d+2}}$} is the subset of
  $\T(x_{d+2}[d],x_0)$ defined inductively as follows:
  \begin{enumerate}
  \item If $d=0$, then $\TodaBracket{f_1,f_2}\coloneqq\set{f_1\circ f_2}$.
  \item If $d\geq1$, then
    \[
      \TodaBracket{f_1,\dots,f_{d+2}}\coloneqq\bigcup_{(\beta,\alpha)\in\TodaFam{f_1,f_2,f_3}}\TodaBracket{\beta[1],\alpha[1],f_4[1],\cdots,f_{d+2}[1]}
    \]
    is the union of Toda brackets of length $d+1$.
  \end{enumerate}
\end{definition}

\begin{remark}
  \Cref{def:Toda_bracket} is equivalent to that in~\cite{Shi02}, which is itself
  based on that in~\cite{Coh68}, see \cite[Prop.~5.8]{CF17}. We also note that
  \Cref{def:Toda_family} is equivalent to the desuspension of that
  in~\cite[Sec.~5]{CF17}.
\end{remark}

\begin{remark}
  The Toda bracket is not quite a $(d+2)$-ary operation as its output is a
  \emph{set} of morphisms rather than a single map (compare with
  \Cref{defprop:TodaBracket_indeterminacy}).
\end{remark}

\begin{example}
  \label{ex:Toda3}
  The Toda bracket of a triple of composable morphisms
  \[
    x_3\xrightarrow{f_3}x_2\xrightarrow{f_2}x_1\xrightarrow{f_1}x_0.
  \]
  is the subset of $\T(x_3[1],x_0)$ given by
  \[
    \TodaBracket{f_1,f_2,f_3}=\bigcup_{(\beta,\alpha)\in\TodaFam{f_1,f_2,f_3}}\TodaBracket{\beta[1],\alpha[1]}=\set{(\beta\circ\alpha)[1]}[(\beta,\alpha)\in\TodaFam{f_1,f_2,f_3}].
  \]
\end{example}

\begin{example}
  \label{ex:Toda3-ex}
  The Toda bracket of the morphisms in an exact triangle
  \[
    x_3\xrightarrow{f_3}x_2\xrightarrow{f_2}x_1\xrightarrow{f_1}x_3[1]
  \]
  satisfies $\id[x_3[1]]\in\TodaBracket{f_1,f_2,f_3}$. Indeed, in view of
  \Cref{ex:Toda3}, it is enough to observe that the tautological diagram
  \[
    \begin{tikzcd}
      x_3\rar{f_3}\dar[swap]{\id[x_3]}&x_2\rar{f_2}\dar[equals]&x_1\rar{f_1}\dar[equals]&x_3[1]\\
      x_3\rar{f_3}&x_2\rar{f_2}&x_1\rar{f_1}&x_3[1]\uar[swap]{\id[x_3][1]}
    \end{tikzcd}
  \]
  exhibits the membership $(\id[x_3],\id[x_3])\in\TodaFam{f_1,f_2,f_3}$.
\end{example}

The following classical result, which is a refinement of \Cref{ex:Toda3-ex},
illustrates the usefulness of Toda brackets.

\begin{theorem}[{\cite[Thm.~13.2]{Hel68}, see also~\cite[Thm.~B.1]{CF17}}]
  \label{thm:TodaBracket_3-angles}
  Consider a sequence of morphisms in $\T$ of the form
  \[
    x_3\xrightarrow{f_3}x_2\xrightarrow{f_2}x_1\xrightarrow{f_1}x_3[1].
  \]
  The following statements are equivalent:
  \begin{enumerate}
  \item The above sequence is an exact triangle in $\T$.
  \item The following two conditions are satisfied:
    \begin{enumerate}
    \item For each $y\in\T$, the induced sequence of $\kk$-modules
      \[
        \T(y,x_3)\to\T(y,x_2)\to\T(y,x_1)\to\T(y,x_3[1])\to\T(y,x_2[1])
      \]
      is exact.
    \item We have $\id[x_3[1]]\in\TodaBracket{f_1,f_2,f_3}$.
    \end{enumerate}
  \end{enumerate}
\end{theorem}

The objective of this section is to prove the following $(d+2)$-angulated
analogue of \Cref{thm:TodaBracket_3-angles} (compare also with
\Cref{thm:GKO-standard}). As is customary, given a subset $X$ of an abelian
group, we let
\[
  -X\coloneqq\set{-x}[x\in X].
\]

\begin{theorem}
  \label{thm:TodaBracket_d+2-angles}
  Let $\C\subseteq\T$ be a $d$-rigid subcategory such that ${\C=\add*{c}}$ for some basic object $c$,
  ${\C[d]=\C}$ and $\thick*{\C}=\T$. Consider a sequence of morphisms in $\C$
  of the form
  \begin{center}
    \begin{tikzcd}
      c_{d+2}\rar{f_{d+2}}&c_{d+1}\rar{f_{d+1}}&\cdots\rar{f_3}&c_2\rar{f_2}&c_1\rar{f_1}&c_{d+2}[d],
    \end{tikzcd}
  \end{center}
  with $d\geq1$. The following statements are equivalent:
  \begin{enumerate}
  \item\label{it:TodaBracket_d+2-angles_isstd} The above sequence fits as the
    spine of a commutative diagram of the form
    \begin{center}
      \begin{tikzcd}[column sep=small]
        &c_{d+1}\drar&\cdots&&&c_3\drar\ar{rr}&&c_2\drar\\
        c_{d+2}\urar&&x_{d.5}\ar{ll}[description]{+1}&\cdots&x_{3.5}\urar&&x_{2.5}\urar\ar{ll}[description]{+1}&&c_1\ar{ll}[description]{+1}
      \end{tikzcd}
    \end{center}
    in which the oriented triangles are exact triangles in $\T$ and such that
    the connecting morphism $c_1\to c_{d+2}[d]$ is given by the obvious
    (shifted) composite along the bottom row of the diagram.
  \item\label{it:TodaBracket_d+2-angles_isAL} The following two conditions are
    satisfied:
    \begin{enumerate}
    \item\label{it:TodaBracket_d+2-angles_exactness} For each $y\in\C$, the induced sequence of vector spaces
      \[
        \begin{tikzcd}[column sep=small]
          \T(y,c_{d+2})\rar&\cdots\rar&\T(y,c_2)\rar&\T(y,c_1)\rar&\T(y,c_{d+2}[d])\rar&\T(y,c_{d+1}[d])
        \end{tikzcd}
      \]
      is exact.
    \item We have
      \[
        \id[c_{d+2}[d]]\in(-1)^{1+\sum_{i=1}^{d+1}i}\TodaBracket{f_1,f_2,\dots,f_{d+2}}.
      \]
    \end{enumerate}
  \end{enumerate}
\end{theorem}

Computing higher Toda brackets can be rather intricate. We establish a few
computation rules that are needed for the proof of
\Cref{thm:TodaBracket_d+2-angles} as well as elsewhere in the article.

\begin{lemma}
  \label{lemma:TodaBracket_red-iso}
  Consider a sequence of morphisms in $\T$ of the form
  \[
    x_{d+2}\xrightarrow{f_{d+2}}x_{d+1}\xrightarrow{f_{d+1}}\cdots\xrightarrow{f_2}x_1\xrightarrow{f_1}x_0
  \]
  and let $\varphi\colon x_1\to x_1'$ be an isomorphism. Then, there is an
  equality of Toda brackets
  \[
    \TodaBracket{f_1,f_2,\dots,f_{d+2}}=\TodaBracket{f_1\varphi^{-1},\varphi
      f_2,\dots,f_{d+2}}.
  \]
\end{lemma}
\begin{proof}
  The claim is obvious for $d=0$ for we have
  \[
    \TodaBracket{f_1,f_2}=\set{f_1\circ
      f_2}=\set{(f_1\varphi^{-1})\circ(\varphi\circ
      f_2)}=\TodaBracket{f_1\varphi^{-1},\varphi f_2}.
  \]
  To prove the claim for $d\geq1$ it is enough to show that
  \[
    \TodaFam{f_1,f_2,f_3}=\TodaFam{f_1\varphi^{-1},\varphi f_2,f_3}.
  \]
  Moreover, it is enough to prove that
  \[
    \TodaFam{f_1,f_2,f_3}\subseteq\TodaFam{f_1\varphi^{-1},\varphi f_2,f_3}
  \]
  for the reverse induction is obtained by replacing $f_1$ and $f_2$ by
  $f_1\varphi^{-1}$ and $\varphi f_2$, respectively, and $\varphi$ by its
  inverse. Indeed, given a pair $(\beta,\alpha)\in\TodaFam{f_1,f_2,f_3}$ there
  is a commutative diagram in $\T$ of the form
  \begin{center}
    \begin{tikzcd}
      &x_3\ar{rr}{f_3}\dlar[swap]{\alpha}\ar[equals]{dd}&&x_2\ar{rr}{f_2}\ar[equals]{dl}\ar[equals]{dd}&&
      x_1\ar{rr}{f_1}\ar[equals]{dl}\ar{dd}[near start]{\varphi}&&x_0\ar[equals]{dd}\\
      c\ar[crossing over]{rr}[near start]{u}\ar[equals]{dd}&&%
      x_2\ar[crossing over]{rr}[near start]{f_2}&&%
      x_1\ar[crossing over]{rr}[near end]{v}&&c[1]\urar[swap]{\beta[1]}\\
      &x_3\ar{rr}[near start]{f_3}\dlar[swap]{\alpha}&&%
      x_2\ar{rr}[near start]{\varphi f_2}\ar[equals]{dl}&&%
      x_1'\ar{rr}[near start]{f_1\varphi^{-1}}\ar[equals]{dl}&&x_0\\
      c\ar{rr}{u}&&%
      x_2\ar{rr}{\varphi f_2}\ar[equals,crossing over]{uu}&&%
      x_1'\ar{rr}{v\varphi^{-1}}%
      \ar[leftarrow,crossing over]{uu}[swap,near end]{\varphi}&&%
      c[1]\urar[swap]{\beta[1]}\ar[equals,crossing over]{uu}
    \end{tikzcd}
  \end{center}
  in which the rows of the front face are exact triangles in $\T$ (the top layer
  exists by assumption). The bottom face of the diagram exhibits the membership
  $(\beta,\alpha)\in\TodaFam{f_1\varphi^{-1},\varphi f_2,f_3}$, which is what we
  needed to prove.
\end{proof}

\begin{notation}
  \label{not:TodaFamS}
  Consider a sequence of morphisms in $\T$ of the form
  \[
    x_3\xrightarrow{f_3}x_2\xrightarrow{f_2}x_1\xrightarrow{f_1}x_0
  \]
  and let $S$ be a set of exact triangles of the form
  \[
    c\to x_2\xrightarrow{f_2}x_1\to c[1].
  \]
  We let $\TodaFam{f_1,f_2,f_3}_S\subseteq\TodaFam{f_1,f_2,f_3}$ be the subset
  of all pairs $(\beta,\alpha)$ that are part of a commutative diagram of the
  form \eqref{eq:TodaFam} in which the bottom horizontal row belongs to the set
  $S$.
\end{notation}

\begin{lemma}
  \label{lemma:TodaBracket_S}
  Let $d\geq1$ and consider a sequence of morphisms in $\T$ of the form
  \[
    x_{d+2}\xrightarrow{f_{d+2}}x_{d+1}\xrightarrow{f_{d+1}}\cdots\xrightarrow{f_2}x_1\xrightarrow{f_1}x_0.
  \]
  Let $S$ be a non-empty class of triangles with middle morphism $f_2$ as in
  \Cref{not:TodaFamS}. Then, in the definition of the Toda bracket
  $\TodaBracket{f_1,\dots,f_{d+2}}$, we can replace $\TodaFam{f_1,f_2,f_3}$ by
  its subset $\TodaFam{f_1,f_2,f_3}_S$.
\end{lemma}
\begin{proof}
  Given an arbitrary exact triangle in $\T$ of the form
  \[
    c'\xrightarrow{u'}x_2\xrightarrow{f_2}x_1\xrightarrow{v'}c'[1],
  \]
  there exists an isomorphism of exact triangles
  \begin{center}
    \begin{tikzcd}
      c'\rar{u'}\dar{\varphi}&x_2\rar{f_2}\dar[equals]&x_1\rar{v'}\dar[equals]&c'[1]\dar{\varphi[1]}\\
      c\rar{u}&x_2\rar{f_2}&x_1\rar{v}&c[1]
    \end{tikzcd}
  \end{center}
  where the bottom row is an exact triangle in $S$ (chosen arbitrarily). Thus,
  given a commutative diagram in $\T$ of the form
  \begin{center}
    \begin{tikzcd}
      x_3\rar{f_3}\dar[swap]{\alpha}&x_2\rar{f_2}\dar[equals]&x_1\rar{f_1}\dar[equals]&x_0\\
      c'\rar{u'}&x_2\rar{f_2}&x_1\rar{v'}&c'[1]\uar[swap]{\beta[1]}
    \end{tikzcd}
  \end{center}
  there exists an extended commutative diagram
  \begin{center}
    \begin{tikzcd}
      &x_3\ar{rr}{f_3}\dlar[swap]{\alpha}\ar[equals]{dd}&&x_2\ar{rr}{f_2}\ar[equals]{dl}\ar[equals]{dd}&&x_1\ar{rr}{f_1}\ar[equals]{dl}\ar[equals]{dd}&&x_0\ar[equals]{dd}\\
      c'\ar[crossing over]{rr}[near start]{u'}\ar{dd}[near start]{\varphi}%
      &&x_2\ar[crossing over]{rr}[near start]{f_2}&&%
      x_1\ar[crossing over]{rr}[near end]{v'}&&c'[1]\urar[swap]{\beta[1]}\\
      &x_3\ar{rr}[near start]{f_3}\dlar[swap]{\varphi\alpha}&&%
      x_2\ar{rr}[near start]{f_2}\ar[equals]{dl}&&%
      x_1\ar{rr}[near start]{f_1}\ar[equals]{dl}&&x_0\\
      c\ar{rr}{u}&&x_2\ar{rr}{f_2}\ar[equals,crossing over]{uu}&&%
      x_1\ar{rr}{v}\ar[equals,crossing over]{uu}&&%
      c[1]\urar[swap]{(\beta\varphi^{-1})[1]}%
      \ar[leftarrow,crossing over]{uu}[swap,near end]{\varphi[1]}
    \end{tikzcd}
  \end{center}
  Finally, since by \Cref{lemma:TodaBracket_red-iso} there is an equality of
  Toda brackets
  \[
    \TodaBracket{\beta[1],\alpha[1],f_4[1],\dots,f_{d+2}[1]}=\TodaBracket{(\beta\varphi^{-1})[1],(\varphi\alpha)[1],f_4[1],\dots,f_{d+2}[1]},
  \]
  we see that we can replace the pair $(\beta,\alpha)\in\T(f_1,f_2,f_3)$
  by the pair
  $(\beta\varphi^{-1},\varphi\alpha)\in\TodaFam{f_1,f_2,f_3}_S$. This finishes the proof.
\end{proof}

We now describe the action of the shift functor on the Toda brackets.

\begin{definition}[{\cite[p.~2709]{CF17}}]
  \label{def:negative_TodaFam}
  Consider a sequence of morphisms in $\T$ of the form
  \[
    x_3\xrightarrow{f_3}x_2\xrightarrow{f_2}x_1\xrightarrow{f_1}x_0.
  \]
  The \emph{negative Toda family $-\TodaFam{f_1,f_2,f_3}$} is the set
  \[
    -\TodaFam{f_1,f_2,f_3}\coloneqq\set{(\beta,-\alpha)}[(\beta,\alpha)\in\TodaFam{f_1,f_2,f_3}]
  \]
  (notice the unusual convention for the negative of a set).
\end{definition}

\begin{remark}
  \label{rmk:negative_TodaFam}
  In the setting of \Cref{def:negative_TodaFam}, notice that
  \[
    -\TodaFam{f_1,f_2,f_3}=\TodaFam{f_1,-f_2,f_3}.
  \]
  Indeed, given a diagram of the form \eqref{eq:TodaFam} exhibiting the
  membership relation $(\beta,\alpha)\in\TodaFam{f_1,f_2,f_3}$, there is a
  commutative diagram in $\T$ of the form
  \begin{center}
    \begin{tikzcd}
      &x_3\ar{rr}{f_3}\dlar[swap]{\alpha}\ar{dd}[near end]{-1}&&%
      x_2\ar{rr}{f_2}\ar[equals]{dl}\ar{dd}[near end]{-1}&&%
      x_1\ar{rr}{f_1}\ar[equals]{dl}\ar[equals]{dd}&&x_0\ar[equals]{dd}\\
      c\ar[crossing over]{rr}[near start]{u}\ar[equals]{dd}&&%
      x_2\ar[crossing over]{rr}[near start]{f_2}&&%
      x_1\ar[crossing over]{rr}[near end]{v}&&%
      c[1]\urar[swap]{\beta[1]}\\
      &x_3\ar{rr}[near start]{f_3}\dlar[swap]{-\alpha}&&%
      x_2\ar{rr}[near start]{-f_2}\ar[equals]{dl}&&%
      x_1\ar{rr}[near start]{f_1}\ar[equals]{dl}&&x_0\\
      c\ar{rr}{-u}&&%
      x_2\ar{rr}{-f_2}\ar[leftarrow,crossing over]{uu}[swap,near start]{-1}&&%
      x_1\ar{rr}{v}\ar[equals,crossing over]{uu}&&%
      c[1]\urar[swap]{\beta[1]}\ar[equals,crossing over]{uu}
    \end{tikzcd}
  \end{center}
  that exhibits the pair $(\beta,-\alpha)\in-\TodaFam{f_1,f_2,f_3}$ as a member
  of $\TodaFam{f_1,-f_2,f_3}$. The converse is then also clear.
\end{remark}

\begin{lemma}
  \label{lemma:neg_TodaFam_f2}
  Consider a sequence of morphisms in $\T$ of the form
  \[
    x_{d+2}\xrightarrow{f_{d+2}}x_{d+1}\xrightarrow{f_{d+1}}\cdots\xrightarrow{f_2}x_1\xrightarrow{f_1}x_0.
  \]
  Then, there is an equality of Toda brackets
  \[
    -\TodaBracket{f_1,f_2,f_3,\dots,f_{d+2}}=\TodaBracket{f_1,-f_2,f_3,\dots,f_{d+2}}.
  \]
\end{lemma}
\begin{proof}
  We proceed by induction on the number $n=2$ of morphisms in the sequence. If
  $n=2$, then the claim is obvious. Suppose then that the claim holds for all
  sequences of $n$ composable morphisms in $\T$ for some $n\geq2$. Let
  $f_1,\dots,f_n,f_{n+1}$ be a sequence of composable morphisms in $\T$. Then,
  in view of \Cref{rmk:negative_TodaFam} and the inductive hypothesis,
  \begin{align*}
    \TodaBracket{f_1,-f_2,f_3,\dots,f_n,f_{n+1}} & =\bigcup_{(\beta,\gamma)\in\TodaFam{f_1,-f_2,f_3}}\TodaBracket{\beta[1],\gamma[1],f_4[1],\dots,f_n[1],f_{n+1}[1]} \\
                                                 & =\bigcup_{(\beta,\gamma)\in-\TodaFam{f_1,f_2,f_3}}\TodaBracket{\beta[1],\gamma[1],f_4[1],\dots,f_n[1],f_{n+1}[1]} \\
                                                 & =\bigcup_{(\beta,\alpha)\in\TodaFam{f_1,f_2,f_3}}\TodaBracket{\beta[1],-\alpha[1],f_4[1],\dots,f_n[1],f_{n+1}[1]} \\
                                                 & =\bigcup_{(\beta,\alpha)\in\TodaFam{f_1,f_2,f_3}}-\TodaBracket{\beta[1],\alpha[1],f_4[1],\dots,f_n[1],f_{n+1}[1]} \\
                                                 & =-\TodaBracket{f_1,f_2,f_3,\dots,f_n,f_{n+1}}.
  \end{align*}
  The claim follows.
\end{proof}

\begin{notation}
  Given a set $X$ of morphisms in $\T$, we let
  \[
    X[1]\coloneqq\set{f[1]}[f\in X],
  \]
  and similarly for sets of pairs of morphisms in $\T$ (we are thinking of Toda
  families as in \Cref{def:Toda_family}).
\end{notation}

\begin{lemma}[{\cite[Lemma~5.12]{CF17}}]
  \label{lemma:neg_TodaFam_shift}
  Consider a sequence of morphisms in $\T$ of the form
  \[
    x_3\xrightarrow{f_3}x_2\xrightarrow{f_2}x_1\xrightarrow{f_1}x_0.
  \]
  Then, there is an equality
  \[
    \TodaFam{f_1[1],f_2[1],f_3[1]}=-\TodaFam{f_1,f_2,f_3}[1].
  \]
\end{lemma}

\begin{lemma}
  \label{lemma:TodaBracket_shift}
  Consider a sequence of morphisms in $\T$ of the form
  \[
    x_{d+2}\xrightarrow{f_{d+2}}x_{d+1}\xrightarrow{f_{d+1}}\cdots\xrightarrow{f_2}x_1\xrightarrow{f_1}x_0.
  \]
  Then, the following relation between Toda brackets holds:
  \[
    \TodaBracket{f_1,f_2,\dots,f_{d+2}}[1]=(-1)^d\TodaBracket{f_1[1],f_2[1],f_3[1],\dots,f_{d+2}[1]}.
  \]
\end{lemma}
\begin{proof}
  We proceed by induction on the number $n$ of morphisms in the sequence. If
  $n=2$, then the claim follows from the functoriality of the shift:
  \[
    \TodaBracket{f_1,f_2}[1]=\set{f_1\circ f_2}{[1]}=\set{(f_1\circ
      f_2)[1]}=\set{f_1[1]\circ f_2[1]}=\TodaBracket{f_1[1],f_2[1]}.
  \]
  Suppose that the claim holds for all sequences of $n$ composable morphisms in
  $\T$ for some $n\geq 2$. Let $f_1,\dots,f_n,f_{n+1}$ be a sequence of
  composable morphisms in $\T$. Then, in view of
  \Cref{lemma:neg_TodaFam_f2,lemma:neg_TodaFam_shift} and the inductive
  hypothesis, the Toda bracket
  \[
    \TodaBracket{f_1,f_2,\dots,f_n,f_{n+1}}[1]
  \]
  equals
  \begin{align*}
    & \bigcup_{(\beta,\alpha)\in\TodaFam{f_1,f_2,f_3}}\TodaBracket{\beta[1],\alpha[1],f_4[1],\dots,f_n[1],f_{n+1}[1]}[1]                          \\
    & =\bigcup_{(\beta,\alpha)\in\TodaFam{f_1,f_2,f_3}}(-1)^{n-2}\TodaBracket{\beta[2],\alpha[2],f_4[2],\dots,f_n[2],f_{n+1}[2]}                  \\
    & =\bigcup_{(\beta,\alpha)\in\TodaFam{f_1,f_2,f_3}}(-1)^{n-1}\TodaBracket{\beta[2],-\alpha[2],f_4[2],\dots,f_n[2],f_{n+1}[2]}                 \\
    & =\bigcup_{(\beta,-\alpha)\in-\TodaFam{f_1,f_2,f_3}}(-1)^{n-1}\TodaBracket{\beta[2],-\alpha[2],f_4[2],\dots,f_n[2],f_{n+1}[2]}               \\
    & =\bigcup_{(\beta[1],-\alpha[1])\in-\TodaFam{f_1,f_2,f_3}[1]}(-1)^{n-1}\TodaBracket{\beta[2],-\alpha[2],f_4[2],\dots,f_n[2],f_{n+1}[2]}      \\
    & =\bigcup_{(\beta[1],-\alpha[1])\in\TodaFam{f_1[1],f_2[1],f_3[1]}}(-1)^{n-1}\TodaBracket{\beta[2],-\alpha[2],f_4[2],\dots,f_n[2],f_{n+1}[2]} \\
    & =(-1)^{n-1}\TodaBracket{f_1[1],\dots,f_n[1],f_{n+1}[1]}.
  \end{align*}
  The claim follows.
\end{proof}

We combine the above results into the following equivalent way to compute the
Toda bracket of a sequence of morphisms.

\begin{proposition}
  \label{prop:TodaBracket_alt}
  Let $d\geq1$ and consider a sequence of morphisms in $\T$ of the form
  \[
    x_{d+2}\xrightarrow{f_{d+2}}x_{d+1}\xrightarrow{f_{d+1}}\cdots\xrightarrow{f_2}x_1\xrightarrow{f_1}x_0.
  \]
  Then, there is an equality
  \[
    \TodaBracket{f_1,\dots,f_{d+2}}=\bigcup_{(\beta,\alpha)\in\TodaFam{f_1,f_2,f_3}}(-1)^{d-1}\TodaBracket{\beta,\alpha,f_4,\dots,f_{d+2}}[1].
  \]
\end{proposition}
\begin{proof}
  According to \Cref{def:Toda_bracket} and \Cref{lemma:TodaBracket_shift}, there
  are equalities
  \begin{align*}
    \TodaBracket{f_1,\dots,f_{d+2}} & =\bigcup_{(\beta,\alpha)\in\TodaFam{f_1,f_2,f_3}}\TodaBracket{\beta[1],\alpha[1],f_4[1],\cdots,f_{d+2}[1]}          \\
                                    & =\bigcup_{(\beta,\alpha)\in\TodaFam{f_1,f_2,f_3}}(-1)^{d-1}\TodaBracket{\beta,\alpha,f_4,\dots,f_{d+2}}[1].\qedhere
  \end{align*}
\end{proof}

\begin{proposition}
  \label{prop:theSign}
  Let $d\geq1$ and consider a sequence of morphisms in $\T$ of the form
  \[
    \begin{tikzcd}
      x_{d+2}\rar{f_{d+2}}&x_{d+1}\rar{f_{d+1}}&\cdots\rar{f_3}&x_2\rar{f_2}&x_1\rar{f_1}&x_{d+2}[d].
    \end{tikzcd}
  \]
  Suppose that the above sequence fits as the spine of a commutative diagram of
  the form
  \begin{center}
    \begin{tikzcd}[column sep=small]
      &x_{d+1}\drar&\cdots&&&x_3\drar\ar{rr}&&x_2\drar\\
      x_{d+2}\urar&&x_{d.5}\ar{ll}[description]{+1}&\cdots&x_{3.5}\urar&&x_{2.5}\urar\ar{ll}[description]{+1}&&x_1\ar{ll}[description]{+1}
    \end{tikzcd}
  \end{center}
  in which the oriented triangles are exact triangles in $\T$ and the morphism
  $f_1$ is the apparent (shifted) composite along the bottom row of the diagram.
  Then,
  \[
    \id[x_{d+2}[d]]\in(-1)^{1+\sum_{i=1}^{d+1}i}\TodaBracket{f_1,f_2,\dots,f_{d+2}}.
  \]
\end{proposition}
\begin{proof}
  We proceed by induction on $d$. The case $d=1$ follows from
  \Cref{thm:TodaBracket_3-angles}. Suppose then that the claim holds for
  $d\geq1$ and consider a sequence of morphisms $f_1,f_2,\dots,f_{d+2},f_{d+3}$
  that satisfies the conditions in the statement of the proposition. Thus, there
  are exact triangles
  \[
    \begin{tikzcd}
      x_{i+1.5}\rar{h_{i+1}}&x_{i+1}\rar{g_{i+1}}&x_{i.5}\rar{u_{i}}&x_{i+1.5}[1],\qquad{i=1,\dots,d+1},
    \end{tikzcd}
  \]
  such that the diagrams
  \[
    \begin{tikzcd}[column sep=small]
      x_{i+1}\ar{rr}{f_{i+1}}\drar[swap]{g_{i+1}}&&x_{i}\\
      &x_{i.5}\urar[swap]{h_i}
    \end{tikzcd},\qquad%
    i=2,\dots,d+1,
  \]
  commute, where we set $x_{1.5}\coloneqq x_1$ and $x_{d+2.5}\coloneqq x_{d+3}$,
  so that also $g_2=f_2$ and $h_{d+2}=f_{d+3}$. We also set
  \[
    f_1^{(d)}=(u_{d+1}[d-1])\circ\cdots\circ(u_3[1])\circ u_2
  \]
  and notice that $f_1^{(d)}\colon x_{2.5}\to x_{d+3}[d]$ and
  $f_1=(f_1^{(d)}[1])\circ u_1$. Having fixed the necessary notation, the
  induction hypothesis implies that
  \[
    \id[x_{d+3}[d]]\in(-1)^{1+\sum_{i=1}^{d+1}i}\TodaBracket{f_1^{(d)},g_3,f_4,\dots,f_{d+2},f_{d+3}}.
  \]
  We claim that $(f_1^{(d)},g_3)\in\TodaFam{f_1,f_2,f_3}$ and therefore,
  according to \Cref{prop:TodaBracket_alt},
  \begin{align*}
    \TodaBracket{f_1,f_2,\dots,f_{d+2},f_{d+3}}&=\bigcup_{(\beta,\alpha)\in\TodaFam{f_1,f_2,f_3}}(-1)^d\TodaBracket{\beta,\alpha,f_3,\dots,f_{d+2},f_{d+3}}[1] \\
                                                &\supseteq(-1)^{d}\TodaBracket{f_1^{(d)},g_3,f_4,\dots,f_{d+2},f_{d+3}}[1]                                      \\
                                                &\ni(-1)^{d+2}((-1)^{1+\sum_{i=1}^{d+1}i}\id[x_{d+3}[d]])[1]                                                    \\
                                                &=(-1)^{1+\sum_{i=1}^{d+2}i}\id[x_{d+3}[d+1]].
  \end{align*}
  To prove the claim, it suffices to observe that the following diagram commutes:
  \[
    \begin{tikzcd}
      x_3\rar{f_3}\dar{g_3}&x_2\rar{f_2}\dar[equals]&x_1\rar{f_1}\dar[equals]&x_{d+3}[d+1]\\
      x_{2.5}\rar{h_2}&x_2\rar{f_2}&x_1\rar{u_1}&x_{2.5}[1]\uar[swap]{f_1^{(d)}[1]}.
    \end{tikzcd}
  \]
  This finishes the proof.
\end{proof}

We also need the following lemma.

\begin{lemma}
  \label{lemma:TodaBracket-3-id}
  Consider a sequence of morphisms in $\T$ of the form
  \[
    x_3\xrightarrow{f_3}x_2\xrightarrow{f_2}x_1\xrightarrow{f_1}x_3[1].
  \]
  The following statements are equivalent:
  \begin{enumerate}
  \item We have $\id[x_3[1]]\in\TodaBracket{f_1,f_2,f_3}$.
  \item There exist an object $y\in\T$ and morphisms $g\colon y\to x_2$ and
    $h\colon x_1\to y[1]$ such that the sequence
    \begin{center}
      \begin{tikzcd}[ampersand replacement=\&]
        x_3\oplus{y}\rar{\smallpmatrix{f_3&g}}\&x_2\rar{f_2}\&x_1\rar{\smallpmatrix{f_1\\h}}\&(x_3\oplus{y})[1]
      \end{tikzcd}
    \end{center}
    is an exact triangle in $\T$.
  \end{enumerate}
\end{lemma}
\begin{proof}
  By definition, $\id[x_2[1]]\in\TodaBracket{f_0,f_1,f_2}$ if and only if there
  exists a commutative diagram in $\T$ of the form
  \begin{center}
    \begin{tikzcd}
      x_3\rar{f_3}\dar[swap]{\alpha}&x_2\rar{f_2}\dar[equals]&x_1\rar{f_1}\dar[equals]&x_3[1]\\
      c\rar{u}&x_2\rar{f_2}&x_1\rar{v}&c[1]\uar[swap]{\beta[1]}
    \end{tikzcd}
  \end{center}
  in which the bottom row is an exact triangle in $\T$ and
  $(\beta\alpha)[1]=\id[x_3[1]]$ or, equivalently $\beta\alpha=\id[x_3]$. In
  particular, $c\cong x_3\oplus y$ for some object $y\in\T$. If we identify $c$
  with $x_3\oplus y$ we see that the morphism $u$ has components $f_3\colon
  x_3\to x_2$ and $g\colon y\to x_2$ for some morphism $g$. Similarly, under the
  identification $c[1]=(x_3\oplus y)[1]=x_3[1]\oplus y[1]$ the morphism $v$ has
  components $f_1\colon x_1\to x_3[1]$ and $h\colon x_1\to y[1]$ for some
  morphism $h$. The claim follows.
\end{proof}

In the proof of \Cref{thm:TodaBracket_d+2-angles} we will use elementary properties of Poincaré polynomials of periodic graded modules.

\begin{definition}
  Let $A$ be a graded algebra with a degree $d$ unit. The \emph{Poincaré polynomial} of a degree-wise finite-dimensional graded $A$-module $M$ is 
  \[p(M)=\sum_{i=0}^{d-1}\dim_{\kk}M^i\cdot t^i\in\ZZ[t]/(t^d-1).\]
  Here $M^i$ denotes the degree $i$ component of $M$.
\end{definition}

We now show some elementary properties of these Poincaré polynomials. All proofs are simple, albeit somewhat tedious, exercises, so we only sketch them.

\begin{proposition}
  \label{prop:poincare}
  Given a graded algebra $A$ with a degree $d$ unit $u\in A^d$, the Poincaré polynomials of degree-wise finite-dimensional graded $A$-modules satisfy the following properties:
  \begin{enumerate}
    \item $p(M\oplus N)=p(M)+p(N)$.
    \item\label{it:short} Given a short exact sequence of degree $0$ morphisms $0\to M\to N\to P\to 0$, $p(N)=p(M)+p(P)$.
    \item\label{it:zero} $p(M)=0\Leftrightarrow M=0$.
    \item\label{it:poincare_Z} $p(M)\in\ZZ$ if and only if $M^i=0$ for $i\notin d\ZZ$.
    \item\label{it:poincare_Z_sub} If $p(M)\in\ZZ$ and $N\subset M$ is a graded sub-$A$-module then $p(N)\in\ZZ$.
    \item\label{it:shift} $p(M(-1))=t\cdot p(M)$.
    \item\label{it:d-shift} $p(M(d))=p(M)$.
    \item\label{it:long} Given an exact sequence of graded $A$-modules
    \begin{center}
      \begin{tikzcd}
        &N\ar[rd,"g"]&\\
        M\ar[ru,"f"]&&P\ar[ll,"+1" description, near start,"h" midway]
      \end{tikzcd}
    \end{center}
    consisting of two degree $0$ morphisms $f,g$ and a degree $+1$ morphism $h$,
    if $K=\ker g=\operatorname{im} f$ then
    \[p(M)+tp(N)-tp(P)=(t+1)p(K).\]
  \end{enumerate}
\end{proposition}
\begin{proof}
  The first five properties are obvious. For \eqref{it:shift}, we use the following three facts:
  \begin{itemize}
    \item $M(-1)^i=M^{i-1}$,
    \item multiplication by $u$ induces an isomorphism $M(-1)^0=M^{-1}\cong M^{d-1}$,
    \item $t^d=1$ in $\ZZ[t]/(t^d-1)$.
  \end{itemize}
  Property \eqref{it:d-shift} is a direct consequence of \eqref{it:shift}.
  For \eqref{it:long}, we break the exact sequence of degree $0$ morphisms
  \[0\to K(-1)\to N(-1)\stackrel{g}{\to} P(-1)\stackrel{h}{\to} M\stackrel{f}{\to} K\to 0\]
  into three short exact sequences and apply \eqref{it:short} and \eqref{it:shift}.
\end{proof}

\begin{proposition}
  \label{prop:circle_polynomial}
  Let $d\geq 1$. Given a graded algebra $A$ with a degree $d$ unit and an exact sequence of degree-wise finite-dimensional $A$-modules and degree $0$ morphisms,
  \begin{center}
    \begin{tikzcd}
      &C_{d+1}\ar[r,"f_{d+1}"]&\cdots\ar[r,"f_{3}"]&C_2\ar[rd,"f_{2}"]&\\
      C_{d+2}\ar[ru,"f_{d+2}"]&&&&C_1\ar[llll,"+d" description, near start, "f_{d+2}" midway]
    \end{tikzcd}
  \end{center}
  if $K=\ker f_2$ then the following relation between Poincaré polynomials holds,
  \[(-1)^d\left(p(C_2)-p(C_1)\right)+\sum_{i=3}^{d+2}(-1)^ip(C_i)=\left((-1)^d-1\right)p(K).\]
\end{proposition}

\begin{proof}
  Consider $K_i=\ker f_i$ and the short exact sequences
  \[\begin{array}{rr}
    0\to K_i\to C_i\to K_{i-1}\to 0,&2\leq i\leq d+2,\\
    0\to K_1\to C_1\to K_{d+2}(d)\to 0.
  \end{array}\]
  Now it is just a matter of applying \eqref{it:short} and \eqref{it:d-shift} from \Cref{prop:poincare}.
\end{proof}

\begin{proposition}
  \label{prop:poincare_consecutive}
  Let $d\geq 1$. Given a graded algebra $A$ with a degree $d$ unit and exact sequences of degree-wise finite-dimensional $A$-modules,
  \begin{center}
    \begin{tikzcd}[column sep=small]
      &M_{d+1}\ar[rd,"g_{d+1}"]&&M_{d}\ar[rd,"g_{d}"]&\cdots&&&M_{2}\ar[rd,"g_{2}"]&\\
      M_{d+2}\ar[ru]&&X_{d.5}\ar[ru]\ar[ll,"+1" description]&&X_{d-1.5}\ar[ll,"+1" description]&\cdots&X_{2.5}\ar[ru]&&M_{1}\ar[ll,"+1" description]
    \end{tikzcd}
  \end{center}
  where diagonal arrows are degree $0$ morphisms and bottom horizontal arrows are degree $+1$ morphisms,
  if $K_i=\ker g_i$ then
  \begin{multline*}
    p(M_2)+p(M_{d+2})-p(M_1)+\sum_{i=3}^{d+1}p(M_i)t^{d+2-i}=\\p(K_2)+p(K_{d+1})+\sum_{i=3}^{d+1}\left(p(K_i)+p(K_{i-1})\right)t^{d+2-i}
  \end{multline*}
\end{proposition}

\begin{proof}
  This follows from \Cref{prop:poincare} \eqref{it:long}.
\end{proof}

\begin{corollary}
  \label{cor:poincare}
  In the setting of \Cref{prop:poincare_consecutive}, if the $M_i$ are concentrated in degrees $d\ZZ$ then 
  \[(-1)^d\left(p(M_2)-p(M_1)\right)+\sum_{i=3}^{d+2}(-1)^ip(M_i)=\left((-1)^d-1\right)p(K_2).\]
\end{corollary}

\begin{proof}
  By \Cref{prop:poincare} \eqref{it:poincare_Z} and \eqref{it:poincare_Z_sub}, $p(M_i),p(K_i)\in\ZZ$. In the polynomial equation of \Cref{prop:poincare_consecutive}, both polynomials must have the same coefficients in the same degrees. This translates into $d$ equations involving $p(M_i),p(K_i)\in\ZZ$. From these equations, it is easy to deduce the equation in the statement.
\end{proof}

We are ready to prove \Cref{thm:TodaBracket_d+2-angles}.

\begin{proof}[Proof of \Cref{thm:TodaBracket_d+2-angles}]
  Let $\C\subseteq\T$ be a $d$-rigid subcategory such that $\add*{\C}=\C$ and
  $\C[d]=\C$, and
  \begin{equation}
    \label{eq:theSeqThm}
    \begin{tikzcd}
      c_{d+2}\rar{f_{d+2}}&c_{d+1}\rar{f_{d+1}}&\cdots\rar{f_3}&c_2\rar{f_2}&c_1\rar{f_1}&c_{d+2}[d]
    \end{tikzcd}
  \end{equation}
  a sequence of morphisms in $\C$ with $d\geq1$.

  \eqref{it:TodaBracket_d+2-angles_isstd}$\Rightarrow$\eqref{it:TodaBracket_d+2-angles_isAL}
  Suppose that the sequence \eqref{eq:theSeqThm} satisfies the conditions in
  \Cref{thm:TodaBracket_d+2-angles}\eqref{it:TodaBracket_d+2-angles_isstd}. The
  fact that, for each $y\in\C$, the induced sequence of vector spaces
  \begin{center}
    $\T(y,c_{d+2})\to\cdots\to\T(y,c_2)\to\T(y,c_1)\to\T(y,c_{d+2}[d])\to\T(y,c_{d+1}[d])$
  \end{center}
  is exact follows from standard arguments using that $\C\subseteq\T$ is a
  $d$-rigid subcategory, see for example the first paragraph in the proof of
  \cite[Lemma~4.6]{Lin19}. That
  \[
    \id[c_{d+2}[d]]\in(-1)^{1+\sum_{i=1}^{d+1}}\TodaBracket{f_1,f_2,\dots,f_{d+2}}
  \]
  follows from \Cref{prop:theSign}.

  \eqref{it:TodaBracket_d+2-angles_isAL}$\Rightarrow$\eqref{it:TodaBracket_d+2-angles_isstd}
  If $d=1$, then we are precisely in the setting of
  \Cref{thm:TodaBracket_3-angles}; hence, we may assume that $d>1$. Inductively,
  we shall construct a commutative diagram
  \begin{equation}\label{eq:theBigDiagram}
    \begin{tikzcd}[column sep=small]
      &c_{d+1}\drar{g_{d+1}}&\cdots&&&c_3\drar{g_3}\ar{rr}{f_3}&&c_2\drar{f_2}\\
      c_{d+2}\urar{f_{d+2}}&&x_{d.5}\ar{ll}[near
      end]{u_d}\ar{ll}[description,near
      start]{+1}&\cdots&x_{3.5}\urar{h_3}&&x_{2.5}\urar{h_2}\ar{ll}[near end]{u_2}\ar{ll}[description,near
      start]{+1}&&c_1\ar{ll}[near end]{u_1}\ar{ll}[description,near start]{+1}
    \end{tikzcd}
  \end{equation}
  with the following properties:
  \begin{itemize}
  \item The triangles
    \begin{center}
      $x_{k.5}\xrightarrow{h_k}c_k\xrightarrow{g_k}
      x_{k-1.5}\xrightarrow{u_{k-1}} x_{k.5}[1],\qquad 2\leq k\leq d,$
    \end{center}
    are exact, where $x_{1.5}=c_1$ and $g_2=f_2$;
  \item We have $\id[c_{d+2}[1]]\in\TodaBracket{u_d,g_{d+1},f_{d+2}}$, although the
    triangle
    \begin{center}
      $c_{d+2}\xrightarrow{f_{d+2}}c_{d+1}\xrightarrow{g_{d+1}}x_{d.5}\xrightarrow{u_d}
      c_{d+2}[1]$
    \end{center}
    is not (yet) known to be exact.
  \item There are equalities $f_k=h_{k-1}\circ g_k$, for all $3\leq k\leq d+1$.
  \item There is an equality $f_1= u_d[d-1]\circ\cdots\circ u_2[1]\circ u_1$.
  \end{itemize}
  In the first step we notice that, since by \Cref{prop:TodaBracket_alt} the morphism
  $(-1)^{1+\sum_{i=1}^{d+1}i}\id[c_{d+2}[d]]$ lies in
  \[
    \TodaBracket{f_1,f_2,\dots,f_{d+2}}=\bigcup_{(\beta,\alpha)\in\TodaFam{f_1,f_2,f_3}}(-1)^{d-1}\TodaBracket{\beta,\alpha,f_4,\dots,f_{d+2}}[1],
  \]
  there exist a pair of morphisms $(f_1^{(d)},g_3)\in\TodaFam{f_1,f_2,f_3}$ such
  that
  \[
    (-1)^{1+\sum_{i=1}^{d+1}i}\id[c_{d+2}[d]]\in(-1)^{d-1}\TodaBracket{f_1^{(d)},g_3,f_4,\dots,f_{d+2}}[1]
  \]
  or, equivalently,
  \begin{equation}\label{eq:noseque}
    (-1)^{1+\sum_{i=1}^{d}i}\id[c_{d+2}[d-1]]\in\TodaBracket{f_1^{(d)},g_3,f_4,\dots,f_{d+2}}.
  \end{equation}
  In particular, by \Cref{rmk:TodaBracket_betaalpha}, there exists a diagram of the form
  \begin{center}
    \begin{tikzcd}[column sep=normal]
      &c_2\ar{rr}{f_2}&&c_1\drar{f_1}\ar{dl}[near end,swap]{u_1}\ar{dl}[description,near start]{+1}\\
      c_3\urar{f_3}\ar{rr}{g_3}&&%
      x_{2.5}\ar{rr}[near
      end]{f_1^{(d)}}\ular[swap]{h_2}\ar{rr}[description,near start]{-1}&&c_{d+2}[d]
    \end{tikzcd}
  \end{center}
  in which the oriented triangle is exact. Thus, we obtain a diagram of the form
  \begin{center}
    \begin{tikzcd}[column sep=small]
      &c_{d+1}\rar{f_{d+1}}&{}&\cdots&{}\rar{f_4}&c_3\drar{g_3}\ar{rr}{f_3}&&c_2\drar{f_2}\\
      c_{d+2}\urar{f_{d+2}}&&&&&&x_{2.5}\urar{h_2}\ar{llllll}[near end]{f_1^{(d)}}%
      \ar{llllll}[description,near start]{+(d-1)}&&%
      c_1\ar{ll}[near end]{u_1}\ar{ll}[description,near start]{+1}%
      \arrow[llllllll, "f_1"', rounded corners, to path={--([yshift=-1em]\tikztostart.south)--([yshift=-1em]\tikztotarget.south)\tikztonodes--(\tikztotarget.south)}]
    \end{tikzcd}
  \end{center}
  We may then continue the inductive procedure with the sequence
  $f_1^{(d)},g_3,f_4,\dots,f_{d+2}$. This clearly finishes the construction of \eqref{eq:theBigDiagram} with the required properties since the analogue of \eqref{eq:noseque} in the last step reads
  \[
    (-1)^{1+(1+2)}\id[c_{d+2}[1]]=\id[c_{d+2}[1]]\in\TodaBracket{u_d,g_{d+1},f_{d+2}},
  \]
  as required). It remains to prove that the sequence
  \begin{center}
    $c_{d+2}\xrightarrow{f_{d+2}}c_{d+1}\xrightarrow{g_{d+1}}x_{d.5}\xrightarrow{u_d}
    c_{d+2}[1]$
  \end{center}
  is an exact triangle. Since $\id[c_{d+2}[1]]\in\TodaBracket{u_d,g_{d+1},f_{d+2}}$,
  \Cref{lemma:TodaBracket-3-id} shows that there exists an object $y\in\T$ and an
  exact triangle of the form
  \begin{center}
    $y\oplus c_{d+2}\xrightarrow{\smallpmatrix{f_{d+2}&g}}c_{d+1}\xrightarrow{g_{d+1}}x_{d.5}\xrightarrow{\smallpmatrix{u_d\\h}}
    (y\oplus c_{d+2})[1].$
  \end{center}
  We need to prove that $y=0$; since $\thick*{c}=\thick*{\C}=\T$ it suffices
  to prove that $\T(c[i],y)=0$ for all $i\in\ZZ$. Moreover, since $\C[d]=\C=\add*{c}$ and $c$ is basic, $c\cong c[d]$ so the graded endomorphism algebra $A$ of $c$ in $\T$ has a degree $d$ unit. Moreover, $M_x=\bigoplus_{i\in\ZZ}\T(c[i],x)$ is a degree-wise finite-dimensional graded $A$-module for any $x\in \T$. We have to show that $M_y=0$. We will denote $p(x)=p(M_x)$, and we will show that $p(y)=0$, which suffices by \Cref{prop:poincare} \eqref{it:zero}.

  Applying $\bigoplus_{i\in\ZZ}\T(c[i],-)$ to 
  \begin{center}
    \begin{tikzcd}
      &c_{d+1}\ar[r,"f_{d+1}"]&\cdots\ar[r,"f_{3}"]&c_2\ar[rd,"f_{2}"]&\\
      c_{d+2}\ar[ru,"f_{d+2}"]&&&&c_1\ar[llll,"+d" description, near start, "f_{d+2}" midway]
    \end{tikzcd}
  \end{center}
  and
  \begin{center}
    \begin{tikzcd}[column sep=small, ampersand replacement = \&]
      \&c_{d+1}\ar[rd,"g_{d+1}"]\&\&c_{d}\ar[rd,"g_{d}"]\&\cdots\&\&\&c_{2}\ar[rd,"f_{2}"]\&\\
      y\oplus c_{d+2}\ar[ru,"\smallpmatrix{f_{d+2}& g}"]\&\&x_{d.5}\ar[ru,"h_d"]\ar[ll,"+1" description, near start, "\smallpmatrix{u_{d+2}\\h}" xshift=-8mm]\&\&x_{d-1.5}\ar[ll,"+1" description, near start, "u_{d-1}" xshift=-6mm]\&\cdots\&x_{2.5}\ar[ru,"h_2"]\&\&c_{1}\ar[ll,"+1" description, near start, "u_1" near end]
    \end{tikzcd}
  \end{center}
  we obtain diagrams of $A$-modules as in \Cref{prop:circle_polynomial} and \Cref{prop:poincare_consecutive}. 
  For the first case, we use the standing hypothesis \eqref{it:TodaBracket_d+2-angles_exactness}, and for the second case we use that all triangles are exact. By \Cref{prop:circle_polynomial}, if $K=\ker\bigoplus_{i\in\ZZ}\T(c[i],f_2)$, 
  \[(-1)^d\left(p(c_2)-p(c_1)\right)+\sum_{i=3}^{d+2}(-1)^ip(c_i)=\left((-1)^d-1\right)p(K).\]
  Moreover, since $c_i\in\C$ and $\C$ is $d$-rigid, $M_{c_i}$ is concentrated in $d\ZZ$. Hence, by \Cref{cor:poincare},
  \[(-1)^d\left(p(c_2)-p(c_1)\right)+\sum_{i=3}^{d+1}(-1)^ip(c_i)+p(y)+p(c_{d+2})=\left((-1)^d-1\right)p(K).\]
  Both equations are identical, except for the fact that the second one has an extra term, $p(y)$. Therefore, $p(y)=0$, as we wanted to prove.
\end{proof}

We conclude this section with the following result.

\begin{proposition}[{\cite[Prop.~4.10]{Sag08}}]
  \label{defprop:TodaBracket_indeterminacy}
  Let $\C\subseteq\T$ be a $d$-rigid subcategory such that $\add*{\C}=\C$ and
  $\C[d]=\C$. Consider a sequence of morphisms in $\C$ of the form
  \[
    c_{d+2}\xrightarrow{f_{d+2}}c_{d+1}\xrightarrow{f_{d+1}}\cdots\xrightarrow{f_2}c_1\xrightarrow{f_1}c_0
  \]
  Then, the Toda bracket $\TodaBracket{f_1,\dots,f_{d+2}}$ is an element of the
  quotient
  \[
    \Toda{f_{d+2}}{f_1}\coloneqq
    \frac{\T(c_{d+2}[d],c_0)}{f_1\cdot\T(c_{d+2}[d],c_1)+\T(c_{d+1}[d],c_0)\cdot
      f_{d+2}[d]}.
  \]
\end{proposition}

\subsection{Massey products versus Toda brackets}

Let $\A$ be a DG category. We recall that, similarly to the cohomology of a
space, the graded category $\dgH{\A}$ is endowed with so-called Massey
products~\cite[Sec.~5.A]{BK90}. Massey products are closely related to Toda
brackets. In fact, it is stated in \loccit~that
\begin{quote}
  ``It can be verified that for an enhanced triangulated category the Massey
  products coincide with the Toda brackets.''
\end{quote}
As it turns out, this statement indeed holds, but only up to a suitable sign.
Since this agreement is crucial to our main results, we include a proof for the
convenience of the reader (see also \cite{Bod14} for an alternative proof that
involves the language of twisted complexes).

We begin by recalling the definition of the Massey products, see for example
\cite{May69}.

\begin{definition}
  \label{def:MasseyProduct}
  Let $\A$ be a DG category with differential $\partial$ and consider a sequence in $\dgH{\A}$ of the
  form
  \[
    x_{d+2}\xrightarrow{f_{d+2}}x_{d+1}\xrightarrow{f_{d+1}}\cdots\xrightarrow{f_2}x_1\xrightarrow{f_1}x_0
  \]
  such that $f_i$ is homogeneous of degree $|f_i|$, where $d\geq0$.
  \begin{enumerate}
  \item A \emph{defining system (for the above sequence)} is a set
    \[
      \set{g_{ij}\colon x_j\to x_i}[0\leq i<j\leq d+2,\ j-i<d+2]
    \]
    of morphisms in $\A$ that are homogeneous of degree
    \[
      |g_{ij}|=i-j+1+\sum_{i<k\leq j}|f_k|
    \]
    with the following two properties:
    \begin{enumerate}
    \item The equality
        $$\partial(g_{ij})=\sum_{i<k<j}(-1)^{|g_{ik}|-1}g_{ik}g_{kj}$$
      holds (notice that $|g_{ik}|-1=i-k+\sum_{i<\ell\leq k}|f_\ell|$). In
      particular,
      \[
        \partial(g_{i-1,i})=0 \] so that $g_{i-1,i}$ is a cocycle of degree
      $|f_i|$. \item The morphism $g_{i-1,i}$ represents $f_i$ in $\dgH{\A}$.
    \end{enumerate}
  \item The \emph{Massey product} is the subset
    \[
      \MasseyProduct{f_1,\dots,f_{d+2}}\subseteq\dgH[-d+\sum_{i=1}^{d+2}|f_i|]{\A(x_{d+2},x_0)}
    \]
    consisting of the classes of the cocycles
    \[
      \sum_{0<k< d+2}(-1)^{|g_{0k}|-1}g_{0k}g_{k,d+2}=\sum_{0<k<
        d+2}(-1)^{-k+\sum_{0<\ell\leq k}|f_\ell|}g_{0k}g_{k,d+2}
    \]
    for all possible defining systems.
  \end{enumerate}
\end{definition}

\begin{remark}
  In the setting of \Cref{def:MasseyProduct}, the degree of the Massey product
  $\MasseyProduct{f_1,\dots,f_{d+2}}$ is more recognisable by letting $n=d+2$ so
  that $\MasseyProduct{f_1,\dots,f_n}$ is a subset of the cohomology of
  $\A(x_n,x_0)$ in degree $2-n+\sum_i |f_i|$.
\end{remark}

\begin{remark}
  \label{rmk:MasseyProduct_d0}
  The Massey product $\MasseyProduct{f_1,f_2}$ of a pair of (homogeneous)
  composable morphisms in $\dgH{\A}$ is always defined and is given by the
  singleton
  \[
    \MasseyProduct{f_1,f_2}=\set{(-1)^{|f_1|-1}f_1\circ f_2}.
  \]
  We see that, if $|f_1|=0$, the Massey product of $f_1$ and $f_2$ is the
  sign-twisted composite $-f_1\circ f_2$ (compare with \Cref{def:Toda_bracket}
  and \Cref{thm:TodaMassey}).
\end{remark}

\begin{remark}
  \label{rmk:MasseyProduct_non-empty}
  In the setting of \Cref{def:MasseyProduct}, if the Massey product
  $\MasseyProduct{f_1,\dots,f_{d+2}}$ is non-empty, then
  \[
    0\in\MasseyProduct{f_i,\dots,f_j},\qquad 0\leq i<j\leq d+2,\ j-i<d+2.
  \]
  If $d=1$, this condition is also sufficient. On the other hand, if $d=2$, then
  it is also necessary that the corresponding \emph{coindeterminacy} also contains $0$,
  see~\cite{Isa15} for details and compare with \Cref{prop:MasseyProduct_Ainfty_m}.
\end{remark}

\begin{remark}
  Let $\A$ be a small DG category. Consider a sequence of degree $0$
  morphisms
  \[
    M_{d+2}\xrightarrow{f_{d+2}}M_{d+1}\xrightarrow{f_{d+1}}\cdots\xrightarrow{f_2}M_1\xrightarrow{f_1}M_0
  \]
  in the triangulated category $\dgH[0]{\dgModdg*{\A}}$. By definition, their
  Massey product satisfies
  \[
    \MasseyProduct{f_1,\dots,f_{d+2}}\subseteq\dgH[-d]{\dgHom[\A]{M_{d+2}}{M_0}}\cong\dgH[0]{\dgModdg*{\A}}(M_{d+2},M_{0}[-d]),
  \]
  while their Toda bracket satisfies
  \[
    \TodaBracket{f_1,\dots,f_{d+2}}\subseteq\dgH[0]{\dgModdg*{\A}}(M_{d+2}[d],M_{0}).
  \]
  Of course, we may compare both brackets by means of the isomorphism
  \[
    \dgH[0]{\dgModdg*{\A}}(M_{d+2},M_{0}[-d])\cong\dgH[0]{\dgModdg*{\A}}(M_{d+2}[d],M_{0})
  \]
  induced by the $d$-fold shift functor.
\end{remark}

We shall prove the following agreement result for Massey products and Toda
brackets.

\begin{theorem}
  \label{thm:TodaMassey}
  Let $\A$ be a small DG category. Consider a sequence of degree $0$
  morphisms
  \[
    M_{d+2}\xrightarrow{f_{d+2}}M_{d+1}\xrightarrow{f_{d+1}}\cdots\xrightarrow{f_2}M_1\xrightarrow{f_1}M_0
  \]
  in the triangulated category $\dgH[0]{\dgModdg*{\A}}$. Then, their Massey
  product and their Toda bracket are related by the formula
  \[
    \MasseyProduct{f_1,\dots,f_{d+2}}[d]=(-1)^{\sum_{i=1}^{d+1}
      i}\TodaBracket{f_1,\dots,f_{d+2}}.
  \]
\end{theorem}

\begin{remark}
  Let $\A$ be a small DG category. Since
  $\DerCat{\A}\subseteq\dgH[0]{\dgModdg*{\A}}$ is a full triangulated
  subcategory, \Cref{thm:TodaMassey} also establishes a relationship between
  Massey products and Toda brackets in $\DerCat{\A}$.
\end{remark}

For the proof of \Cref{thm:TodaMassey}, we need the following observation. Given
a small DG category $\A$ and a morphism $f\colon M\to N$ in $\dgMod*{\A}$,
recall that the \emph{cone of $f$}, denoted by $\cone{f}$, is the DG $\A$-module
whose underlying graded $\A$-module is $M(1)\oplus N$ equipped with the
differential
\[
  d_{\cone{f}}\coloneqq d_{M[1]}\oplus
  d_N+\smallpmatrix{0&0\\f&0}=\smallpmatrix{-d_M&0\\f&d_N};
\]
above, $M\mapsto M(1)$ denotes the shift of the underlying graded $\A$-module of $M$.
The definition above ensures that the canonical split short exact sequence of
graded $\A$-modules
\begin{center}
  \begin{tikzcd}[column sep=small]
    0\rar&N\rar{\iota}&\cone{f}\rar{\pi}&M[1]\rar&0
  \end{tikzcd}
\end{center}
is in fact a (not necessarily split) short exact sequence of DG $\A$-modules.
The above sequence yields an exact triangle
\begin{center}
  \begin{tikzcd}[column sep=small]
    M\rar{f}&N\rar{\iota}&\cone{f}\rar{\pi}&M[1]
  \end{tikzcd}
\end{center}
in the triangulated category $\dgH[0]{\dgModdg*{\A}}$.

\begin{proposition}
  \label{prop:MasseyProduct_alt}
  Let $\A$ be a small DG category. Consider a sequence of degree $0$
  morphisms
  \[
    M_{d+2}\xrightarrow{f_{d+2}}M_{d+1}\xrightarrow{f_{d+1}}\cdots\xrightarrow{f_2}M_1\xrightarrow{f_1}M_0
  \]
  in the triangulated category $\dgH[0]{\dgModdg*{\A}}$, with $d\geq1$. Let $S$
  be the set consisting of the exact triangles of the precise form
  \begin{center}
    \begin{tikzcd}[column sep=small]
      \cone{g_{12}}[-1]\rar{\pi[-1]}&M_2\rar{f_2}&M_1\rar{-\iota}&\cone{g_{12}},
    \end{tikzcd}
  \end{center}
  where $g_{12}$ ranges over all degree $0$ cocycles $M_2\to M_1$ in
  $\dgMod*{\A}$ that represent the morphism $f_2$. Then, the Massey product
  $\MasseyProduct{f_1,\dots,f_{d+2}}$ satisfies the formula
  \[
    \MasseyProduct{f_1,\dots,f_{d+2}}=\bigcup_{(\beta,\alpha)\in\TodaFam{f_1,f_2,f_3}_S}\MasseyProduct{\beta,\alpha,f_4,\dots,f_{d+2}}.
  \]
\end{proposition}
\begin{proof}
  Let
  \[
    \set{g_{ij}\colon M_j\to M_i}[0\leq i<j\leq d+2,\ j-i<d+2]
  \]
  be a defining system for the Massey product
  $\MasseyProduct{f_1,\dots,f_{d+2}}$. Define the morphisms in $\dgMod*{\A}$
  \begin{align*}
    \alpha & \coloneqq\begin{pmatrix}g_{23}\\g_{13}\end{pmatrix}\colon M_3\longrightarrow \cone{g_{12}}[-1], \\
    \beta  & \coloneqq\begin{pmatrix}g_{02}&-g_{01}\end{pmatrix}\colon \cone{g_{12}}[-1]\longrightarrow M_0[-1].
  \end{align*}
  Notice that $|\alpha|=0=|\beta|$. We claim that the pair formed by the
  cohomology classes of these morphisms
  $(\class{\beta},\class{\alpha})\in\TodaFam{f_1,f_2,f_3}_S$. Notice first that
  $\alpha$ and $\beta$ are indeed cocycles since, in view of the equations
  satisfied by a defining system, we have
  \begin{align*}
    \partial(\alpha) & =                                                                      %
                       \begin{pmatrix}
                         d_{M_2} & 0 \\ -g_{12} & -d_{M_1}
                       \end{pmatrix}\circ%
                       \begin{pmatrix}
                         g_{23} \\ g_{13}
                       \end{pmatrix}-%
                       \begin{pmatrix}
                         g_{23} \\ g_{13}
                       \end{pmatrix}\circ d_{M_3}                                                                \\
                     & =\begin{pmatrix}
                          d_{M_2}\circ g_{23}-g_{23}\circ d_{M_3} \\ -g_{12}g_{23}-(d_{M_1}\circ
                          g_{13}+g_{13}\circ d_{M_3})
                        \end{pmatrix} \\
                     & =\begin{pmatrix}
                          \partial(g_{23}) \\
                          \partial(g_{13})-\partial(g_{13})
                        \end{pmatrix}=0
  \end{align*}
  and
  \begin{align*}
    \partial(\beta) & = -d_{M_0}\circ                                                                                       
                      \begin{pmatrix}
                        g_{02} & -g_{01}
                      \end{pmatrix}-%
                      \begin{pmatrix}
                        g_{02} & -g_{01}
                      \end{pmatrix}\circ%
                      \begin{pmatrix}
                        d_{M_2} & 0 \\ -g_{12} & -d_{M_1}
                      \end{pmatrix}                                                                                       \\
                    & =\begin{pmatrix}
                         -(d_{M_0}\circ g_{02}+g_{02}\circ d_{M_2})-g_{01}g_{12} & d_{M_0}\circ g_{01}-g_{01}\circ d_{M_1}
                       \end{pmatrix} \\
                    & =\begin{pmatrix}
                         -\partial(g_{02})+\partial(g_{02}) & \partial(g_{01})
                       \end{pmatrix}=0.
  \end{align*}
  Here we use that $d_{M_0[-1]}=-d_{M_0}$. Thus, to exhibit the membership
  relation $(\class{\beta},\class{\alpha})\in\TodaFam{f_1,f_2,f_3}$, it suffices
  to observe that the diagram (see \Cref{rmk:TodaBracket_betaalpha})
  \begin{center}
    \begin{tikzcd}[column sep=small]
      &M_2\ar{rr}{f_2}&&M_1\drar{f_1}\ar{dl}[very near start]{-\iota}[description]{+1}\\
      M_3\urar{f_3}\ar{rr}[swap]{\class{\alpha}}&&%
      \cone{g_{12}}[-1]\ar{rr}[swap,very near start]{\class{\beta}}[description]{-1}\ar{ul}{\pi[-1]}&&M_0
    \end{tikzcd}
  \end{center}
  commutes in $\dgH[0]{\dgModdg*{\A}}$, for $g_{01}$ and $g_{23}$ represent the
  morphisms $f_1$ and $f_3$ and $\pi[-1]$ and $\iota$ are the apparent
  projection and inclusion into the corresponding direct summands.

  We wish to use the given defining system for the Massey product
  $\MasseyProduct{f_1,\dots,f_{d+2}}$ to construct a defining system for the
  Massey product
  $\MasseyProduct{\class{\beta},\class{\alpha},f_4,\dots,f_{d+2}}$. We claim
  that the collection of morphisms
  \begin{align*}
    \og_{01} & \coloneqq \beta = \begin{pmatrix}g_{02} & -g_{01}\end{pmatrix}\colon \cone{g_{12}}[-1]\longrightarrow M_0;                   \\
    \og_{0j} & \coloneqq g_{0,j+1}\colon M_{j+1}\longrightarrow M_0[-1],\qquad d+2>j\geq2;                                                      \\
    \og_{1j} & \coloneqq \begin{pmatrix}g_{2,j+1} \\g_{1,j+1}\end{pmatrix}\colon M_{j+1}\longrightarrow\cone{g_{12}}[-1],\qquad d+2>j\geq2; \\
    \og_{ij} & \coloneqq g_{i+1,j+1}\colon M_{j+1}\longrightarrow M_{i+1},\qquad d+2>j>i\geq2;
  \end{align*}
  is such a defining system. Notice that
  \begin{itemize}
  \item $\og_{01}=\beta$ represents $\class{\beta}$;
  \item $\og_{12}=\alpha$ represents $\class{\alpha}$;
  \item $\og_{i-1,i}=g_{i,i+1}\colon M_{i+1}\to M_i$ represents $f_{i+1}$.
  \end{itemize}
  Moreover, since $|\og_{01}|=|\og_{02}|$,
  \[
    \sum_{0<k<d+2}(-1)^{|\og_{0k}|-1}\og_{0k}\og_{k,d+2}=\sum_{0<k<d+2}(-1)^{|g_{0k}|-1}g_{0k}g_{k,d+2}
  \]
  so that the new defining system yields the same cocycle as the original one.
  We now analyse the differentials $\partial(\og_{i,j})$ in the four cases:
  \begin{itemize}
  \item We have already shown that $\partial(\og_{01})=\partial(\beta)=0$. \item
    If $j\geq2$, then
    \begin{align*}
      \partial(\og_{0j}) & =-\partial(g_{0,j+1})                                                                                                          \\ & =-\sum_{0<\ell<j+1}(-1)^{-\ell}g_{0\ell}g_{\ell,j+1} \\  &
                                                                                                                                                                                                                          =-\begin{pmatrix}g_{02} & -g_{01}\end{pmatrix}\begin{pmatrix}g_{2,j+1} \\g_{1,j+1}\end{pmatrix}+\sum_{2<\ell<j+1}(-1)^{1-\ell}g_{0\ell}g_{\ell,j+1} \\                       &
                                                                                                                                                                                                                                                                                                                                                                                                         =-\og_{01}\og_{1j}+\sum_{2<\ell<j+1}(-1)^{1-\ell}\og_{0,\ell-1}\og_{\ell-1,j}
      \\ & =\sum_{0<k<j}(-1)^{-k}\og_{0k}\og_{kj},
    \end{align*}
    where the equality $\partial(\og_{0j})=-\partial(g_{0,j+1})$ stems from the
    fact that the target of $\og_{0j}$ is the DG $\A$-module $M_0[-1]$, which
    has the differential $-d_{M_0}$. \item If $i=1$ and $j\geq 2$, then
    \begin{align*}
      \partial(\og_{1j}) & =\begin{pmatrix}
                              d_{M_2} & 0 \\ -g_{12} & -d_{M_1}
                            \end{pmatrix}\circ\begin{pmatrix}
                                                g_{2,j+1} \\ g_{1,j+1}
                                              \end{pmatrix}-(-1)^{|\og_{1j}|}                                           %
                           \begin{pmatrix}
                             g_{2,j+1} \\ g_{1,j+1}
                           \end{pmatrix}\circ d_{M_{j+1}}                                                                                    \\
                         & =\begin{pmatrix}
                              d_{M_2}\circ g_{2,j+1}-(-1)^{|g_{2,j+1}|}g_{2,j+1}\circ d_{M_{j+1}} \\
                              -g_{12}g_{2,j+1}-(d_{M_1}\circ g_{1,j+1}-(-1)^{|g_{1,j+1}|}g_{1,j+1}\circ
                              d_{M_{j+1}})
                            \end{pmatrix}  \\
                         & =\begin{pmatrix}
                              \partial(g_{2,j+1}) \\ -g_{12}g_{2,j+1}-\partial(g_{1,j+1})
                            \end{pmatrix}                                  \\
                         & =\begin{pmatrix}
                              \sum_{2<\ell<j+1}(-1)^{2-\ell}g_{2\ell}g_{\ell,j+1} \\
                              -\sum_{2<\ell<j+1}(-1)^{1-\ell}g_{1\ell}g_{\ell,j+1}
                            \end{pmatrix}                                         \\
                         & =\sum_{2<\ell<j+1}(-1)^{2-\ell}\begin{pmatrix}g_{2\ell} \\g_{1\ell}\end{pmatrix}g_{\ell,j+1} \\
                         & =\sum_{2<\ell<j+1}(-1)^{2-\ell}\og_{1,\ell-1}\og_{\ell-1,j}                                  \\
                         & =\sum_{1<k<j}(-1)^{1-k}\og_{1k}\og_{kj}.
    \end{align*}
  \item If $i\geq2$, then
    \begin{align*}
      \partial(\og_{ij})=\partial(g_{i+1,j+1}) & =\sum_{i+1<k+1<j+1}(-1)^{(i+1)-(k+1)}g_{i+1,k+1}g_{k+1,j+1}
      \\ & =\sum_{i<k<j}(-1)^{i-k}\og_{ik}\og_{kj}.
    \end{align*}
  \end{itemize}
  This shows that the collection of morphisms $\set{\og_{ij}}$ is indeed a
  defining system for the Massey product
  $\MasseyProduct{\class{\beta},\class{\alpha},f_4,\dots,f_{d+2}}$.

  Finally, the above process is clearly reversible: Given
  $(\class{\beta},\class{\alpha})\in\TodaFam{f_1,f_2,f_3}_S$ and a defining
  system for the Massey product
  $\MasseyProduct{\class{\beta},\class{\alpha},f_4,\dots,f_{d+2}}$, one obtains
  a defining system for the Massey product $\MasseyProduct{f_1,\dots,f_{d+2}}$
  by means of the above formulas. This finishes the proof.
\end{proof}

We are ready to prove \Cref{thm:TodaMassey}.

\begin{proof}[Proof of \Cref{thm:TodaMassey}]
  Given a sequence of degree $0$ morphisms
  \[
    M_{d+2}\xrightarrow{f_{d+2}}M_{d+1}\xrightarrow{f_{d+1}}\cdots\xrightarrow{f_2}M_1\xrightarrow{f_1}M_0
  \]
  in the triangulated category $\dgH[0]{\dgModdg*{\A}}$, for $\A$ a small DG
  category, we need to prove that
  \[
    \MasseyProduct{f_1,\dots,f_{d+2}}[d]=(-1)^{\sum_{i=1}^{d+1}
      i}\TodaBracket{f_1,\dots,f_{d+2}}.
  \]
  If $d=0$, then the claim follows from the discussion in
  \Cref{rmk:MasseyProduct_d0}. For $d\geq1$, the claim follows from a
  straightforward induction using
  \Cref{prop:TodaBracket_alt,prop:MasseyProduct_alt} and
  \Cref{lemma:TodaBracket_S}. Indeed, suppose that the claim holds for all
  sequences of $n$ composable morphisms for some $n\geq2$. Let
  $f_1,\dots,f_n,f_{n+1}$ be a sequence of composable homogeneous morphisms in
  $\dgH[0]{\dgModdg*{\A}}$. Then, by \Cref{prop:MasseyProduct_alt} the $(n-1)$-fold shifted Massey product
  \[
    \MasseyProduct{f_1,\dots,f_n,f_{n+1}}[n-1]
  \]
  equals
  \begin{align*}
    & \bigcup_{(\beta,\alpha)\in\TodaFam{f_1,f_2,f_3}_S}\MasseyProduct{\beta,\alpha,f_4,\dots,f_n,f_{n+1}}[n-1]                                         \\
    & =\bigcup_{(\beta,\alpha)\in\TodaFam{f_1,f_2,f_3}_S}(-1)^{\sum_{i=1}^{n-1} i}\TodaBracket{\beta,\alpha,f_4,\dots,f_n,f_{n+1}}[1]                   \\
    & =(-1)^{\sum_{i=1}^{n-1} i}(-1)^{n}\bigcup_{(\beta,\alpha)\in\TodaFam{f_1,f_2,f_3}_S}(-1)^{n-2}\TodaBracket{\beta,\alpha,f_4,\dots,f_n,f_{n+1}}[1] \\
    & =(-1)^{\sum_{i=1}^{n} i}\TodaBracket{f_1,\dots,f_n,f_{n+1}},
  \end{align*}
  where $S$ is as in \Cref{prop:MasseyProduct_alt}. The claim follows.
\end{proof}

\subsection{Massey products and minimal $A_\infty$-structures}

In this subsection we relate (under some conditions) higher operations in a
minimal $A_\infty$-category and Massey products. We begin by recalling the
necessary definitions.

\subsubsection{Reminder on $A_\infty$-categories}

We recall the most basic aspects of the theory of $A_\infty$-ca\-te\-go\-ries
and its relationship to that of DG categories. We refer the reader to
\cite{Kel01,Lef03} for details.

An \emph{$A_\infty$-category} is a categorical structure $\A$ consisting of a
class of objects and graded vector spaces of morphisms
\[
  \A(x,y),\qquad x,y\in\A.
\]
Moreover, $\A$ is equipped with higher composition operations
\[
  m_n\colon\A(x_1,x_0)\otimes\A(x_2,x_1)\otimes\cdots\otimes\A(x_n,x_{n-1})\longrightarrow\A(x_n,x_0),\qquad
  n\geq1
\]
given by graded morphisms of degree $2-n$ that must satisfy a certain infinite
system of equations that we do not recall here. We only mention that, in
particular, these equations imply the following:
\begin{itemize}
\item For each $x,y\in\A$, the pair $(\A(x,y),m_1)$ is a cochain complex, that
  is $m_1\circ m_1=0$.
\item The differential $m_1$ is a graded derivation with respect to the
  composition operation
  \[
    m_2\colon\A(y,z)\otimes\A(x,y)\longrightarrow\A(x,z),
  \]
  (that is, it satisfies the graded Leibniz rule).
\item Although the composition operation $m_2$ is not necessarily associative,
  the failure of the associativity is controlled by the ternary operation $m_3$
  in a precise sense. In particular, as in the case of DG algebras, $\A$ has an
  associated graded category $\dgH{\A}$ as well as an ordinary category
  $\dgH[0]{\A}$, whose (associative) composition laws are induced by $m_2$.
\end{itemize}
We also note that DG categories can be identified with the $A_\infty$-categories
with $m_n=0$ for all $n\geq3$; in particular, the composition operation $m_2$ is
associative in this case. An $A_\infty$-category $\A$ is \emph{minimal} if
$m_1=0$; in this case the composition operation is also associative.

Let $\A$ and $\B$ be $A_\infty$-categories. An \emph{$A_\infty$-functor}
$F\colon\A\to\B$ consists of the following data:
\begin{itemize}
\item A map $F_0\colon\ob{\A}\to\ob{\B}$ between the classes of objects of $\A$
  and $\B$.
\item Graded morphisms
  \begin{center}
    $F_n\colon\A(x_1,x_0)\otimes\A(x_2,x_1)\otimes\cdots\otimes\A(x_n,x_{n-1})\longrightarrow\B(F_0(x_n),F_0(x_0)),$
  \end{center}
  for $n\geq 1$, of degree $1-n$.
\end{itemize}
Moreover, the above morphisms are required to satisfy a further infinite system
of equations that imply that the morphisms
\[
  F_1\colon\A(x,y)\longrightarrow\B(F_0(x),F_0(y)),\qquad x,y\in\A,
\]
are cochain maps and, moreover, $F$ induces a graded functor
\[
  \dgH{F}\colon\dgH{\A}\longrightarrow\dgH{\B}
\]
via the apparent formulas
\[
  x\longmapsto F_0(x),\quad x\in\ob*{\A},\qquad\text{and}\qquad
  \class{f}\longmapsto \class{F_1(f)},\quad f\in\dgZ{\A(x,y)},
\]
and in particular an ordinary functor
\[
  \dgH[0]{F}\colon\dgH[0]{\A}\longrightarrow\dgH[0]{\B}.
\]
An $A_\infty$-functor $F\colon\A\to\B$ is a \emph{quasi-equivalence} if the
induced graded functor
\[
  \dgH{F}\colon\dgH{\A}\longrightarrow\dgH{\B}
\]
is an equivalence of graded categories.

\begin{remark}
  There are different notions of `unitality' for the composition operation in an
  $A_\infty$-category. First of all, the existence of units (identities) may not
  be required, in which case the corresponding $A_\infty$-categories are termed
  \emph{non-unital}. Second, one may require that the graded category $\dgH{\A}$ is
  unital, in which case one speaks of \emph{cohomologically unital}
  $A_\infty$-categories. Finally, one may require that the (higher) composition
  operations $m_n$ are \emph{strictly unital} in a precise sense. All these
  variants of unitality extend to $A_\infty$-functors. We do not wish to
  elaborate on how the various notions of unitality relate to each other, but
  only remark that strict unitality implies cohomological unitality, and that
  every cohomologically unital $A_\infty$-category is quasi-equivalent to a
  strictly unital $A_\infty$-category, see for example~\cite[Sec.~I.2a]{Sei08}.
  The $A_\infty$-categories in this article are cohomologically unital.
\end{remark}

Let $\A$ be a DG category. By the Homotopy Transfer Theorem~\cite{Kad80}, the
graded category $\dgH{\A}$ admits a (non-unique) minimal $A_\infty$-structure
\[
  (\dgH{\A},m_3,m_4,\dots)
\]
whose composition operation $m_2$ is the composition law in the graded category
$\dgH{\A}$ (and is implicit in the notation), and such that there is a
quasi-equivalence of $A_\infty$-categories
\[
  F\colon(\dgH{\A},m_3,m_4,\dots)\stackrel{\sim}{\longrightarrow}\A
\]
that is the identity on the objects: $F_0=\id$. In this case,
$(\dgH{\A},m_3,m_4,\dots)$ is called a \emph{minimal $A_\infty$-model of $\A$}.
Up to quasi-equivalence of $A_\infty$-categories, the DG category structure on
$\A$ is determined by the additional structure on its cohomology $\dgH{\A}$.
This is the main motivation for the use of $A_\infty$-categories in this article
(see the proof of \Cref{thm:projLambda_existence_and_uniqueness} to see why this
is important for us).

\subsubsection{Massey products in DG categories with sparse cohomology}

Let $\A$ be a DG category. It was a longstanding belief that the higher Massey
products in the graded category $\dgH{\A}$ were related to the higher operations
in a minimal model for $\A$ via the formula
\[
  \pm m_n(f_1,\dots,f_n)\in\MasseyProduct{f_1,\dots,f_n}.
\]
However, counter-examples were found recently by Buijs, Moreno-Fern{\'a}ndez and
Murillo in~\cite{BMM20}. There, the authors give sufficient conditions for the
above formula to hold; however, these conditions are not suitable for our
purposes. Below we show that the above relationship between Massey products and
higher operations holds under a different sufficient condition that is better
adapted to our framework.

The following definition is motivated by \Cref{rmk:sparseness}.

\begin{definition}
  \label{def:sparse}
  A graded category $\C$ is \emph{$d$-sparse} if
  \[
    \forall x,y\in\C,\ \forall i\not\in d\ZZ,\qquad \C(x,y)^i=0.
  \]
  Similarly, a DG category $\A$ has \emph{$d$-sparse cohomology} or is 
  \emph{cohomologically $d$-sparse} if the graded category $\dgH{\A}$ is
  $d$-sparse. Explicitly, $\A$ is cohomologically $d$-sparse if
  \[
    \forall x,y\in\A,\qquad \forall i\not\in d\ZZ,\qquad \dgH[i]{\A(x,y)}=0.
  \]
\end{definition}

\begin{remark}
  For $d=1$, the sparseness condition on a (DG) category is, of course, vacuous.
\end{remark}

\begin{proposition}
  \label{prop:vanishing_d-sparse}
  Let $\A$ be a DG category with $d$-sparse cohomology, for some $d\geq1$. Suppose
  given a sequence of composable morphisms in $\dgH{\A}$ of the form
  \[
    x_{k+2}\xrightarrow{f_{k+2}}x_{k+1}\xrightarrow{f_{k+1}}\cdots\xrightarrow{f_2}x_1\xrightarrow{f_1}x_0
  \]
  with $k\not\in d\ZZ$. Let $(\dgH{\A},m_3,m_4,\dots)$ be a minimal
    $A_\infty$-model of $\A$. Then,
    \[
      m_{k+2}(f_1,f_2,\dots,f_{k+2})=0.
    \]
    In other words, the higher operation $m_{k+2}$ is identically $0$.
\end{proposition}
\begin{proof}
By definition, $m_{k+2}$ is a homogeneous morphism
  of degree $2-(k+2)=-k$. Since $\A$ is cohomologically $d$-sparse by assumption
  and
  \[
    |m_{k+2}(f_1,f_2,\dots,f_{k+2})|=-k+\sum_{i=1}^{k+2}|f_i|,
  \]
  we must have $m_{k+2}(f_1,f_2,\dots,f_{k+2})=0$, because either some
  $|f_i|\notin d\ZZ$ and hence $f_i=0$ or all $|f_i|\in d\ZZ$ and then
  $-k+\sum_{i=1}^{k+2}|f_i|\not\in d\ZZ$.
\end{proof}

\begin{notation}\label{not:sparse_minimal}
  Let $(\A,m_3,m_4,\dots)$ be a minimal $A_\infty$-category whose underlying
  graded category $\A=\dgH{\A}$ is $d$-sparse. In view of
  \Cref{prop:vanishing_d-sparse}, we write
  \[
    (\A,m_{d+2},m_{2d+2},\dots)=(\A,m_3,m_4,\dots)
  \]
  since the higher operations $m_{i+2}$ with $i\not\in d\ZZ$ must vanish.
\end{notation}

The following proposition shows that, in a DG category with $d$-sparse
cohomology, the lowest non-trivial higher operation of any minimal
$A_\infty$-model, that is $m_{d+2}$, determines the Massey products of the
corresponding length.

\begin{proposition}
  \label{prop:MasseyProduct_Ainfty_m}
  Let $\A$ be DG category with $d$-sparse cohomology, for some $d\geq1$, and
  consider a sequence of composable morphisms in $\dgH{\A}$ of the form
  \[
    x_{d+2}\xrightarrow{f_{d+2}}x_{d+1}\xrightarrow{f_{d+1}}\cdots\xrightarrow{f_2}x_1\xrightarrow{f_1}x_0
  \]
  and such that $f_if_{i+1}=0$ for each $0<i<d+2$. Then, for every minimal
  $A_\infty$-model $(\dgH{\A},m_{d+2},m_{2d+2},\dots)$ of $\A$, we have
  \[
    -(-1)^{\sum_{0< k\leq
        d+2}(d+2-k)|f_k|}m_{d+2}(f_1,\dots,f_{d+2})\in\MasseyProduct{f_1,\dots,f_{d+2}}.
  \]
  In particular, $\MasseyProduct{f_1,\dots,f_{d+2}}$ is non-empty.
\end{proposition}
\begin{proof} 
  Let
  $(\dgH{\A},m_{d+2},m_{2d+2},\dots)$ be a minimal $A_\infty$-model of $\A$ and
  choose a quasi-equivalence
  \[
    F\colon(\dgH{\A},m_{d+2},m_{2d+2},\dots)\stackrel{\sim}{\longrightarrow}\A
  \]
  with $F_0=\id$. For $0\leq i<j\leq d+2$ and $j-i<d+2$, we define
  \[
    g_{ij}\coloneqq (-1)^{\sum_{i<k\leq j}(j-k)|f_k|}F_{j-i}(f_{i+1},\dots,f_j).
  \]

  We claim that the $g_{ij}$'s form a defining system for computing the Massey
  product $\MasseyProduct{f_1,\dots,f_{d+2}}$. Indeed, in view of the equations
  satisfied by the $A_\infty$-functor $F$~\cite[Sec.~3.4]{Kel01}, if $\partial$ denotes the differential of $\A$, the morphism
  $\partial(F_{j-i}(f_{i+1},\dots,f_j))$ equals
  \[
    -\sum_{i<k<j}(-1)^{(k-i-1)+(j-k-1)\sum_{i<\ell\leq
        k}|f_{\ell}|}F_{k-i}(f_{i+1},\dots,f_k)F_{j-k}(f_{k+1},\dots,f_j)
  \]
  and the claim follows by multiplying by the sign $(-1)^{\sum_{i<k\leq
      j}(j-k)|f_k|}$ in the definition of $g_{ij}$.

  Finally, again by \cite[Sec.~3.4]{Kel01}, the morphism $\partial(F_{d+2}(f_1,\dots,f_{d+2}))$ equals
  \begin{align*}
    & F_1(m_{d+2}(f_1,\dots,f_{d+2}))                    \\
    & -\sum_{0<k<d+2}(-1)^{(k-1)+(d+1-k)\sum_{0<\ell\leq
      k}|f_{\ell}|}F_k(f_1,\dots,f_k)F_{d+2-k}(f_{k+1},\dots,f_{d+2}).
  \end{align*}
  From this, it follows that the morphism
  \[
    -(-1)^{\sum_{0< k\leq d+2}(d+2-k)|f_k|}m_{d+2}(f_1,\dots,f_{d+2})
  \]
  represents the same cohomology class in $\MasseyProduct{f_1,\dots,f_{d+2}}$ as
  the one induced by the above defining system.
\end{proof}

\subsection{The universal Massey product}

In this subsection we introduce the universal Massey product of length $d+2$
associated to a cohomologically $d$-sparse DG category
(\Cref{def:universal_Massey}).

\subsubsection{Reminder on Hochschild cohomology}

We refer the reader to~\cite[Sec.~1]{Mur22} for more details on the material in
this section. Let $\C$ be a category and $\C^e\coloneqq\C\otimes\C^\op$ its
enveloping category. As usual, we identify the abelian category of
$\C$-bimodules with that of (right) $\C^e$-modules; in particular, the category
$\C$ yields the \emph{diagonal $\C$-bimodule}
\[
  (x,y)\longmapsto\C(x,y),\qquad x,y\in\C,
\]
that, with some abuse of notation, we also denote by $\C$. Recall that the
\emph{Hochschild cohomology of $\C$ with coefficients in a $\C$-bimodule $M$},
also known as the \emph{Hochschild--Mitchell cohomology}~\cite{Mit72}, is the
graded vector space
\[
  \HH*{\C}[M]\coloneqq\Ext[\C^e]{\C}{M}[\bullet].
\]
In particular, the Hochschild cohomology $\HH*[n]{\C}[M]$ vanishes if $n$ is
negative. If $M=\C$ the diagonal $\C$-bimodule, we write
\[
  \HH*{\C}\coloneqq\HH*{\C}[\C].
\]
If $\C$ is a graded category, then the category of $\C$-bimodules is also graded
and the Hochschild cohomology is then a bi-graded vector space $\HH{\C}[M]$ and
\[
  \HH*[p]{\C}[M]=\prod_{q\in\ZZ}\HH[p][q]{\C}[M],\qquad p\geq0.
\]
Here, $p$ is the \emph{Hochschild (=horizontal) degree} and $q$ is the
\emph{internal (=vertical) degree}. Hochschild cohomology is Morita invariant
and suitably functorial with respect to graded functors in the first variable
and bimodule morphisms in the second variable, see for example~\cite{Mur06}.

The Hochschild cohomology of a graded category can be computed by means of the
\emph{bar complex}, which is the projective resolution of the diagonal bimodule
whose component in degree $n$ is
\[
  B_n(\C)\coloneqq\bigoplus_{x_0,x_1,\dots,x_n}\C(x_0,-)\otimes\C(x_1,x_0)\otimes\C(x_2,x_1)\otimes\cdots\otimes\C(x_n,x_{n-1})\otimes\C(-,x_n)
\]
and whose differential $d\colon B_n(\C)\to B_{n-1}(\C)$ and augmentation
$d\colon B_0(\C)\to \C$ act on an elementary tensor by the formula
\[
  d(f_0\otimes f_1\otimes\cdots\otimes f_n\otimes
  f_{n+1})=\sum_{i=0}^n(-1)^if_0\otimes f_1\otimes\cdots\otimes
  f_{i}f_{i+1}\otimes\dots\otimes f_n\otimes f_{n+1}
\]
that uses the composition operation in $\C$. The \emph{normalised bar complex} $\bar{B}_\bullet(\C)$, consisting of replacing $\C(x,x)$ with $\C(x,x)/\kk\cdot\id[x]$ above whenenever $x_{i+1}=x_i=x\in\C$, is also a projective resolution of the diagonal bimodule.
The \emph{Hochschild cochain
  complex of $\C$} is the the cochain complex
\[
  \HC{\C}[M]\coloneqq\Hom[\C^e]{B_\bullet(\C)}{M}[*].
\]
Explicitly, in Hochschild degree $n$,
\[
  \HC[n]{\C}[M]=\prod_{x_0,x_1,\dots,x_n}\Hom[\kk]{\C(x_1,x_0)\otimes\C(x_2,x_1)\otimes\cdots\otimes\C(x_n,x_{n-1})}{M(x_n,x_0)}[*]
\]
and the differential $\varphi\mapsto d(\varphi)$, of bidegree $(1,0)$, is given
by the formula
\begin{align*}
  d(\varphi)(f_1\otimes\cdots\otimes f_{n+1}) &
                                                =(-1)^{|\varphi|_v|f_1|}f_1\cdot\varphi(f_2\otimes\cdots\otimes f_{n+1}) \\  &
                                                                                                                             +\sum_{i=1}^{n-1}(-1)^i\varphi(f_1\otimes\cdots\otimes
                                                                                                                             f_{i}f_{i+1}\otimes\dots\otimes f_n\otimes f_{n+1})                    \\  &
                                                                                                                                                                                                          +(-1)^{n}\varphi(f_1\otimes\cdots\otimes f_n)\cdot f_{n+1},
\end{align*}
where $|\varphi|_v$ denotes the vertical degree. 
We can alternatively use the \emph{normalised Hochschild cochain complex,}
\[
  \NHC{\C}[M]\coloneqq\Hom[\C^e]{\bar{B}_\bullet(\C)}{M}[*],
\]
which is the subcomplex formed by the Hochschild cochains which vanish when at least one entry is an identity morphism in $\C$. 
If $M=\C$ the diagonal $\C$-bimodule, we
also write
\[
  \HC*{\C}\coloneqq\HC*{\C}[\C].
\]
Essentially by definition, the Hochschild cohomology
\[
  \HH[0]{\C,M}=\Hom[\C^e]{\C}{M}[*]
\]
is isomorphic to the so-called \emph{end} of the bimodule $M$,
see~\cite[Sec.~2.1]{Kel05a} for the definition. In particular,
\[
  \HH[0]{\C}=\Hom[\C^e]{\C}{\C}[*]\cong Z(\C)
\]
is isomorphic to the end of the diagonal bimodule, that is the \emph{graded centre} of
$\C$.

If $\C$ is an ungraded category (=graded category concentrated in degree $0$),
then the Hochschild cohomology of $\C$ with coefficients in a graded (!)
$\C$-bimodule can be computed in terms of ungraded Hochschild cohomology by
means of the formula
\[
  \HH[p][q]{\C}[M]\cong\HH*[p]{\C}[M^q],\qquad p\geq0,\quad q\in\ZZ.
\]

The Hochschild cohomology
\[
  \HH{\C}=\Ext[\C^e]{\C}{\C}[\bullet,*]
\]
of a graded category has the structure of a \emph{Gerstenhaber algebra}
\cite{Ger63} and \cite[Sec.~1]{Mur20b}. Thus, $\HH{\C}$ is equipped with a
graded commutative product
\[
  \HH[p][q]{\C}\otimes\HH[r][s]{\C}\longrightarrow\HH[p+r][q+s]{\C},\qquad
  a\otimes b\longmapsto a\cdot b,
\]
called the \emph{cup product}; up to sign, it agrees with the Yoneda product on the
corresponding extension spaces. Thus, the cup product satisfies the identities
\begin{align*}
  (a\cdot b)\cdot c & =a\cdot (b\cdot c),     \\
  a\cdot b          & =(-1)^{|a||b|}b\cdot a.
\end{align*}
Here $|x|=p+q$ is the \emph{total degree} of an element $x\in \HH[p][q]{\C}$. 
Moreover, there is a degree $-1$ Lie bracket
\[
  \HH[p][q]{\C}\otimes\HH[r][s]{\C}\longrightarrow\HH[p+r-1][q+s]{\C},\qquad
  a\otimes b\longmapsto [a,b],
\]
that satisfies the following graded variants of the usual Lie algebra
identities:
\begin{align*}
  [a,a]     & =0,\qquad |a|\text{ is odd},               \\
  [a,b]     & =-(-1)^{(|a|-1)(|b|-1)}[b,a],              \\
  [a,[b,c]] & =[[a,b],c]+(-1)^{(|a|-1)(|b|-1)}[b,[a,c]], \\
  [a,[a,a]] & =0,\qquad |a|\text{ is even}.
\end{align*}
The cup product and the Lie bracket are compatible in the sense that
\[
  [a,b\cdot c]=[a,b]\cdot c+(-1)^{(|a|-1)|b|}b\cdot[a,c].
\]
This is know as the \emph{Gerstenhaber relation}. 
Finally, there is a \emph{Gerstenhaber squaring operation} or \emph{Gerstenhaber
  square}
\[
  \HH[p][q]{\C}\longrightarrow\HH[2p-1][2q]{\C},\qquad a\longmapsto\Sq[a],
\]
defined whenever $p+q$ is even or $\cchar{\kk}=2$. The Gerstenhaber square
satisfies the following compatibility relations with respect to the cup product
and the Lie bracket:
\begin{align*}
  \Sq[a+b]      & =\Sq[a]+\Sq[b]+[a,b],                               \\
  \Sq[a\cdot b] & =\Sq[a]\cdot b^2+a\cdot[a,b]\cdot b+a^2\cdot\Sq[b], \\
  [\Sq[a],b]    & =[a,[a,b]].
\end{align*}
If $\cchar{\kk}\neq2$, then the Gerstenhaber square is in fact determined by the
Lie bracket according to the formula
\[
  \Sq(a)=\frac{1}{2}[a,a],
\]
which follows from the previous relations. In particular, if $\cchar{\kk}\neq2$,
then the relations that involve the Gerstenhaber square follow from the usual Lie algebra identities and the Gerstenhaber relation.

\begin{remark}\label{rmk:automorphism_action_on_cochains}
For a graded algebra $A$ (that we can regard as a graded category with only one object), the automorphism group $\Aut*{A}$ of the graded algebra $A$ acts on the right on the cochain complex $\HC{A}$ by conjugation,
\begin{align*}
	\varphi^{f}(x_1,\dots,x_n)&=f^{-1}\varphi(f(x_1),\dots,f(x_n)),
	\qquad \varphi\in\HC[n]{A},\quad f\in\Aut*{A}.
\end{align*}
This induces a right action of $\Aut*{A}$ on $\HH{A}$ which is determined by \[x^f=(f^{-1})_*f^*(x).\]
Here $f^*\colon \HH{A}[A]\to \HH{A}[{}_fA_f]$ is induced by the algebra automorphism $f$ in the first variable and $f^{-1}_*\colon \HH{A}[{}_fA_f]\to \HH{A}[A]$ is induced by the $A$-bimodule morphism $f^{-1}\colon {}_fA_f\to A$ in the second variable.
\end{remark}

The following result is probably well known but we have not found any reference.

\begin{proposition}
\label{prop:inner_auto}
	The right action of $\Aut*{A}$ on $\HH{A}$ (\Cref{rmk:automorphism_action_on_cochains}) factors through the outer automorphism group $\Out*{A}$ of the graded algebra $A$.
\end{proposition}

\begin{proof}
	We have to show that the inner automorphism $f_u\colon A\stackrel{\sim}{\to} A$, $f_u(x)\coloneqq (-1)^{|u||x|}uxu^{-1}$, induced by a unit $u\in A$ acts trivially on $\HH{A}$. The sign in the definition of the outer automorphism comes from the fact that $\id[A]$ and $f_u$ can be regarded as endofunctors of the graded category with one object $A$ and multiplication by $u$ can then be regarded as a degree $|u|$ natural isomorphism $u\colon\id[A]\Rightarrow f_u$. We will produce a null-homotopy for the map
	\begin{align*}
		\HC{A}&\longrightarrow\HC{A},\\
		\varphi&\longmapsto\varphi^{f_u}.
	\end{align*}
	The null-homotopy is the map 
	\begin{align*}
		h\colon\HC{A}&\longrightarrow\HC[\bullet-1]{A}
	\end{align*}
	defined by 
	\begin{align*}
		h(\varphi)(x_1,\dots,x_{n-1})&=\sum_{i=0}^{n-1}(-1)^{i+|u||\varphi|_v}u^{-1}\varphi(ux_1u^{-1},\dots, ux_iu^{-1},u,x_{i+1},\dots,x_{n-1}).
	\end{align*}
	A tedious but straightforward computation shows that
	\[dh+hd=\id[\HC{A}]-(-)^{f_u}.\qedhere\]
\end{proof}

\begin{remark}
	\Cref{prop:inner_auto} can be generalized from algebras $A$ to  categories $\C$. In theory, the role of $\Aut*{A}$ should be replaced by the self-equivalences of $\C$, but these do not form a group (taking pseudo-inverses is not well defined). Nevertheless, $\Out*{A}$ can be replaced by the group of natural isomorphism classes of self-equivalences of $\C$, which happens to act on $\HH{\C}$ by the same formulas as above. This can be checked by using the functoriality properties of Hochschild cohomology of categories, compare \cite{Mur06}.
\end{remark}

\subsubsection{The universal Massey product of lenght $d+2$}

Let $(\A,m_3,m_4,\dots)$ be a minimal $A_\infty$-category. We are interested in
the Hochschild cohomology of the graded category $\A=\dgH{\A}$ (recall that
$m_1=0$ by definition and $m_2$ is the composition law in $\A$). Firstly, notice
that by definition the higher operations are Hochschild cochains:
\[
  m_n\in\HC[n][2-n]{\A},\qquad n\geq 3.
\]
Similarly, if
\[
  F\colon(\A,m_3,m_4,\dots)\longrightarrow(\B,m_3,m_4,\dots)
\]
is an $A_\infty$-functor between minimal $A_\infty$-categories, then $F_0$ and
$F_1$ define a graded functor
\[
  F_{0,1}\colon\A\longrightarrow\B.
\]
between the underlying graded categories, and the rest of components are
Hochschild cochains:
\[
  F_n\in\HC[n][1-n]{\A}[\B(F_{0,1},F_{0,1})],\qquad n\geq 2.
\]
The following result is a straightforward consequence of results of
Lef{\`e}vre-\-Ha\-se\-ga\-wa~\cite{Lef03}.

\begin{proposition}
  \label{prop:universal_Massey_welldef}
  Let $(\A,m_{d+2},m_{2d+2},\dots)$ be a minimal $A_\infty$-category whose
  underlying graded category $\A$ is $d$-sparse. The following statements hold:
  \begin{enumerate}
  \item\label{it:m_d+2_is_cocycle}The first possibly non-trivial higher operation (see
    \Cref{prop:vanishing_d-sparse})
    \[
      m_{d+2}\in\HC[d+2][-d]{\A}
    \]
    is a Hochschild cocycle. 
    \item\label{it:m_d+2-cohomologous} Let
    $(\A,m_{d+2}',m_{2d+2}',\dots)$ be another minimal $A_\infty$-category with the same 
    underlying graded category $\A$ and
    \[
      F\colon(\A,m_{d+2},m_{2d+2},\dots)\longrightarrow(\A,m_{d+2}',m_{2d+2}',\dots)
    \]
    an $A_\infty$-functor with $F_0=\id$ and $F_1=\id$. Then,
    \[
      \class{m_{d+2}}=\class{m_{d+2}'}\quad\text{in}\quad\HH[d+2][-d]{\A}.
    \]
  \end{enumerate}
\end{proposition}
\begin{proof}
  \eqref{it:m_d+2_is_cocycle} Since $\A$ is $d$-sparse, the higher operations
  $m_{i+2}$ and $m_{i+2}'$ vanish whenever $i\not\in d\ZZ$
  (\Cref{prop:vanishing_d-sparse}). The claim follows from
  ~\cite[Lemma~B.4.1]{Lef03}, which shows that the differential $d(m_{d+2})$ of
  $m_{d+2}$ in the Hochschild complex $\HC{\A}$ equals an expression of the form
  \[
    \sum\pm m_{i+2}(\cdots),
  \]
  where in particular $0<i<d$, and hence the above expression must vanish.

  \eqref{it:m_d+2-cohomologous} Similarly as in the previous statement, since
  $\A$ is $d$-sparse, the higher components $F_{i+1}$ of the $A_\infty$-functor
  $F$ must vanish whenever $i\not\in d\ZZ$ as, by definition, $|F_{i+1}|=-i$.
  The claim follows from~\cite[Lemma~B.4.2]{Lef03}, which shows that the
  differential $d(F_{d+1})$ of $F_{d+1}$ in the Hochschild complex
  $\HC{\A}$ equals an expression of the form
  \[
    m_{d+2}-m_{d+2}'+\sum\pm F_{i+1}(\cdots)-\sum\pm m_{i+2}'(\cdots),
  \]
  where in particular $0<i<d$, and hence the two right-most terms in the
  expression must vanish. This finishes the proof.
\end{proof}

\begin{definition}
  \label{def:universal_Massey}
  Let $\A$ be a cohomologically $d$-sparse DG category and $\dgH{\A}$ is graded
  cohomology category. The \emph{universal
    Massey product of length $d+2$} is the Hochschild cohomology class
  \[
    \class{m_{d+2}}\in\HH[d+2][-d]{\dgH{\A}}
  \]
  represented by the operation $m_{d+2}$ in any minimal $A_\infty$-model of
  $\A$.
\end{definition}

\begin{remark}
  Given a cohomologically $d$-sparse DG category,
  \Cref{prop:universal_Massey_welldef} shows that the universal Massey product
  of length $d+2$ is well defined and, moreover, is natural with respect to
  quasi-equivalences: if $F\colon\A\to\B$ is a quasi-equivalence between
  cohomologically $d$-sparse DG categories, then the induced isomorphism
  \[
    \HH[d+2][-d]{\dgH{\A}}\cong \HH[d+2][-d]{\dgH{\B}}
  \]
  takes the universal Massey product of length $d+2$ of $\A$ to that of $\B$.
  The universal Massey product of length $d+2$ is also considered in \cite[Ch.~3]{Sei15}.
  For $d=1$, it has been investigated for example in
  \cite{BKS04,Kad82} and plays a crucial role in \cite{Mur22} as well. 
\end{remark}

\begin{remark}
  Occasionally, we consider, more generally, the universal Massey product
  \[
    \class{m_{d+2}}\in\HH[d+2][-d]{\A}
  \]
  associated to a minimal $A_\infty$-category $\A=\dgH{\A}$ whose underlying
  graded category is $d$-sparse (here, $\HH[d+2][-d]{\A}$ denotes the Hochschild
  cohomology of the underlying graded category of $\A$).
\end{remark}

\subsection{The restricted universal Massey product}

In this subsection we relate universal Massey products to standard and
Amiot--Lin $(d+2)$-angulations (\Cref{thm:GKO_AL_Massey}). The key ingredient
for establishing the desired relationship is the restricted universal Massey
product (\Cref{def:restricted_universal_Massey}).

\subsubsection{The restricted universal Massey product}

Let $\A$ be a small DG category and $\dgH{\A}$ its graded cohomology category.
The inclusion $j\colon\dgH[0]{\A}\hookrightarrow\dgH{\A}$ of the degree $0$ part
induces a restriction morphism
\[
  j^*\colon\HH{\dgH{\A}}[\dgH{\A}]\longrightarrow\HH{\dgH[0]{\A}}[\dgH{\A}],
\]
that, in bidegree $(d+2,-d)$, yields a morphism
\[
  j^*\colon\HH[d+2][-d]{\dgH{\A}}[\dgH{\A}]\longrightarrow\HH[d+2][-d]{\dgH[0]{\A}}[\dgH{\A}].
\]
Notice also that there is an isomorphism of vector spaces
\[
  \HH[d+2][-d]{\dgH[0]{\A}}[\dgH{\A}]\cong\HH*[d+2]{\dgH[0]{\A}}[\dgH[-d]{\A}],
\]
where $\dgH[-d]{\A}$ is the ungraded $\dgH[0]{\A}$-bimodule
\[
  \dgH[-d]{\A}\colon(x,y)\longmapsto\dgH[-d]{\A(x,y)}.
\]

\begin{definition}
  \label{def:restricted_universal_Massey}
  Let $\A$ be a small DG category that is cohomologically $d$-sparse. The
  \emph{restricted universal Massey product of length $d+2$} is the class
  \[
    j^*\class{m_{d+2}}\in\HH*[d+2]{\dgH[0]{\A}}[\dgH[-d]{\A}],
  \]
  where
  \[
    \class{m_{d+2}}\in\HH[d+2][-d]{\dgH{\A}}[\dgH{\A}]
  \]
  is the universal Massey product of length $d+2$ associated to any minimal
  $A_\infty$-model of $\A$,
  \[
    (\dgH{\A},m_{d+2},m_{2d+2},\dots).
  \]
\end{definition}

\begin{remark}
  Occasionally, we consider, more generally, the restricted universal Massey product
  \[
    j^*\class{m_{d+2}}\in \HH[d+2][-d]{\A^0}[\A]=\HH*[d+2]{\A^0}[\A^{-d}]
  \]
  associated to a minimal $A_\infty$-category $\A=\dgH{\A}$ whose underlying
  graded category is $d$-sparse, where $\A^0$ is the degree $0$ part of $\A$ (in
  particular, $\HH[d+2][-d]{\A^0}[\A]$ denotes the Hochschild cohomology of the
  ungraded category $\A^0$ with coefficients in the graded $\A^0$-bimodule
  $\A$).
\end{remark}

\begin{remark}
  The cohomology class investigated in \cite{BD89} is a topological version of
  the restricted universal Massey products considered here in the case $d=1$. It
  computes Toda brackets, which are the topological counterpart of Massey
  products. The connection is explained in \cite{Mur20a}.
\end{remark}

\subsubsection{The $d$-sparse graded algebra $\gLambda$}\label{sec:gLambda}

We fix the following setting until further notice.

\begin{setting}
  \label{setting:resticted_universal_Massey}
  We fix a small DG category $\A$ with the following properties:
  \begin{enumerate}
  \item The DG category $\A$ is cohomologically $d$-sparse. We fix a minimal
    $A_\infty$-model
    \[
      (\dgH[\bullet]{\A},m_{d+2},m_{2d+2},\dots)
    \]
    of $\A$ (see \Cref{not:sparse_minimal}).
  \item The essential image of $\dgH[0]{\A}$ under the canonical fully faithful
    functor
    \begin{center}
      $\dgH[0]{\Yoneda}\colon\dgH[0]{\A}\longrightarrow\DerCat[c]{\A}$
    \end{center}
    is closed under the action of the $d$-fold shift and its inverse. In
    particular, there is a commutative diagram
    \begin{center}
      \begin{tikzcd}
        \dgH[0]{\A}\rar\dar[swap]{[d]}&\DerCat[c]{\A}\dar{[d]}\\
        \dgH[0]{\A}\rar&\DerCat[c]{\A}
      \end{tikzcd}
    \end{center}
    where $[d]\colon\dgH[0]{\A}\stackrel{\sim}{\to}\dgH[0]{\A}$ is an
    autoequivalence such that
    \begin{center}
      $\forall x,y\in\A,\qquad
      \dgH[0]{\A(x,y[d])}\cong\dgH[0]{\A(x,y)[d]}\cong\dgH[d]{\A(x,y)}$
    \end{center}
    as vector spaces (functorially on $x$ and $y$). 
  \item The category $\dgH[0]{\A}$ is additive and has split idempotents and
    finite-dimensional morphism spaces.
  \item There exists a basic object $c\in\dgH[0]{\A}$ such that
    $\add*{c}=\dgH[0]{\A}$. In particular, the Yoneda functor
    \begin{center}
      $\dgH[0]{\A}\longrightarrow\mmod*{\Lambda},\qquad
      x\longmapsto\dgH[0]{\A}(c,x)$
    \end{center}
    restricts to an equivalence of categories
    \begin{center}
      $\dgH[0]{\A}\stackrel{\sim}{\longrightarrow}\proj*{\Lambda},$
    \end{center}
    where $\Lambda\coloneqq\dgH[0]{\A(c,c)}$.
  \end{enumerate}
\end{setting}

\begin{notation}
  \label{not:the_phi}
  Since the finite-dimensional algebra $\Lambda$ is basic, there exists an
  algebra automorphism $\sigma\colon\Lambda\stackrel{\sim}{\to}\Lambda$ (unique
  up to inner automorphisms) such that the diagram
  \[
    \begin{tikzcd}
      \dgH[0]{\A}\rar{\sim}\dar[swap]{[d]}&\proj*{\Lambda}\dar{-\otimes_{\Lambda}\twBim[\sigma]{\Lambda}}\\
      \dgH[0]{\A}\rar{\sim}&\proj*{\Lambda}
    \end{tikzcd}
  \]
  commutes up to natural isomorphism. Explicitly, we fix an isomorphism
  $\varphi\colon c\stackrel{\sim}{\to}c[d]$ in $\dgH[0]{\A}$ (which exists since
  $c$ is a basic additive generator of $\dgH[0]{\A}$, which is closed under $[d]$). The algebra automorphism
  \[
    \sigma=\sigma_\varphi\colon\dgH[0]{\A}(c,c)\stackrel{\sim}{\longrightarrow}\dgH[0]{\A}(c,c),\qquad
    a\longmapsto \varphi^{-1}(a[d])\varphi,
  \]
  has the desired property. In particular, the autoequivalence
  \[
    \Sigma\coloneqq-\otimes_{\Lambda}\twBim[\sigma]{\Lambda}\colon\proj*{\Lambda}\stackrel{\sim}{\longrightarrow}\proj*{\Lambda}
  \]
  can be identified with the restriction of scalars $P\longmapsto
  P_{\sigma^{-1}}$ along $\sigma^{-1}$.
\end{notation}

The isomorphism $\varphi\colon c\stackrel{\sim}{\to}c[d]$ (\Cref{not:the_phi})
can be interpreted as an invertible cohomology class
\[
  \varphi\in\dgH[0]{\A}(c,c[d])\cong\dgH[d]{\A(c,c)};
\]
under this identification, the algebra automorphism $\sigma=\sigma_\varphi$ of
$\Lambda$ acts by conjugation $a\mapsto \varphi^{-1}a\varphi$. Notice also that
$\varphi$ induces isomorphisms of vector spaces
\[
  \dgH[0]{\A(c,c)}\stackrel{\sim}{\longrightarrow}\dgH[di]{\A(c,c)},\qquad
  b\longmapsto\varphi^{i}b,
\]
for each $i\in\ZZ$. Moreover, as a bimodule over $\Lambda=\dgH[0]{\A(c,c)}$, the
cohomology space $\dgH[di]{\A(c,c)}$ identifies with the twisted
$\Lambda$-bimodule $\twBim[\sigma^i]{\Lambda}$. Indeed, for
$a\in\dgH[0]{\A(c,c)}$ and
\[
  b\in\dgH[0]{\A(c,c)}\stackrel{\sim}{\longrightarrow}\dgH[di]{\A(c,c)}
\]
we have
\begin{align*}
  b\cdot a=\varphi^{-i}((\varphi^i b) a)=b a\qquad\text{and}\qquad
  a\cdot b=\varphi^{-i}(a  (\varphi^{i} b))=\sigma^i(a) b.
\end{align*}
More generally, we introduce the $d$-sparse graded algebra with underlying
graded vector space
\[
  \gLambda\coloneqq\bigoplus_{di\in d\ZZ}\twBim[\sigma^i]{\Lambda}
\]
and multiplication law
\[
  a\cdot b\coloneqq\sigma^j(a)b\in\gLambda[d(i+j)],\qquad
  a\in\twBim[\sigma^i]{\Lambda},\ b\in\twBim[\sigma^j]{\Lambda}.
\]
It is straightforward to verify that the linear map
\[
  \dgH{\A}(c,c)\stackrel{\sim}{\longrightarrow}\gLambda,\qquad
  a\longmapsto\varphi^{-i}a,\qquad a\in\dgH[di]{\A}(c,c),
\]
is an isomorphism of graded algebras and, therefore, we obtain an induced
isomorphism of Gerstenhaber algebras
\[
  \HH{\dgH{\A}(c,c)}\stackrel{\sim}{\longrightarrow}\HH{\gLambda},
\]
which we treat as an identification in what follows. Finally, since the object
$c\in\dgH[0]{\A}$ is an additive generator, there is also a canonical isomorphism
of Gerstenhaber algebras
\[
  \HH{\dgH{\A}(c,c)}\stackrel{\sim}{\to}\HH{\dgH{\A}}
\]
and, consequently, a further isomorphism of Gerstenhaber algebras
\[
  \HH{\dgH{\A}}\stackrel{\sim}{\longrightarrow}\HH{\gLambda}.
\]

\begin{remark}
  \label{rmk:Lambda-sigma-d}
  There is a further isomorphism of ($d$-sparse) graded algebras
  \[
    \gLambda\cong\frac{\Lambda\langle\imath^{\pm 1}\rangle}{(\imath
      a-\sigma(a)\imath)_{a\in\Lambda}},\qquad |\imath|=-d,
  \]
  induced by the $\Lambda$-linear map
  \[
    \Lambda\langle\imath^{\pm
      1}\rangle\longrightarrow\gLambda,\qquad\imath\longmapsto1\in\gLambda[-d]=\twBim[\sigma^{-1}]{\Lambda}.
  \]
  Indeed, for $a\in\Lambda=\gLambda[0]$ and $1\in\gLambda[-d]$ we
  have
  \[
    \sigma(a)\cdot1=\sigma^{-1}(\sigma(a))1=a=1a=\sigma^0(1)a=1\cdot a.
  \]
  We treat this isomorphism as an identification simply for the purpose of
  removing the ambiguities that arise when considering homogeneous elements of
  the graded algebra $\gLambda$ (whose non-zero components are all equal to
  $\Lambda$ as vector spaces).
\end{remark}

\begin{remark}
  \label{rmk:skew-group-algebras}
  The following description of the graded algebra $\gLambda$ might
  be more enlightening to some readers. The choice of algebra automorphism
  $\sigma\in\Aut*{\Lambda}$ uniquely determines an action of the cyclic group
  $\ZZ$ on $\Lambda$. If we identify the Laurent polynomial algebra
  $\kk[\imath^{\pm1}]$ with the group algebra $\kk[\ZZ]$, we see that
  $\gLambda$ is isomorphic to the graded skew-group algebra
  \[
    \kk[\imath^{\pm1}]*\Lambda=\kk[\ZZ]*\Lambda,\qquad |\imath|=-d,
  \]
  whose underlying graded vector space is $\kk[\imath^{\pm1}]\otimes\Lambda$ and
  whose multiplication is given by
  \[
    (\imath^{-i}\otimes a)*(\imath^{-j}\otimes b)=\imath^{-(i+j)}\otimes
    \sigma^j(a)b,\qquad a,b\in\Lambda,\ i,j\in\ZZ.
  \]
  The claimed isomorphism is given by the apparent map
  \[
    \kk[\imath^{\pm1}]*\Lambda\stackrel{\sim}{\longrightarrow}\gLambda,\qquad \imath\otimes
    a\longmapsto \imath a;
  \]
  indeed,
  \[
    (\id*[k]\otimes
    \sigma(a))*(\imath\otimes\id*[\Lambda])=\imath\otimes\sigma^{-1}(\sigma(a))=\imath\otimes a.
  \]
  In particular, if $\sigma=\id[\Lambda]$ then
  $\gLambda=\kk[\imath^{\pm1}]\otimes\Lambda$ as graded algebras.
\end{remark}

As before, the inclusion $j\colon\Lambda\hookrightarrow\gLambda$ of the degree $0$ part
induces a restriction morphism
\[
  j^*\colon\HH{\gLambda}[\gLambda]\longrightarrow\HH{\Lambda}[\gLambda]
\]
that, in bidegree $(d+2,-d)$, yields a morphism
\[
  j^*\colon\HH[d+2][-d]{\gLambda}[\gLambda]\longrightarrow\HH[d+2][-d]{\Lambda}[\gLambda]\cong\HH*[d+2]{\Lambda}[\twBim{\Lambda}[\sigma]],
\]
since
\[
  \gLambda[-d]=\twBim[\sigma^{-1}]{\Lambda}\cong\twBim{\Lambda}[\sigma]
\]
is the degree $-d$ part of $\gLambda$ and
\[
  \HH[d+2][-d]{\Lambda}[\gLambda]=\HH*[d+2]{\Lambda}[\gLambda[-d]]=\HH*[d+2]{\Lambda}[\twBim[\sigma^{-1}]{\Lambda}];
\]
notice also that the target of $j^*$ is (isomorphic to) the extension space
\[
  \HH*[d+2]{\Lambda}[\twBim{\Lambda}[\sigma]]=\Ext[\Lambda^e]{\Lambda}{\twBim{\Lambda}[\sigma]}[d+2].
\]
Using the graded Morita invariance of Hochschild cohomology, we may identify
the above morphism $j^*$ with the restriction morphism used
in~\Cref{def:restricted_universal_Massey}. In particular, in view of the
isomorphism
\[
  \HH{\dgH{\A}}\stackrel{\sim}{\longrightarrow}\HH{\gLambda},
\]
the universal Massey product $\class{m_{d+2}}$ can be identified with a class
\[
  \class{m_{d+2}}\in\HH{\gLambda}[\gLambda]
\]
and therefore
\[
  j^*\class{m_{d+2}}\in\HH[d+2][-d]{\Lambda}[\gLambda]\cong\Ext[\Lambda^e]{\Lambda}{\twBim{\Lambda}[\sigma]}[d+2]
\]
can be represented by an exact sequence of $\Lambda$-bimodules.

\subsubsection{Standard $(d+2)$-angles versus Amiot--Lin $(d+2)$-angulations}

\begin{notation}
  \label{not:Massey_angles}
  Choose an exact sequence
  \[
    \eta\colon\qquad0\to\twBim{\Lambda}[\sigma]\stackrel{\iota}{\to}
    P_{d+1}\to\cdots\to P_1\to P_0\to\Lambda\to0
  \]
  that represents $j^*\class{m_{d+2}}$ in
  $\Ext[\Lambda^e]{\Lambda}{\twBim{\Lambda}[\sigma]}[d+2]$; we may and we will
  assume that $P_0,P_1,\dots,P_d$ are projective $\Lambda$-bimodules (see \cite[p.~151]{hilton_stammbach_1971_course_homological_algebra}). We let
  \[
    \pentagon_{j^*\class{m_{d+2}}}\coloneqq\pentagon_\eta
  \]
  be the class of $(d+2)$-angles in $\proj*{\Lambda}$ obtained as in
  \Cref{const:AL_angulation}. Note, however, that we do not claim that
  $j^*\class{m_{d+2}}$ can be represented by a extension in which \emph{all}
  middle terms are projective $\Lambda$-bimodules, and therefore
  \emph{a priori} the class $\pentagon_{j^*\class{m_{d+2}}}$ is not known to be
  a $(d+2)$-angulation of $(\proj*{\Lambda},\Sigma)$.
\end{notation}

We aim to prove the following theorem. We remind the reader of our standing
assumptions (\Cref{setting:resticted_universal_Massey}).

\begin{theorem}
  \label{thm:GKO_AL_Massey}
  Suppose that the DG category $\A$ is pre-$(d+2)$-angulated. Then, the
  following statements hold:
  \begin{enumerate}
  \item\label{it:std=AL} The class of standard $(d+2)$-angles on
    \[
      (\proj*{\Lambda},\Sigma)\simeq(\dgH[0]{\A},[d])
    \]
    coincides with the class
    $(-1)^{\Sigma_{i=1}^{d+2}i}\pentagon_{j^*\class{m_{d+2}}}$ from
    \Cref{not:Massey_angles} (see also \Cref{rmk:several_angulations}). In
    particular, the class
    $(-1)^{\Sigma_{i=1}^{d+2}i}\pentagon_{j^*\class{m_{d+2}}}$ does not depend
    on the choice of an exact sequence representing $j^*\class{m_{d+2}}$ in
    $\Ext[\Lambda^e]{\Lambda}{\twBim{\Lambda}[\sigma]}[d+2]$.
  \item\label{it:is_Tate_unit} Suppose that $\Lambda/J_\Lambda$ is separable.
    Let
    \[
      0\to\twBim{\Lambda}[\sigma]\to P_{d+1}\to\cdots\to P_1\to
      P_0\to\Lambda\to0
    \]
    be an exact sequence that represents $j^*\class{m_{d+2}}$ in
    $\Ext[\Lambda^e]{\Lambda}{\twBim{\Lambda}[\sigma]}[d+2]$ such that the
    $\Lambda$-bimodules $P_i$, $0\leq i<d+1$, are projective-injective
    $\Lambda$-bimodules. Then, $P_{d+1}$ is also a projective-injective
    $\Lambda$-bimodule.
  \end{enumerate}
  In particular, if $\Lambda/J_\Lambda$ is separable, the standard
  $(d+2)$-angulation and the AL $(d+2)$-angulation of $(\proj*{\Lambda},\Sigma)$
  coincide.
\end{theorem}

\begin{notation}
  \label{not:H0A_is_C}
  To alleviate the notation, we denote by $\C\subseteq\DerCat[c]{\A}$ be the
  essential image of $\dgH[0]{\A}$ under the canonical fully faithful functor
  \[
    \dgH[0]{\Yoneda}\colon\dgH[0]{\A}\longrightarrow\DerCat[c]{\A}.
  \]
  Under the induced equivalence of categories $\dgH[0]{\A}\simeq\C$, the
  $\dgH[0]{\A}$-bimodule
  \[
    \dgH[di]{\A}\colon(x,y)\longmapsto\dgH[di]{\A(x,y)}\cong\dgH[0]{\A}(x,y[di]),\qquad
    i\in\ZZ,
  \]
  identifies with the twisted $\C$-bimodule
  \[
    \twBim[[di]]{\C}\colon(M,N)\longmapsto\C(M,N[di]),\qquad i\in\ZZ.
  \]
  Similarly, the autoequivalence
  $[d]\colon\dgH[0]{\A}\stackrel{\sim}{\longrightarrow}\dgH[0]{\A}$ induces the
  autoequivalence
  \[
    \Sigma\coloneqq-\otimes_\C\twBim[[d]]{\C}\colon\Mod*{\C}\stackrel{\sim}{\longrightarrow}\Mod*{\C}
  \]
  with quasi-inverse
  \[
    \Sigma^{-1}\coloneqq-\otimes_\C\twBim[[-d]]{\C}\colon\Mod*{\C}\stackrel{\sim}{\longrightarrow}\Mod*{\C}.
  \]
\end{notation}

Working towards the proof of \Cref{thm:GKO_AL_Massey}, we first investigate the
relationship between the restricted universal Massey product and extensions. The
reader might find it useful to compare the discussion below with
\Cref{const:AL_angulation}. We begin with a general observation. Let $\X$ be a
category and consider an exact sequence of $\X$-bimodules of the form
\[
  0\to M\to E_{d+1}\to\cdots\to E_1\to E_0\to\X\to0,
\]
where $E_0,E_1,\dots, E_d$ are projective $\X$-bimodules. Notice that for each $x\in\X$ the sequence of \emph{left} $\X$-modules
\[
  0\to M(x,-)\to E_{d+1}(x,-)\to\cdots\to E_1(x,-)\to E_0(x,-)\to\X(x,-)\to0
\]
is split-exact (=contractible as a complex), for the representable left
$\X$-module $\X(x,-)$ as well as $E_i(x,-)$, $i=0,1,\dots,d$, are projective as left $\X$-modules. From
this follows that, for every (right) $\X$-module $N$, the sequence of (right)
$\X$-modules
\[
  N\otimes_\X(0\to M\to E_{d+1}\to\cdots\to E_1\to E_0\to\X\to0)
\]
is exact. Thus, in view of the canonical isomorphism $N\otimes_\X\X\cong N$,
there is a well-defined morphism
\[
  \Phi_N\colon\Ext[\X^e]{\X}{M}[\bullet]\longrightarrow\Ext[\X]{N}{N\otimes_\X
    M}[\bullet],
\]
which is an algebra morphism with respect to the Yoneda product if $M=\X$ is the
diagonal $\X$-bimodule.

Returning to the context of \Cref{setting:resticted_universal_Massey}, if we let
$\X=\C$ and $M=\twBim[[q]]{\C}$ in the previous discussion, for each $\C$-module
$N$ we obtain a morphism
\[
  \Phi_N\colon\Ext[\C^e]{\C}{\twBim[[q]]{\C}}[d+2]\longrightarrow\Ext[\C]{N}{N\otimes_\C\twBim[[q]]{\C}}[d+2].
\]
In particular, if $q=-d$, we may consider the class
\[
  \Phi_N(j^*\class{m_{d+2}})\in\Ext[\C]{N}{N\otimes_\C\twBim[[-d]]{\C})}[d+2]\cong\Ext[\C]{N}{\Sigma^{-1}N}[d+2]
\]
induced by the restricted universal Massey product of length $d+2$. Thus, given
an exact sequence of $\C$-bimodules
\[
  \eta\colon\qquad 0\to\twBim[[-d]]{\C}\to P_{d+1}\to\cdots\to P_1\to
  P_0\to\C\to0
\]
that represents the restricted universal Massey product $j^*\class{m_{d+2}}$, we
see that the class $\Phi_N(j^*\class{m_{d+2}})$ is represented by the exact
sequence
\[
  N\otimes_\C(0\to\twBim[[-d]]{\C}\to P_{d+1}\to\cdots\to P_1\to P_0\to\C\to0).
\]

Consider a sequence of composable morphisms in $\C$ of the form
\[
  M_{d+2}\xrightarrow{f_{d+2}}M_{d+1}\xrightarrow{f_{d+1}}\cdots\xrightarrow{f_2}M_1\xrightarrow{f_1}M_0
\]
and such that $f_if_{i+1}=0$ for all $0\leq i<d+2$, so that the Massey product
$\MasseyProduct{f_1,\dots,f_{d+2}}$ in $\dgH{\A}$ is non-empty (see
\Cref{rmk:MasseyProduct_non-empty}). Let
\[
  N\coloneqq\coker\C(-,f_1)\colon\C^\op\longrightarrow\mmod{\kk}.
\]
In view of \Cref{defprop:TodaBracket_indeterminacy} and \Cref{thm:TodaMassey},
the Massey product is an element of the vector space $\Massey{f_{d+2}}{f_1}$
defined as the quotient
\[
  \frac{\Hom[\DerCat[c]{\A}]{M_{d+2}}{M_0[-d]}}{f_1[-d]\cdot\Hom[\DerCat[c]{\A}]{M_{d+2}}{M_1[-d]}+\Hom[\DerCat[c]{\A}]{M_{d+1}}{M_0}\cdot
    f_{d+2}}.
\]
We wish to construct a further morphism
\[
  \Psi\colon\Ext[\C]{N}{\Sigma^{-1}N}[d+2]\longrightarrow\Massey{f_{d+2}}{f_1}.
\]
For this, let $B_\bullet(\C)$ be the bar complex of $\C$, so that $N\otimes_\C
B_\bullet(\C)$ is a projective resolution of $N$ that we use to compute the
extension space $\Ext[\C]{N}{\Sigma^{-1}N}[d+2]$ as
\[
  \Ext[\C]{N}{\Sigma^{-1}N}[d+2]=\dgH[d+2]{\Hom[\C]{N\otimes_\C
      B_\bullet(\C)}{\Sigma^{-1}N}}.
\]
Following essentially the same arguments as in~\cite[Sec.~5]{Mur20a}, one can
check that under the above isomorphism the class $\Phi_N(j^*\class{m_{d+2}})$ is
represented by the class of the morphism
\[
  N\otimes_\C B_{d+2}(\C)\longrightarrow \Sigma^{-1} N
\]
that corresponds, under the Yoneda embedding, to the morphism
\begin{align*}
  \bigoplus_{L_0,L_1,\dots,L_{d+2}}N(L_0)\otimes\C(L_1,L_0)\otimes\cdots\otimes\C(L_{d+2},L_{d+1}) & \longrightarrow \Sigma^{-1}N(L_{d+2})
\end{align*}
given by
\begin{equation}
  \label{eq:aux_PhiN}
  [h]\otimes g_1\otimes\cdots\otimes g_{d+2}\longmapsto[(h[-d])\circ m_{d+2}(g_1,\dots,g_{d+2})],
\end{equation}
where we use that
\[
  m_{d+2}(g_1,\dots,g_{d+2})\in\Hom[\DerCat[c]{\A}]{L_{d+2}}{L_0[-d]}
\]
and therefore the composite
\[
  \begin{tikzcd}[column sep={6.5em}]
    L_{d+2}\rar{m_{d+2}(g_1,\dots,g_{d+2})}&L_0[-d]\rar{h[-d]}&N[-d]=\Sigma^{-1}N
  \end{tikzcd}
\]
indeed represents an element in $\Sigma^{-1}N(L_{d+2})$. We refer the reader to
the last paragraph of the proof of~\cite[Prop.~5.2]{Mur20a} as well
as~\cite[Prop.~5.6 and Rmks.~5.4 and 5.7]{Mur20a} for details.

Continuing with the construction of the morphism $\Psi$, consider the morphism
of augmented chain complexes
\[
  \begin{tikzcd}[column sep=tiny]
    \cdots\rar&0\rar\dar&\C(-,M_{d+2})\rar\dar&\cdots\rar&\C(-,M_1)\rar\dar&\C(-,M_0)\rar{p}\dar&N\dar[equals]\\
    \cdots\rar&N\otimes_\C B_{d+3}(\C)\rar&N\otimes_\C
    B_{d+2}(\C)\rar&\cdots\rar&N\otimes_\C B_1(\C)\rar&N\otimes_\C
    B_0(\C)\rar&N.
  \end{tikzcd}
\]
that corresponds, under the Yoneda embedding, to the elements
\begin{equation}
  \label{eq:aux_chain}
  (-1)^{\sum_{j=1}^{i}j}p\otimes f_1\otimes\cdots\otimes
  f_i\otimes\id[M_i]\in (N\otimes_\C B_i(\C))(M_i),
\end{equation}
compare with the proof of~\cite[Prop.~5.9]{Mur20a}. Here $p$ is the natural projection onto the cokernel. The morphism $\Psi$ is
obtained by applying the functor $\Hom[\C]{-}{\Sigma^{-1}N}$ to the above
morphism of chain complexes and passing to cohomology in degree $d+2$, while
keeping in mind the isomorphisms
\[
  \Hom[\C]{\C(-,M_{d+2})}{\Sigma^{-1}N}\cong\Sigma^{-1}N(M_{d+2})\cong\frac{\C(M_{d+2},M_0[-d])}{(f_1[-d])\cdot\C(M_{d+2},M_1[-d])}.
\]

As a consequence of the above discussion, we obtain the following relationship
between the restricted universal Massey product $j^*\class{m_{d+2}}$ and the
Massey product $\MasseyProduct{f_1,\dots,f_{d+2}}$.

\begin{proposition}
  \label{prop:PsiPhiN}
  Consider a sequence of composable morphisms in $\C$ of the form
  \[
    M_{d+2}\xrightarrow{f_{d+2}}M_{d+1}\xrightarrow{f_{d+1}}\cdots\xrightarrow{f_2}M_1\xrightarrow{f_1}M_0
  \]
  and such that $f_if_{i+1}=0$ for all $0\leq i<d+2$. Set $N=\coker\C(-,f_1)$.
  Then,
  \[
    (\Psi\circ\Phi_N)(j^*\class{m_{d+2}})=-(-1)^{\sum_{i=1}^{d+2}i}\MasseyProduct{f_1,\dots,f_{d+2}}
  \]
  in $\Massey{f_{d+2}}{f_1}$.
\end{proposition}
\begin{proof}
  In view of \eqref{eq:aux_PhiN} and \eqref{eq:aux_chain},
  \begin{align*}
    (\Psi\circ\Phi_N)(j^*\class{m_{d+2}}) & =(-1)^{\sum_{i=1}^{d+2}i}\class{m_{d+2}(f_1,\dots,f_{d+2})}\in\Massey{f_{d+2}}{f_1}.
  \end{align*}
  The claim then follows from \Cref{prop:MasseyProduct_Ainfty_m}, which in this
  case says that
  \[
    -m_{d+2}(f_1,\dots,f_{d+2})\in\MasseyProduct{f_1,\dots,f_{d+2}}.\qedhere
  \]
\end{proof}

\begin{remark}
  In view of \Cref{thm:TodaMassey}, under the assumptions of
  \Cref{prop:PsiPhiN},
  \[
    (\Psi\circ\Phi_N)(j^*\class{m_{d+2}})[d]=(-1)^{d+1}\TodaBracket{f_1,\dots,f_{d+2}}\in\Toda{f_{d+2}}{f_1}.
  \]
\end{remark}

The following proposition shows that the standard $(d+2)$-angles in $\C$ are
detected by the restricted universal Massey product of length $d+2$. We remind the reader that we are in \Cref{setting:resticted_universal_Massey}, however in the statement we do not assume $\A$ to be pre-$(d+2)$-angulated.

\begin{proposition}
  \label{prop:std_d+2=Massey}
  Let $(\dgH[\bullet]{\A},m_{d+2},m_{2d+2},\dots)$ be a minimal $A_\infty$-model
  of $\A$. Consider a sequence of composable morphisms
  \begin{equation}
    \label{eq:the_angle}
    \begin{tikzcd}
      M_{d+2}\rar{f_{d+2}}&M_{d+1}\rar{f_{d+1}}&\cdots\rar{f_2}&M_1\rar{f_1}&\Sigma
      M_{d+2}.
    \end{tikzcd}
  \end{equation}
  in $(\C,\Sigma)\simeq(\dgH[0]{\A},[d])$. Then, the following statements are
  equivalent:
  \begin{enumerate}
  \item\label{it:is_std} The above sequence satisfies the conditions in
    \Cref{thm:TodaBracket_d+2-angles}.
  \item\label{it:is_Massey} The above sequence lies in
    $(-1)^{\Sigma_{i=1}^{d+2}i}\pentagon_{j^*\class{m_{d+2}}}$ (see
    \Cref{not:Massey_angles} and \Cref{rmk:several_angulations_and_extensions}).
  \end{enumerate}
\end{proposition}
\begin{proof}
  For simplicity, we identify \eqref{eq:the_angle} with a sequence of projective
  $\Lambda$-modules. In view of \Cref{thm:TodaMassey}, statement
  \eqref{it:is_std} is equivalent to the following two conditions:
  \begin{itemize}
  \item The sequence
    \begin{center}
      \begin{tikzcd}
        M_{d+2}\rar{f_{d+2}}&M_{d+1}\rar{f_{d+1}}&\cdots\rar{f_2}&M_1\rar{f_1}&\Sigma
        M_{d+2}\rar{\Sigma f_{d+2}}&\Sigma M_{d+1}
      \end{tikzcd}
    \end{center}
    is exact.
  \item We have $-\id[M_{d+2}]\in\MasseyProduct{f_1,\dots,f_{d+2}}$.
  \end{itemize}
  The first condition is precisely the first condition that a sequence must
  satisfy to be a member of the class
  $(-1)^{\Sigma_{i=1}^{d+2}i}\pentagon_{j^*\class{m_{d+2}}}$.

  On the other hand, the second (and last) condition that a sequence must
  satisfy to be a member of the class
  $(-1)^{\Sigma_{i=1}^{d+2}i}\pentagon_{j^*\class{m_{d+2}}}$ says the following:
  Let $N=\coker{f_1}$, so that there is an exact sequence
  \[
    0\to\Sigma^{-1}N\xrightarrow{i} M_{d+1}\to\cdots\to
    M_1\xrightarrow{f_1}\Sigma M_{d+2}\xrightarrow{p} N\to 0
  \]
  with $f_{d+2}=i\circ\Sigma^{-1}p$. Then, the above sequence must represent the
  class
  \[
    \Phi_N((-1)^{\Sigma_{i=1}^{d+2}i}j^*\class{m_{d+2}})\in\Ext[\Lambda^e]{N}{\Sigma^{-1}N}[d+2]
  \]
  that, by definition, is represented by the exact sequence
  \[
    N\otimes_\Lambda(0\to\twBim{\Lambda}[\sigma]\stackrel{(-1)^{\Sigma_{i=1}^{d+2}i}\iota}{\longrightarrow}
    P_{d+1}\to\cdots\to P_1\to P_0\to\Lambda\to0)
  \]
  (the $\Lambda$-bimodule extension on the right of the tensor product
  represents the class $(-1)^{\Sigma_{i=1}^{d+2}i}j^*\class{m_{d+2}}$, see
  \Cref{not:Massey_angles} and \Cref{rmk:several_angulations_and_extensions}).
  The extension space can be computed by means of the projective resolution
  \[
    \cdots\to\Sigma^{-1}M_1\xrightarrow{\Sigma^{-1}f_1}M_{d+2}\xrightarrow{}M_{d+1}\to\cdots\to
    M_1\xrightarrow{f_1}\Sigma M_{d+2}\xrightarrow{p} N\to 0
  \]
  and, in these terms, we see that the former sequence is represented by the
  class
  \[
    [\Sigma^{-1}p]\in\frac{\ker(?\circ\Sigma^{-1}f_1)}{\Hom[\Lambda]{M_{d+1}}{\Sigma^{-1}N}\cdot
      f_{d+2}}\cong\Ext[\Lambda]{N}{\Sigma^{-1}N}[d+2].
  \]
  Here
  \[?\circ\Sigma^{-1}f_1\colon \Hom[\Lambda]{M_{d+2}}{\Sigma^{-1}N}
    \longrightarrow \Hom[\Lambda]{\Sigma^{-1}M_{1}}{\Sigma^{-1}N}\] is the
  morphism given by pre-composition with $\Sigma^{-1}f_1$. Thus, to summarise,
  the second condition for membership in the class
  $(-1)^{\Sigma_{i=1}^{d+2}i}\pentagon_{j^*\class{m_{d+2}}}$ says that
  \[
    \left[\Sigma^{-1}p
    \right]=\Phi_N((-1)^{\Sigma_{i=1}^{d+2}i}j^*\class{m_{d+2}})
  \]
  in $\Ext[\Lambda^e]{N}{\Sigma^{-1}N}[d+2]$.

  Now, the morphism $\Psi$ is injective by construction and identifies the
  extension space $\Ext[\Lambda]{N}{\Sigma^{-1}N}[d+2]$ with the subspace
  \[
    \frac{\ker(?\circ\Sigma^{-1}f_1)}{\Hom[\Lambda]{M_{d+1}}{\Sigma^{-1}N}\cdot
      f_{d+2}}\subseteq\frac{\Hom[\Lambda]{M_{d+2}}{\Sigma^{-1}N}}{\Hom[\Lambda]{M_{d+1}}{\Sigma^{-1}N}\cdot
      f_{d+2}}\cong\Massey{f_{d+2}}{f_1}.
  \]
  Thus, to conclude the proof it is enough to prove that
  \[
    \Psi(\Phi_N((-1)^{\Sigma_{i=1}^{d+2}i}j^*\class{m_{d+2}}))=\Psi(\left[\Sigma^{-1}p
    \right])
  \]
  if and only if $-\id[M_{d+2}]\in\MasseyProduct{f_1,\dots,f_{d+2}}$. Indeed,
  \[
    \Psi\colon
    \left[\Sigma^{-1}p\right]\longmapsto\left[\id[M_{d+2}]\right]\in\Massey{f_{d+2}}{f_1}.
  \]
  and
  \[
    \Psi\colon
    \Phi_N((-1)^{\Sigma_{i=1}^{d+2}i}j^*\class{m_{d+2}})\longmapsto-\MasseyProduct{f_1,\dots,f_{d+2}}
  \]
  by \Cref{prop:PsiPhiN}. The claim follows.
\end{proof}

The graded algebra structure on $\gLambda$ endows the Hochschild cohomology
\[
  \HH{\Lambda}[\gLambda]=\Ext[\Lambda^e]{\Lambda}{\gLambda}[\bullet,*]
\]
with the structure of a bigraded algebra, where the internal grading $*$ is induced
by the grading of $\gLambda$. In particular, it is concentrated in internal degrees
$d\ZZ$ and the internal degree $0$ part
\[
  \HH[\bullet][0]{\Lambda}[\gLambda]=\Ext[\Lambda^e]{\Lambda}{\gLambda[0]}[\bullet]=\Ext[\Lambda^e]{\Lambda}{\Lambda}[\bullet]
\]
is the Yoneda algebra of the diagonal $\Lambda$-bimodule. Similarly, since
$\Lambda^e$ is also self-injective, we may consider the \emph{Hochschild--Tate
  cohomology}
\[
  \TateHH{\Lambda}[\gLambda]\coloneqq\TateExt[\Lambda^e]{\Lambda}{\gLambda}[\bullet,*],
\]
which is again a bigraded algebra concentrated in internal degrees $d\ZZ$, and note that there are isomorphisms
\[
  \TateHH[p][dq]{\Lambda}[\gLambda]\cong\HH[p][dq]{\Lambda}[\gLambda]=\Ext[\Lambda^e]{\Lambda}{\twBim[\sigma^q]{\Lambda}}[p],\qquad
  p>0,\quad q\in\ZZ.
\]
Similarly, given a right $\Lambda$-module $N$, the Yoneda product and the graded algebra structure on $\gLambda$ induce a bigraded algebra structure on
\[\TateExt[\Lambda]{N}{N\otimes_{\Lambda}\gLambda}[\bullet,*],\]
concentrated in internal degrees $d\ZZ$, such that 
\[\TateExt[\Lambda]{N}{N\otimes_{\Lambda}\gLambda}[p,dq]\cong \Ext[\Lambda]{N}{N\otimes_{\Lambda}\twBim[\sigma^q]{\Lambda}}[p],\qquad
  p>0,\quad q\in\ZZ.\]
We refer the reader to \cite[Sec.~5]{Mur22} for details.

\begin{remark}
  \label{rmk:TateUnits}
  Recall that the automorphism $\sigma$ of the algebra $\Lambda$ induces an isomorphism of $\Lambda$-bimodules $\twBim[\sigma^{-1}]{\Lambda}\cong\twBim{\Lambda}[\sigma]$.
  An extension
  \[
    0\to\twBim{\Lambda}[\sigma]\to P_{d+1}\to\cdots\to P_1\to P_0\to\Lambda\to0
  \]
  of $\Lambda$-bimodules such that $P_i$, $0\leq i<d+1$, is a
  projective-injective $\Lambda$-bimodule represents a unit in the
  Hochschild--Tate cohomology $\TateHH{\Lambda}[\gLambda]$ if and only if
  $P_{d+1}$ is also a projective-injective $\Lambda$-bimodule,
  see \cite[Rmk.~5.8]{Mur22}. 
  Similarly, given a right $\Lambda$-module $N$,
  an extension
  \[
    0\to N\otimes_{\Lambda}\twBim{\Lambda}[\sigma]\to Q_{d+1}\to\cdots\to Q_1\to Q_0\to N\to0
  \]
  of right $\Lambda$-modules such that $Q_i$, $0\leq i<d+1$, is a
  projective-injective right $\Lambda$-module represents a unit in
  $\TateExt[\Lambda]{N}{N\otimes_{\Lambda}\gLambda}[\bullet,*]$ if and only if
  $Q_{d+1}$ is also a projective-injective right $\Lambda$-module; 
  see \cite[Rmk.~5.9]{Mur22}.
\end{remark}

We are finally ready to prove \Cref{thm:GKO_AL_Massey}.

\begin{proof}[Proof of \Cref{thm:GKO_AL_Massey}]
  Statement \eqref{it:std=AL} follows from \Cref{prop:std_d+2=Massey}. We prove
  statement \eqref{it:is_Tate_unit}. Choose an exact sequence
  \[
    \eta\colon\qquad 0\to\twBim{\Lambda}[\sigma]\stackrel{\iota}{\to} P_{d+1}\to\cdots\to P_1\to
    P_0\to\Lambda\to0
  \]
  that represents the restricted universal Massey product
  \[
    j^*\class{m_{d+2}}\in\Ext[\Lambda^e]{\Lambda}{\twBim{\Lambda}[\sigma]}[d+2]
  \]
  and such that $P_i$, $0\leq i<d+1$, is a projective-injective
  $\Lambda$-bimodule (such a representative always exists, see \cite[p.~151]{hilton_stammbach_1971_course_homological_algebra}). We need to prove
  that the $\Lambda$-bimodule $P_{d+1}$ is also projective-injective. It is
  enough to prove that $N\otimes P_{d+1}$ is projective for every
  finite-dimensional $\Lambda$-module $N$, see \cite[Thm.~3.1]{AR91a} and notice
  that the proof remains valid under the assumption of $\Lambda/J_\Lambda$ being
  separable, since in this case the finite-dimensional algebra
  $(\Lambda/J_\Lambda)\otimes(\Lambda/J_\Lambda)^{\op}$ is semisimple
  \cite[Cor.~18]{ERZ57}.

  Thus, let $N$ be a finite-dimensional $\Lambda$-module and choose a projective
  presentation
  \[
    Q_1\stackrel{f_1}{\longrightarrow}\Sigma Q_{d+2}\to N\to 0.
  \]
  Using axioms \eqref{dTR1} and \eqref{dTR2}, complete the morphism $f_1$ to a
  standard $(d+2)$-angle
  \[
    Q_{d+2}\to Q_{d+1}\to\cdots\to Q_1\xrightarrow{f_1}\Sigma Q_{d+2}.
  \]
  In view of the validity of statement \eqref{it:std=AL}, the above
  $(d+2)$-angle lies in the class $\pm\pentagon_{j^*\class{m_{d+2}}}$. In
  particular, the exact sequence
  \[
    N\otimes_\Lambda(0\to\twBim{\Lambda}[\sigma]\stackrel{\pm\iota}{\to} P_{d+1}\to\cdots\to P_1\to
    P_0\to\Lambda\to0)
  \]
  is equivalent to an exact sequence of the form
  \[
    0\to N\otimes_{\Lambda}\twBim[\sigma]{\Lambda}=\Sigma^{-1}N\xrightarrow{} Q_{d+1}\to\cdots\to Q_1\xrightarrow{}\Sigma
    Q_{d+2}\xrightarrow{} N\to 0.
  \]
  Notice that all of the middle terms in the latter exact sequence are
  projective-injective $\Lambda$-bimodules and, therefore, by \Cref{rmk:TateUnits}, the class of this
  sequence in $\TateExt[\Lambda]{N}{N\otimes_{\Lambda}\gLambda}[\bullet,*]$ is a
  unit. Consequently, the former (equivalent) sequence also represents a unit in
  this bigraded algebra and hence the $\Lambda$-module
  ${N\otimes_\Lambda P_{d+1}}$ must be projective-injective again by \Cref{rmk:TateUnits}, which is what we
  needed to prove. This finishes the proof of the theorem.
\end{proof}

\Cref{thm:GKO_AL_Massey} and \Cref{prop:std_d+2=Massey} yield the following important corollaries.

\begin{corollary}
  \label{coro:pre-d+2-ang=unit}
  Let $\A$ be a small DG category satisfying the assumptions in
  \Cref{setting:resticted_universal_Massey}. Consider the following statements:
  \begin{enumerate}
  \item\label{it:isKaroubian_d+2} The DG category $\A$ is pre-$(d+2)$-angulated.
  \item\label{it:isdZCT} The canonical fully faithful functor
    \[
      \dgH[0]{\Yoneda}\colon\dgH[0]{\A}\longrightarrow\DerCat[c]{\A}
    \]
    identifies $\dgH[0]{\A}$ with a $d\ZZ$-cluster tilting subcategory of
    $\DerCat[c]{\A}$.
  \item\label{it:H0_isAL-angulated} The restricted universal Massey product
    \[
      j^*\class{m_{d+2}}\in\Ext[\Lambda]{\Lambda}{\twBim{\Lambda}[\sigma]}[d+2]
    \]
    is represented by an extension all of whose middle terms are
    projective-injective $\Lambda$-bimodules.
  \end{enumerate}
  In general \eqref{it:isKaroubian_d+2}$\Leftrightarrow$\eqref{it:isdZCT} and
  \eqref{it:isKaroubian_d+2}$\Leftarrow$\eqref{it:H0_isAL-angulated}. If
  $\Lambda/J_\Lambda$ is separable, then
  \eqref{it:isKaroubian_d+2}$\Rightarrow$\eqref{it:H0_isAL-angulated}.
\end{corollary}
\begin{proof}
  The equivalence \eqref{it:isKaroubian_d+2}$\Leftrightarrow$\eqref{it:isdZCT}
  is \Cref{prop:d+2=dZ-CT}.
  
  \eqref{it:isKaroubian_d+2}$\Rightarrow$\eqref{it:H0_isAL-angulated} This is
  precisely \Cref{thm:GKO_AL_Massey}\eqref{it:is_Tate_unit}, which requires the
  algebra $\Lambda/J_\Lambda$ to be separable.

  \eqref{it:H0_isAL-angulated}$\Rightarrow$\eqref{it:isKaroubian_d+2} By
  \Cref{prop:std_d+2=Massey} and \Cref{rmk:several_angulations_and_extensions}, the Amiot--Lin $(d+2)$-angulation associated to a representative of $(-1)^{\Sigma_{i=1}^{d+2}i}j^*\class{m_{d+2}}$ coincides with
  the class of sequences that satisfy the conditions in
  \Cref{thm:TodaBracket_d+2-angles}, which therefore forms a $(d+2)$-angulation.
  Consequently, the canonical fully faithful functor
  \[
    \dgH[0]{\Yoneda}\colon\dgH[0]{\A}\longrightarrow\DerCat[c]{\A}
  \]
  identifies the left-hand side with a standard $(d+2)$-angulated subcategory of
  the right-hand side, which is precisely the definition of $\A$ being
  pre-$(d+2)$-angulated.
\end{proof}

\begin{remark}
  \label{rmk:Dugas_isomorphism}
  Suppose for simplicity that the field $k$ is perfect. In the context of
  \Cref{coro:pre-d+2-ang=unit}, it follows from~\cite[Thm.~3.2]{Dug12} that
  there is a stable bimodule isomorphism
  $\Omega_{\Lambda^e}^{d+2}(\Lambda)\cong{}_1\Lambda_\sigma$. An important aspect of 
  \Cref{coro:pre-d+2-ang=unit} is that this stable bimodule isomorphism is
  witnessed by the restricted universal Massey product of $\A$.
\end{remark}

Recall that the Nakayama automorphism $\eta$ of a Frobenius algebra $\Lambda$ is
characterised (up to inner automorphisms) by the existence of an isomorphism of
$\Lambda$-bimodules $D\Lambda\cong{}_{\eta}\Lambda_{1}$.

\begin{corollary}[{\cite[Prop.~3.3]{Dug12}, \cite[Sec.~5.4]{GKO13}, \cite[Cor.~4.6]{IO13}}]
  \label{coro:Nakayama_automorphism}
  Let $\A$ be a small DG category satisfying the assumptions in
  \Cref{setting:resticted_universal_Massey} and such that the canonical fully faithful functor
  \[
    \dgH[0]{\Yoneda}\colon\dgH[0]{\A}\longrightarrow\DerCat[c]{\A}
  \]
  identifies $\dgH[0]{\A}$ with a $d\ZZ$-cluster tilting subcategory of
  $\DerCat[c]{\A}$. If $\dgH[0]{\A}$ is $d\ell$-Calabi--Yau, $\ell\in\ZZ$, in the sense that there
  are functorial isomorphisms
  \[
    \dgH[0]{\A}(y,x[d\ell])\stackrel{\sim}{\longrightarrow}D\dgH[0]{\A}(x,y),\qquad x,y\in\dgH[0]{A},
  \]
  then $\sigma^{\ell}$ is a Nakayama automorphism for $\Lambda$. In particular,
  the stable module category $\smmod{\Lambda}$ is $((d+2)\ell-1)$-Calabi--Yau.
\end{corollary}
\begin{proof}
  The argument is standard; we reproduce it for the
  convenience of the reader. The functorial isomorphism
  \[
    \dgH[0]{\A}(c,c[d\ell])\stackrel{\sim}{\longrightarrow}D\dgH[0]{\A}(c,c)
  \]
  can be interpreted as an isomorphism of $\Lambda$-bimodules
  ${}_{\sigma^{\ell}}\Lambda_{1}\cong D\Lambda$ that exhibits $\sigma^{\ell}$ as a
  Nakayama automorphism for $\Lambda$  (recall that\[
    \dgH[0]{\A}(c,c[d])\cong\twBim[\sigma]{\Lambda}
  \]
  as $\Lambda$-bimodules, see~\Cref{setting:resticted_universal_Massey}).  This proves the first claim. To prove
  the second claim, observe that \Cref{coro:pre-d+2-ang=unit} yields the
  existence of an isomorphism
  \[
    \Omega_{\Lambda^e}^{d+2}(\Lambda)\cong{}_1\Lambda_{\sigma},
  \]
  in $\smmod{\Lambda^e}$ that in turn induces an isomorphism of functors
  \[
    \Omega_{\Lambda}^{(d+2)\ell}\cong-\otimes_{\Lambda}\twBim{\Lambda}[\sigma^{\ell}]\cong\nu^{-1}
  \]
  on $\smmod{\Lambda}$, where $\nu=-\otimes_\Lambda D\Lambda$ is the Nakayama
  functor. The claim follows from the fact the suspension functor on
  $\smmod{\Lambda}$ is given by Heller's cosyzygy functor
  $\Omega_{\Lambda}^{-1}$ and that
  \[
    \nu\Omega_{\Lambda}\cong\Omega_{\Lambda}^{-((d+2)\ell-1)}
  \]
  is a Serre
  functor on $\smmod{\Lambda}$, see \cite[Prop.~1.2 and Cor.~1.3]{ES06} and
  compare also with~\cite[Thm.~1.8]{IV14}. This finishes the proof.
\end{proof}

\begin{corollary}
  \label{coro:atmostone}
  Let $\Lambda$ be a finite-dimensional self-injective algebra that is twisted
  $(d+2)$-periodic with respect to an algebra automorphism $\sigma$ and let
  \[
    \Sigma\coloneqq-\otimes_\Lambda\twBim[\sigma]{\Lambda}\colon\proj*{\Lambda}\stackrel{\sim}{\longrightarrow}\proj*{\Lambda}.
  \]
  If $\Lambda/J_\Lambda$ is separable, then, up to equivalence, the pair
  $(\proj*{\Lambda},\Sigma)$ admits at most one algebraic $(d+2)$-angulation (in
  the sense of \Cref{def:DG_enhancement_d+2-angulated}).
\end{corollary}
\begin{proof}
  Indeed, \Cref{thm:GKO_AL_Massey} shows that any two algebraic
  $(d+2)$-angulations of $(\proj*{\Lambda},\Sigma)$ coincide with the AL
  $(d+2)$-angulations given by any choice of exact sequences with
  projective-middle terms in
  $\Ext[\Lambda^e]{\Lambda}{\twBim{\Lambda}[\sigma]}[d+2]$ that represent the
  corresponding restricted universal Massey products of the enhancements, up to
  a sign depending on $d$, see also
  \Cref{rmk:several_angulations_and_extensions}. But such AL $(d+2)$-angulations
  are unique up to equivalence by \Cref{prop:Amiot-Lin_independence_of_res}.
\end{proof}

\begin{remark}
  Notice that \Cref{coro:atmostone} \emph{does not} prove that any two
  enhancements are equivalent, provided at least one exists. This is part of the
  content of \Cref{thm:dZ-Auslander_correspondence}.
\end{remark}

\subsection{Example of a universal Massey product}

In this subsection, for $d=2$, we partially compute the universal Massey product of a specific
Karoubian pre-$(d+2)$-angulated DG category $\A$ or, rather, the first non-zero
higher operation in a minimal $A_\infty$-model of $\A$. Let $A=A_3$ be the path
algebra of the quiver $1\to 2\to 3$ and consider its compact derived category
$\DerCat[c]{A}$. It is well known and easy to verify that $A$ is DG Morita
equivalent to the factor algebra $B=A/J_A^2$ (an explicit tilting object in
$\DerCat[c]{A}$ with endomorphism algebra isomorphic to $B$ is given by the
direct sum of the simple $A$-modules at the vertices $1$ and $3$ together with
the unique indecomposable projective-injective $A$-module). Thus, we may and we
will identify minimal $A_\infty$-models of $\DerCat[c]{A}$ with those of
$\DerCat[c]{B}$ (we do this to minimise the number of projective resolutions
needed in the computation below).

The algebra $B$ has global dimension $2$ and $\DerCat[c]{B}$ admits a
$2\ZZ$-cluster tilting subcategory given by
\[
  \U=\U(B)\coloneqq\add{\set{(B\oplus S_1)[2i]}[i\in\ZZ]}\subset\DerCat[c]{B},
\]
where $B$ is the regular representation and $S_1$ is the simple $B$-module
concentrated at the vertex $1$, see ~\cite[Cor.~1.15 and Thm.~1.21]{Iya11} (the
claim can also be verified directly using the explicit description of the
Auslander--Reiten quiver of $\DerCat[c]{A}\simeq\DerCat[c]{B}$ available
in~\cite[Sec.~I.5]{Hap88} for example). By \Cref{thm:GKO-standard} the
pair $(\U,[2])$ has an induced standard $4$-angulation; it is also easy to
verify that all $4$-angles in $\U$ are
finite direct sums of rotations of trivial $4$-angles as in \eqref{dTR1b} and the apparent $4$-angle
\[
  P_3 \stackrel{a}\to P_2\stackrel{b}\to P_1\to S_1\to P_3[2],
\]
where $P_i$ denotes the indecomposable projective $B$-module at the
vertex $i$ and 
the connecting morphism $S_1\to P_3[2]$ classifies the (essentially unique)
non-zero class in $\Ext[B]{S_1}{P_3}[2]$; thus, the above $4$-angle consists of
the minimal projective resolution of $S_1$ together with its canonical augmentation.
\Cref{thm:TodaBracket_d+2-angles} implies that the Toda
bracket of the morphisms in the above $4$-angle contains
$-\id[P_3[2]]$, and by \Cref{thm:TodaMassey} the Massey product contains $-\id[P_3]$.
By \Cref{thm:TodaMassey} and \Cref{defprop:TodaBracket_indeterminacy}, this Massey product is an element of the quotient 
\[\frac{\hom_B(P_3,P_3)}{\hom_B(P_2,P_3)\cdot a}=\frac{\kk}{0\cdot a}\cong \kk,\]
hence it must be equal to $-\id[P_3]$. \Cref{prop:MasseyProduct_Ainfty_m} then implies that, for any minimal $A_\infty$-model of the
canonical DG enhancement of $\DerCat[c]{B}$, applying the operation $m_4$ to the morphisms in the above $4$-angle yields $-\id[P_3]$, in particular $m_4$ is non-trivial\footnote{This can also be verified directly using explicit formula for computing
minimal models available for example in~\cite{Mar06}.}

\begin{remark}
  The above example shows that \cite[Thm.~D]{Her16} cannot hold as stated.
  Indeed, \cite[Thm.~D]{Her16} claims that the compact derived category of an
  acyclic quiver has a minimal $A_\infty$-model with only one non-zero higher
  operation, namely $m_3$; the previous example shows that there are higher
  non-vanishing operations in general. Notice also that~\cite{Her16} would
  contradict our~\Cref{coro:pre-d+2-ang=unit}, as the latter implies that the
  Hochschild class of $m_4$ in the above example must be non-zero. The error in
  the proof of \cite[Thm.~D]{Her16} occurs in the proof of Prop.~3.8.1 therein,
  where it is claimed that the higher operations $m_n$, $n>3$, vanish for the
  minimal $A_\infty$-models of bounded derived categories of hereditary abelian categories
  with enough projectives. 
\end{remark}

\subsection{Hochschild cohomology of sparse graded algebras}

\begin{setting}
  \label{setting:gLambda}
  We fix a finite-dimensional self-injective algebra $\Lambda$ and an algebra
  automorphism $\sigma\colon\Lambda\simto\Lambda$.
  
  In this technical subsection we study the Hochschild cohomology of the
$d$-sparse graded algebra
\[
  \gLambda\coloneqq\frac{\Lambda\langle\imath^{\pm 1}\rangle}{(\imath
    x-\sigma(x)\imath)_{x\in\Lambda}},\qquad |\imath|=-d,
\]
(compare with \Cref{rmk:Lambda-sigma-d}). The results in this section extend
results obtained in \cite{Mur22} in the case $d=1$; these results are used in a
crucial way in \Cref{subsec:secondary_formality,subsec:thmA-Ainfty-version},
where we prove some our main results on uniqueness of enhancements
(\Cref{thm:secondary_formality} and \Cref{thm:A-Ainfty-version}).
\end{setting}

We begin with some general observations that hold for an arbitrary sparse graded
algebra.

\begin{definition}
  Let $A$ be a $d$-sparse graded algebra. The \emph{fractional Euler derivation}
  $\EulerDer[d]\colon A\to A$ is the degree $0$ map
  \[
    \EulerDer[d](a)=\frac{|a|}{d}\cdot a,
  \]
  where the coefficient is indeed an integer in view of the sparseness of $A$.
\end{definition}

\begin{remark}
  Let $A$ be a $d$-sparse graded algebra. The (usual) Euler derivation is
  $\delta=\EulerDer$. If $d$ is invertible in the ground field, then
  $\EulerDer[d]=\frac{1}{d}\delta$.
\end{remark}

\begin{proposition}
  \label{prop:euler_lie}
  Let $A$ be a $d$-sparse graded algebra and $\varphi\in\HC[p][q]{A}$ a
  Hochschild cochain. Then,
  \[
    [\EulerDer[d],\varphi]=\frac{q}{d}\cdot\varphi.
  \]
  Moreover, $\EulerDer[d]$ is a Hochschild cocycle.
\end{proposition}
\begin{proof}
  The proof is almost identical to those of \cite[Prop.~3.2 and Cor.~3.3]{Mur22}
  and is left to the reader.
\end{proof}

\begin{notation}
  Let $A$ be a $d$-sparse graded algebra. The Hochschild cohomology class of the
  fractional Euler derivation, called \emph{fractional Euler class}, is denoted
  \[
    \class{\EulerDer[d]}\in\HH[1][0]{A}.
  \]
\end{notation}

\begin{proposition}
  \label{prop:euler_square}
  The square of the fractional Euler class vanishes:
  \[
    \class{\EulerDer[d]}^2=0\in\HH[2][0]{A}.
  \]
\end{proposition}
\begin{proof}
  The proof of \cite[Prop.~3.8]{Mur22} applies almost verbatim, only that one
  needs to replace the cochain $\beta\in\HC[1][0]{A}$ in \loccit~by the cochain
  \[
    \beta\colon x\longmapsto
    \frac{\frac{|x|}{d}\left(1-\frac{|x|}{d}\right)}{2}\cdot x.\qedhere
  \]
\end{proof}

\begin{proposition}
  If $\cchar{\kk}=2$, then $\Sq(\class{\EulerDer[d]})=\class{\EulerDer[d]}$.
\end{proposition}
\begin{proof}
  Indeed, the claim follows from the fact that every integer its congruent to
  its square modulo $2$ (compare with \cite[Prop.~3.5]{Mur22}).
\end{proof}

Straightforward computations that combine the above propositions with the laws
of a Gerstenhaber algebra yield the following identities (compare with
\cite[Props.~3.6, 3.9, 3.10]{Mur22}). We leave the proof to the reader.

\begin{proposition}
  \label{prop:HH_computations}
  Let $A$ be a $d$-sparse graded algebra. The following identities hold for
  $x\in\HH[p][dq]{A}$ and $y\in\HH[s][dt]{A}$:
  \begin{align*}
    [\class{\EulerDer[d]}\cdot x,\class{\EulerDer[d]}\cdot y] & =(t-q)\cdot\class{\EulerDer[d]}\cdot x\cdot y.                                                          \\
    \Sq(\class{\EulerDer[d]}\cdot x)                          & =0\quad\text{if $\cchar{\kk}\neq2$ and $p+dq$ is odd.}                                                  \\
    \Sq(\class{\EulerDer[d]}\cdot x)                          & =(q+1)\cdot\class{\EulerDer[d]}\cdot x^2\quad\text{if $\cchar{\kk}=2$.}                                 \\
    y\cdot x                                                  & =[y,\class{\EulerDer[d]}\cdot x]+\class{\EulerDer[d]}\cdot[y,x]\quad\text{if $t=-1$ and $d-s$ is even.}
  \end{align*}
\end{proposition}

We return to our object of interest in this subsection, which is the Hochschild
cohomology of the $d$-sparse graded algebra $\gLambda$ (see
\Cref{setting:gLambda}).

\begin{remark}
  \label{rmk:graded_sigma}
  Recall that, by definition, the following identity
  holds in $\gLambda$:
  \[
    \imath x\imath^{-1}=\sigma(x),\qquad x\in\Lambda.
  \]
  Now, the automorphism $\sigma\colon\Lambda\simto\Lambda$ extends to $\gLambda$
  via the formula
  \[
    \sigma(\imath)\coloneqq(-1)^{|\imath|}\imath=(-1)^{-d}\imath.
  \]
  Since $\gLambda[dk]=\set{\imath^{-k}x}[x\in\Lambda]$, we see that the equality
  \[
    \imath x\imath^{-1}=(-1)^{|x|}\sigma(x)
  \]
  holds for all homogeneous elements $x\in\gLambda$.
\end{remark}

Recall that we have a canonical comparison map
\[
  \HH{\Lambda}[\gLambda]\longrightarrow\TateHH{\Lambda}[\gLambda]
\]
from Hochschild cohomology to Hochschild--Tate cohomology, that is an
isomorphism in positive horizontal degree and an epimorphism in horizontal
degree $0$. In particular, we may identify $\HH[>0]{\Lambda}[\gLambda]$ with
the (horizontal) positive part of $\TateHH{\Lambda}[\gLambda]$.
We introduce the following auxiliary definition, see \cite[Sec.~6]{Mur22} for
justification for the terminology (in particular \cite[Prop.~6.5 and
Cor.~6.7]{Mur22}).

\begin{definition}
  \label{def:edge_units}
  Let $p>0$.
  \begin{enumerate}
  \item A class
    \[
      x\in\HH[p][q]{\Lambda}[\gLambda]=\Ext[\Lambda^e]{\Lambda}{\twBim[\sigma^q]{\Lambda}}[p]
    \]
    is an \emph{edge unit} if its image in $\TateHH{\Lambda}[\gLambda]$ under
    the canonical map is a unit.
  \item Let $j\colon\Lambda\hookrightarrow\gLambda$ be the inclusion of
    $\Lambda$ as the degree $0$ part. A class
    $x\in\HH[p][q]{\gLambda}[\gLambda]$ is an \emph{edge unit} if the image of
    $x$ under the map
    \[
      j^*\colon\HH[p][q]{\gLambda}[\gLambda]\longrightarrow\HH[p][q]{\Lambda}[\gLambda]
    \]
    is an edge unit.
  \end{enumerate}
\end{definition}

\begin{proposition}
  \label{rmk:action}
  Consider the group $\Aut*{\gLambda}$ of graded algebra automorphisms of
  $\gLambda$. The following statements hold:
  \begin{enumerate}
    \item\label{it:canonical_equivariant_map} The canonical comparison map
  \[\HH{\Lambda}[\gLambda]\longrightarrow\TateHH{\Lambda}[\gLambda]\]
  is equivariant with respect to the natural right action of $\Aut*{\gLambda}$
  by conjugation. Moreover, the action of $\Aut*{\gLambda}$ on
  $\HH{\Lambda}[\gLambda]$ preserves edge units.
\item\label{it:def_g_u} A unit $u\in Z(\Lambda)^\times$ in the centre of
  $\Lambda$ induces a graded algebra automorphism $g_u\colon\gLambda\simto\gLambda$ that is uniquely
  determined by the requirements
  \[
    g_u(x)=x, \quad x\in\Lambda=\gLambda[0],\quad\text{and}\quad g_u(\imath)=\imath u^{-1}.
  \]
  Moreover, the map
 \begin{align*}
   \gamma\colon Z(\Lambda)^\times&\longrightarrow\Aut*{\gLambda},\\
   u&\longmapsto g_u
 \end{align*}
 is a group homomorphism.
  \item\label{it:restriction_along_gamma} The right action of $\Aut*{\gLambda}$ on $\HH[d+2][-d]{\Lambda}[\gLambda]$ restricts along the group homomorphism $\gamma$ to the multiplication action of $Z(\Lambda)^\times$ on $\HH[d+2][-d]{\Lambda}[\gLambda]$.
  \end{enumerate}
\end{proposition}
\begin{proof}
  \eqref{it:canonical_equivariant_map} Observe that the group $\Aut*{\gLambda}$ of graded algebra automorphisms of
  $\gLambda$ acts naturally on the right by conjugation on the bigraded algebra
  $\HH{\Lambda}[\gLambda]$ as in \Cref{rmk:automorphism_action_on_cochains},
  since $\Lambda=\gLambda[0]$ is the degree $0$ part. It also acts on
  $\TateHH{\Lambda}[\gLambda]$ since Hochschild--Tate cohomology satisfies the
  same functoriality properties as Hochschild cohomology. In particular the
  latter action preserves units. It is clear that the canonical comparison map
  \[\HH{\Lambda}[\gLambda]\longrightarrow\TateHH{\Lambda}[\gLambda]\] is $\Aut*{\gLambda}$-equivariant, and hence also that the action of $\Aut*{\gLambda}$ on
  $\HH{\Lambda}[\gLambda]$ preserves edge units.

  \eqref{it:def_g_u}  Indeed, the formulas define $g_u$ on the graded algebra generators of $\gLambda$. In order to check compatibility with the defining relations of $\gLambda$ in \Cref{setting:gLambda}, we observe that $g_u(\imath)=\imath u^{-1}$ is a unit with inverse $u\imath^{-1}$ and, moreover, for $x\in\Lambda=\gLambda[0]$
 \begin{align*}
 	g_u(\imath)g_u(x)-g_u(\sigma(x))g_u(\imath)&= \imath u^{-1}x-\sigma(x) \imath u^{-1}
 	= \imath (x u^{-1})-\imath (x u^{-1})=0.
 \end{align*}
 Here we use that $u^{-1} x=x u^{-1}$ since $u\in Z(\Lambda)^\times$ and
 $\sigma(x)\imath =\imath x$ by the defining relations of $\gLambda$. It is
 then clear that the map
 \begin{align*}
   \gamma\colon Z(\Lambda)^\times&\longrightarrow\Aut*{\gLambda},\\
   u&\longmapsto g_u,
 \end{align*}
 is a group homomorphism.

 \eqref{it:restriction_along_gamma}
 Since $g_u$ is the identity in degree $0$, the right action of $g_u\in \Aut*{\gLambda}$ on $x\in\HH{\Lambda}[\gLambda]$ consists of applying the morphism induced by $g_u^{-1}$, regarded as an automorphism of the coefficient graded $\Lambda$-bimodule, on Hochschild cohomology, 
 \[x^{g_u}=(g_u^{-1})_*(x)=(g_{u^{-1}})_*(x).\]
Furthermore, given $x\in\HH[d+2][-d]{\Lambda}[\gLambda]$
\[x^{g_u}=(g_{u^{-1}})_*(x)=x\cdot u=u\cdot x\]
since $g_{u^{-1}}$ is given in degree $-d$ by right multiplication with $u\in Z(\Lambda)^\times\subset Z(\Lambda)=\HH[0][0]{\Lambda}[\gLambda]$ and the product in Hochschild cohomology is graded commutative. Therefore, the right action of $\Aut*{\gLambda}$ on $\HH[d+2][-d]{\Lambda}[\gLambda]$ restricts along the group homomorphism $\gamma$ to the multiplication action of $Z(\Lambda)^\times$ on $\HH[d+2][-d]{\Lambda}[\gLambda]$, which is what we needed to show. This finishes the proof of the proposition.
\end{proof}

\begin{proposition}
  \label{prop:long_ex_seqHH}
  Let $j\colon\Lambda\hookrightarrow\gLambda$ be the inclusion of $\Lambda$ as
  the degree $0$ part, and $\langle\sigma\rangle\subset \Aut*{\gLambda}$ the
  subgroup generated by the extension of the automorphism $\sigma$ to
  $\gLambda$, see \Cref{rmk:graded_sigma}. Then, the following statements hold:
  \begin{enumerate}
  \item\label{it:the_long_ex_seqHH} There is a long exact sequence in Hochschild
    cohomology
    \begin{center}
      \scalebox{0.85}{
        \begin{tikzcd}[ampersand replacement=\&]
          \cdots\rar\&\HH[n]{\gLambda}[\gLambda]\rar{j^*}\&\HH[n]{\Lambda}[\gLambda]\arrow[phantom]{d}[coordinate, name=Z]{}\rar{\id-\sigma_*^{-1}\sigma^*}\&\HH[n]{\Lambda}[\gLambda]
          \arrow[
          rounded corners, to path={ - - ([xshift=2ex]\tikztostart.east) |- (Z)
            [near end]\tikztonodes -| ([xshift=-2ex]\tikztotarget.west) - -
            (\tikztotarget) }
          ]{dll}[description]{\partial}
          \\
          \&\HH[n+1]{\gLambda}[\gLambda]\rar\&\cdots
        \end{tikzcd}
      }
    \end{center}
  \item\label{it:easy_props_of_the_long_ex_seqHH} The maps
    \[
      \begin{split}
        j^*\colon\HH{\gLambda}[\gLambda]&\longrightarrow\HH{\Lambda}[\gLambda],\\
        \sigma_*^{-1}\sigma^*\colon \HH{\Lambda}[\gLambda]&\longrightarrow\HH{\Lambda}[\gLambda],
      \end{split}
    \]
    are morphisms of graded algebras.
  \item\label{it:_props__delta_of_the_long_ex_seqHH} The map
    \[
      \partial\colon\HH[\bullet-1]{\Lambda}[\gLambda]\longrightarrow\HH{\gLambda}[\gLambda] \]
    is a morphism of $\HH{\gLambda}$-bimodules with square-zero image.
  \item\label{it:_props__delta_j_of_the_long_ex_seqHH} The
    map
    \[
      \partial\circ
      j^*\colon\HH{\gLambda}[\gLambda]\longrightarrow\HH[\bullet+1]{\gLambda}[\gLambda]
    \] is given by $x\mapsto\class{\EulerDer[d]}\cdot x$.
  \end{enumerate}
\end{proposition}
\begin{proof}
  This proof is similar to the union of the proofs of \cite[Prop.~2.3]{Mur20a} and \cite[Prop.~4.3]{Mur22}. The statement here is simpler because we work with graded algebras instead of categories. However, for algebras, this proposition is more general. It is actually a generalization of the case \(d=1\). Therefore, since this result plays a crucial role in this paper, we include the proof with some straightforward computations left to the reader.

  \eqref{it:the_long_ex_seqHH} The long exact sequence in the statement is
  isomorphic to the
  long exact sequence associated to the
  short exact sequence  of \(\gLambda\)-bimodules
  \[\gLambda\otimes_{\Lambda}\gLambda\stackrel{\id{}-\Gamma}{\hookrightarrow}\gLambda\otimes_{\Lambda}\gLambda\stackrel{\mu}{\twoheadrightarrow}\gLambda,\]
  where \(\mu\) is the product in \(\gLambda\otimes_{\Lambda}\gLambda\) and \(\Gamma\) is defined as
  \[\Gamma(a\otimes b)=a\imath^{-1}\otimes\imath b,\]
  and the derived functors of
  \[
    \hom^q_{\gLambda[e]}(-,\gLambda)=\Hom[\gLambda[e]]{-}{\gLambda(q)},\qquad q\in\ZZ.
  \]
  Firstly, using the isomorphisms of \(\Lambda\)-bimodules,
  \begin{align*}
      {}_{\sigma^i}\Lambda_1\otimes_{\Lambda}{}_{\sigma^j}\Lambda_1 & \stackrel{\cong}{\longrightarrow} {}_{\sigma^{i+j}}\Lambda_1, \\
      a\otimes b                                                    & \;\longmapsto\;\sigma^j(a)b,
  \end{align*}
  we see that, as \(\Lambda\)-bidmoules,
  \[(\gLambda\otimes_{\Lambda}\gLambda)^k\cong\bigoplus_{i\in\ZZ}{}_{\sigma^k}\Lambda_1.\]
  Through this decomposition, \(\mu\) maps all copies of \({}_{\sigma^k}\Lambda_1\) to itself identically, and
  \(\Gamma\) maps the \(i\)-th direct summand to the
  \((i+1)\)-st direct summand identically. Hence, the aforementioned short
  exact sequence sequence is a well-known short exact sequence degree-wise.

  The \(\gLambda\)-bimodule \(\gLambda\otimes_{\Lambda}\gLambda\) is the
  extension of scalars of \(\Lambda\) along \(j^{\otimes^2}\colon
  \Lambda^{e}\to\gLambda[e]\). It remains to show that the following square is commutative
  \begin{center}
      \begin{tikzcd}
          \hom^*_{\gLambda[e]}(\gLambda\otimes_{\Lambda}\gLambda,\gLambda)\ar{r}{\cong}\arrow{d}[swap]{\hom^*_{\gLambda[e]}(\Gamma,\gLambda)}&\hom^*_{\Lambda^{e}}(\Lambda,\gLambda)\arrow{d}{\sigma^{-1}_*\sigma^*}\\
          \hom^*_{\gLambda[e]}(\gLambda\otimes_{\Lambda}\gLambda,\gLambda)\arrow{r}{\cong}&\hom^*_{\Lambda^{e}}(\Lambda,\gLambda)
      \end{tikzcd}
  \end{center}
  Here, the horizontal isomorphisms are the adjunction isomorphisms. Let us check this. Let \(f\colon \gLambda\otimes_{\Lambda}\gLambda\to \gLambda\) be a \(\gLambda\)-bimodule morphism. Since the source and target are \(d\)-sparse, we can suppose that \(f\) has degree \(kd\) for some \(k\in\ZZ\). The image of \(f\) through the upper right corner is the \(\Lambda\)-bimodule morphism \(g\colon\Lambda\to\gLambda\) given by \(g(x)=\sigma^{-1}f(\sigma(x)\otimes 1)\). The image through the lower left corner is the \(\Lambda\)-bimodule morphism \(h\colon\Lambda\to\gLambda\) given by
  \begin{align*}
      h(x) & =f\Gamma(x\otimes 1)                                            \\
           & =f(x\imath^{-1}\otimes\imath)                                   \\
           & =f(\imath^{-1}\imath x\imath^{-1}\otimes\imath)                 \\
           & =(-1)^{kd^2}\imath^{-1}f(\imath x\imath^{-1}\otimes1)\imath \\
           & = \sigma^{-1} f(\sigma(x)\otimes 1)                             \\
           & =g(x).
  \end{align*}
  Here we use that \(f\) is a \(\gLambda\)-bimodule morphism and that \(d^2\equiv d\mod 2\).

  \eqref{it:easy_props_of_the_long_ex_seqHH} The morphisms \(j^*\) and
  \(\sigma^{-1}_*\sigma^*\) in the statement are morphisms of graded algebras by functoriality
  properties of Hochschild cohomology.

  \eqref{it:_props__delta_of_the_long_ex_seqHH} The claimed multiplicative
  property of \(\partial\) requires a deeper analysis. In order to establish it,
  we shall give yet another presentation of the long exact sequence in statement \eqref{it:the_long_ex_seqHH}.

  Observe that the \(\gLambda\)-bimodule endomorphism \(\Gamma\) lifts to the projective resolution \(\gLambda\otimes_{\Lambda}B_{\bullet}(\Lambda)\otimes_{\Lambda}\gLambda\) in the following way,
  \begin{align*}
      \Gamma\colon\gLambda\otimes_{\Lambda}B_{k}(\Lambda)\otimes_{\Lambda}\gLambda & \longrightarrow \gLambda\otimes_{\Lambda}B_{k}(\Lambda)\otimes_{\Lambda}\gLambda,  \\
      a\otimes  x_1\otimes\cdots\otimes  x_k\otimes b                                                  & \longmapsto a\imath^{-1}\otimes \sigma(x_1)\otimes\cdots\otimes \sigma(x_k)\otimes\imath b.
  \end{align*}
  Therefore, the mapping cone of the endomorphism \(\id{}-\Gamma\) of
  \(\gLambda\otimes_{\Lambda}B_{\bullet}(\Lambda)\otimes_{\Lambda}\gLambda\) is
  a projective resolution of \(\gLambda\) as a \(\gLambda\)-bimodule. This
  mapping cone is given explicitly by
  \[C_{\id{}-\Gamma}=\gLambda\otimes_{\Lambda}B_{\bullet}(\Lambda)\otimes_{\Lambda}\gLambda
      \oplus
      \gLambda\otimes_{\Lambda}B_{\bullet-1}(\Lambda)\otimes_{\Lambda}\gLambda\]
  with the differential
  \[\begin{pmatrix}
          d_B            & 0  \\
          \id{}-\Gamma & -d_B
      \end{pmatrix}.\]
  Here, $d_B$ stands for the bar complex differential.

  Consider the maps
  \[\gLambda\otimes_{\Lambda}B_{\bullet}(\Lambda)\otimes_{\Lambda}\gLambda
      \stackrel{\id{}-\Gamma}{\longrightarrow}
      \gLambda\otimes_{\Lambda}B_{\bullet}(\Lambda)\otimes_{\Lambda}\gLambda
      \stackrel{j_*}{\longrightarrow}
      B_{\bullet}(\gLambda)
  \]
  where \(j_*\) is defined by
  \[j_*(a\otimes  x_1\otimes\cdots\otimes  x_j\otimes b)=a\otimes  j(x_1)\otimes\cdots\otimes  j(x_j)\otimes b.\]
  We could omit the \(j\)'s on the right because it is just the inclusion \(\Lambda\subset\gLambda\).
  The composite is null-homotopic, with explicit homotopy
  \begin{multline*}
      h(a\otimes  x_1\otimes\cdots\otimes  x_j\otimes b)=\\
      \sum_{k=0}^j(-1)^k
      a\imath^{-1}\otimes  \sigma(x_1)\otimes\cdots\otimes\sigma(x_k)\otimes\imath\otimes x_{k+1}\otimes\cdots\otimes  x_j\otimes b.
  \end{multline*}
  Indeed, it is easy to check that
  \[j_*(\id{}-\Gamma)=d_Bh+hd_B.\]
  This gives rise to an explicit map of resolutions
  \begin{align*}
      C_{\id{}-\Gamma} & \longrightarrow B_\bullet(\gLambda), \\
      (x,y)            & \longmapsto j_*(x)+h(y).
  \end{align*}
  Taking \(\hom^*_{\gLambda[e]}(-,\gLambda)\), we obtain a quasi-isomorphism
  \[\HC{\gLambda}[\gLambda]\stackrel{\sim}{\longrightarrow}\hom^*_{\gLambda[e]}(C_{\id-\Gamma},\gLambda).\]
  The Hochschild complex (source) is a differential graded algebra. We can endow the target with a differential graded algebra structure turning this map into an algebra map. It is the following square-zero extension
  \[\hom^*_{\gLambda[e]}(C_{\id -\Gamma},\gLambda)=
      \HC{\Lambda}[\gLambda]\oplus {}_{\sigma^{-1}_*\sigma^*}\HC[\bullet-1]{\Lambda}[\gLambda]_1\]
  of the differential graded algebra \(\HC{\Lambda}[\gLambda]\) by the twisted and then shifted bimodule \({}_{\sigma^{-1}_*\sigma^*}\HC[\bullet-1]{\Lambda}[\gLambda]_1\).
  It is straightforward to check that the differential satisfies the Leibniz rule with respect to the product. Moreover, the quasi-isomorphism is an algebra map because, given \(f,g\in \HC{\gLambda}[\gLambda]\),
  \[h^*(f\cdot g)=h^*(f)\cdot j^*(g)+(-1)^{|f|}\sigma^{-1}_*\sigma^*j^*(f)\cdot h^*(g).\]
  This is tedious but straightforward to check.

  Up to isomorphism, the long exact sequence in statement \eqref{it:the_long_ex_seqHH} is induced by the square-zero extension
  \begin{equation}
    \label{eq:square_zero_ext_in_proof}
    {}_{\sigma^{-1}_*\sigma^*}\HC[\bullet-1]{\Lambda}[\gLambda]_1
      \stackrel{i_2}{\hookrightarrow}
      \hom^*_{\gLambda[e]}(C_{\id -\Gamma},\gLambda)
      \stackrel{p_1}{\twoheadrightarrow},
      \HC{\Lambda}[\gLambda]
    \end{equation}
  where \(p_1\) is an algebra map and \(i_2\) is a \(\hom^*_{\gLambda[e]}(C_{\id -\Gamma},\gLambda)\)-bimodule morphism.
  This proves that \(\partial\), which is induced by \(-i_2\), is a bimodule morphism. In cohomology, we can remove the left twisting by \(\sigma^{-1}_*\sigma^*\) since \(\HH{\gLambda}[\gLambda]\) maps to the kernel of \(\id -\sigma^{-1}_*\sigma^*\).

  \eqref{it:_props__delta_j_of_the_long_ex_seqHH} The coordinates of the image of the fractional Euler derivation under the quasi-isomorphism are
  \[j^*(\delta_{/d})=0,\]
  because \(\Lambda\) is ungraded, and
  \begin{equation*}
      h^*(\delta_{/d})=\imath^{-1}\delta_{/d}(\imath)
      =\frac{|\imath|}{d}\imath^{-1}\imath=-1,
  \end{equation*}
  hence it corresponds to \((0,-1)\in \hom^*_{\gLambda[e]}(C_{\id -\Gamma},\gLambda)\).
  The composite \(\partial\circ j^*\) is therefore induced by \(-i_2p_1\) and, in \(\hom^*_{\gLambda[e]}(C_{\id -\Gamma},\gLambda)\),
  \[(0,-1)\cdot(f,g)=(0,-f)=-i_2p_1(f,g).\]
  This proves the claim in the statement, and hence the proof of the proposition
  is finished.
\end{proof}

We record the following immediate corollary of \Cref{prop:long_ex_seqHH}. Below,
we denote by $\langle\sigma\rangle\subseteq\Aut{\gLambda}$ be the subgroup
generated by the extension of $\sigma$ to $\gLambda$

\begin{corollary}
  \label{cor:sq-zero}
  In the setting of \Cref{prop:long_ex_seqHH}, the inclusion
  ${j:\Lambda\to\gLambda}$ induces a square-zero extension of bigraded algebras
\[
  \HH[\bullet-1]{\Lambda}[\gLambda]_{\langle\sigma\rangle}\hookrightarrow\HH{\gLambda}[\gLambda]\twoheadrightarrow\HH{\Lambda}[\gLambda]^{\langle\sigma\rangle},
\]
where $\HH{\Lambda}[\gLambda]^{\langle\sigma\rangle}$ and
$\HH[\bullet-1]{\Lambda}[\gLambda]_{\langle\sigma\rangle}$ denote the bigraded algebra of
invariants and the shifted bimodule of coinvariants of the action of
${\langle\sigma\rangle\subset\Aut*{\gLambda}}$ on $\HH{\Lambda}[\gLambda]$,
respectively (see \Cref{rmk:action}).
\end{corollary}
\begin{proof}
  This is an immediate consequence of the existence of the square-zero extension
  of DG algebras
  \eqref{eq:square_zero_ext_in_proof} established in the proof of \Cref{prop:long_ex_seqHH}\eqref{it:_props__delta_of_the_long_ex_seqHH}.
\end{proof}

\begin{remark}
  \Cref{cor:sq-zero} implies that $\HH{\Lambda}[\gLambda]^{\langle\sigma\rangle}$ is graded
  commutative with respect to the total degree as is
  $\HH{\gLambda}[\gLambda]$. Moreover
  $\HH{\Lambda}[\gLambda]_{\langle\sigma\rangle}$ is a symmetric
  $\HH{\Lambda}[\gLambda]^{\langle\sigma\rangle}$-bimodule, that is a module in
  the commutative sense. Furthermore, when the obvious map
  \[
    \HH{\Lambda}[\gLambda]^{\langle\sigma\rangle}\longrightarrow\HH{\Lambda}[\gLambda]_{\langle\sigma\rangle}
  \]
  is an isomorphism, \Cref{prop:long_ex_seqHH} identifies the kernel of the
  square-zero extension with the ideal generated by $\class{\EulerDer[d]}$. In
  particular, in this case, multiplication by $\class{\EulerDer[d]}$ in
  $\HH{\gLambda}[\gLambda]$ induces a splitting of the square-zero extension. A
  sufficient condition is that the order of $\sigma$ as an automorphism of
  $\gLambda$ is finite (compare with the Periodicity
  Conjecture of Erdmann and Skowor{\'n}ski~\cite{ES08}) and invertible in the ground field $\kk$. Notice that the
  order of $\sigma$ as an automorphism of $\gLambda$ is twice the order of
  $\sigma$ as an automorphism of $\Lambda$ if $d$ is odd and $\kk$ has odd
  characteristic, and they coincide otherwise, see \Cref{rmk:graded_sigma}.
\end{remark}

The following property of edge units is key for later computations.

\begin{proposition}
  \label{prop:key_edege_units}
  Let $x\in\HH[d+2][-d]{\gLambda}$ be an edge unit. Then, the
  left-multiplication map
  \[
    \HH[p]{\gLambda}\longrightarrow\HH[p+d+2][*-d]{\gLambda},\qquad y\longmapsto
    x\cdot y,
  \]
  is bijective for $p\geq2$ and surjective for $p=1$.
\end{proposition}
\begin{proof}
  Using \Cref{prop:long_ex_seqHH}, both claims can be proved with the same
  argument used in the proof of \cite[Lemma~6.8]{Mur22}.
\end{proof}

  The following result---a direct consequence of
  \Cref{prop:Amiot-Lin_independence_of_res}---is responsible for the
  \emph{uniqueness} of enhancements in \Cref{thm:A-Ainfty-version} (compare with
  \cite[Prop.~9.4]{Mur22}\footnote{The formula for $g_u$ in the proof
    of~\cite[Prop.~9.4]{Mur22} (called $g(x)$ therein for $x=u\in
    Z(\Lambda)^\times$), where the case $d=1$ is considered, is wrong in degrees
    $\neq 0,-1$ unless $\sigma(u)=u$, since otherwise $u$ is not a central
    element in $\gLambda$ and hence $g(x)$ does not preserve products. The correct definition is the one in \Cref{rmk:action} in terms of generators and relations, which yields different degree-wise formulas. The rest of the argument in the proof of \cite[Prop.~9.4]{Mur22} is nevertheless valid since it only uses the explicit formulas in degrees $0$ and $-1$.}). In particular, it is here that it is crucial for us
  to work with finite-dimensional algebras and not with more general objects
  such as locally-finite categories (see the proof of
  \Cref{prop:Amiot-Lin_independence_of_res}, in particular the use of
  \cite[Cor.~2.3]{Che21}).

\begin{proposition}
  \label{prop:single_orbit}
  Let $\Lambda$ be a finite-dimensional self-injective algebra that is twisted
  $(d+2)$-periodic with respect to an algebra automorphism
  $\sigma\colon\Lambda\simto\Lambda$, $d\geq 1$. Then, the set formed by the
  edge units in $\HH[d+2][-d]{\Lambda}[\gLambda]$ is an orbit for
  the action of $\Aut*{\gLambda}$.
\end{proposition}
\begin{proof}
  By \Cref{rmk:action}, the orbit of an edge unit consists of edge units. 
  Let $x$ and $x'$ be edge units in
  $\HH[d+2][-d]{\Lambda}[\twBim{\Lambda}[\sigma]]$ and choose exact sequences
  \[
    \eta\colon\quad 0\to\twBim{\Lambda}[\sigma]\to P_{d+1}\to\cdots\to P_1\to
    P_0\to\Lambda\to0
  \]
  and
  \[
    \eta'\colon\quad 0\to\twBim{\Lambda}[\sigma]\to Q_{d+1}\to\cdots\to Q_1\to
    Q_0\to\Lambda\to0
  \]
  of $\Lambda$-bimodules with projective-injective middle terms that represent
  $x$ and $y$, respectively (see \Cref{rmk:TateUnits}).
  \Cref{prop:Amiot-Lin_independence_of_res} yields a morphism of exact sequences
  \begin{center}
    \begin{tikzcd}[column sep=small,row sep=small]
      \eta\colon&0\rar&\twBim{\Lambda}[\sigma]\rar\dar{u}&P_{d+1}\rar\dar&\cdots\rar&P_1\rar\dar&P_0\dar\rar&\Lambda\dar[equals]\rar&0\\
      \eta'\colon&0\rar&\twBim{\Lambda}[\sigma]\rar&Q_{d+1}\rar&\cdots\rar&Q_1\rar&Q_0\rar&\Lambda\rar&0
    \end{tikzcd}
  \end{center}
  where $u\in Z(\Lambda)^\times$, which implies that
  \[
    x'=u\cdot x=x^{g_u},
  \]
  see again \Cref{rmk:action}.
\end{proof}

\section{Existence and uniqueness of enhancements}
\label{sec:existence_and_uniqueness}

\begin{setting}
  \label{setting:existence_and_uniqueness}
  We fix a finite-dimensional basic self-injective algebra $\Lambda$ that is
  twisted ${(d+2)}$-periodic for some $d\geq 1$ with respect to an algebra automorphism
  $\sigma\colon \Lambda\simto{\Lambda}$ and let
  \[
    \Sigma\coloneqq-\otimes_{\Lambda}\twBim[\sigma]{\Lambda}\colon\proj*{\Lambda}\stackrel{\sim}{\longrightarrow}\proj*{\Lambda}.
  \]
  As in \Cref{rmk:Lambda-sigma-d}, we consider the $d$-sparse graded algebra
  \[
    \gLambda\coloneqq\frac{\Lambda\langle \imath^\pm\rangle}{(\imath
      x-\sigma(x)\imath)_{x\in\Lambda}},\qquad |\imath|=-d.
  \]
\end{setting}

\subsection{Proof of~\Cref{thm:dZ-Auslander_correspondence}}

In this subsection we give the proof \Cref{thm:dZ-Auslander_correspondence},
modulo the proof of the following main result (\Cref{thm:A-Ainfty-version}) that
we postpone until \Cref{subsec:thmA-Ainfty-version}.

\begin{notation}
  \label{not:g_action}
  Let $g\in\Aut*{\gLambda}$ be a graded algebra automorphism and
  \[
    (\gLambda,m_{d+2},m_{2d+2},m_{3d+2},\dots)
  \]
  a minimal $A_\infty$-algebra structure on $\gLambda$. We let
  \[
    (\gLambda,m_{d+2},m_{2d+2},\dots)*g\coloneqq(\gLambda,m_{d+2}^g,m_{2d+2}^g,\dots),
  \]
  where $m_i^g\coloneqq g^{-1}m_ig^{\otimes i}$, and note that this is also a
  minimal $A_\infty$-algebra structure on the $d$-sparse graded algebra
  $\gLambda$.
\end{notation}

\begin{theorem}
  \label{thm:A-Ainfty-version}
  The following statements hold:
  \begin{enumerate}
  \item\label{it:existence}There exists a minimal $A_\infty$-algebra structure
    \[
      (\gLambda,m_{d+2},m_{2d+2},\dots)
    \]
    on the $d$-sparse graded algebra $\gLambda$ whose restricted universal
    Massey product
    \begin{center}
      $j^*\class{m_{d+2}}\in\HH[d+2][-d]{\Lambda}[\gLambda]=\Ext[\Lambda^e]{\Lambda}{\twBim{\Lambda}[\sigma]}[d+2]$
    \end{center}
    is a unit in the Hochschild--Tate cohomology $\TateHH{\Lambda}[\gLambda]$
    (equivalently, the class $j^*\class{m_{d+2}}$ is represented by an extension
    all of whose middle terms are projective-injective $\Lambda$-bimodules, see~\Cref{rmk:TateUnits}).
  \item\label{it:uniqueness} Let
    \begin{center}
      $(\gLambda,m_{d+2},m_{2d+2},\dots)\quad\text{and}\quad(\gLambda,m_{d+2}',m_{2d+2}',\dots)$
    \end{center}
    be minimal $A_\infty$-algebra structures on $\gLambda$ whose corresponding
    restricted universal Massey products $j^*\class{m_{d+2}}$ and
    $j^*\class{m_{d+2}'}$ are units in the Hochs\-schild--Tate cohomology
    $\TateHH{\Lambda}[\gLambda]$. Then, there exist a graded algebra
    automorphism $g\in\Aut*{\gLambda}$ and a quasi-iso\-mor\-phism of minimal
    $A_\infty$-algebras
    \begin{center}
      $F\colon(\gLambda,m_{d+2},m_{2d+2},\dots)\stackrel{\simeq}{\longrightarrow}(\gLambda,m_{d+2}',m_{2d+2}',\dots)*g$
    \end{center}
    such that $F_0=\id$ and $F_1=\id$.
  \end{enumerate}
\end{theorem}

We need some recollections on the derived category of an $A_\infty$-algebra.
Given an $A_\infty$-algebra $A$, we denote by $\AiMod*{A}$ the DG (!) category
of \emph{(right) $A_\infty$-modules over $A$}, that is the DG category of
$A_\infty$-functors $A^\op\to\dgCh{\kk}$ from the opposite $A_\infty$-algebra of
$A$ (viewed as an $A_\infty$-category with a single object $*$) to the DG
category of cochain complexes of vector spaces, see for
example~\cite[Sec.~I.1j]{Sei08}. The $A_\infty$-categorical Yoneda embedding
\[
  A\longrightarrow\AiMod*{A},\qquad *\longmapsto A_A
\]
yields a quasi-isomorphism of $A_\infty$-algebras
\[
  A\stackrel{\sim}{\longrightarrow}\overline{A}\coloneqq\dgHom[\A]{A_A}{A_A}
\]
where $\overline{A}$ is the DG (!) algebra of endomorphisms of the free
$A_\infty$-module $A_A$, see for example~\cite[Sec.~I.2g]{Sei08} (to our
knowledge, this important fact was first observed by Kontsevich and
Soibelman~\cite[Sec.~6.2]{KS06}, but see also \cite[Sec.~7.5]{Lef03}). In particular, there is a canonical
isomorphism of graded algebras
\[
  \dgH{A}\stackrel{\sim}{\longrightarrow}\dgH{\overline{A}}.
\]
Finally, the derived category of $A$ is defined as
\[
  \DerCat{A}\coloneqq\dgH[0]{\AiMod*{A}};
\]
this is a sensible definition since quasi-isomorphisms of $A_\infty$-modules are
also homotopy equivalences, see \cite[Sec.~4]{Kel01}
or~\cite[Thm.~4.1.3.1]{Lef03} while keeping in mind that we work over a field.
By construction, there is a canonical equivalence of triangulated categories
\[
  \DerCat{\overline{A}}\stackrel{\sim}{\longrightarrow}\DerCat{A},
\]
which is in fact induced by the quasi-equivalence of DG categories
\[
  \AiMod*{\overline{A}}\stackrel{\sim}{\longrightarrow}\AiMod*{A}
\]
given by restriction of scalars along the quasi-isomorphism
$A\stackrel{\sim}{\to}\overline{A}$, see for example~\cite[Lemma~I.2.10]{Sei08}.
The following result can be regarded as a $(d+2)$-angulated analogue of one of
the main theorems in~\cite{Mur22}.

\begin{theorem}
  \label{thm:projLambda_existence_and_uniqueness}
  Let $\Lambda$ be a finite-dimensional self-injective algebra that is twisted
  $(d+2)$-periodic with respect to an algebra automorphism $\sigma$ and let
  \[
    \Sigma\coloneqq-\otimes_\Lambda\twBim[\sigma]{\Lambda}\colon\proj*{\Lambda}\stackrel{\sim}{\longrightarrow}\proj*{\Lambda}.
  \]
  Then, the AL $(d+2)$-angulation $\pentagon_\eta$ of $(\proj*{\Lambda},\Sigma)$
  induced by any exact sequence
  of $\Lambda$-bimodules
  \[
    \eta\colon\qquad 0\to\twBim{\Lambda}[\sigma]\to P_{d+1}\to\cdots\to P_1\to
    P_0\to\Lambda\to0
  \]
  with projective-injective middle terms (\Cref{thm:Amiot-Lin}) admits an
  enhancement. Moreover, if $\Lambda/J_\Lambda$ is separable, then the
  enhancement is unique.
\end{theorem}
\begin{proof}
  Recall from \Cref{prop:Amiot-Lin_independence_of_res} that, up to equivalence, the AL
  $(d+2)$-angulation is independent of the choice of the exact sequence $\eta$
  (as long as it satisfies the conditions in the statement of the theorem).

  We prove the first claim, that the AL $(d+2)$-angulation on
  $(\proj*{\Lambda},\Sigma)$ admits an enhancement. Firstly,
  \Cref{thm:A-Ainfty-version}\eqref{it:existence} yields a minimal
  $A_\infty$-algebra structure
  \[
    A\coloneqq(\gLambda,m_{d+2},m_{2d+2},\dots)
  \]
  on the $d$-sparse graded algebra $\gLambda$ whose restricted universal Massey
  product
  \begin{center}
    $j^*\class{m_{d+2}}\in\HH[d+2][-d]{\Lambda}[\gLambda]=\Ext[\Lambda^e]{\Lambda}{\twBim{\Lambda}[\sigma]}[d+2]$
  \end{center}
  is represented by an exact sequence all of whose middle terms are
  projective-injective $\Lambda$-bimodules. Secondly, let
  $\overline{\A}\subseteq\DerCat{A}$ be the smallest subcategory containing $A$
  that is closed under finite direct sums and direct summands, and
  $\A\subseteq\AiMod*{A}$ the full DG subcategory spanned by the objects of
  $\overline{\A}$. By construction,
  \begin{itemize}
  \item $\dgH{\A}$ and $\gLambda$ are graded Morita equivalent (indeed, the
    graded category $\dgH{\A}$ identifies with that of projective graded
    $\gLambda$-modules) and
  \item the induced diagram of equivalences of categories
    \begin{center}
      \begin{tikzcd}
        \dgH[0]{\A}\rar{\sim}\dar[swap]{[d]}&\proj*{\Lambda}\dar{-\otimes_{\Lambda}\twBim[\sigma]{\Lambda}}\\
        \dgH[0]{\A}\rar{\sim}&\proj*{\Lambda}
      \end{tikzcd}
    \end{center}
  \end{itemize}
  commutes up to natural isomorphism. It follows that there is a commutative
  diagram
  \begin{center}
    \begin{tikzcd}
      \HH{\dgH[\bullet]{\A}}[\dgH[\bullet]{\A}]\rar{j^*}\dar{\sim}&\HH{\dgH[0]{\A}}[\dgH[\bullet]{\A}]\dar{\sim}\\
      \HH{\gLambda}[\gLambda]\rar{j^*}&\HH{\Lambda}[\gLambda]
    \end{tikzcd}
  \end{center}
  where the vertical maps are isomorphisms of bigraded algebras and the
  horizontal maps are induced by the corresponding inclusion of the degree $0$
  part. In particular, the restricted universal Massey product
  \[
    j^*\class{m_{d+2}}\in\HH[d+2][-d]{\dgH[0]{\A}}[\dgH[\bullet]{\A}]
  \]
  is also unit. Consequently, the DG category $\A$ is Karoubian
  pre-$(d+2)$-angulated (\Cref{coro:pre-d+2-ang=unit}
  \eqref{it:isKaroubian_d+2}$\Leftarrow$\eqref{it:H0_isAL-angulated}) and the
  standard $(d+2)$-angulation on $(\dgH[0]{\A},[d])$ coincides with the AL
  $(d+2)$-angulation induced by any representative of $(-1)^{\Sigma_{i=1}^{d+2}i}j^*\class{m_{d+2}}$ with projective-injective middle terms
  (\Cref{thm:GKO_AL_Massey}\eqref{it:std=AL}). Finally,
  by construction, $\A$ is an enhancement of the AL $(d+2)$-angulation of the
  pair $(\proj*{\Lambda},\Sigma)$, which proves the first claim.

  We prove the second claim, that any two enhancements of the AL
  $(d+2)$-angulation on $(\proj*{\Lambda},\Sigma)$ are equivalent under the
  separability hypothesis. Indeed, suppose that $\A$ and $\B$ are two such
  enhancements. By assumption, we may identify the graded endomorphism algebras
  of basic additive generators $c\in\dgH[0]{\A}$ and $c'\in\dgH[0]{\B}$ with the
  $d$-sparse graded algebra $\gLambda$. It follows that any choice of minimal
  $A_\infty$-models on the graded categories $\dgH[\bullet]{\A}$ and
  $\dgH[\bullet]{\B}$ yield minimal $A_\infty$-algebra structures
  \[
    (\gLambda,m_{d+2},m_{2d+2},\dots)\quad\text{and}\quad(\gLambda,m_{d+2}',m_{2d+2}',\dots)
  \]
  on $\gLambda$. Moreover, as a consequence of \Cref{coro:pre-d+2-ang=unit}
  \eqref{it:isKaroubian_d+2}$\Rightarrow$\eqref{it:H0_isAL-angulated}, which
  requires the separability hypothesis on $\Lambda/J_\Lambda$, we see
  that the restricted universal Massey products $j^*\class{m_{d+2}}$ and
  $j^*\class{m_{d+2}'}$ are units in $\TateHH{\Lambda}[\gLambda]$. Thus, by
  \Cref{thm:A-Ainfty-version}\eqref{it:uniqueness}, there exist a graded algebra
  automorphism $g\in\Aut*{\gLambda}$ and a quasi-iso\-mor\-phism of
  $A_\infty$-algebras
  \[
    F\colon(\gLambda,m_{d+2},m_{2d+2},\dots)\stackrel{\simeq}{\longrightarrow}(\gLambda,m_{d+2}',m_{2d+2}',\dots)*g
  \]
  such that $F_1=\id$. It is straightforward to verify that the tuple
  \[
    (gF_1,gF_2,gF_3,\dots)\colon(\gLambda,m_{d+2},m_{2d+2},\dots)\stackrel{\simeq}{\longrightarrow}(\gLambda,m_{d+2}',m_{2d+2}',\dots)
  \]
  is also quasi-isomorphism of $A_\infty$-algebras. By well-known rectification
  results for $A_\infty$-algebras \cite[Sec.~11.4.3]{LV12}, the existence of the
  latter quasi-isomorphism of $A_\infty$-algebras implies that the DG algebras
  $A=\A(c,c)$ and $B=\B(c',c')$ are quasi-isomorphic. However the
  quasi-isomorphism might not be direct, in the sense that it might be given by
  zig-zag through intermediate DG algebras,
  \begin{equation*}
    A=A^{(0)}\stackrel{\sim}{\longleftarrow}A^{(1)}\stackrel{\sim}{\longrightarrow}\cdots\stackrel{\sim}{\longleftarrow}A^{(n-1)}\stackrel{\sim}{\longrightarrow}A^{(n)}=B.
  \end{equation*}
  Moreover, the quasi-isomorphisms
  in the zig-zag might not preserve the unit, and the intermediate DG algebras
  might not even be unital. Notwithstanding, since both $A$ and $B$ are unital
  DG algebras, everything can be made unital using the results in
  \cite[Prop.~6.2]{Mur14}. Finally, any
  unital quasi-isomorphism between $A$ and $B$ can be extended to a
  quasi-equivalence between $\A$ and $\B$ in an essentially unique way (this is a well-known property of the
  homotopy Karoubi envelope of a DG algebra, see for example
  \cite[Sec.~4]{GHW21}), 
  \begin{equation*}
    \A=\A^{(0)}\stackrel{\sim}{\longleftarrow}\A^{(1)}\stackrel{\sim}{\longrightarrow}\cdots\stackrel{\sim}{\longleftarrow}\A^{(n-1)}\stackrel{\sim}{\longrightarrow}\A^{(n)}=\B.
  \end{equation*}
  This finishes the proof of the theorem.
\end{proof}

\begin{remark}\label{rmk:algebraic=AL}
\Cref{thm:GKO_AL_Massey,thm:projLambda_existence_and_uniqueness} show that if $\Lambda$ is a finite-dimensional self-injective algebra that its twisted $(d+2)$-periodic and such that $\Lambda/J_\Lambda$ is separable, then the algebraic and the AL $(d+2)$-angulated structures on $\proj*{\Lambda}$ coincide, see also   \Cref{rmk:several_angulations_and_extensions}. 
\end{remark}

\begin{remark}
  The use of derived categories of $A_\infty$-algebras in the proof of
  \Cref{thm:projLambda_existence_and_uniqueness} is not essential; these are
  only used for constructing an explicit DG algebra that is quasi-isomorphic
  to a given $A_\infty$-algebra, and any other choice of such a DG algebra would
  suffice.
\end{remark}

\begin{theorem}\label{thm:injective_correspondence}
Let $d\geq 1$. For $i=1,2$, let $\T_i$ be algebraic triangulated categories
with finite-dimensional morphism spaces and split idempotents, and $c_i\in\T_i$
basic $d\ZZ$-cluster tilting objects. Assume that the algebras $\T_1(c_1,c_1)$ and
$\T_2(c_2,c_2)$ are isomorphic and write $\Lambda$ for a representative of their isomorphism
class. Suppose, moreover, that the algebra $\Lambda/J_\Lambda$ is separable and
that there exists a $\Lambda$-bimodule $I$ such
the diagram
  \[
    \begin{tikzcd}
           \T_i\supset \add*{c_i}\rar{\T(c_i,-)}\dar[swap]{[-d]}&\proj*{\Lambda}\dar{-\otimes_{\Lambda}I}\\
\T_i\supset      \add*{c_i}\rar{\T(c_i,-)}&\proj*{\Lambda}
    \end{tikzcd}
  \]
  commutes up to natural isomorphism, for $i=1,2$. Then, there exists an equivalence of triangulated categories
  \[
    F\colon\T_1\stackrel{\sim}{\longrightarrow}\T_2
  \]
  that restricts to an equivalence
  \[
    F\colon\add*{c_1}\stackrel{\sim}{\longrightarrow}\add*{c_2}.
  \]
  \end{theorem} 
  
  \begin{proof}
  The algebra $\Lambda$ is finite-dimensional, self-injective and twisted $(d+2)$-periodic by \Cref{prop:dZ_CT_twisted_periodic}, and the $\Lambda$-bimodule $I$ is invertible because the left vertical and the horizontal arrows in the commutative square of the statement are equivalences. Moreover, $\Lambda$ is basic since the $c_i$ are, so $I\cong\twBim{\Lambda}[\sigma]$ for a some automorphism $\sigma$ of $\Lambda$ by \cite[Prop.~3.8]{Bol84}. Furthermore, $\sigma$ realises the $(d+2)$-periodicity of $\Lambda$ by \Cref{thm:GKO-standard} and \Cref{prop:GKO-Freyd-Heller}\eqref{it:Sigma-d+2-syzygy} (compare with the hypotheses of \Cref{thm:projLambda_existence_and_uniqueness}).
  
  Let
  $\B_i$, $i=1,2$, be pre-triangulated enhancements of $\T_i$, $i=1,2$, so
  that there are equivalences of triangulated categories
  \[
    \T_i\simeq\dgH[0]{\B_i}\qquad i=1,2.
  \]
  Set $\C_i\coloneqq\add*{c_i}$, $i=1,2$, and let $\A_i\subseteq\B_i$, $i=1,2$,
  be the full DG subcategories spanned by the objects in $\C_i$, $i=1,2$. In
  particular, since $\thick*{\C_i}=\T_i$ (see \Cref{thm:IY-Bel-d-CT}), the
  inclusions $\A_i\subseteq\B_i$, $i=1,2$, are Morita equivalences of DG
  categories and, consequently, there are canonical equivalences of triangulated
  categories
  \[
    \DerCat[c]{\A_i}\stackrel{\sim}{\longrightarrow}\DerCat[c]{\B_i}\stackrel{\sim}{\longleftarrow}\dgH[0]{\B_i}\qquad
    i=1,2,
  \]
  see \Cref{thm:Keller-perf}. Next, by construction, $\A_i$, $i=1,2$, yields a
  pre-$(d+2)$-angulated enhancements of
  \[
    (\C_i,[d]_{\T_i})\simeq(\proj*{\Lambda},-\otimes_{\Lambda}I^{-1}),\qquad
    i=1,2,
  \]
  where the pair of the left-hand side is equipped with the standard
  $(d+2)$-angulation as in \Cref{thm:GKO-standard}. We may apply
  \Cref{thm:projLambda_existence_and_uniqueness} (see also \Cref{rmk:algebraic=AL}) to deduce that $\A_1$ and
  $\A_2$ are isomorphic as objects of the homotopy category $\Ho{\dgcat}$. Thus,
  $\A_1$ and $\A_2$ are connected by a finite zig-zag of quasi-equivalences of
  DG categories
  \begin{equation}
    \label{eq:zigzag}
    \A_1=\A^{(0)}\stackrel{\sim}{\longleftarrow}\A^{(1)}\stackrel{\sim}{\longrightarrow}\cdots\stackrel{\sim}{\longleftarrow}\A^{(n-1)}\stackrel{\sim}{\longrightarrow}\A^{(n)}=\A_2
  \end{equation}
  that induces a finite zig-zag of equivalences of triangulated categories
  (\Cref{thm:Keller-perf})
  \[
    \DerCat[c]{\A_1}=\DerCat[c]{\A^{(0)}}\stackrel{\sim}{\longleftarrow}\DerCat[c]{\A^{(1)}}\stackrel{\sim}{\longrightarrow}\cdots\stackrel{\sim}{\longleftarrow}\DerCat[c]{\A^{(n-1)}}\stackrel{\sim}{\longrightarrow}\DerCat[c]{\A^{(n)}}=\DerCat[c]{\A_2}.
  \]
  Choosing quasi-inverses of the left-pointing functors in the latter zig-zag
  yields an equivalence of triangulated categories
  \[
    \dgH[0]{\B_1}\simeq\DerCat[c]{\A_1}\stackrel{\sim}{\longrightarrow}\DerCat[c]{\A_2}\simeq\dgH[0]{\B_2}
  \]
  which restricts to an equivalence
  \[
    \C_1\simeq\dgH[0]{\A_1}\stackrel{\sim}{\longrightarrow}\dgH[0]{\A_2}\simeq\C_2,
  \]
  which is what we needed to prove. Here, we apply to the quasi-equivalences in the zig-zag \eqref{eq:zigzag} the general fact that a DG functor between small DG categories
  \[
    F\colon\B\longrightarrow\A
  \]
  induces an exact functor
  \[
    \mathbb{L}F_!\colon\DerCat{\B}\longrightarrow\DerCat{\A}
  \]
  which sends compact objects to compact objects and, for an object $x\in\B$, satisfies $\mathbb{L}F_!(\B(-,x))\cong\A(-,F(x))$ in $\DerCat{\A}$, see for example~\cite[Sec.~6.1]{Kel94}.
\end{proof}

\begin{theorem}\label{thm:uniqueness_triangulated}
	Let $\T$ be an algebraic triangulated category with finite-dimensional
  morphism spaces and split idempotents, and $c\in\T$ a basic $d\ZZ$-cluster
  tilting object. Set $\Lambda=\T(c,c)$ and suppose that the algebra $\Lambda/J_\Lambda$ is separable. Then $\T$ has a unique enhancement.
\end{theorem}

\begin{proof}
The argument is identical to the proof of \Cref{thm:injective_correspondence}, but choosing enhancements
  $\B_i$ of $\T_i=\T$, $i=1,2$ and with $c_1=c_2=c$. Indeed, letting
  $\A_i\subseteq\B_i$, $i=1,2$, as above, by
  \Cref{rmk:quasi-equivalence_vs_Morita_equivalence} the zig-zag of
  quasi-equivalences \eqref{eq:zigzag} shows that $\A_1$ and $\A_2$ are Morita
  equivalent DG categories, but then so are the pre-triangulated DG categories
  $\B_1$ and $\B_2$. This finishes the proof of the theorem.
\end{proof}

\begin{theorem}\label{thm:surjective_correspondence}
	Let $d\geq 1$. Let $\Lambda$ be a
  basic finite-dimensional algebra that is twisted $(d+2)$-periodic with respect
  to an algebra automorphism $\sigma$. Then, there exists an algebraic triangulated category $\T$ with finite-dimensional morphism spaces, split idempotents, and a $d\ZZ$-cluster tilting object such that $\T(c,c)\cong\Lambda$ and the square
    \[
    \begin{tikzcd}
      \T\supset\add*{c}\rar{\T(c,-)}\dar[swap]{[-d]}&\proj*{\Lambda}\dar{-\otimes_{\Lambda} \twBim{\Lambda}[\sigma]}\\
            \T\supset\add*{c}\rar{\T(c,-)}&\proj*{\Lambda}
    \end{tikzcd}
  \]
  commutes up to natural isomorphism. In particular, $\T(c,c[-d])\cong \twBim{\Lambda}[\sigma]$.
\end{theorem}

\begin{proof}
	\Cref{thm:Amiot-Lin} shows that the pair
  $(\proj*{\Lambda},-\otimes_\Lambda \twBim[\sigma]{\Lambda})$ can be endowed with an AL $(d+2)$-angulation, and
  the latter admits an enhancement $\A$ by
  \Cref{thm:projLambda_existence_and_uniqueness}. Let ${\T=\DerCat[c]{\A}}$, which is an algebraic triangulated category with split idempotents, and
let $c\in\T$ be the image of $\Lambda$ under the identification
  \[
    \proj*{\Lambda}\simeq\dgH[0]{\A},
  \]
  so that $\T(c,c)\cong\Lambda$. Since the canonical inclusion functor
  identifies $\dgH[0]{\A}$ with a $d\ZZ$-cluster tilting subcategory
  $\C\subseteq\DerCat[c]{\A}$ (\Cref{prop:d+2=dZ-CT}), it follows that $\T$ has finite-dimensional morphism spaces and $\proj*{\Lambda}\simeq\dgH[0]{\A}$ is equivalent to the standard $(d+2)$-angulated category $\C\subseteq\DerCat[c]{\A}$, so the square in the statement commutes.
\end{proof}

We now use \Cref{thm:injective_correspondence,thm:uniqueness_triangulated,thm:surjective_correspondence} to prove
\Cref{thm:dZ-Auslander_correspondence}.

\begin{proof}[Proof of \Cref{thm:dZ-Auslander_correspondence}] We analyse each
  correspondence separately.

  \emph{The map $(\T,c)\mapsto(\T(c,c),\T(c,c[-d]))$ is well defined.}
  Let $(\T,c)$ be a pair satisfying the conditions in
  \Cref{thm:dZ-Auslander_correspondence}\eqref{it:Tc}, so that $c\in\T$ is a basic
  $d\ZZ$-cluster tilting object.
  Since $\T$ is algebraic and has split idempotents, there exists a Karoubian pre-triangulated DG category $\X$ such that $\T\simeq\dgH[0]{\X}$ as triangulated categories. For the sake of simplicity, we identify $\T$ with $\dgH[0]{\X}$ via a chosen equivalence, in particular $\T$ and $\X$ have the same objects.
  
  We take $\A\subset\X$ to be a full DG subcategory spanned by the objects of $\C\coloneqq \add*{c}$, in particular $\dgH[0]{\A}=\C$. Let us check that $\A$ satisfies the assumptions of \Cref{setting:resticted_universal_Massey}:
  \begin{itemize}
    \item[(2)] Since $\C$ is $d\ZZ$-cluster tilting, $\thick*{\C}=\T$ (\Cref{thm:dZ-ct_characterisation}\eqref{itit:thick}), so $\DerCat[c]{\A}\simeq\DerCat[c]{\X}\simeq\dgH[0]{\X}$ by \Cref{thm:Keller-perf}. Therefore, the essential image of $\dgH[0]{\Yoneda}\colon\dgH[0]{\A}\rightarrow\DerCat[c]{\A}$ identifies with $\C\subset\T=\dgH[0]{\X}$, which is closed under the action of the $d$-fold shift and its inverse.
    
    \item[(1)] Since $\C$ is $d\ZZ$-cluster tilting, that $\A$ is cohomologically $d$-sparse follows from \eqref{itit:d-rigid} and \eqref{itit:d-stable} in \Cref{thm:dZ-ct_characterisation}.
    \item[(3)] Since $\T$ has finite-dimensional morphism spaces and split idempotents by assumption, $\dgH[0]{\A}=\C=\add*{c}\subset\T$ is additive and has split idempotents and finite-dimensional morphism spaces too.

    \item[(4)] The object $c\in \add*{c}=\C=\dgH[0]{\A}$ satisfies the required assumption by definition.
  \end{itemize}
  The small DG category $\A$ is Karoubian pre-$(d+2)$-angulated in the sense of \Cref{def:pre-d+2-anguled_DG} by \Cref{prop:d+2=dZ-CT}.
  
  As noted in \Cref{not:the_phi}, $\Lambda\coloneqq \dgH[0]{\A(c,c)}=\C(c,c)=\T(c,c)$ is a finite-dimensional basic algebra. Moreover, since $\kk$ is perfect, $\Lambda/J_\Lambda$ is separable and $\Lambda$ is self-injective by \Cref{prop:dZ_CT_twisted_periodic}. Any finite-dimensional self-injective basic algebra is Frobenius. 

  The graded algebra $\dgH{\A(c,c)}$ is isomorphic to the graded algebra $\gLambda$ described in \Cref{sec:gLambda} for $\sigma$ a certain algebra automorphism of $\Lambda$. In particular,
  \[\T(c,c[-d])=\C(c,c[-d])\cong\dgH[-d]{\A(c,c)}\cong \gLambda[-d]=\twBim[\sigma^{-1}]{\Lambda}\cong\twBim{\Lambda}[\sigma].\]
  The $\Lambda$-bimodule $\twBim{\Lambda}[\sigma]$ is invertible (actually, all invertible $\Lambda$-bimodules arise in this way by \cite[Prop.~3.8]{Bol84}). Moreover, as noted in \Cref{rmk:Dugas_isomorphism}, $\Omega_{\Lambda^e}^{d+2}(\Lambda)\cong \twBim{\Lambda}[\sigma]$ in the stable category of $\Lambda$-bimodules. This proves that this part of the correspondence is well defined.
  
  \emph{The map $(\T,c)\mapsto(\T(c,c),\T(c,c[-d]))$ is surjective.} Let $\Lambda$ be a basic finite-dimensional algebra that is twisted $(d+2)$-periodic and $I$ an invertible $\Lambda$-bimodule isomorphic to $\Omega^{d+2}_{\Lambda^e}(\Lambda)$ in the stable category of $\Lambda$-bimodules. Then $I\cong\twBim{\Lambda}[\sigma]$ for a certain automorphism $\sigma$ of $\Lambda$ by \cite[Prop.~3.8]{Bol84}. Furthermore, $\sigma$ realises the $(d+2)$-periodicity of $\Lambda$ by the stable isomorphism $I\cong \Omega^{d+2}_{\Lambda^e}(\Lambda)$. Hence \Cref{thm:surjective_correspondence} applies.
  
  \emph{The map $(\T,c)\mapsto(\T(c,c),\T(c,c[-d]))$ is injective.} This follows from \Cref{thm:injective_correspondence}. Here $\Lambda/J_\Lambda$ is separable for $\Lambda=\T(c,c)$ because the ground field $\kk$ is perfect.
  
  \emph{The map $A\mapsto (\DerCat[c]{A},A)$ is
    surjective.} Indeed,
  given an algebraic triangulated category $\T$ with split idempotents and a basic $d\ZZ$-cluster
  tilting object $c\in\T$, any choice of enhancement $\A$ of $\T$ yields a DG
  algebra $A$ together with an equivalence of triangulated categories
  \[
    \DerCat[c]{A}\stackrel{\sim}{\longrightarrow}\T,\qquad A\longmapsto c.
  \]
  Explicitly, by assumption there exists an equivalence of
  triangulated categories $\T\simeq\dgH[0]{\A}$; identify $\T$ with
  $\dgH[0]{\A}$ along this equivalence and let $A\coloneqq\A(c,c)$ be the DG
  algebra of endomorphisms of $c$ viewed as an object of $\A$. Since
  $d\ZZ$-cluster tilting objects are in particular classical generators of the
  ambient triangulated category
  (\Cref{thm:IY-Bel-d-CT}), \Cref{thm:Keller-perf} yields an equivalence of
  triangulated categories
  \[
    \DerCat[c]{A}\stackrel{\sim}{\longrightarrow}\dgH[0]{\A},\qquad A\longmapsto c,    
  \]
  and the claim follows. 

  \emph{The map $A\mapsto (\DerCat[c]{A},A)$ is
    injective.} The injectivity of the map is a direct consequence of
  \Cref{thm:A-Ainfty-version}\eqref{it:uniqueness}. Indeed, let $A$ and $B$ be DG
  algebras satisfying the conditions in \eqref{it:dZ-CT_DGAs} and suppose that
  the pairs $(\DerCat[c]{A},A)$ and $(\DerCat[c]{B},B)$ are equivalent; explicitly, this means
  that there exists an equivalence of triangulated categories
  \[
    \DerCat[c]{A}\stackrel{\sim}{\longrightarrow}\DerCat[c]{B},\qquad A\longmapsto B;
  \]
  here we use that $A$ and $B$ are assumed to be basic objects of $\DerCat[c]{A}$.
  In particular, we may and we will identify the graded algebras $\dgH{A}$ and $\dgH{B}$ via
  this equivalence. Choose minimal
  $A_\infty$-models
  \[
    (\dgH{A},m_{d+2}^A,m_{2d+2}^A,\dots)\text{ and }(\dgH{B},m_{d+2}^B,m_{2d+2}^B,\dots).
  \]
  of $A$ and $B$, respectively. By \Cref{coro:pre-d+2-ang=unit}, that the restricted universal Massey
  products $j^*\class{m_{d+2}^A}$ and $j^*\class{m_{d+2}^B}$ are edge units, and
  \Cref{thm:A-Ainfty-version}\eqref{it:uniqueness} shows that, up to
  quasi-isomorphism of $A_\infty$-algebras, there is a unique
  $A_\infty$-structure on $\dgH{A}=\dgH{B}$ with this property. Consequently,
  $A$ and $B$ are quasi-isomorphic DG algebras (compare with the proof of
  \Cref{thm:projLambda_existence_and_uniqueness}). This finishes the proof of
  the theorem.

  \emph{Uniqueness of enhancements.} Finally, we need to show that the triangulated categories that satisfy the
  conditions in \Cref{thm:projLambda_existence_and_uniqueness}\eqref{it:Tc}
  admit a unique enhancement. This follows from \Cref{thm:uniqueness_triangulated}.
\end{proof}

\subsection{Proof of \Cref{thm:secondary_formality}}
\label{subsec:secondary_formality}

Let $A$ be a graded vector space. Essentially by definition, a minimal
$A_\infty$-algebra structure on $A$ can be identified with a point in the
mapping space (a simplicial set)
\[
  \Map{\opA}{\opEnd{A}}=\Map[\operatorname{dgOp}]{\opA}{\opEnd{A}}
\]
of operadic maps from the $A_\infty$-operad $\opA$ to the endomorphism operad
$\opEnd{A}$ of $A$ viewed as a cochain complex with vanishing differential,
computed in the model category $\operatorname{dgOp}$ of differential graded
(non-symmetric) operads
\cite{lyubashenko_2011_homotopy_unital_algebras,muro_2011_homotopy_theory_nonsymmetric}.
By \cite[Thm.~5.2.1]{Fre09} or \cite[Cor.~2.3]{Mur16},
two points in the space $\Map{\opA}{\opEnd{A}}$ lie in the same path-connected
component if and only if the corresponding minimal $A_\infty$-structures are related by
a quasi-isomorphism of $A_\infty$-algebras $F$ with $F_0=\id$ and $F_1=\id$
(compare with \Cref{thm:secondary_formality}). Moreover, the mapping space
$\Map{\opA}{\opEnd{A}}$ is the homotopy limit of the tower of fibrations
\[
  \cdots\twoheadrightarrow\Map{\opA[i+2]}{\opEnd{A}}\twoheadrightarrow\cdots\twoheadrightarrow\Map{\opA[3]}{\opEnd{A}}\twoheadrightarrow\Map{\opA[2]}{\opEnd{A}},
\]
where $\opA[n]$ is the operad of $A_n$-algebra structures.\footnote{Recall that
  an $A_n$-algebra has operations
  \[
    m_i\colon A^{\otimes i}\longrightarrow A(2-i),\qquad 1\leq i\leq n,
  \]
  that satisfy all the $A_\infty$-equations that involve only these operations;
  in particular, an $A_\infty$-algebra yields an $A_n$-algebra by discarding the
  operations $m_i$, $i>n$. Notice also that there is no condition of the $n$-ary
  multiplicaton in an $A_n$-algebra when $m_1=0$, that is when $A$ is minimal.
  For example, a minimal $A_3$-algebra consists of an \emph{associative} graded
  algebra with an \emph{arbitrary} ternary operation.}
Thus, a minimal $A_\infty$-algebra structure $x_\infty\in\Map{\opA}{\opEnd{A}}$
on $A$ yields minimal $A_{i+2}$-algebra structures
\[
  x_{i}\in\Map{\opA[i+2]}{\opEnd{A}},\qquad i\geq0,
\]
by restriction. To simplify the notation, we write
\[
  X_i\coloneqq\Map{\opA[i+2]}{\opEnd{A}},\qquad i\geq0,
\]
so that we have a tower of \emph{pointed} fibrations
\[
  \cdots\twoheadrightarrow X_{i+2}\twoheadrightarrow\cdots\twoheadrightarrow
  X_1\twoheadrightarrow X_0.
\]

The above perspective on $A_\infty$-structures is advantageous in that it
enables us to apply robust techniques from homotopy theory to compute the set of
path-connected components of the homotopy limit $\Map{\opA}{\opEnd{A}}$. Indeed,
suppose that $A$ is a $d$-sparse graded algebra. In this case, the results in
\cite{Mur20b} provide an extension of the classical Bousfield--Kan (fringed)
spectral sequence \cite[Ch.~IX, Sec.~4]{BK72} with useful computational
properties, of which we highlight the following:
\begin{itemize}
\item The classical Bousfield--Kan spectral sequence is only defined on the
  upper-half of the bisection of the first quadrant, that is $\BK[r][s][t]$ is
  defined for $t\geq s\geq 0$. In contrast, the extension of the Bousfield--Kan
  spectral sequence is defined in most of the right half-plane; however, the
  extended region of definition `converges' to that of the classical
  Bousfield--Kan spectral sequence as ${r\to\infty}$, see
  \Cref{fig:range_of_def} and \cite{Mur20b} for details.
  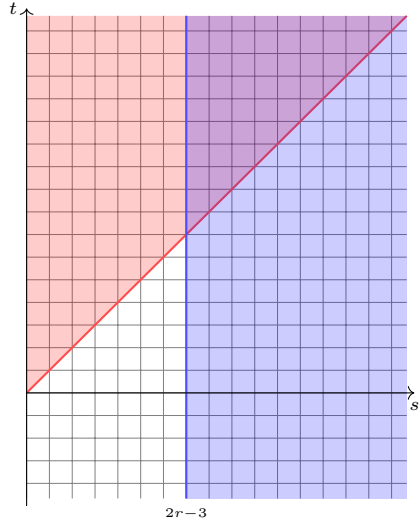
\begin{figure}[t] 
    \setcounter{rpage}{5}    
\begin{tikzpicture}


  \draw[step=\rescale,gray,very thin] (0,{-\abajo*\rescale+\margen}) grid ({(3*\rpage -3 +\derecha)*\rescale-\margen},{(3*\rpage -3 +\arriba)*\rescale-\margen});


  \filldraw[fill=red,draw=none,opacity=0.2] (0,0) -- ({(3*\rpage -3 +\arriba)*\rescale-\margen},{(3*\rpage -3 +\arriba)*\rescale-\margen}) --
  (0,{(3*\rpage -3 +\arriba)*\rescale-\margen}) -- cycle;
  \draw[red!70,thick]  (0,0) -- ({(3*\rpage -3 +\arriba)*\rescale-\margen},{(3*\rpage -3 +\arriba)*\rescale-\margen});


  \filldraw[fill=blue,draw=none,opacity=0.2]  ({(2*\rpage-3)*\rescale},{{(3*\rpage -3 +\arriba)*\rescale-\margen}}) -- ({(2*\rpage-3)*\rescale},{-\abajo*\rescale+\margen}) -- ({(3*\rpage -3 +\derecha)*\rescale-\margen},{-\abajo*\rescale+\margen}) --
  ({(3*\rpage -3 +\derecha)*\rescale-\margen},{(3*\rpage -3 +\arriba)*\rescale-\margen}) -- cycle;

  \draw[blue!70,thick] ({(2*\rpage-3)*\rescale},{(3*\rpage -3 +\arriba)*\rescale-\margen}) -- ({(2*\rpage-3)*\rescale},{-\abajo*\rescale+\margen}) node[black,anchor=north] {$\scriptscriptstyle 2r-3$};


  \draw [->] (0,0)  -- ({(3*\rpage -3 + \derecha)*\rescale},0) node[anchor=north] {$\scriptstyle s$};


  \draw [->] (0,-{\abajo*\rescale}) -- (0,{(3*\rpage -3 +\arriba)*\rescale})node[anchor=east] {$\scriptstyle t$};

\end{tikzpicture}
    \caption{Range of definition of the extended Bousfield--Kan spectral
      sequence ($r=5$). The red region is where the classical Bousfield--Kan
      spectral sequence is defined; the extended spectral sequence is defined
      also in the blue region.}
    \label{fig:range_of_def}
  \end{figure}
\item The $\BK[2]$-terms of the extended Bousfield--Kan spectral sequence are given
  by \cite[Cor.~6.3]{Mur20b}
  \[
    \BK[2][s][t]=\begin{cases}\HH[s+2][-t]{A}&s>0,\
      t\in\ZZ;\\\dgZ[2,-t]{\HC{A}}&s=0,\ t>0.\end{cases}
  \]
  In particular, since the graded algebra $A$ is $d$-sparse, $\BK[2][s][t]=0$ if
  $t\not\in d\ZZ$. Therefore, the spectral sequence differentials
  $\partial_{i}=0$ are necessarily trivial for $2\leq i\leq d$ and
  $\BK[d+1][s][t]=\BK[2][s][t]$ wherever they are defined (see next bullet point).
\item The extended Bousfield--Kan spectral sequence is partially-defined even if
  we only have a minimal $A_{i+2}$-structure on $A$, so just the bottom part of the tower of fibrations is pointed. More precisely, the extended
  spectral sequence is defined up to the terms of page
  $\lfloor\frac{i+3}{2}\rfloor$, see \cite[Def.~5.1]{Mur20b}.
\item Consider a minimal $A_{i+2}$-algebra structure on $A$, that is a base point
  \[
    x_i\in X_i=\Map{\opA[i+2]}{\opEnd{A}}.
  \]
  The obstruction to the existence of a minimal $A_{i+3}$-algebra structure
  \[
    x_{i+1}\in x_{i+1}=\Map{\opA[i+3]}{\opEnd{A}}
  \]
  such that $x_i$ and $x_{i+1}$ have the same underlying minimal $A_{i-r+3}$-algebra
  structure
  \[
    x_{i-r+1}\in X_{i-r+1}=\Map{\opA[i-r+3]}{\opEnd{A}}
  \]
  lies in the term $\BK[r][i+1][i]$, $1\leq r\leq\lfloor\tfrac{i+3}{2}\rfloor$.
  This obstruction vanishes precisely when such a minimal $A_{i+3}$-algebra structure
  exists, see \cite[Prop.~5.4]{Mur20b}.
\end{itemize}

\begin{remark}
  By construction, the spectral sequence described above is related to the
  homotopy groups of the space $\Map{\opA}{\opEnd{A}}$, which turn out to be computed in
  terms of Hochschild cohomology of $A_\infty$-algebras,
  see~\cite[Prop.~6.9]{Mur20b} for a precise statement. In particular, notice that the spectral sequence
  above is different from those usually considered in the context of deformation
  problems, such as the one discussed in~\cite[Ch.~3]{Sei15}. Seidel's spectral
  sequence is a particular case of the third sequence in
  Elmendorf--Kriz--Mandell--May's \cite[Thm.~1.6]{EKMM07} in the context of `brave new algebra.' For the relation between these spectral sequences and ours, see \cite[Sec.~2]{Ang11} and \cite[Sec.~7]{Mur20b}.
\end{remark}

A minimal $A_{i+2}$-algebra structure on $A$ has an $m_{d+2}$ operation for any $i\geq d$, but we need $i>d$ to ensure that $m_{d+2}$ is a cocycle, whose cohomology class
\[\{m_{d+2}\}\in\HH[d+2][-d]{A}\]
also deserves to be called \emph{universal Massey product of length $d+2$}, and $i>2d$ so that \[\Sq(\{m_{d+2}\})=0.\]

\begin{proposition}
  \label{prop:differential_d+1}
  Let $A$ be a $d$-sparse graded algebra equipped with a minimal $A_{2d+3}$-algebra structure, so the spectral sequence is defined up to the $\BK[d+2]$ terms. The spectral sequence differential \[\partial_{d+1}\colon \BK[d+1][s][t]\longrightarrow \BK[d+1][s+d+1][t+d]\]
  is given by 
  \[\partial_{d+1}(x)=\pm[\{m_{d+2}\},x]\]
  wherever it is defined, except for $(s,t)=(d-1,d)$ in $\cchar{\kk}=2$, where it is given by
  \[\partial_{d+1}(x)=[\{m_{d+2}\},x]+x^2.\]
\end{proposition}

\begin{remark}
  The differential defined in \Cref{prop:differential_d+1} is the first possibly
  non-trivial spectral sequence differential beyond page $1$. Notice also that $x^2=0$ in $\cchar{\kk}\neq 2$ by graded commutativity.
\end{remark}

\begin{proof}[Proof of \Cref{prop:differential_d+1}]
  The general formula follows from \cite[Lem.~6.1]{Mur20b} by the same argument as in the proof of \cite[Thm.~6.5]{Mur20b}, which is the case $d=1$. Condition $m\leq s+1$ in the statement of \cite[Lem.~6.1]{Mur20b} is actually unnecessary. Indeed, as indicated therein, this lemma follows directly from \cite[Prop.~3.5]{Mur20b} where no upper bound is required. In our current context, where $A$ is $d$-sparse, \cite[Prop.~3.5]{Mur20b} is used for $r=s+1$, $s=d+1$ and $t=1$ and only the first and last summands ($p=2,d+2$) produce possibly non-vanishing elements since $m_{i+2}=0$ for $0<i<d$. Hence the proof goes through as for \cite[Thm.~6.5]{Mur20b}.

  In the special case $(s,t)=(d-1,d)$ we must invoke \cite[Cors.~3.9,3.10]{Mur20b} instead, as in the proof of \cite[Thm.~6.5]{Mur20b}, here for $r=2d+1$ and $s=1$.
\end{proof}

\begin{remark}
  For $s=0$ and $t>0$, $x\in E^{0,t}_2=\dgZ[2,-t]{\HC{A}}$ is a cocycle. Hence, on the right hand side of the formula for $\partial_{d+1}(x)$ in \Cref{prop:differential_d+1}, we understand that $x$ stands for its Hochschild cohomology class.
\end{remark}

\begin{definition}\label{def:Massey_algebra}
  A \emph{$d$-sparse Massey algebra} $(A,m)$ is a pair given by a $d$-sparse graded algebra $A$ and a Hochschild cohomology class
  \[m\in\HH[d+2][-d]{A}\]
  with trivial Gerstenhaber square
  \[\Sq(m)=0.\]
  The \emph{Hochschild--Massey complex} of a $d$-sparse Massey algebra is the bigraded vector space $\SHC{A}[m]$ defined as $\HH{A}$ in horizontal degrees $\bullet\geq 2$ and zero otherwise, with bidegree $(d+1,-d)$ differential 
  \begin{align*}
    \HH[s][t]{A}&\longrightarrow\HH[s+d+1][t-d]{A},\\
    x&\longmapsto [m,x],
  \end{align*}
  except in bidegree $(d+1,-d)$, where it is given by
  \begin{align*}
    \HH[d+1][-d]{A}&\longrightarrow\HH[2(d+1)][-2d]{A},\\
    x&\longmapsto [m,x]+x^2.
  \end{align*}
  The cohomology of this complex is the \emph{Hochschild--Massey cohomology} of the $d$-sparse Massey algebra, and it is denoted by $\SHH{A}[m]$.
\end{definition}

\begin{remark}
  \label{rmk:massey_algebra}
  The Hochschild complex of a $d$-sparse Massey algebra is indeed a complex since 
  \[[m,[m,x]]=[\Sq(m),x]=0\]
  by the Gerstenhaber algebra relations. Moreover, if $x$ is in bidegree $(d+1,-d)$
  \[[m,[m,x]+x^2]=[m,[m,x]]+[m,x^2]=[\Sq(m),x]+[m,x]\cdot x-x\cdot [m,x]=0.\]
  Notice also that $x^2=0$ in $\cchar{k}\neq 2$ by graded commutativity.
\end{remark}

\begin{remark}
  \label{rmk:BK_massey_algebra}
  Let $A$ be a $d$-sparse graded algebra equipped with a minimal $A_{2d+3}$-algebra structure.
  The terms $\BK[d+1][s][t]$ of the (truncated) extended Bousfield--Kan spectral sequence constructed in \cite{Mur20b} are defined for $t\geq s$ and for $s\geq 2d-1$ and $t\in\ZZ$. Moreover, the differential
  \[\partial_{d+1}\colon \BK[d+1][s][t]\longrightarrow \BK[d+1][s+d+1][t+d]\]
  is defined except for $0\leq s=t\leq d$. Furthemore, $\BK[d+2][s][t]$ is given
  by the cohomology of $(\BK[d+1],\partial_{d+1})$ whenever the incoming and
  outgoing differentials are defined on $\BK[d+1][s][t]$. Notice that the
  assumptions on $A$ imply that $(A,\{m_{d+2}\})$ is a $d$-sparse Massey algebra. Moreover, \Cref{prop:differential_d+1} shows that 
  $(\BK[d+1],\partial_{d+1})$ surjects onto $\SHC[\bullet+2][-*]{A}[\{m_{d+2}\}]$ whenever the source is defined, and this surjection is actually an isomorphism for $\bullet>0$. Therefore, we deduce that 
  \[E^{s,t}_{d+2}\cong\SHH[s+2][-t]{A}[\{m_{d+2}\}]\]
  in the following cases:
  \begin{itemize}
    \item $t>s$,
    \item $t=s>d$,
    \item $s\geq 3d$, $t\in\ZZ$.
  \end{itemize}
\end{remark}

We are ready to prove \Cref{thm:secondary_formality}.

\begin{proof}[Proof of \Cref{thm:secondary_formality}] Let $A$ be a $d$-sparse
  graded algebra and
  \[
    B=(A,m_{d+2}^B,m_{2d+2}^B,m_{3d+2}^B,\dots)\qquad\text{and}\qquad C=(A,m_{d+2}^C,m_{2d+2}^C,m_{3d+2}^C,\dots)
  \]
  minimal $A_\infty$-algebra structures on $A$, which we identify with points
  \[
    x_{\infty},y_\infty\in X_{\infty}=\Map{\opA}{\opEnd{A}},
  \]
  respectively. Suppose that there is an equality of universal Massey products
  \[
    m=\class{m_{d+2}^B}=\class{m_{d+2}^C}\in\HH[d+2][-d]{A}
  \]
  and that
  \[
    \SHH[p+2][-p]{A}[m]=0,\qquad p>d.
  \]
  As explained at the beginning of this subsection, we need to prove that $x_\infty$ and $y_{\infty}$ lie in
  the same connected component of $X_\infty$. We take $x_\infty\in X_\infty$ as a base point for the definition of the spectral sequence in \cite{Mur20b}. By \Cref{rmk:BK_massey_algebra}
  \[
    E^{p,p}_{d+2}=\SHH[p+2][-p]{A}[m]=0,\qquad p>d,
  \]
  therefore
  \begin{equation}
    \label{eq:E_r_pp-vanishing}
    E^{p,p}_{r}=0,\qquad p>d,\quad d+2\leq r\leq\infty.
  \end{equation}
  
  Denote by $x_i, y_i\in X_i$ the images of $x_\infty,y_\infty\in X_\infty$ along the restriction map $X_\infty\to X_i$. We have chosen $x_i$ as the base point of $X_i$, $0\leq i\leq\infty$, hence $\pi_0(X_i)$ is a set pointed by $[x_i]$. We will inductively show that $[x_i]=[y_i]\in\pi_0(X_i)$ for $0\leq i\leq \infty$.

($i=0$) Since $A$ has trivial differential, the set $\pi_0(X_0)$ is in
      bijection with the set of degree $0$ maps of graded vector spaces
      $A\otimes A\to A$, and the class $[x_0]$ corresponds to the product
      operation of
      $A$; the class $[y_0]$ also corresponds to the product operation of $A$
      since $A$ is the underlying graded algebra of both minimal
      $A_\infty$-algebra structures $B$ and $C$. We distinguish between the
      various natural cases:

      ($0<i<d)$ Since the graded algebra $A$ is $d$-sparse, in this range the
      set $\pi_0(X_i)$ is in bijection with the set of graded associative algebra structures on the underlying graded vector space of $A$, and both $[x_i]$ and $[y_i]$ correspond the product of $A$ by the same reason as above. 

      ($i=d$) The previous case shows, in particular, that
      $[x_{d-1}]=[y_{d-1}]$. Therefore the classes $[x_d]$ and $[y_d]$ lie in 
    \[\ker(\pi_0 X_d\to \pi_0 X_{d-1})\cap\operatorname{im}(\pi_0
      X_{2d}\to\pi_0 X_d)\]
    and, according to \cite[Eqn.~4.8]{Mur20b}, this intersection is in fact equal to
    \[
      E^{d,d}_{d+1}=E^{d,d}_{2}\cong\HH[d+2][-d]{A}.
    \]
    This identification is induced by the class $\{m_{d+2}^B\}\in\HH[d+2][-d]{A}$ as follows. An element $[z_d]\in E^{d,d}_{d+1}$ is represented by a minimal $A_{d+2}$-algebra structure on $A$ which can be extended to a minimal $A_{2d+2}$-algebra structure. Such an $A_{d+2}$-algebra $(A,m^D_{d+2})$ has a universal Massey product of length $d+2$ given by the Hochschild cohomology class $\{m^D_{d+2}\}\in\HH[d+2][-d]{A}$.
    Since since $x_d\in X_d$ is the base point, which corresponds to the minimal $A_{d+2}$-algebra structure $(A,m^B_{d+2})$, the previous indentification sends $[z_d]$ to $\{m_{d+2}^D\}-\{m_{d+2}^B\}$. In particular, $[x_d]\in E^{d,d}_{d+1}$ corresponds to $\{m_{d+2}^B\}-\{m_{d+2}^B\}=0$ and $[y_d]$ too since
    $\class{m_{d+2}^B}=\class{m_{d+2}^C}$ by hypothesis, so $\{m_{d+2}^C\}-\{m_{d+2}^B\}=\{m_{d+2}^B\}-\{m_{d+2}^B\}=0$.

   ($d<i<\infty$) As in the previous case, the classes $[x_i]$ and $[y_i]$
      lie in the intersection
    \[
      \ker(\pi_0 X_i\to \pi_0 X_{i-1})\cap\operatorname{im}(\pi_0
      X_{2i}\to\pi_0 X_i)=\BK[i+1][i][i]=0,
    \]
    see \cite[Eqn.~4.8]{Mur20b} and \eqref{eq:E_r_pp-vanishing}, therefore $[x_i]=[y_i]$.

    ($i=\infty$) We deduce the desired equality $[x_\infty]=[y_\infty]$ from the following two facts:
    \begin{itemize}
    \item \cite[Sec.~IX.3.1]{BK72} There is a Milnor short
      exact sequence of pointed sets
      \begin{center}
        $*\longrightarrow\varprojlim\nolimits^1\pi_{1}(X_i)\longrightarrow\pi_0(X_\infty)\longrightarrow\varprojlim
        \pi_0(X_i)\longrightarrow *$
      \end{center}
      and, by the previous discussion, the classes $[x_\infty]$ and $[y_\infty]$ have
      the same image under the canonical map
      $\pi_0(X_\infty)\to\varprojlim\pi_0(X_i)$, namely the base point of $\varprojlim
      \pi_0(X_i)$.
    \item \cite[Sec.~IX.5.4]{BK72} For every $s\geq 0$, the term
    $\BK[r][s][s+1]$ stabilises for $r\geq \max\{d+2,s+1\}$ since the source of the incoming differential is $\BK[r][s-r][s-r+1]=0$ being in the left half-plane ($s-r\leq s-(s+1)<0$), and the target of the outgoing differential is $\BK[r][s+r][s+r]=0$ since $s+r\geq r\geq d+2>d$. Therefore  $\varprojlim\nolimits^1\pi_1(X_i)=0$. This implies that the inverse image of the base point along 
    the natural map of pointed sets
    $\pi_0(X_\infty)\to\varprojlim\pi_0(X_i)$ consists just of the source's base point, hence $[x_\infty]=[y_\infty]$ by the last observation in the previous item.
    \end{itemize}
    This finishes the proof of the theorem.
\end{proof}

For the sake of completeness, we observe that \Cref{thm:secondary_formality}
implies Kadeishvili's intrinsic formality theorem, which we state in its version
for minimal $A_\infty$-algebras.

\begin{corollary}[{\cite{Kad88}, see also \cite[Thm.~4.7]{ST01}}]
  \label{coro:Kadeishvili}
  Let $A$ be a graded algebra such that
  \[
    \HH[p+2][-p]{A}[A]=0,\qquad p\geq1.
  \]
  If $B$ and $C$ are minimal $A_\infty$-algebras with $\dgH{B}=\dgH{C}=A$ as
  graded algebras, then $B$ and $C$ are $A_\infty$-isomorphic through an
  $A_\infty$-isomorphism with identity linear part. In particular, the graded
  algebra $A$ is intrinsically formal.
\end{corollary}
\begin{proof}
  Endow the graded algebra $A$ with the trivial $1$-sparse Massey algebra
  structure $(A,0)$ and notice that, in fact, this is the only possible
  such structure since $\HH[3][-1]{A}[A]=0$ by assumption. We will apply
  \Cref{thm:secondary_formality} to $(A,0)$. By definition, the cochain
  complex $\SHC{A}[0]$ has vanishing differential and therefore the
  Hochschild--Massey cohomology of the $1$-sparse Massey algebra $(A,0)$
  coincides with the (ordinary) Hochschild cohomology of the graded algebra $A$.
  In particular,
  \[
    \HH[p+2][-p]{A}[0]=\HH[p+2][-p]{A}=0,\qquad p>1,
  \]
  by assumption. Thus, we may apply \Cref{thm:secondary_formality} to
  the $1$-sparse Massey algebra $(A,0)$ and the first claim follows. The second
  claim follows by taking $C=A$ viewed as a minimal $A_\infty$-algebra with
  vanishing higher operations.
\end{proof}

We record the following vanishing properties that are crucial to the proof of
\Cref{thm:final_thm} below, compare with the hypothesis
in~\Cref{rmk:BK_massey_algebra} and the assumptions in~\Cref{thm:secondary_formality}.

\begin{proposition}
  \label{prop:BK_vanishing}
  Let $A=\gLambda$ as in \Cref{setting:existence_and_uniqueness}. Assume we have (at least) a minimal $A_{2d+3}$-algebra structure on $A$ such that $\{m_{d+2}\}$ is an edge unit. Then the terms $\BK[d+2][s][t]$ in the
  extended Bousfield--Kan spectral sequence associated to the tower of
  fibrations
  \[
    \cdots\twoheadrightarrow X_i\twoheadrightarrow\cdots\twoheadrightarrow
    X_1\twoheadrightarrow X_0
  \]
  can be non-zero only in the following cases (\Cref{fig:vanishing}):
  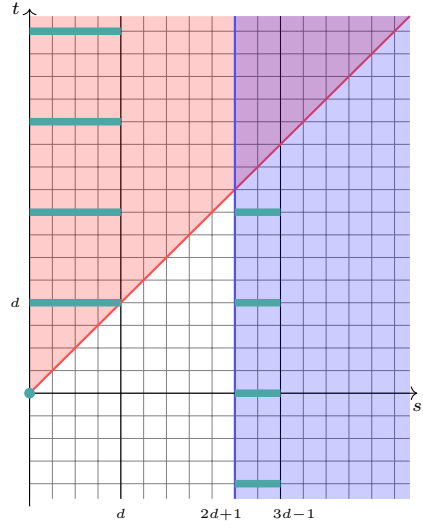
\begin{figure}[t]
            \begin{tikzpicture}


            \draw[step=\rescale,gray,very thin] (0,{-\abajo*\rescale+\margen}) grid ({(3*\rpage -3 +\derecha)*\rescale-\margen},{(3*\rpage -3 +\arriba)*\rescale-\margen});

            \node[black,anchor=east] at (0, {(\rpage -1)*\rescale}) {$\scriptscriptstyle d$};

            \draw[black] ({(\rpage-1)*\rescale},{(3*\rpage -3 +\arriba)*\rescale-\margen}) -- ({(\rpage-1)*\rescale},{-\abajo*\rescale+\margen}) node[black,anchor=north] {$\scriptscriptstyle d$};

            \draw[black] ({(3*\rpage-4)*\rescale},{(3*\rpage -3 +\arriba)*\rescale-\margen}) -- ({(3*\rpage-4)*\rescale},{-\abajo*\rescale+\margen}) node[black,anchor=north] {$\scriptscriptstyle \quad 3d-1$};


            \filldraw[fill=red,draw=none,opacity=0.2] (0,0) -- ({(3*\rpage -3 +\arriba)*\rescale-\margen},{(3*\rpage -3 +\arriba)*\rescale-\margen}) --
            (0,{(3*\rpage -3 +\arriba)*\rescale-\margen}) -- cycle;
            \draw[red!70,thick]  (0,0) -- ({(3*\rpage -3 +\arriba)*\rescale-\margen},{(3*\rpage -3 +\arriba)*\rescale-\margen});


            \filldraw[fill=blue,draw=none,opacity=0.2]  ({(2*\rpage-1)*\rescale},{{(3*\rpage -3 +\arriba)*\rescale-\margen}}) -- ({(2*\rpage-1)*\rescale},{-\abajo*\rescale+\margen}) -- ({(3*\rpage -3 +\derecha)*\rescale-\margen},{-\abajo*\rescale+\margen}) --
            ({(3*\rpage -3 +\derecha)*\rescale-\margen},{(3*\rpage -3 +\arriba)*\rescale-\margen}) -- cycle;

            \draw[blue!70,thick] ({(2*\rpage-1)*\rescale},{(3*\rpage -3 +\arriba)*\rescale-\margen}) -- ({(2*\rpage-1)*\rescale},{-\abajo*\rescale+\margen}) node[black,anchor=north] {$\scriptscriptstyle 2d+1\quad$};


            \draw [->] (0,0)  -- ({(3*\rpage -3 + \derecha)*\rescale},0) node[anchor=north] {$\scriptstyle s$};


            \draw [->] (0,-{\abajo*\rescale}) -- (0,{(3*\rpage -3 +\arriba)*\rescale})node[anchor=east] {$\scriptstyle t$};


            \foreach  \x in {0,...,3}
            \draw[teal!70,line width=1mm] ({(2*\rpage-1)*\rescale},{(\x-1)*(\rpage-1)*\rescale}) -- ({(2*\rpage-1+\rpage-3)*\rescale},{(\x-1)*(\rpage-1)*\rescale});


            \foreach  \x in {2,...,\rpage}
            \draw[teal!70,line width=1mm] ({0},{(\x-1)*(\rpage-1)*\rescale}) -- ({(\rpage-1)*\rescale},{(\x-1)*(\rpage-1)*\rescale});

            \node[fill=teal!70,draw=none,circle,inner sep=.5mm,opacity=1]   at (0,0) {};

          \end{tikzpicture}
    \caption{The $\BK[d+2]$-page of the extended Bousfield--Kan spectral sequence
      can be non-trivial only in the green part $(d=4)$, see
      \Cref{prop:BK_vanishing}.}
    \label{fig:vanishing}
  \end{figure}
  \begin{itemize}
  \item $s=0$ and $t=0$.
  \item $t=kd$, $k\geq1$, and $0\leq s\leq d$.
  \item $t=kd$, $k\leq2$, and $2d+1\leq s\leq 3d-1$.
  \end{itemize}
\end{proposition}
\begin{proof}
  We prove that
  \[\SHH[s][t]{A}[\{m_{d+2}\}]=0,\qquad s> d+2.\]
  The proposition then follows from the relation between $\BK[d+2]$ and the Hochschild--Massey cohomology of the $d$-sparse Massey algebra $(A,\{m_{d+2}\})$ in \Cref{rmk:BK_massey_algebra}, and from the fact that $\SHH[s][t]{A}[\{m_{d+2}\}]$ is concentrated in $t\in d\ZZ$ since $A$ is $d$-sparse.

  Consider the degree $(d+2,-d)$ endomorphism of $\SHC{A}[\{m_{d+2}\}]$ given by the isomorphisms
  \begin{align*}
    \HH[s][t]{A}&\stackrel{\cong}{\longrightarrow}\HH[s+d+2][t-d]{A},\\
    x&\longmapsto \{m_{d+2}\}\cdot x,
  \end{align*}
  if $(s,t)\neq (d+1,-d)$,
  see \Cref{prop:key_edege_units}, and by
  \begin{align*}
    \HH[d+1][-d]{A}&\stackrel{\cong}{\longrightarrow}\HH[2d+3][-2d]{A},\\
    x&\longmapsto \{m_{d+2}\}\cdot x+\{\EulerDer[d]\}\cdot x^2.
  \end{align*}
  The latter map is also bijective. This can be checked as in the proof of \cite[Thm.~9.1]{Mur22}. The remaining components of this endomorphism are maps from $0$ to possibly non-trivial targets, so it not a point-wise isomorphism everywhere.
  
  Indeed, the previous maps assemble to a cochain map because, by the Gerstenhaber algebra relations,
  \begin{align*}
    [\{m_{d+2}\},\{m_{d+2}\}\cdot x]&=[\{m_{d+2}\},\{m_{d+2}\}]\cdot x+\{m_{d+2}\}\cdot [\{m_{d+2}\},x]\\
    &=2\Sq(\{m_{d+2}\})\cdot x+\{m_{d+2}\}\cdot [\{m_{d+2}\},x]\\
    &=\{m_{d+2}\}\cdot [\{m_{d+2}\},x],
  \end{align*}
  and if $x$ has bidegree $(d+1,-1)$ then
  \begin{align*}
    [\{m_{d+2}\},\{m_{d+2}\}\cdot x+\{\EulerDer[d]\}\cdot x^2]&=
    [\{m_{d+2}\},\{m_{d+2}\}\cdot x]+[\{m_{d+2}\},\{\EulerDer[d]\}\cdot x^2]\\
    &=\{m_{d+2}\}\cdot [\{m_{d+2}\},x]+[\{m_{d+2}\},\{\EulerDer[d]\}]\cdot x^2\\
    &=\{m_{d+2}\}\cdot [\{m_{d+2}\},x]+\{m_{d+2}\}\cdot x^2\\
    &=\{m_{d+2}\}\cdot \left([\{m_{d+2}\},x]+x^2\right),
  \end{align*}
  see \Cref{prop:euler_lie}.

  Since the previous cochain endomorphism is mostly given by isomorphisms, it induces isomorphisms in most of the cohomology. Namely,
  \[\SHH[s][t]{A}[\{m_{d+2}\}]\cong \SHH[s+d+2][t-d]{A}[\{m_{d+2}\}],\qquad s> d+2,\]
  because in these bidegrees the endomorphism of $\SHC{A}[\{m_{d+2}\}]$ is an isomorphism on the source and target of the incoming and outgoing differentials.

  Strikingly, the previous endomorphism of the complex $\SHC{A}[\{m_{d+2}\}]$ is  null-homotopic. The chain homotopy is given by the bidegree $(1,0)$ maps
  \begin{align*}
    \HH[s][t]{A}&\stackrel{\cong}{\longrightarrow}\HH[s+1][t]{A},\\
    x&\longmapsto \{\EulerDer[d]\}\cdot x.
  \end{align*}
  Indeed, by \Cref{prop:HH_computations},
  \begin{align*}
    \{m_{d+2}\}\cdot x                                                  & =[\{m_{d+2}\},\class{\EulerDer[d]}\cdot x]+\class{\EulerDer[d]}\cdot[\{m_{d+2}\},x],
  \end{align*}
  in particular if $x$ has bidegree $(d+1,-d)$,
  \begin{align*}
    \{m_{d+2}\}\cdot x+\{\EulerDer[d]\}\cdot x^2                                                  & =[\{m_{d+2}\},\class{\EulerDer[d]}\cdot x]+\class{\EulerDer[d]}\cdot\left([\{m_{d+2}\},x]+x^2\right),
  \end{align*}
  and if $d>1$ and $x$ has bidegree $(d,-d)$,
  \begin{align*}
    \{m_{d+2}\}\cdot x                                                  & =[\{m_{d+2}\},\class{\EulerDer[d]}\cdot x]+\left(\class{\EulerDer[d]}\cdot x\right)^2+\class{\EulerDer[d]}\cdot[\{m_{d+2}\},x]\\
    & =[\{m_{d+2}\},\class{\EulerDer[d]}\cdot x]+\class{\EulerDer[d]}\cdot[\{m_{d+2}\},x],
  \end{align*}
  since the product is graded commutative and $\class{\EulerDer[d]}^2=0$ by \Cref{prop:euler_square}.

  Therefore, the previous isomorphisms on the Hochschild--Massey cohomology of the $d$-sparse Massey algebra $(A,\{m_{d+2}\})$ are trivial maps. This implies that 
  \[\SHH[s][t]{A}[\{m_{d+2}\}]=0,\qquad s> d+2,\]
  as desired.
\end{proof}

\subsection{Proof of \Cref{thm:A-Ainfty-version}}
\label{subsec:thmA-Ainfty-version}

In this subsection we complete the proof of
\Cref{thm:dZ-Auslander_correspondence} by providing a proof of
\Cref{thm:A-Ainfty-version}. We remind the reader of our standing assumptions
for this section (\Cref{setting:existence_and_uniqueness}).

\begin{notation}
  We denote by $\AS{\Lambda}{\sigma}$ the set of (minimal) $A_\infty$-algebra structures
  \[
    (\gLambda,m_{d+2},m_{2d+2},m_{3d+2},\dots)
  \]
  on the $d$-sparse graded algebra $\gLambda$ whose restricted universal Massey
  product
  \[
    j^*\class{m_{d+2}}\in\HH[d+2][-d]{\Lambda}[\gLambda]=\Ext[\Lambda^e]{\Lambda}{\twBim{\Lambda}[\sigma]}[d+2]
  \]
  is an edge unit. Similarly, we denote by $\AS*{\Lambda}{\sigma}$ the quotient
  of $\AS{\Lambda}{\sigma}$ by the equivalence relation given by the existence
  of a quasi-isomorphism of $A_\infty$-algebras $F$ with 
  $F_1=\id$.
\end{notation}

\begin{remark}
  The group $\Aut*{\gLambda}$ of graded algebra automorphisms of $\gLambda$ acts
  on the right on the set $\AS{\Lambda}{\sigma}$ via the formula in
  \Cref{not:g_action}, and this action descends to the quotient
  $\AS*{\Lambda}{\sigma}$ (compare with \Cref{rmk:automorphism_action_on_cochains,rmk:action}). Moreover, the set of
  path-connected components of the space $\Map{\opA}{\opEnd{\gLambda}}$ of
  minimal $A_\infty$-algebra structures on $\gLambda$ contains the quotient set
  $\AS*{\Lambda}{\sigma}$, see \cite[Prop.~6.9]{Mur20b}.
\end{remark}

We shall prove the following theorem.

\begin{theorem}
  \label{thm:final_thm}
  Suppose that $\Lambda/J_\Lambda$ is separable. The following statements hold:
  \begin{enumerate}
  \item\label{it:equivmap_is_injective} The $\Aut*{\gLambda}$-equivariant map
    \begin{align*}
      \Phi\colon\AS*{\Lambda}{\sigma}    & \longrightarrow\HH[d+2][-d]{\Lambda}[\gLambda] \\
      (\gLambda,m_{d+2},m_{2d+2},\dots) & \longmapsto j^*\class{m_{d+2}}
    \end{align*}
    is injective.
  \item\label{it:consists_of_units} The image of $\Phi$ consists precisely of
    the edge units in $\HH[d+2][-d]{\Lambda}[\gLambda]$.
  \item\label{it:is_singleton} The set of orbits
    \[
      \AS*{\Lambda}{\sigma}/\Aut*{\gLambda}
    \]
    is a singleton.
  \end{enumerate}
\end{theorem}
\begin{proof}
  The map
  \begin{align*}
    \Phi\colon\AS*{\Lambda}{\sigma}    & \longrightarrow\HH[d+2][-d]{\Lambda}[\gLambda] \\
    (\gLambda,m_{d+2},m_{2d+2},\dots) & \longmapsto j^*\class{m_{d+2}}
  \end{align*}
  is clearly $\Aut*{\gLambda}$-equivariant and, by definition, its image
  consists of units in the Hochschild--Tate cohomology
  $\TateHH{\Lambda}[\gLambda]$ (the set of Hochschild--Tate units in bidegree
  $(d+2,-d)$ is non-empty since the algebra $\Lambda$ is twisted
  $(d+2)$-periodic with respect to the algebra automorphism
  $\sigma\colon\Lambda\stackrel{\sim}{\rightarrow}\Lambda$, see
  \Cref{setting:existence_and_uniqueness} and \Cref{rmk:TateUnits}).

  Notice also that statement \eqref{it:is_singleton} follows immediately from
  statements \eqref{it:equivmap_is_injective} and \eqref{it:consists_of_units}
  and \Cref{prop:single_orbit}: statements \eqref{it:equivmap_is_injective} and
  \eqref{it:consists_of_units} imply that we may identify
  $\AS*{\Lambda}{\sigma}/\Aut*{\gLambda}$ with the set of orbits of edge units
  in $\HH[d+2][-d]{\Lambda}[\gLambda]$, and the latter set is a singleton by
  \Cref{prop:single_orbit}. Recall that
  \[
    X_i=\Map{\opA[i+2]}{\opEnd{\gLambda}},\qquad i\geq0,
  \]
  is the space of minimal $A_{i+2}$-algebra structures on $\gLambda$, and that
  we have a tower of fibrations
  \[
    \cdots\twoheadrightarrow X_{i+2}\twoheadrightarrow\cdots\twoheadrightarrow
    X_1\twoheadrightarrow X_0
  \]
  as well as a distinguished base point $x_0\in X_0$ given by multiplication on
  $\gLambda$. As in \Cref{subsec:secondary_formality}, we identify
  \[
    X_{\infty}=\Map{\opA}{\opEnd{\gLambda}},
  \]
  with the (homotopy) limit of the above tower. In particular, we have
  restriction/truncation maps $X_\infty\to X_i$ for all $i\geq0$.
  
  \emph{Proof of statement \eqref{it:consists_of_units}:} We have to prove that
  every edge unit is in the image of $\Phi$. Let
  $x\in\HH[d+2][-d]{\Lambda}[\gLambda]$ be an edge unit. We need to prove that
  there exists an $A_\infty$-algebra structure
  \[
    x_\infty\in X_\infty=\Map{\opA}{\opEnd{\gLambda}},\qquad x_\infty\longmapsto
    x_0,
  \]
  whose underlying graded algebra structure is the given one
  (\Cref{rmk:Lambda-sigma-d}) and whose restricted universal Massey product of
  length $d+2$ is $x$. We make the following observations:
  \begin{itemize}
  \item The class $x$ lies in the image of the restriction map
    \begin{center}
      $j^*\colon\HH[d+2][-d]{\gLambda}[\gLambda]\longrightarrow\HH[d+2][-d]{\Lambda}[\gLambda]$
    \end{center}
    induced by the inclusion $j\colon\Lambda\hookrightarrow\gLambda$ of the
    degree $0$ part. Indeed, in view of the long exact sequence described in
    \Cref{prop:long_ex_seqHH}, it is enough to check that
    $\sigma_*^{-1}\sigma^*(x)=x$; this can be shown exactly as in the proof of
    \cite[Cor.~7.7]{Mur22}, that is using \cite[Prop.~7.6]{Mur22} for $n=d+2$,
    replacing $\sigma$ by $\sigma^{-1}$ and keeping in mind that the extension
    of $\sigma^{-1}$ acts on $\gLambda[-d]=\twBim[\sigma^{-1}]{\Lambda}$ by
    \[
      (-1)^{-d}\sigma^{-1}=(-1)^{(d+2)^2}\sigma^{-1}.
    \]
  \item There exists a unique $u\in\HH[d+2][-d]{\gLambda}[\gLambda]$ such that
    $j^*(u)=x$ and $\Sq[u]=0$. To prove the existence of such a class $u$,
    notice that if $j^*(u)=x$, then also
    \[
      j^*(u+\class{\EulerDer[d]}\cdot y)=x
    \]
    for arbitrary $y\in\HH[d+1][-d]{\gLambda}[\gLambda]$; this follows from the
    fact that $\Lambda$ is concentrated in degree $0$ and therefore
    $j^*\class{\EulerDer[d]}=0$. Arguing as in the proof of
    \cite[Prop.~7.9]{Mur22}, one shows that one can choose an appropriate $y$
    such that $\Sq[u+\class{\EulerDer[d]}\cdot y]=0$ (in the argument, one must
    replace \cite[Props.~3.6, 3.9 and 3.10]{Mur22} by
    \Cref{prop:HH_computations} and \cite[Lemma~6.8]{Mur22} by
    \Cref{prop:key_edege_units}). The proof that such a pre-image $u$ is unique
    is completely analogous to that of \cite[Prop.~7.10]{Mur22}, replacing
    \cite[Prop.~3.6]{Mur22} by the last equation in \Cref{prop:HH_computations}
    and \cite[Lemma~6.8]{Mur22} by \Cref{prop:key_edege_units}.
  \item Since the graded algebra $\gLambda$ is $d$-sparse, tautologically, a
    minimal $A_{d+2}$-algebra structure $x_d\in X_d$ is simply a Hochschild
    cochain
    \[
      m_{d+2}\in\HC[d+2][-d]{\gLambda}[\gLambda],
    \]
    which we choose to be a representing cocycle of the unique class
    \begin{center}
      $u\in\HH[d+2][-d]{\gLambda}[\gLambda]\quad\text{with}\quad\Sq(u)=0\quad\text{and}\quad
      j^*(u)=x$
    \end{center}
    (compare with \Cref{prop:universal_Massey_welldef}). The fact that $m_{d+2}$
    is a cocycle implies that it extends to an $A_{d+3}$-algebra structure on
    $\gLambda$ choosing $m_{d+3}$ arbitrarily, e.g.~$m_{d+3}=0$ (this is
    actually the only possible choice if $d>1$). This is immediate from the
    definition of an $A_\infty$-algebra, see \cite[Sec.~3.1]{Kel01}. Thus, there
    exists a point $x_{d+1}\in X_{d+1}$ that maps to $x_d$ under the fibration
    $X_{d+1}\twoheadrightarrow X_d$. Moreover, the same argument shows that we
    can extend the $A_{d+3}$-algebra structure to an $A_{2d+2}$-algebra
    structure, necessarily setting $m_{i+2}=0$ for $d<i<2d$ and choosing
    $m_{2d+2}$ arbitrarily, for example~$m_{2d+2}=0$. Hence we have $x_{2d}\in X_{2d}$
    that maps to $x_d$ under the fibration $X_{2d}\twoheadrightarrow X_d$. In
    other words, the Hochschild cocycle $m_{d+2}$ with $\class{m_{d+2}}=u$ gives
    rise to a $d$-sparse minimal $A_{d+2}$-algebra $(\gLambda,m_{d+2})$ which
    can be extended to a $d$-sparse minimal $A_{2d+2}$-algebra
    $(\gLambda,m_{d+2},0)$, with all other operations being trivial due to the
    fact that $\gLambda$ is $d$-sparse. Notice, however, that
    $(\gLambda,m_{d+2},0,0,\ldots)$ \emph{need not} be an $A_\infty$-algebra
    since, although $\Sq(\class{m_{d+2}})=\Sq(u)=0$, the Gerstenhaber square of
    $m_{d+2}$ need not vanish at the cochain level.
    
  \item The condition $\Sq(u)=0$ is equivalent to the existence of an
    $A_{2d+3}$-algebra structure on $\gLambda$ with the same underlying
    $A_{d+2}$-algebra as above. In other words, $\Sq(u)=0$ if and only if there
    exists a point $y_{2d+1}\in X_{2d+1}$ such that $y_{2d+1}\mapsto x_d$. This
    can be shown exactly as in the proof of \cite[Prop.~6.7]{Mur20b}, since the
    relevant results \cite[Props.~3.4 and 5.4]{Mur20b} are not particular to the
    case $d=1$. Indeed, in the notation of \cite[Prop.~5.4]{Mur20b} we have
    $k=2d$ and $l=2d-1$, the recipient of the obstruction is
    $E^{k+1,k}_{k+1-l}=E^{2d+1,2d}_{2}=\HH[2d+3][-2d]{\gLambda[]}$, and the
    obstruction is $\Sq(u)$ by \cite[Prop.~3.4]{Mur20b} (here, only the last
    summand counts, since the other ones correspond to the elements $m_{i+2}=0$,
    $d<i<2d$, which vanish since $\gLambda$ is $d$-sparse, as remarked in the
    previous item). In other words, we have obtained a $d$-sparse minimal
    $A_{2d+3}$-algebra $(\gLambda,m_{d+2},m_{2d+2})$ with $\class{m_{d+2}}=u$
    and whose operation of top arity $2d+3$ must vanish since $\gLambda$ is
    $d$-sparse. 
  \end{itemize}
  
  Let $j\geq d$ and suppose that we have constructed points $x_i\in X_i$, $d\leq
  i\leq j$, as well as an auxiliary point $y_{j+d+1}\in X_{j+d+1}$ that map to
  one another under the bonding maps in the tower:
  \[
    \begin{tikzcd}[column sep=1em, row sep=tiny]
      \cdots\rar[twoheadrightarrow]&X_{j+d+2}\rar[twoheadrightarrow]&X_{j+d+1}\rar[twoheadrightarrow]&\cdots\rar[twoheadrightarrow]&X_{j+1}\rar[twoheadrightarrow]&X_j\rar[twoheadrightarrow]&\cdots\rar[twoheadrightarrow]&X_d\rar[twoheadrightarrow]&\cdots\rar[twoheadrightarrow]&X_0\\
      &\exists y_{j+d+2}\ar[mapsto,dotted,bend
      right]{rrrr}&y_{j+d+1}\ar[mapsto]{rrr}&&&x_j\rar[mapsto]&\cdots\rar[mapsto]&x_d\ar[mapsto]{rr}&&x_0
    \end{tikzcd}
  \]
  (the base of the induction $j=d$ was discussed above). Since
  $\BK[d+2][j+d+2][j+d+1]=0$ by \Cref{prop:BK_vanishing}, the obstruction for
  the existence of a point
  \[
    y_{j+d+2}\in X_{j+d+2},\qquad y_{j+d+2}\longmapsto x_j,
  \]
  vanishes; therefore such a $y_{j+d+2}$ exists and we define $x_{j+1}\in
  X_{j+1}$ as its image. Here we use \cite[Prop.~5.4]{Mur20b} for $k=j+d+1$ and
  $l=j$. This process results in a sequence $\set{x_i\in X_i}[i\geq d]$ that
  defines a point $x_\infty\in X_\infty$ in the (homotopy) limit with
  $x_\infty\mapsto x_d$, that is a minimal $A_\infty$-algebra structure on the
  graded algebra $\gLambda$ whose restricted universal Massey product is
  $j^*\class{m_{d+2}}=x$. This finishes the proof of statement
  \eqref{it:consists_of_units}.

  \emph{Proof of statement \eqref{it:equivmap_is_injective}:} Let
  $x_\infty$ and $y_\infty$ be $A_\infty$-algebra structures on $\gLambda$ with
  universal Massey products $\class{m_{d+2}^{x_{\infty}}}$ and
  $\class{m_{d+2}^{y_{\infty}}}$, respectively. We need to prove that, if there
  is an equality of restricted universal Massey products
  \[
    j^*\class{m_{d+2}^{y_{\infty}}}=j^*\class{m_{d+2}^{x_{\infty}}}\in\HH[d+2][-d]{\Lambda}[\gLambda],
  \]
  then $x_\infty$ and $y_{\infty}$ lie in the same connected component of
  $X_\infty$. The claim follows from \Cref{thm:secondary_formality} since we
  have
  \[
    \Sq(\class{m_{d+2}^{x_{\infty}}})=0=\Sq(\class{m_{d+2}^{y_\infty}})
  \]
  and therefore $\class{m_{d+2}^{y_\infty}}=\class{m_{d+2}^{x_\infty}}$ (as
  explained in the proof of statement \eqref{it:consists_of_units}). This
  finishes the proof of statement \eqref{it:equivmap_is_injective}.
\end{proof}

We are finally ready to prove \Cref{thm:A-Ainfty-version}, the last logical step
for establishing \Cref{thm:dZ-Auslander_correspondence}.

\begin{proof}[Proof of \Cref{thm:A-Ainfty-version}] The theorem is a direct
  consequence of \Cref{thm:final_thm}\eqref{it:is_singleton}. Indeed,
  \Cref{thm:A-Ainfty-version}\eqref{it:existence}, that is the existence of an
  $A_\infty$-algebra structure on the $d$-sparse graded algebra $\gLambda$ whose
  restricted universal Massey product is a unit in Hochschild--Tate cohomology,
  follows immediately from the fact that the quotient
  $\AS*{\Lambda}{\sigma}/\Aut*{\gLambda}$ is non-empty. Finally,
  \Cref{thm:A-Ainfty-version}\eqref{it:uniqueness}, that is the essential
  uniqueness of such $A_\infty$-algebra structures, is readily seen to be
  equivalent to the fact that the quotient
  $\AS*{\Lambda}{\sigma}/\Aut*{\gLambda}$ has exactly one element.
\end{proof}

\subsection{Strong differential graded enhancements}
\label{subsec:strong}

The following definition extends that of a `strong enhancement' of a
triangulated category (the case $d=1$) introduced by Lunts and Orlov
in~\cite{LO10} to the setting of $(d+2)$-angulated categories.

\begin{definition}
  \label{def:strong}
  Let $\F$ be a $(d+2)$-angulated category with split idempotents.
  \begin{enumerate}
  \item A \emph{strong (DG) enhancement of $\F$}
    is a pair $(\A,\varphi)$ consisting of a (necessarily Karoubian) pre-$(d+2)$-angulated DG category $\A$ and a 
    $\kk$-linear equivalence of $(d+2)$-angulated categories
    \[
      \varphi\colon\dgH[0]{\A}\stackrel{\sim}{\longrightarrow}\F.
    \]
  \item\label{it:strong_equivalence}  An \emph{equivalence} of strong enhancements $(\A,\varphi)$ and $(\B,\psi)$ of $\F$ is a quasi-equivalence of DG categories
    $F\colon\A\stackrel{\sim}{\to}\B$ such that the diagram
    \begin{center}
      \begin{tikzcd}[column sep=large,row sep=tiny]
        \dgH[0]{\A}\ar[swap]{dd}{\dgH[0]{F}}\drar{\varphi}\\
        &\F\\
        \dgH[0]{\B}\urar[swap]{\psi}
      \end{tikzcd}
    \end{center}
    commutes up to (unspecified) natural isomorphism. More generally, we can define an equivalence as an isomorphism $F\colon\A\cong\B$ in $\Ho{\dgcat}$ such that the previous diagram commutes up to natural isomorphism. In this case, the induced equivalence $\dgH[0]{F}\colon\dgH[0]{\A}\simeq\dgH[0]{\B}$ is only well defined up to natural isomorphism, but this is not a problem.
    \item\label{it:strong_uniqueness} We say that $\F$ \emph{admits a unique strong (DG)
      enhancement}\footnote{In this case one also says that $\F$ `admits a strongly unique enhancement.'} if any two strong enhancements of $\F$ are equivalent.
  \end{enumerate}
\end{definition}

\begin{remark}
  \label{def:strong_triangulated}
  In the case $d=1$ of triangulated categories with split idempotents, we could have alternatively defined strong enhancements and equivalences in the following way, which yields the same set of equivalence classes, in particular the same notion of strong uniqueness. Compare \Cref{def:DG_enhancement} and \Cref{def:DG_enhancement_equivalent}.

  Let $\T$ be a triangulated category with split idempotents.
  \begin{enumerate}
  \item A \emph{strong (DG) enhancement of $\T$}
    is a pair $(\A,\varphi)$ consisting of a DG category $\A$ and a 
    $\kk$-linear equivalence of triangulated categories
    \[
      \varphi\colon\DerCat[c]{\A}\stackrel{\sim}{\longrightarrow}\T.
    \]
  \item\label{it:strong_equivalence_triangulated}  An \emph{equivalence} of strong enhancements $(\A,\varphi)$ and $(\B,\psi)$ of $\T$ is a Morita equivalence of DG categories
    $F\colon\A\stackrel{\sim}{\to}\B$ such that the diagram
    \begin{center}
      \begin{tikzcd}[column sep=large,row sep=tiny]
        \DerCat[c]{\A}\ar[swap]{dd}{\mathbb{L}F_!}\drar{\varphi}\\
        &\T\\
        \DerCat[c]{\B}\urar[swap]{\psi}
      \end{tikzcd}
    \end{center}
    commutes up to (unspecified) natural isomorphism. More generally, we can define an equivalence as an isomorphism $F\colon\A\cong\B$ in $\Hmo$ such that the previous diagram commutes up to natural isomorphism. In this case, the induced equivalence $\mathbb{L}F_!\colon\DerCat[c]{\A}\simeq\DerCat[c]{\B}$ is only well defined up to natural isomorphism, but this is not a problem.
    \item\label{it:strong_uniqueness_triangulated} We say that $\T$ \emph{admits a unique strong (DG)
      enhancement}\footnote{In this case one also says that $\T$ `admits a strongly unique enhancement.'} if any two strong enhancements of $\T$ are equivalent.
  \end{enumerate}

  An strong enhancement $(\A,\varphi)$ in the sense of \Cref{def:strong} for $d=1$ yields a strong enhancement in the sense of this remark, composing with the inverse of the equivalence $\dgH[0]{\Yoneda}\colon\dgH[0]{\A}\simeq\DerCat[c]{\A}$ induced by the DG Yoneda embedding, $(\A,\varphi\dgH[0]{\Yoneda}^{-1})$. Moreover, two strong enhancements $(\A,\varphi)$, $(\B,\psi)$ in the sense of
  \Cref{def:strong} are equivalent if and only if the corresponding strong enhancements $(\A,\varphi\dgH[0]{\Yoneda}^{-1})$, $(\B,\psi\dgH[0]{\Yoneda}^{-1})$ in this sense are equivalent, see \Cref{rem:Morita_fibrant_pretriangulated}. Furthermore, any enhancement $(\A,\varphi)$ in this sense is equivalent to an enhancement in the sense of \Cref{def:strong}, namely to $(\DerCat[c][dg]{\A},\varphi)$, via the DG Yoneda embedding $\Yoneda\colon\A\rightarrow\DerCat[c][dg]{\A}$, which is a Morita equivalence, see \Cref{subsubsec:Morita_equivalences}. Therefore, as claimed, both definitions yield the same set of equivalence classes of strong enhancements. In particular, $\T$ admits a unique strong enhancement in the sense of \Cref{def:DG_enhancement} if and only if it admits a unique strong enhancement in this sense.
\end{remark}

In order to investigate strong enhancements of $(d+2)$-angulated categories, we extend the notion to pairs.

\begin{definition}
  \label{def:strong_pairs}
  Let $(\C,\Sigma)$ be a pair consisting of an additive category with split idempotents and a
  self-equivalence, that we often omit from the notation. 
  \begin{enumerate}
  \item A \emph{strong $(d+2)$-angulated (DG) enhancement of $\C$}
    is a pair $(\A,\varphi)$ consisting of a pre-$(d+2)$-angulated DG category $\A$ and a 
    $\kk$-linear equivalence of pairs, in the sense of \Cref{def:enhanced_angulated_structure},
    \[
      \varphi\colon\dgH[0]{\A}\stackrel{\sim}{\longrightarrow}\C.
    \]
  \item\label{it:strong_equivalence_pairs}  An \emph{equivalence} of strong enhancements $(\A,\varphi)$ and $(\B,\psi)$ of $\C$ is
    a quasi-equivalence of DG categories
    $F\colon\A\stackrel{\sim}{\to}\B$ such that the diagram
    \begin{center}
      \begin{tikzcd}[column sep=large,row sep=tiny]
        \dgH[0]{\A}\ar[swap]{dd}{\dgH[0]{F}}\drar{\varphi}\\
        &\C\\
        \dgH[0]{\B}\urar[swap]{\psi}
      \end{tikzcd}
    \end{center}
    commutes up to (unspecified) natural isomorphism. More generally, we can define an equivalence as an isomorphism $F\colon\A\cong\B$ in $\Ho{\dgcat}$ such that the previous diagram commutes up to natural isomorphism. In this case, the induced equivalence $\dgH[0]{F}\colon\dgH[0]{\A}\simeq\dgH[0]{\B}$ is only well defined up to natural isomorphism, but this is not a problem.
  \end{enumerate}
\end{definition}

\begin{remark}
  \label{rmk:strong_pairs_triangulated}
  For $d=1$, \Cref{def:strong_pairs} also admits an equivalent formulation in
  terms of perfect derived categories, compare
  \Cref{def:strong,def:strong_triangulated}.
\end{remark}

\begin{remark}
  A strong $(d+2)$-angulated enhancement of $(\C,\Sigma)$ induces a
  $(d+2)$-angulated structure and two equivalent enhancements induce
  \emph{identical} $(d+2)$-angulated structures, not just equivalent ones.
  Therefore, the set of equivalence classes of strong enhancements of a
  $(d+2)$-angulated category $(\C,\Sigma,\pentagon)$ is a subset of the set of
  equivalence classes of strong $(d+2)$-angulated enhancements of $(\C,\Sigma)$. Moreover, the latter can be partitioned by subsets of the former form as $\pentagon$ runs over the set of $(d+2)$-angulated structures on the pair $(\C,\Sigma)$.
\end{remark}

Recall the standing assumptions of this section,
\Cref{setting:existence_and_uniqueness}. We endow $\proj*{\Lambda}$ with any of
the equivalent algebraic, that is AL $(d+2)$-angulated structures $\pentagon$ with suspension
functor $\Sigma=-\otimes_{\Lambda}\twBim[\sigma]{\Lambda}$, which we fix, see
\Cref{def:AL_angulation}, \Cref{prop:Amiot-Lin_independence_of_res} and \Cref{rmk:algebraic=AL}. Our aim
is to compute the set of equivalence classes of strong enhancements of the
$(d+2)$-angulated category $(\proj*{\Lambda},\Sigma,\pentagon)$. For this, we
need some auxiliary definitions.

\begin{definition}
  The \emph{stable centre} $\underline{Z}(\Lambda)$ of $\Lambda$ is the quotient
  of the centre $Z(\Lambda)$ by the ideal formed by the elements $a\in
  Z(\Lambda)$ such that the $\Lambda$-bimodule morphism
  \[
    \Lambda\to\Lambda,\quad x\mapsto xa,
  \]
  factors through a projective-injective $\Lambda$-bimodule.
\end{definition}

\begin{remark}
  In the same way that the centre of $\Lambda$ can be interpreted in terms of
  Hochschild cohomology as $Z(\Lambda)=\HH*[0]{\Lambda}$, the stable centre is
  interpreted in terms of Hochschild--Tate cohomology as
  \[
    \underline{Z}(\Lambda)=\TateHH*[0]{\Lambda}=\sHom[\Lambda^e]{\Lambda}{\Lambda}.
  \]
\end{remark}

We recall the following classical definition.

\begin{definition}
  \label{def:centers}
  The \emph{centre} $Z(\D)$ of a $\kk$-linear category $\D$ is the endomorphism
  algebra of the identity functor $\id[\D]$.
\end{definition}

\begin{remark}
  There is a commutative square of algebra maps
  \begin{center}
    \begin{tikzcd}
      Z(\Lambda)\arrow{r}{\cong}\arrow[two heads]{d}[swap]{\rho}&Z(\mmod{\Lambda})\arrow{d}\\
      \underline{Z}(\Lambda)\arrow{r}{\zeta}&Z(\smmod{\Lambda})
    \end{tikzcd}
  \end{center}
  where the top arrow is the well-known isomorphism that sends $a\in Z(\Lambda)$
  to the natural transformation with components
  \[
    M\stackrel{}{\longrightarrow} M,\quad x\mapsto xa,\qquad
    M\in\mmod*{\Lambda},
  \]
  the left vertical arrow is the quotient map, and the right vertical arrow is
  induced by the ideal quotient
  \[
    \mmod*{\Lambda}\twoheadrightarrow\smmod*{\Lambda}=\mmod*{\Lambda}/[\proj*{\Lambda}].
  \]
  We denote by
  \[
    \zeta^{\times}\colon\underline{Z}(\Lambda)^\times\longrightarrow
    Z(\smmod*{\Lambda})^\times
  \]
  the morphism induced by $\zeta$ between the corresponding groups of units.
\end{remark}

\begin{theorem}
  \label{thm:strong_enhancements}
  Suppose that the algebra $\Lambda/J_\Lambda$ is separable and endow the pair $(\proj*{\Lambda},\Sigma)$ with an algebraic $(d+2)$-angulation $\pentagon$. The following
  statements hold:
  \begin{enumerate}
  \item\label{it:uZtimes} The set of equivalence classes of strong $(d+2)$-angulated enhancements
    of the pair $(\proj*{\Lambda},\Sigma)$ is in bijection with the set
    $\underline{Z}(\Lambda)^\times$ of units in the stable centre of $\Lambda$.
  \item\label{it:kerZetatimes} The subset of equivalence classes of strong enhancements of the
    $(d+2)$-angulated category $(\proj*{\Lambda},\Sigma,\pentagon)$ is in
    bijection with $\ker\zeta^\times$.
  \end{enumerate}
\end{theorem}

\begin{proof}
  Let $(\A,\varphi)$ be a strong enhancement of $(\proj*{\Lambda},\Sigma,\pentagon)$, which is in particular a strong $(d+2)$-angulated enhancement of
  $(\proj*{\Lambda},\Sigma)$. If $(\B,\psi)$ is another strong $(d+2)$-angulated enhancement of
  $(\proj*{\Lambda},\Sigma)$, then $\A$ and $\B$ are quasi-equivalent, were we use
  \Cref{prop:Amiot-Lin_independence_of_res},
  \Cref{thm:projLambda_existence_and_uniqueness} and the separability
  hypothesis. Therefore, without loss of generality, we may replace $\B$ by $\A$
  and consider only strong enhancements of the form $(\A,\psi)$.
  
  A strong $(d+2)$-angulated enhancement of the form $(\A,\psi)$ is equivalent
  to $(\A,\xi\varphi)$ for $\xi=\psi\varphi^{-1}$, where $\varphi^{-1}$ is a
  choice of quasi-inverse of the equivalence of $(d+2)$-angulated categories
  $\varphi$; indeed, the diagram
  \begin{center}
    \begin{tikzcd}[column sep=large,row sep=tiny]
      \dgH[0]{\A}\ar[swap]{dd}{\id[\dgH[0]{\A}]=\dgH[0]{\id[\A]}}\drar{\xi\varphi}\\
      &\C\\
      \dgH[0]{\A}\urar[swap]{\psi}
    \end{tikzcd}
  \end{center}
  commutes up to natural isomorphism. Moreover, given two self-equivalences
  $\xi_1,\ \xi_2$ of $\proj*{\Lambda}$, the strong $(d+2)$-angulated
  enhancements $(\A,\xi_1\varphi)$ and $(\A,\xi_2\varphi)$ are equivalent if and
  only if there exists an automorphism $F$ of $\A$ in $\Ho{\dgcat}$ such that
  the following square commutes up to natural isomorphism
  \begin{center}
    \begin{tikzcd}
      \dgH[0]{\A}\arrow{r}{\varphi}\arrow{d}[swap]{\dgH[0]{F}}&\proj*{\Lambda}\arrow{d}{\xi_2^{-1}\xi_1}\\
      \dgH[0]{\A}\arrow{r}{\varphi}&\proj*{\Lambda}
    \end{tikzcd}
  \end{center}
  In this context, talking about the self-equivalence of $\dgH[0]{\A}$ induced by an automorphism $F$ of $\A$ in $\Ho{\dgcat}$ makes sense because any such automorphism is represented by a zig-zag of quasi-equivalences
  \begin{equation*}
    \A=\A^{(0)}\stackrel{\sim}{\longleftarrow}\A^{(1)}\stackrel{\sim}{\longrightarrow}\cdots\stackrel{\sim}{\longleftarrow}\A^{(n-1)}\stackrel{\sim}{\longrightarrow}\A^{(n)}=\A.
  \end{equation*}
  Here, taking cohomology we obtain a zig-zag of equivalences of pairs
     \begin{equation*}
    \dgH[0]{\A}= \dgH[0]{\A^{(0)}}\stackrel{\sim}{\longleftarrow} \dgH[0]{\A^{(1)}}\stackrel{\sim}{\longrightarrow}\cdots\stackrel{\sim}{\longleftarrow} \dgH[0]{\A^{(n-1)}}\stackrel{\sim}{\longrightarrow} \dgH[0]{\A^{(n)}}= \dgH[0]{\A}.
  \end{equation*}
  We can then take quasi-inverses of arrows pointing backwards
       \begin{equation*}
    \dgH[0]{\A}= \dgH[0]{\A^{(0)}}\stackrel{\sim}{\longrightarrow} \dgH[0]{\A^{(1)}}\stackrel{\sim}{\longrightarrow}\cdots\stackrel{\sim}{\longrightarrow} \dgH[0]{\A^{(n-1)}}\stackrel{\sim}{\longrightarrow} \dgH[0]{\A^{(n)}}= \dgH[0]{\A}
  \end{equation*}
  and the composite is well defined up to natural isomorphism.
  We conclude that the set of equivalence classes of strong $(d+2)$-angulated
  enhancements of $(\proj*{\Lambda},\Sigma)$ is in bijection with the set of left
  cosets of the group of natural isomorphism classes of self-equivalences of
  $(\proj*{\Lambda},\Sigma)$ modulo the subgroup defined by the
  self-equivalences induced by an automorphism of $\A$ in $\Ho{\dgcat}$, see
  \Cref{subsubsec:hocat}.

  Recall that the group of self-equivalences of the category $\proj*{\Lambda}$
  modulo natural isomorphisms---the Picard group of $\Lambda$---identifies with
  the outer automorphism group $\Out*{\Lambda}$, see
  \Cref{subsec:AL_angulations}. Similarly, the group of self-equivalences of the
  pair $(\proj*{\Lambda},\Sigma)$ modulo natural isomorphisms identifies with
  the outer automorphism group $\Out*{\gLambda}$ of the graded algebra
  $\gLambda$. We now investigate the self-equivalences of $\A$. We use
  homotopical techniques that have also been used in~\cite[Sec.~8]{Mur22} with
  similar purposes.

  Let $\B\subset\A$ be the full DG subcategory spanned by the objects which are
  basic additive generators of $\dgH[0]{\A}\simeq\proj*{\Lambda}$. Denote by
  $\dgcat_\A,\dgcat_\B\subset\dgcat$ the non-full (!) subcategories whose objects are the DG categories which are quasi-isomorphic to $\A$ and $\B$, respectively, and whose morphisms are quasi-equivalences. 

  The following argument is based in the proof of~\cite[Prop.~8.5]{Mur22}.
  Consider the functors
  \begin{center}
    \begin{tikzcd}
      \dgcat_\B\arrow[shift left = .5ex]{r}{R}&\dgcat_\A\arrow[shift left = .5ex]{l}{j}
    \end{tikzcd}
  \end{center}
  defined as follows:
  \begin{itemize}
  \item The functor $R$ is the homotopy Karoubi envelope (defined for example in
    \cite[Sec.~4]{GHW21}), which is an endofunctor of $\dgcat$.
  \item The functor $j$ takes a DG category $\D$ quasi-equivalent to $\A$ to
    its full DG subcategory spanned by the objects that are basic additive
    generators in $\dgH[0]{\D}\simeq \dgH[0]{\A}\simeq\proj*{\Lambda}$.
  \end{itemize}
   For an arbitrary DG category $\D$, there is a natural functor $\D\to R(\D)$
   that is a quasi-equivalence if and only if $\dgH[0]{\D}$ is idempotent
   complete. If $\D$ is quasi-equivalent to $\B$, then the previous canonical
   functor corestricts to a natural quasi-equivalence $\D\stackrel{\sim}{\to} jR(\D)$.
   If $\D$ is quasi-equivalent to $\A$, then there is a zig-zag of natural quasi-equivalences
   \[
     Rj(\D)\stackrel{\sim}{\longrightarrow}R(\D)\stackrel{\sim}{\longleftarrow}\D.
   \]
  This proves that $R$ and $j$ induce inverse homotopy equivalences between the
  classifying spaces of the categories $\dgcat_\A$ and $\dgcat_\B$,
  \begin{center}
    \begin{tikzcd}
      {|\dgcat_\B|}\arrow[shift left = .5ex]{r}{|R|}&{|\dgcat_\A|.}\arrow[shift left = .5ex]{l}{|j|}
    \end{tikzcd}
  \end{center}
  Taking fundamental groups (these spaces are connected by construction) we obtain an
  isomorphism between the automorphism groups of $\A$ and $\B$ in $\Ho{\dgcat}$.
  
  Let $c\in\A$ be an object such that $\varphi(c)\cong\Lambda$ in
  $\proj*{\Lambda}$; in particular, $c\in\B$. Since all objects in
  $\dgH[0]{\B}$ are isomorphic, the homotopy fiber sequence in the proof of
  \cite[Prop.~2.3.3.5]{TV08} gives an isomorphism between the automorphism group
  of $\B$ in $\Ho{\dgcat}$ and the automorphism group of $\B(c,c)$ in
  $\Ho{\dgalg}$, the homotopy category of the model category of DG algebras with
  quasi-isomorphisms as weak equivalences and surjections as fibrations.
  
  A choice of isomorphism $\varphi(c)\cong\Lambda$ in $\proj*{\Lambda}$ induces
  an isomorphism of graded algebras $\dgH{\B(c,c)}\cong\gLambda$. Fix
  a minimal model
  \[
    (\gLambda,m_{d+2},m_{2d+2},\dots)
  \]
  of the DG algebra $\B(c,c)$. The automorphism group of $\B(c,c)$ in
  $\Ho{\dgalg}$ is isomorphic to the group of homotopy classes of
  $A_\infty$-automorphisms of this minimal model (an $A_\infty$-isomorphism of
  minimal $A_\infty$-algebras is an $A_\infty$-morphism $F$ such that $F_1$ is
  an isomorphism); this follows from~\cite[Cor.~1.3.1.3]{Lef03} (despite the
  fact that only non-unital DG algebras are considered in \cite{Lef03}, the
  claim is still valid by \cite[Prop.~6.2]{Mur14}).

  An $A_\infty$-automorphism $F$ of $(\gLambda,m_{d+2},m_{2d+2},\dots)$ is the same as an $A_\infty$-isomorphism
  \begin{align*}
    G\colon (\gLambda,m_{d+2},m_{2d+2},\dots) \longrightarrow &(\gLambda,m_{d+2},m_{2d+2},\dots)*F_1\\
    &=(\gLambda,m_{d+2}^{F_1},m_{2d+2}^{F_1},\dots)
  \end{align*}
  with $G_1=\id$ and $G_n=F_1^{-1}F_n$, see \Cref{not:g_action} and compare with
  the proof of \Cref{thm:projLambda_existence_and_uniqueness}.  
  By \Cref{thm:final_thm}\eqref{it:equivmap_is_injective}, such an
  $A_\infty$-isomorphism exists if and only if there is an equality of
  restricted universal Massey products
  \begin{equation*}
    j^*\{m_{d+2}\}=j^*\{m_{d+2}^{F_1}\}\in\HH[d+2][-d]{\Lambda}[\gLambda].
  \end{equation*}
  Moreover, using the right conjugation action of $\Aut*{\gLambda}$ on $\HH[d+2][-d]{\Lambda}[\gLambda]$ that appeared in \Cref{rmk:action}, 
  \[j^*\{m_{d+2}^{F_1}\}=j^*\{m_{d+2}\}^{F_1}.\]
  Hence, if $G\subset \Aut*{\gLambda}$ is now the isotropy group of this action at
  $j^*\{m_{d+2}\}$ and $H\subset \Out*{\gLambda}$ is the image of $G$ along the
  quotient map
  \[
    \Aut*{\gLambda}\twoheadrightarrow\Out*{\gLambda},
  \]
  the set of cosets we
  wish to describe is $\Out*{\gLambda}/H$.

  The right conjugation action of $\Aut*{\gLambda}$ on $\HH{\gLambda}[\gLambda]$
  factors through $\Out*{\gLambda}$ (\Cref{prop:inner_auto}).
  Since $j^*\{m_{d+2}\}$ is in the image of the $\Aut*{\gLambda}$-equivariant
  morphism
  \[
    j^*\colon\HH[d+2][-d]{\gLambda}[\gLambda]\longrightarrow\HH[d+2][-d]{\Lambda}[\gLambda],
  \]
  we deduce that $G$ contains all inner automorphisms of $\gLambda$ and, therefore, $\Out*{\gLambda}/H$ is in bijection with $\Aut*{\gLambda}/G$. 

  By \Cref{prop:single_orbit}, $\Aut*{\gLambda}/G$ is in
  bijection with the edge units of $\HH[d+2][-d]{\Lambda}[\gLambda]$, that is
  with the Hochschild cohomology classes in bidegree $(d+2,-d)$ that become
  units in the Hochschild--Tate cohomology $\TateHH{\Lambda}[\gLambda]$. These
  clases are in (non-canonical) bijection with the units in
  \[
    \TateHH[0][0]{\Lambda}[\gLambda]=\underline{Z}(\Lambda),
  \]
  since the latter
  group of units acts freely and transitively on the set of edge units of
  bidegree $(d+2,-d)$. An explicit bijection between $\Aut*{\gLambda}/G$ and
  $\underline{Z}(\Lambda)^\times$ is induced by the principal crossed
  homomorphism
  \begin{align*}
    \chi\colon\Aut*{\gLambda}&\longrightarrow\underline{Z}(\Lambda)^\times,\\
    g&\longmapsto j^*\{m_{d+2}\}^{g}j^*\{m_{d+2}\}^{-1}
  \end{align*}
  whose kernel is $\ker\chi=G$. 
  This proves statement \eqref{it:uZtimes}.

  To prove statement \eqref{it:kerZetatimes} we have to identify the image along $\chi$ of the subgroup $K\subset \Aut*{\gLambda}$ corresponding to the self-equivalences of $(\proj*{\Lambda},\Sigma)$ which preserve the $(d+2)$-angulated structure $\pentagon$. 
  
  The subgroup $K$ contains $G$ since the latter is the subgroup of
  $\Aut*{\gLambda}$ formed by the automorphisms of $\gLambda$ induced by an
  automorphism of the pre-$(d+2)$-angulated DG category $\A$ in $\Ho{\dgcat}$.
  Since $\A$ is an enhancement of $(\proj*{\Lambda},\Sigma,\pentagon)$, the
  self-equivalences of the pair $\dgH[0]{\A}\simeq(\proj*{\Lambda},\Sigma)$
  induced by the automorphisms of $\A$ also preserve the $(d+2)$-angulated
  structure $\pentagon$, that is $G\subset K$. Since $G=\ker\chi$ we deduce that $K=\chi^{-1}(\chi(K))$ (this is well known for group homomorphisms and easy to deduce for crossed homomorphisms).

  Recall that the automorphisms of $\gLambda$ introduced in \Cref{rmk:action} define a group homomorphism
  \begin{align*}
    \gamma\colon Z(\Lambda)^\times&\longrightarrow\Aut*{\gLambda},\\
    u&\longmapsto g_u.
  \end{align*}
  The composite $\chi\gamma$ is the surjective group homomorphism
  $Z(\Lambda)^\times\twoheadrightarrow\underline{Z}(\Lambda)^\times$ induced by
  the natural projection $Z(\Lambda)\twoheadrightarrow\underline{Z}(\Lambda)$
  (the former homomorphism is surjective by~\cite[Cor.~2.3]{Che21}). Since $K=\chi^{-1}(\chi(K))$ and $\chi\gamma$ is surjective,
  \[\chi(K\cap\gamma(Z(\Lambda)^\times))=\chi(K)\cap\chi \gamma(Z(\Lambda)^\times)=\chi(K).\]
  Hence,
  $\chi(K)\subset \underline{Z}(\Lambda)^\times$ consists of the classes of
  those units $u\in Z(\Lambda)^\times$ such that $g_u$ preserves the
  $(d+2)$-angulation $\pentagon$.  The self-equivalence of
  $(\proj*{\Lambda},\Sigma)$ induced by $g_u$, $u\in Z(\Lambda)^\times$, is
  $(\id[\proj*{\Lambda}],u)$ (considered already in the proof of \Cref{prop:Amiot-Lin_independence_of_res}). Therefore, in order to finish the proof we have to check that $(\id[\proj*{\Lambda}],u)$ preserves $\pentagon$ if and only if $[u]\in\ker\zeta^\times$.

  First, assume $[u]\in\ker\zeta^\times$. We have to show that, for every $(d+2)$-angle
  \[
    x_1\stackrel{f_1}{\longrightarrow} x_2\stackrel{f_2}{\longrightarrow}x_3\to\cdots\to x_{d+2}\stackrel{f_{d+2}}{\longrightarrow}\Sigma(x_1)
  \]
  in $\pentagon$,
  the sequence
  \[
    x_1\stackrel{f_1}{\longrightarrow} x_2\stackrel{f_2}{\longrightarrow}x_3\to\cdots\to x_{d+2}\stackrel{uf_{d+2}}{\longrightarrow}\Sigma(x_1)
  \]
  is also in $\pentagon$; notice that we can in fact place the multiplication
  by $u$ in any given arrow, since the resulting sequences are isomorphic. We
  therefore consider the sequence
  \[
    x_1\stackrel{uf_1}{\longrightarrow} x_2\stackrel{f_2}{\longrightarrow}x_3\to\cdots\to x_{d+2}\stackrel{f_{d+2}}{\longrightarrow}\Sigma(x_1)
  \]
  and show that it is isomorphic to the first diagram. This will suffice by \eqref{dTR1}. We factor $f_1$ in $\mmod*{\Lambda}$ as 
  \begin{center}
    \begin{tikzcd}
      x_1\arrow[two heads]{r}{p}&M\arrow[hook]{r}{i}& x_2
    \end{tikzcd}
  \end{center}
 and consider the following commutative square of solid arrows:
  \begin{center}
    \begin{tikzcd}
      M\arrow[hook]{r}{i}\arrow{d}{u}& x_2\arrow{d}{u}\arrow[dashed]{ld}{\alpha}\\
      M\arrow[hook]{r}{i}& x_2
    \end{tikzcd}
  \end{center}
  Since $[u]\in\ker\zeta^\times$, multiplication by $u$ induces the identity in the stable module category $\smmod*{\Lambda}$ for evey finite-dimensional $\Lambda$-module. This implies the existence of a dashed arrow $\alpha$ such that
  \[
    \alpha i=u-\id[M].
  \]
  An easy computation shows that the following diagram commutes:
  \begin{center}
    \begin{tikzcd}
      x_{1}\dar[equal]\rar{f_1}&x_2\rar{f_2}\dar{\id[x_2]+i\alpha}&x_{3}\rar\dar[equal]&\cdots\rar&x_{d+1}\rar{f_{d+1}}\dar[equal]&x_{d+2}\rar{f_{d+2}}\dar[equal]&\Sigma(x_{1})\dar[equal]\\
      x_{1}\rar{uf_1}&x_2\rar{f_2}&x_{3}\rar&\cdots\rar&x_{d+1}\rar{f_{d+1}}&x_{d+2}\rar{f_{d+2}}&\Sigma(x_{1}).
    \end{tikzcd}
  \end{center}
  The arrow $\id[x_2]+i\alpha$ is an isomorphism by the five lemma, which is
  valid in the setting of $(d+2)$-angulated categories by \eqref{dTR2} and~\cite[Prop.~2.5(a)]{GKO13}.

  Assume now that $(\id[\proj*{\Lambda}],u)$ preserves the $(d+2)$-angulation
  $\pentagon$. Let $M$ be a finite-dimensional $\Lambda$-module. Recall that $\Sigma\coloneqq-\otimes_{\Lambda}\twBim[\sigma]{\Lambda}$. There exists a $(d+2)$-angle in $\pentagon$
  \[
    x_1\stackrel{f_1}{\longrightarrow} x_2\stackrel{f_2}{\longrightarrow}x_3\to\cdots\to x_{d+2}\stackrel{f_{d+2}}{\longrightarrow}\Sigma(x_1)
  \]
  such that $f_{d+2}$ factors as 
  \begin{center}
    \begin{tikzcd}
      x_{d+2}\arrow[two heads]{r}{p}&\Sigma M\arrow[hook]{r}{\Sigma (i)}& \Sigma(x_1).
    \end{tikzcd}
  \end{center}
  in $\mmod*{\Lambda}$.
  Indeed, choose an epimorphism $p\colon x_{d+2}\twoheadrightarrow \Sigma M$ with projective source and a monomorphism $i\colon M\hookrightarrow x_1$ with injective target; define $f_{d+2}=\Sigma(i)p$ and then apply \eqref{dTR1} and \eqref{dTR2}. 
  The previous factorisation and $(d+2)$-angle induce an isomorphism in $\smmod*{\Lambda}$,
  \[\omega\colon\Sigma M\stackrel{\sim}{\longrightarrow} \Omega^{-(d+2)}M.\]  
  Since $(\id[\proj*{\Lambda}],u)$ preserves $\pentagon$, the following diagram is also in $\pentagon$
  \[
    x_1\stackrel{f_1}{\longrightarrow} x_2\stackrel{f_2}{\longrightarrow}x_3\to\cdots\to x_{d+2}\stackrel{uf_{d+2}}{\longrightarrow}\Sigma(x_1).
  \]
  We can factor $uf_{d+2}$ as 
  \begin{center}
    \begin{tikzcd}
      x_{d+2}\arrow[two heads]{r}{up}&\Sigma M\arrow[hook]{r}{\Sigma (i)}& \Sigma(x_1).
    \end{tikzcd}
  \end{center}
  The natural isomorphism in $\smmod*{\Lambda}$ induced by this factorisation and the previous diagram is
  \[u\omega\colon\Sigma M\stackrel{\sim}{\longrightarrow} \Omega^{-(d+2)}M.\]
  By \Cref{prop:GKO-Freyd-Heller}\eqref{it:Sigma-d+2-syzygy} both isomorphisms in $\smmod*{\Lambda}$
  must coincide:
  \[\omega=u\omega.\]
  Therefore, the isomorphism $u\colon M\stackrel{\sim}{\to} M$ in
  $\smmod*{\Lambda}$ represented by multiplication by $u$ must be the identity
  for every $M\in\mmod*{\Lambda}$. This shows that $[u]\in\ker\zeta^\times$, as
  required. This finishes the proof of the theorem.
\end{proof}

\begin{remark}
  It is an interesting problem to study the set of equivalence classes of strong
  enhancements of an algebraic triangulated category $\T$ with
  finite-dimensional morphism spaces, split idempotents, and a $d\ZZ$-cluster
  tilting object $c\in\T$. \Cref{thm:strong_enhancements} provides a complete
  answer in the case $d=1$ when $\add*{c}=\T$. In general, there is an apparent
  map (given by restriction) from the set of equivalence classes of strong
  enhancements of the ambient triangulated category $\T$ to the set of
  equivalence classes of strong enhancements of the standard $(d+2)$-angulated
  category $(\add*{c},[d],\pentagon)$. We do not know, however, whether this map
  is injective or
  surjective.
\end{remark}

The following result provides, in particular, the simplest example of an
algebraic triangulated category with a unique enhancement but with non-unique
strong enhancements.

\begin{corollary}
  \label{coro:non_strong_examples}
  Let $\Lambda=\kk[\varepsilon]/(\varepsilon^2)$ be the algebra of dual numbers. The following statements hold:
  \begin{enumerate}
    \item\label{it:dual_num_is_twisted} The algebra $\Lambda$ is connected and twisted $(d+2)$-periodic for every $d\geq -1$. Additionally, $\Lambda/J_\Lambda$ is separable. Therefore, $\proj*{\Lambda}$ is algebraic $(d+2)$-angulated for every $d\geq 1$, and there is only one choice of suspension functor.
  \item\label{it:dual_num_char0-strongly_unique} If $\cchar{\kk}\neq 2$ and $d\geq1$, then every algebraic $(d+2)$-angulated structure on
    $\proj*{\Lambda}$ admits a unique strong enhancement.
  \item\label{it:dual_num_char2-non-strongly_unique} If $\cchar{\kk}=2$ and $d\geq1$, then for every algebraic $(d+2)$-angulated structure on
    $\proj*{\Lambda}$ the set of equivalence classes of strong
    enhancements is in bijection with the elements of the ground field $\kk$,
    and therefore $\proj*{\Lambda}$ does not admit a unique strong enhancement.
  \end{enumerate}
\end{corollary}
\begin{proof}
  \eqref{it:dual_num_is_twisted} The algebra $\Lambda$ is obviously connected and $\Lambda/J_\Lambda=\Lambda/(\varepsilon)=\kk$ is separable. The algebra $\Lambda$  is also twisted $1$-periodic, as witnessed by the short exact sequence
  of $\Lambda$-bimodules
  \[
    0\longrightarrow \twBim{\Lambda}[\sigma]\stackrel{\nu}{\longrightarrow}\Lambda\otimes\Lambda\stackrel{\mu}{\longrightarrow} \Lambda\longrightarrow 0,
  \]
  where $\mu$ is the multiplication map, the automorphism $\sigma$ is given by
  $\sigma(\varepsilon)=-\varepsilon$, and
  \[
    \nu(1)=1\otimes\varepsilon-\varepsilon\otimes 1.
  \]
  Therefore, $\Lambda$ is twisted $(d+2)$-periodic for every $d\geq -1$. By
  \Cref{thm:d-Hanihara}, the category $\proj*{\Lambda}$ is $(d+2)$-angulated for any $d\geq1$. Indeed, by \Cref{thm:Amiot-Lin}, \Cref{prop:unique_suspension} and \Cref{coro:pre-d+2-ang=unit,coro:atmostone} it has a unique AL $(d+2)$-angulated structure, up to equivalence, and the only choice of suspension functor is $\Sigma=-\otimes_\Lambda\twBim[\sigma^d]{\Lambda}$. Algebraic and AL $(d+2)$-angulated structures are the same thing (\Cref{rmk:algebraic=AL}). Now, the first claim follows. 
  
  \eqref{it:dual_num_char0-strongly_unique} and
  \eqref{it:dual_num_char2-non-strongly_unique} 
  In view of \Cref{thm:strong_enhancements}, we need a set-theoretic description
  of the kernel of the canonical morphism
  \[
    \zeta^\times\colon\underline{Z}(\Lambda)^\times\longrightarrow Z(\smmod*{\Lambda})^\times.
  \]
  Clearly $Z(\Lambda)=\Lambda$ and the stable category
  \[
    \smmod*{\Lambda}\simeq\mmod*{\kk}
  \]
  has centre isomorphic to the ground field $\kk$. Notice that all
  $\Lambda$-bimodule morphisms $\Lambda\to\Lambda\otimes\Lambda$ are of the form
  \[
    1\mapsto \alpha(1\otimes\varepsilon+\varepsilon\otimes
    1)+\beta(\varepsilon\otimes\varepsilon),\qquad \alpha,\ \beta\in\kk.
  \]
  The composition of such a $\Lambda$-bimodule morphism with $\mu$ satisfies
  \[
    1\mapsto 2\alpha\varepsilon.
  \]
  Since $\mu$ is a surjection from a projective $\Lambda$-bimodule, the stable
  centre is the quotient of $Z(\Lambda)$ by the ideal $(2\varepsilon)$. Hence,
  if $\cchar{\kk}\neq 2$, then $\zeta$ is an isomorphism. On the other hand, if
  $\cchar{\kk}=2$ then $\zeta\colon\Lambda\twoheadrightarrow\kk$ is the natural
  projection and
  \[
    \ker\zeta^\times=1+(\varepsilon)=\set{1+\alpha\varepsilon}[\alpha\in\kk].
  \]
  The two last claims now follow from \Cref{thm:strong_enhancements}.
\end{proof}

\begin{example}
  Let $\cchar{\kk}=2$ and $\Lambda=\kk[\varepsilon]/(\varepsilon^2)$ the algebra
  of dual numbers. Following the proof of \Cref{thm:strong_enhancements} and \Cref{coro:non_strong_examples}, we give an explicit description of the strong enhancements of
  the triangulated category $\proj*{\Lambda}$ (for any triangulated structure, all of them are algebraic).

  Since $\cchar{\kk}=2$
  the suspension functor must be $\Sigma=\id[\proj*{\Lambda}]$ (see the
  proof of \Cref{coro:non_strong_examples}). Moreover, the
  stable category $\smmod*{\Lambda}$ is equivalent to $\mmod*{\kk}$ and
  $\Omega=\Sigma=\id[\mmod*{\kk}]$. \Cref{prop:GKO-Freyd-Heller} establishes a
  bijection between the triangulated structures on $(\proj*{\Lambda},\Sigma)$
  and the set $\kk^\times$ of units of the ground field. It is straightforward
  to verify that the triangulated structure $\pentagon_\alpha$ corresponding to
  $\alpha\in\kk^\times$ is characterised by the fact that the triangle
  \[
    \Lambda\stackrel{\varepsilon}{\longrightarrow}\Lambda\stackrel{\varepsilon}{\longrightarrow}\Lambda\stackrel{\alpha\varepsilon}{\longrightarrow}\Lambda
  \]
  belongs to $\pentagon_\alpha$, compare with~\cite[Rmk.~8]{MSS07}.
  
  Consider the DG algebra with underlying graded algebra
  \[
    A=\frac{k[u,v^{\pm1}]\langle
      a\rangle}{(a^2,av+va,au+ua+1)},\qquad |u|=0,\quad |v|=1,\quad
    |a|=0,
  \]
  and differential
  \begin{align*}
    d(u)&=0,&
    d(v)&=0,&
    d(a)&=u^2v.
  \end{align*}
	It is proved 
  in~\cite[Rmk.~8]{MSS07} that there is an equivalence of pairs
  \begin{align*}
    \varphi\colon\DerCat[c]{A}&\stackrel{\sim}{\longrightarrow}\proj*{\Lambda},\\
    M&\longmapsto H^0(M),
  \end{align*}
  which is a triangulated equivalence if we endow the target with the triangulated structure $\pentagon_1$. For any $\alpha\in\kk^\times$, we have a triangulated equivalence
  \[(\id[\proj*{\Lambda}],\alpha)\colon(\proj*{\Lambda},\Sigma,\pentagon_1)\longrightarrow (\proj*{\Lambda},\Sigma,\pentagon_\alpha)\]
  of the kind considered in the proofs of \Cref{prop:Amiot-Lin_independence_of_res} and \Cref{coro:non_strong_examples}.  
  Given $\beta\in\kk$ we consider the self-equivalence $(\id[\proj*{\Lambda}],1+\beta\varepsilon)$ and define
  \[\varphi_{\alpha,\beta}= (\id[\proj*{\Lambda}],1+\beta\varepsilon) (\id[\proj*{\Lambda}],\alpha) \varphi = (\id[\proj*{\Lambda}],(1+\beta\varepsilon)\alpha) \varphi.\]
  The map
  \[
    \beta\longmapsto (A,\varphi_{\alpha,\beta})
  \]
  realises the bijection between $\kk$ and the set of equivalence classes of
  strong enhancements of $(\proj*{\Lambda},\Sigma,\pentagon_\alpha)$ given by
  \Cref{coro:non_strong_examples}. 
\end{example}

\subsection{Comments on the (non-)uniqueness of enhancements}
\label{sec:comms_non_uniqueness}

We wish to illustrate the necessity of the assumptions in
\Cref{thm:projLambda_existence_and_uniqueness,thm:uniqueness_triangulated,thm:strong_enhancements} by analysing the case when
$\Lambda=\kk$ is an arbitrary field, which is the most basic example of a
periodic algebra of any period $d\geq 1$. Thus, let $d\geq 1$ and $\sigma=\id_\Lambda$ so that
\[
  \gLambda=\kk\langle\imath^\pm\rangle,\qquad |\imath|=-d,
\]
is the algebra of Laurent polynomials in a single variable of degree $-d$, which
we view as a DG algebra with vanishing differential (this algebra is considered
implicitly in~\cite[Ex.~3.30]{Lad22} and explicitly in~\cite[Ex.~6.7]{Lor25},
see also \Cref{rmk:is_cluster_cat_of_k}).

\begin{proposition}\label{prop:just_a_field}
  The functor
\[
  \DerCat{\kk\langle\imath^\pm\rangle}\longrightarrow\prod_{i=0}^{d-1}\Mod*{\kk},\qquad
  M\longmapsto(\dgH[0]{M},\dgH[1]{M},\dots,\dgH[d-1]{M}),
\]
is an equivalence of $\kk$-linear categories. Under this equivalence, the shift
  functor on the source corresponds to the automorphism
  \[
    (V^0,V^1,\dots,V^{d-1})\longmapsto(V^1,\dots,V^{d-1},V^0).
  \]
  on the target. Moreover, the previous equivalence restricts to an equivalence
  of $\kk$-linear categories
  \[
    \DerCat[c]{\kk\langle\imath^\pm\rangle}\longrightarrow\prod_{i=0}^{d-1}\mmod*{\kk}.
  \]
  Furthemore, the free DG module
  ${\kk\langle\imath^\pm\rangle\in\DerCat[c]{\kk\langle\imath^\pm\rangle}}$ is a $d\ZZ$-cluster tilting object, since
  this is obviously the case for its image
  \[
    (\kk,0,\dots,0)\in\prod_{i=0}^{d-1}\mmod*{\kk}.
  \]
\end{proposition}

\begin{proof}
  By definition, a DG $\kk\langle\imath^\pm\rangle$-module $M=(V,\varphi)$ consists of a cochain
  complex of vector spaces $V$ equipped with an isomorphism
  $\varphi\colon V\stackrel{\sim}{\to} V[-d]$. Similarly, a morphism of DG
  $\kk\langle\imath^\pm\rangle$-modules $f\colon(V,\varphi)\to(W,\psi)$ is a morphism of cochain
  complexes of vector spaces $f\colon V\to W$ such that the diagram
  \[
    \begin{tikzcd}
      V\dar{f}\rar{\varphi}&V[-d]\dar{f[-d]}\\
      W\rar{\psi}&W[-d]
    \end{tikzcd}
  \]
  commutes. In particular, such an $f$ is a quasi-isomorphism if and only if the
  induced linear map
  \[
    \dgH[i]{f}\colon\dgH[i]{V}\longrightarrow\dgH[i]{W}
  \]
  is an isomorphism for all $0\leq i<d$. 
  
  There is an apparent fully faithful functor
  \[
    \prod_{i=0}^{d-1}\Mod*{\kk}\longrightarrow\dgMod*{\kk\langle\imath^\pm\rangle}
  \]
  which sends a tuple $(V^0,V^1,\dots,V^{d-1})$ to the DG $\kk\langle\imath^\pm\rangle$-module with trivial differential
  $(V,\id[V])$ where
  \[
    \begin{tikzcd}[column sep=small]
      V\colon&\cdots&V^{d-1}&V^0&V^1&\cdots&V^{d-1}&V^0&\cdots
    \end{tikzcd}
  \]
  Notice that the composite
  \[
    \prod_{i=0}^{d-1}\Mod*{\kk}\longrightarrow\dgMod*{\kk\langle\imath^\pm\rangle}\longrightarrow\dgH[0]{\dgModdg*{\kk\langle\imath^\pm\rangle}}
  \]
  is also fully faithful. To complete the proof, it is enough to observe the following:
  \begin{itemize}
  \item Every DG $\kk\langle\imath^\pm\rangle$-module is homotopy equivalent to its cohomology
    (viewed as a DG $\kk\langle\imath^\pm\rangle$-module in the obvious way).
  \item A morphism of DG $\kk\langle\imath^\pm\rangle$-modules is a quasi-isomorphism if and only if
    it is a homotopy equivalence and therefore
    \[
      \dgH[0]{\dgModdg*{\kk\langle\imath^\pm\rangle}}=\DerCat{\kk\langle\imath^\pm\rangle}.
    \]
  \item The composite
    \[
      \prod_{i=0}^{d-1}\Mod*{\kk}\longrightarrow\DerCat{\kk\langle\imath^\pm\rangle}\longrightarrow\prod_{i=0}^{d-1}\Mod*{\kk}
    \]
    is isomorphic to the identity functor.
  \end{itemize}
  The first claim is obvious and the second and third claims can be shown exactly
  as in the case of cochain complexes of vector spaces, observing that the
  category $\dgMod*{\kk\langle\imath^\pm\rangle}$ is equivalent to the category of $d$-periodic
  cochain complexes of vector spaces. 
  
  The final claims about the shift, the restriction to compact objects and the $d\ZZ$-cluster tilting object are
  straightforward.
\end{proof}

\begin{remark}
  \label{rmk:modk_unique_ang}
  According to \Cref{thm:GKO-standard}, the additive category
  \[
    \mmod*{\kk}\simeq\add*{\kk\langle\imath^\pm\rangle}\subseteq\DerCat[c]{\kk\langle\imath^\pm\rangle}, \qquad |\imath|=-d,
  \]
  admits a $(d+2)$-angulated structure with suspension functor $\Sigma=\id$ (up
  to natural isomorphism, the only autoequivalence of $\mmod*{\kk}$) whose
  $(d+2)$-angles are induced by the triangles in $\DerCat[c]{\kk\langle\imath^\pm\rangle}$. In
  fact, $\mmod*{\kk}$ admits a unique $(d+2)$-angulation: A $(d+2)$-angle
  \[
    V_1\to V_2\to \cdots \to V_{d+2}\to V_1
  \]
  in $\mmod*{\kk}$ must yield a $(d+2)$-periodic exact sequence
  \[
    \cdots\to V_{d+2}\to V_1\to V_2\to \cdots\to V_{d+2}\to V_1\to \cdots
  \]
  (see~\cite[Prop.~2.5]{GKO13}) and, conversely, every $(d+2)$-angulation must
  contain all such sequences for these are isomorphic to a finite direct sum of
  rotations of the trivial sequence
  \[
    \kk\xrightarrow{\id}\kk\to0\to\cdots\to 0\to\kk
  \]
  with $d$ zeroes. Moreover, $\mmod*{\kk}$ has a unique DG-enhancement, namely the graded category of finitely generated $\kk\langle\imath^\pm\rangle$-modules endowed with the trivial differential.
\end{remark}

\begin{remark}
  \label{rmk:is_cluster_cat_of_k}
  \Cref{prop:just_a_field,thm:injective_correspondence} show that the triangulated category $\DerCat[c]{\kk\langle\imath^\pm\rangle}$, $|\imath|=-d$, is
  equivalent to the $d$-cluster category of $\kk$, that is the orbit
  category~\cite{Kel05}
  \[
    \C_d(\kk)\coloneqq\DerCat[c]{\mmod*{\kk}}/[d].
  \]
  Indeed, $\C_d(\kk)$ is an algebraic triangulated category with $[d]=\id$,
  finite-dimensional morphism spaces, split idempotents and $\kk\in\C_d(\kk)$ is a $d\ZZ$-cluster
  tilting object with endomorphism algebra isomorphic to $\kk$ (alternatively,
  one can also invoke the Recognition Theorem of Keller and Reiten~\cite{KR08}). 
\end{remark}

\subsubsection{Non-unique enhancements in the $\operatorname{Hom}$-infinite
  case}
\label{subsubsec:non-unique_Hom-infinite}

Let $\ell$ be a field and $\kk=\ell(x_0,x_1,\dots,x_{d-1})$; notice that $\kk$
is perfect if $\ell$ has characteristic $0$. Rizzardo and Van den
Bergh~\cite{RVdB19} show that $\DerCat[c]{\kk\langle\imath^\pm\rangle}$, viewed as an $\ell$-linear
(!) algebraic triangulated category, admits two inequivalent enhancements. This
does not contradict \Cref{thm:dZ-Auslander_correspondence}, \Cref{thm:uniqueness_triangulated} or \Cref{thm:strong_enhancements}: as an $\ell$-linear
category, the morphism spaces in $\DerCat[c]{\kk\langle\imath^\pm\rangle}$ are infinite dimensional
and therefore the hypotheses in \Cref{thm:dZ-Auslander_correspondence} and \Cref{thm:uniqueness_triangulated,thm:strong_enhancements} are not
satisfied. Thus, this example explains why it is necessary to restrict to
triangulated categories with finite-dimensional morphism spaces.

\subsubsection{Non-unique enhancements for non-separable algebras}
\label{subsubsec:non-unique_non-separable}

Although \Cref{thm:dZ-Auslander_correspondence} requires the ground field $\kk$
to be perfect, in \Cref{thm:projLambda_existence_and_uniqueness,thm:injective_correspondence,thm:strong_enhancements} it is enough to
assume that $\Lambda/J_\Lambda$ is separable. We shall see that the latter
assumption is essential. Indeed, let $p>0$ be a prime number and $\kk=\FF(x)$;
thus, $\kk$ is not perfect. Let
\[
  \Lambda\coloneqq\FF(x^{\frac{1}{p}})=\frac{\kk[X]}{(X^p-x)},
\]
which is a purely inseparable field extension of $\kk$ of degree $p$ with
trivial Galois group; thus, in particular, $\Lambda=\Lambda/J_\Lambda$ is not
separable over $\kk$. Notice also that the algebras $\kk$ and $\Lambda$ are
isomorphic over $\FF$, but not over $\kk$. Moreover, there is an isomorphism
of $\kk$-algebras
\[
  \Lambda^e=\Lambda\otimes_\kk\Lambda^\op\cong\frac{\kk[X,Y]}{(X^p-x,Y^p-x)}\cong\frac{\Lambda[Y]}{(Y-X)^p};
\]
we treat the above isomorphism as an identification. Over $\Lambda$, we have
seen that the semisimple category $\mmod*{\Lambda}$ admits a unique
$(d+2)$-angulated structure (\Cref{rmk:modk_unique_ang}), and the latter admits
a unique enhancement, represented by the graded category finitely generated $\Lambda\langle\imath^\pm\rangle$-modules, $|\imath|=-d$, with trivial differential. We shall see
that $\mmod*{\Lambda}$ admits two inequivalent enhancements over the field
$\kk$. Indeed, notice that the diagonal $\Lambda$-bimodule has infinite
projective dimension and admits a $2$-periodic projective resolution
\[
  \begin{tikzcd}[column sep=large]
    \cdots\rar{(Y-X)^{p-1}}&\Lambda^e\rar{Y-X}&\Lambda^e\rar{(Y-X)^{p-1}}&\Lambda^e\rar{Y-X}&\Lambda^e\rar&0
  \end{tikzcd}
\]
that is in fact $1$-periodic if $p=2$. Thus, for even $d$ (or every $d\geq1$ if
$p=2$), we obtain a non-trivial extension of length $d+2$
\[
  \begin{tikzcd}[column sep=11mm]
    0\rar&\Lambda\rar&\Lambda^e\rar{Y-X}&\cdots \rar{(Y-X)^{p-1}}&\Lambda^e\rar{Y-X}&\Lambda^e\rar&\Lambda\rar&0
  \end{tikzcd}
\]
and we may consider the AL $(d+2)$-angulation associated to this extension
(\Cref{thm:Amiot-Lin}). By \Cref{thm:projLambda_existence_and_uniqueness}, this
AL $(d+2)$-angulation admits an enhancement $\B$ whose restricted universal
Massey product is represented, up to sign, by the class of the above non-trivial extension;
in particular, $\B$ cannot be quasi-isomorphic to $\dgH{\B}$ as DG categories
over $\kk$ and, therefore, $\B$ cannot be quasi-equivalent to the enhancement of
$\mmod*{\Lambda}$ given by the graded category finitely generated $\Lambda\langle\imath^\pm\rangle$-modules, $|\imath|=-d$, with trivial differential.

Moreover, just as in \cite{RVdB20}, $\DerCat[c]{\B}$ is equivalent to $\DerCat[c]{\Lambda\langle\imath^\pm\rangle}$ as triangulated categories, but the two aforementioned DG enhancements cannot be Morita equivalent. Indeed, if they were, the arguments in the proof of \Cref{thm:strong_enhancements} would show that they also have to be quasi-equivalent, and we have already argued that they are not. This proves that the separability hypothesis is necessary in \Cref{thm:uniqueness_triangulated,thm:strong_enhancements}.

\subsubsection{$d\ZZ$-cluster tilting does not guarantee twisted
  $(d+2)$-periodicity}

According to \Cref{prop:dZ_CT_twisted_periodic}, the endomorphism algebra
$\Lambda$ of a $d\ZZ$-cluster tilting object in a triangulated category with
finite-dimensional morphism spaces is twisted $(d+2)$-periodic, provided that
$\Lambda/J_\Lambda$ is separable. This need not be the case if the algebra
$\Lambda/J_\Lambda$ is not separable, even if the ambient triangulated category
is algebraic. Indeed, as in \Cref{subsubsec:non-unique_non-separable} above, let
$p>0$ be a prime number, $\kk=\mathbb{F}_p(x)$ and
$\Lambda=\mathbb{F}_p(x^{\frac{1}{p}})$. As explained in
\Cref{sec:comms_non_uniqueness}, the triangulated category
\[
  \DerCat[c]{\Lambda\langle\imath^{\pm}\rangle},\qquad |\imath|=-d,
\]
is algebraic and admits a $d\ZZ$-cluster tilting object with endomorphism
algebra isomorphic to $\Lambda$ (since $\Lambda$ is itself a field). However, if
$p\neq2$ and $d$ is odd (in particular if $d=1$), then $\Lambda$ is not twisted
$(d+2)$-periodic as a $\kk$-algebra and, therefore, the restricted universal
Massey product associated to the canonical $\kk$-linear (!) enhancement of
$\DerCat[c]{\Lambda\langle\imath^{\pm}\rangle}$ cannot be represented by a unit
in the Hochschild--Tate cohomology of $\Lambda$ (viewed as a $\kk$-algebra).
This does not contradict \Cref{prop:dZ_CT_twisted_periodic} nor the implication
\eqref{it:isKaroubian_d+2}$\Rightarrow$\eqref{it:H0_isAL-angulated} in
\Cref{coro:pre-d+2-ang=unit}, since $\Lambda=\Lambda/J_\Lambda$ is not separable
as a $\kk$-algebra.

\section{Recognition theorems}
\label{sec:recognition_thms}

In this section we discuss several recognition theorems for algebraic
triangulated categories of interest in representation theory and algebraic and
symplectic geometry that are immediate consequences of
\Cref{thm:dZ-Auslander_correspondence} (given the existent knowledge about these
categories). Except in special cases, the triangulated categories considered
below do not admit an additive generator and therefore the results in
\cite{Mur22} generally do not apply to them.

\begin{setting}
  For simplicity, we assume that $\kk$ is a perfect field throughout this section.
\end{setting}

\begin{remark}
  \label{rmk:apps_Theorem_A}
Excluding \Cref{thm:application_BIKR}, all of the applications of
\Cref{thm:dZ-Auslander_correspondence} that we discuss in this section hold,
more generally, when the endomorphism algebra $\Lambda$ of the (basic)
$d\ZZ$-cluster tilting objects under consideration is such that the quotient
$\Lambda/J_\Lambda$ is a separable algebra over the ground field. This
assumption is satisfied automatically in the following cases:
\begin{itemize}
\item When the ground field is perfect (for example if it is algebraically
  closed or it has characteristic $0$).
\item When the ground field is arbitrary and $\Lambda\cong\kk Q/I$ is
  (isomorphic to) the bounded path algebra of a finite quiver. Indeed, in this
  case there is an isomorphism $\Lambda/J_\Lambda\cong\kk^{Q_0}$, where $Q_0$
  denotes the set of vertices of $Q$. The analogous statement holds when
  the finite-dimensional algebra $\Lambda$ is the (complete) Jacobian algebra of
  a quiver with potential.
\end{itemize}
Under this more general assumption, references to
\Cref{thm:dZ-Auslander_correspondence} should be replaced by references to
\Cref{thm:injective_correspondence,thm:uniqueness_triangulated}. Similarly,
non-separable algebras should be considered in place of non-semisimple algebras
when appropriate (since these notions coincide for perfect fields but not for
arbitrary fields).
\end{remark}

\subsection{Amiot--Guo--Keller cluster categories}

We begin by recalling the following theorem, proven by Amiot in the case $d=2$
and by Guo in the general case, both using results of Keller
\cite{Kel05,Kel11} in an essential way.

Recall that a triangulated category $\T$ with finite-dimensional morphism spaces
is \emph{$n$-Calabi--Yau} if there exists a natural isomorphism
\[
  \T(y,x[n])\stackrel{\simeq}{\longrightarrow}D\T(x,y),\qquad x,y\in\T.
\]
where $V\mapsto DV$ denotes the passage to the $\kk$-linear dual.

Given a DG algebra $\Gamma$, we denote by $\DerCat[fd]{\Gamma}$ the full
triangulated subcategory of $\DerCat{\Gamma}$ spanned by the DG $\Gamma$-modules
$M$ such that
\[
  \sum_{i\in\ZZ}\dim \dgH[i]{M}<\infty.
\]
Following Kontsevich, we say that $\Gamma$ is \emph{homologically smooth} if the
diagonal bimodule $\Gamma$ is perfect as a DG $\Gamma$-bimodule, that is
$\Gamma\in\DerCat[c]{\Gamma^e}$. Also, given an integer $n$, the DG algebra
$\Gamma$ is \emph{bimodule $n$-Calabi--Yau} if there is an isomorphism
\[
  \RHom[\Gamma^e]{\Gamma}{\Gamma^e}\cong\Gamma[-n]
\]
in the derived category of DG $\Gamma$-bimodules. If $\Gamma$ is homologically
smooth and bimodule $n$-Calabi--Yau, then the triangulated category
$\DerCat[fd]{\Gamma}$ is $n$-Calabi--Yau, see~\cite[Lemma~4.1]{Kel08}.

\begin{theorem}[{\cite[Thm.~2.1]{Ami09} and~\cite[Thm.~2.2]{Guo11}, see also
  \cite[Thm.~5.8]{IY18} for the case $d=1$}]
  \label{thm:Amiot-Guo}
  Let $d\geq1$ and $\Gamma$ a DG algebra that satisfies the following
  properties:
  \begin{itemize}
  \item The DG algebra $\Gamma$ is homologically smooth.
  \item The cohomology of $\Gamma$ is concentrated in non-positive degrees.
  \item The (ordinary) algebra $H^0(\Gamma)$ is finite-dimensional.
  \item The DG algebra $\Gamma$ is bimodule $(d+1)$-Calabi--Yau.
  \end{itemize}
  Then, the following statements hold:
  \begin{enumerate}
  \item The Verdier quotient
    \[
      \C(\Gamma)=\DerCat[c]{\Gamma}/\DerCat[fd]{\Gamma}
    \]
    is well defined and is an (algebraic) $d$-Calabi--Yau triangulated category
    with split idempotents. In particular, $\C(\Gamma)$ has finite-dimensional
    morphism spaces. \item\cite[Lemma~2.7]{Lad22} There are isomorphisms of
    vector spaces
    \[
      \C(\Gamma)(\Gamma[i],\Gamma)=\dgH[-i]{\Gamma},\qquad 0\leq i<d.
    \]
  \item The image of $\Gamma$ in $\C(\Gamma)$ is a $d$-cluster tilting object
    whose endomorphism algebra is isomorphic to $H^0(\Gamma)$.
  \end{enumerate}
\end{theorem}

\begin{definition}
  Let $\Gamma$ be a DG algebra that satisfies the conditions in
  \Cref{thm:Amiot-Guo} for some $d\geq1$. Following~\cite{IY18}, we call
  $\C(\Gamma)$ the \emph{Amiot--Guo--Keller (AGK) cluster category of $\Gamma$}.
\end{definition}

In the context of \Cref{thm:Amiot-Guo}, the following result of Iyama and
Oppermann gives necessary and sufficient conditions for the image of $\Gamma$ in
$\C(\Gamma)$ to be a $d\ZZ$-cluster tilting object. We introduce the following
definition, simply as a matter of convenience, compare with the `vosnex
property' in~\cite[Not.~3.5]{IO13}.

\begin{definition}
  \label{def:vosnex}
  Let $\Gamma$ be a DG algebra and $d\geq1$. We say that the cohomology of
  $\Gamma$ \emph{vanishes in small negative degrees (relative to $d$)}, if
  \[
    \forall 0<i<d-1,\qquad \dgH[-i]{\Gamma}=0.
  \]
  (notice that this condition is in fact vacuous for $d=1,2$).
\end{definition}

\begin{proposition}[{\cite[Prop.~3.6]{IO11}}]
  \label{prop:Gamma_dZ_in_CGamma}
  Let $\Gamma$ be a DG algebra that satisfies the conditions in
  \Cref{thm:Amiot-Guo} for some $d\geq1$. The following statements are
  equivalent:
  \begin{enumerate}
  \item The image of $\Gamma$ in $\C(\Gamma)$ is a $d\ZZ$-cluster tilting
    object.
  \item The following conditions are satisfied:
    \begin{itemize}
    \item The finite-dimensional algebra $\dgH[0]{\Gamma}$ is self-injective.
    \item The cohomology of $\Gamma$ vanishes in small negative degrees
      (relative to $d$).
    \end{itemize}    
  \end{enumerate}
\end{proposition}

Our main result yields the following recognition theorem for the AGK cluster
category of a DG algebra with vanishing cohomology in small negative degrees.

\begin{theorem}
  \label{thm:AGK_uniqueness}
  Let $\Gamma$ be a DG algebra that satisfies the conditions in
  \Cref{thm:Amiot-Guo} for some $d\geq1$. Assume that
  \begin{itemize}
  \item The finite-dimensional algebra $\dgH[0]{\Gamma}$ is self-injective.
  \item The cohomology of $\Gamma$ vanishes in small negative degrees (relative
    to $d$).
  \end{itemize}
  (equivalently, the image of $\Gamma$ in $\C(\Gamma)$ is a $d\ZZ$-cluster
  tilting object). Then, the following statements hold:
  \begin{enumerate}
  \item The $d$-Calabi--Yau AGK cluster category $\C(\Gamma)$ admits a unique DG
    enhancement.
  \item\label{it:CGamma_recognition}
    Suppose, moreover, that the algebra $\dgH[0]{\Gamma}$ is connected and
    non-semisimple. Let $\T$ be an algebraic
    triangulated category with finite-dimensional morphism spaces and split
    idempotents. If there exists a $d\ZZ$-cluster tilting object $c\in\T$ such
    that the algebras $\T(c,c)$ and $\dgH[0]{\Gamma}$ are isomorphic, then
    \[
      \T\simeq\C(\Gamma)
    \]
    as triangulated categories.
  \end{enumerate}  
\end{theorem}
\begin{proof}
  In view of \Cref{thm:Amiot-Guo} and \Cref{prop:Gamma_dZ_in_CGamma}, both
  claims are immediate consequences of \Cref{thm:dZ-Auslander_correspondence}
  and \Cref{prop:unique_suspension}.
\end{proof}

\subsection{The AGK cluster category of a $d$-representation finite algebra}

\begin{setting}
  We assume that $\kk$ is a perfect field throughout this subsection.
\end{setting}

We highlight a class of DG algebras that satisfy the conditions in
\Cref{thm:AGK_uniqueness}, and hence their associated AGK cluster category is
essentially unique. These DG algebras arise naturally in the context of Iyama's
higher-dimensional Auslander--Reiten theory~\cite{Iya07a,Iya07}.

The following class of algebras was introduced by Iyama and Oppermann. Our
choice of terminology is inspired by~\cite[Def.~3.2 and Thm.~3.4]{HIO14}.

\begin{definition}[{\cite[Def.~2.2]{IO11}}]
  \label{def:d-CT_module}
  Let $A$ be a finite-dimensional algebra. We say that $A$ is
  \emph{$d$-representation-finite ($d$-hereditary)} if the following conditions
  are satisfied:
  \begin{enumerate}
  \item The algebra $A$ has global dimension at most $d$.
  \item There exists a \emph{$d$-cluster tilting $A$-module} $M\in\mmod{A}$,
    that is such that the following conditions are satisfied:
    \begin{enumerate}
    \item An $A$-module $L$ lies in $\add*{M}$ if and only if
      \[
        \forall 0<i<d,\qquad \Ext[A]{L}{M}[i]=0.
      \]
    \item An $A$-module $N$ lies in $\add*{M}$ if and only if
      \[
        \forall 0<i<d,\qquad \Ext[A]{M}{N}[i]=0.
      \]
    \end{enumerate}
  \end{enumerate}
\end{definition}

\begin{remark}
  Notice that the $1$-representation finite algebras are precisely the
  hereditary finite-dimensional algebras of finite representation type. Thus,
  $d$-re\-pre\-sen\-ta\-tion finite algebras can be regarded as
  `higher-dimensional analogues' of the latter class of algebras.
\end{remark}

\begin{remark}
  In the context of \Cref{def:d-CT_module}, the $d$-cluster tilting module turns
  out to be unique, up to multiplicity of its indecomposable direct
  summands~\cite[Prop.~1.3]{Iya11}.
\end{remark}

Following Keller~\cite{Kel11}, given a homologically smooth DG algebra $A$ and
an integer $n$, we consider its \emph{derived n-preprojective algebra
  $\bfPi_n(A)$}, which is defined as the tensor DG algebra of the shifted
inverse dualising DG $A$-bimodule
\[
  \RHom[A^e]{A}{A^e}{[n-1]}.
\]
Crucially, the DG algebra $\bfPi_n(A)$ is homologically smooth and bimodule
$n$-Calabi--Yau~\cite[Thm.~4.8]{Kel11}.

Notice that we may apply the above construction when $A$ is a finite-dimensional
algebra of global dimension at most $d$, viewed as a DG algebra concentrated in
degree $0$ (recall that we assume $\kk$ to be a perfect field).
In this case, we consider the derived $(d+1)$-preprojective algebra
$\bfPi_{d+1}(A)$ and, following Iyama and Oppermann, we also consider the (non-derived)
\emph{(d+1)-preprojective algebra}
\[
  \Pi_{d+1}(A)\coloneqq\dgH[0]{\bfPi_{d+1}(A)},
\]
compare with~\cite[Def.~2.11 and Rmk.~2.12]{IO13}. As explained in the paragraph
before~\cite[Cor.~4.8]{Guo11}, the cohomology of $\bfPi_{d+1}(A)$ is
concentrated in non-positive degrees.

In the context of this article, our interest in $d$-representation finite
 algebras stems from the following theorem, see also the proof
 of~\cite[Prop.~8.6]{CDIM25}.

\begin{theorem}[{\cite[Lemma 2.13 and Cor.~3.7]{IO13}}]
  \label{thm:IO-dRF_Pi}
  Let $A$ be a $d$-representation finite algebra. Then, its derived
  $(d+1)$-preprojective algebra $\bfPi_{d+1}(A)$ satisfies the equivalent
  conditions in \Cref{prop:Gamma_dZ_in_CGamma}.
\end{theorem}

We obtain the following recognition theorem for the AGK cluster category of a
$d$-representation finite algebra.

\begin{theorem}
  \label{thm:AGK_uniqueness-dRF}
  Let $A$ be a $d$-representation finite algebra and $\bfPi_{d+1}(A)$ its
  derived $(d+1)$-preprojective algebra. Then, the following statements hold:
  \begin{enumerate}
  \item The $d$-Calabi--Yau AGK cluster category $\C(\bfPi_{d+1}(A))$ admits a
    unique DG enhancement.
  \item Suppose, moreover, that the algebra $\Pi_{d+1}(A)$ is connected and
    non-semisimple. Let $\T$ be an algebraic triangulated category with finite-dimensional
    morphism spaces and split idempotents. If there exists a $d\ZZ$-cluster
    tilting object $c\in\T$ such that the algebra $\T(c,c)$ is isomorphic to the
    (non-derived) $(d+1)$-preprojective algebra $\Pi_{d+1}(A)$, then
    \[
      \T\simeq\C(\bfPi_{d+1}(A))
    \]
    as triangulated categories.
  \end{enumerate}  
\end{theorem}
\begin{proof}
  In view of \Cref{thm:IO-dRF_Pi}, the claims is immediate from
  \Cref{thm:AGK_uniqueness} and \Cref{prop:dZ_CT_twisted_periodic}.
\end{proof}

\subsection{The Amiot cluster category of a self-injective quiver with
  potential}

Recall that a \emph{quiver with potential} is a pair $(Q,W)$ consisting of a
finite quiver and a possibly-infinite linear combination $W$ of (oriented)
cycles in $Q$. To these data, one associates its \emph{(completed) Jacobian
  algebra $J(Q,W)$}, see~\cite[Def.~3.1]{DWZ08} for the precise definition.
Following Herschend and Iyama~\cite{HI11}, we say that $(Q,W)$ is
\emph{self-injective} if its completed Jacobian algebra is a finite-dimensional
self-injective algebra.

Given a quiver with potential $(Q,W)$, we are also interested in the
\emph{completed ($3$-Calabi--Yau) Ginzburg DG algebra $\Gamma(Q,W)$ of $(Q,W)$},
see~\cite{Gin06} and~\cite[Sec.~2.6]{KY11}. Replacing DG algebras and DG modules
by their pseudo-compact counter-parts, \Cref{thm:Amiot-Guo} holds \emph{mutatis
  mutandis} for the pseudo-compact DG algebra $\Gamma(Q,W)$, see
\cite[Thm.~A.27]{KY11}. In particular, associated to $(Q,W)$ there is a
well-defined $2$-Calabi--Yau triangulated category
\[
  \C(Q,W)=\C(\Gamma(Q,W)),
\]
which we call the \emph{Amiot cluster category of $(Q,W)$}, and the image of
$\Gamma$ in $\C(\Gamma)$ is a $2$-cluster tilting object whose endomorphism
algebra is isomorphic to the Jacobian algebra of $(Q,W)$, see also
\cite{Ami09}. Consequently, we obtain the following recognition theorem for the
Amiot cluster category of a self-injective quiver with potential.

\begin{theorem}
  \label{thm:recognition_cluster_cat_self-inj_QP}
  Let $\kk$ be an arbitrary field. Let $(Q,W)$ be a quiver with potential whose
  Jacobian algebra is finite-dimensional and self-injective. Then, the following statements hold:
  \begin{enumerate}
  \item The (2-Calabi--Yau) Amiot cluster category $\C(Q,W)$ admits a unique DG
    enhancement.
  \item Suppose, moreover, that the Jacobian algebra $J(Q,W)$ is connected and
    non-semisimple. Let $\T$ be an algebraic triangulated category with
    finite-dimensional morphism spaces and split idempotents. If there exists a
    $2\ZZ$-cluster tilting object $c\in\T$ such that the algebras $\T(c,c)$ 
    and $J(Q,W)$ are isomorphic, then
    \[
      \T\simeq\C(Q,W)
    \]
    as triangulated categories.
  \end{enumerate}  
\end{theorem}
\begin{proof}
  In view of the previous discussion, both claims are immediate consequences of
  \Cref{thm:injective_correspondence,thm:uniqueness_triangulated} and \Cref{prop:unique_suspension},
  keeping in mind that, for $d=2$, the condition on the cohomology of $\Gamma$
  vanishing in small negative degrees is vacuous (we also remind the reader of \Cref{rmk:apps_Theorem_A}).
\end{proof}

\begin{remark}
  Suppose that $\kk$ is an algebraically closed field. In this case,
  \Cref{thm:recognition_cluster_cat_self-inj_QP} can be seen as an instance of
  \Cref{thm:AGK_uniqueness-dRF} for completed derived higher preprojective
  algebras. Indeed, the completed Ginzburg DG algebra of a self-injective quiver
  with potential is quasi-isomorphic to the completed derived $3$-preprojective
  algebra of some $2$-representation finite algebra, see the proof of
  ~\cite[Thm.~3.11(a)]{HI11} and~\cite[Prop.~2.4]{HI11}.
\end{remark}

\begin{remark}
  Classifying all quivers with potential whose Jacobian algebra is
  self-injective seems to be an intractable problem. Families of such quivers
  with potential are constructed in~\cite{HI11,Pas20}, see also
  ~\cite[Thm.~1.3]{Jas15} where the self-injective cluster tilted algebras of
  canonical type are classified.
\end{remark}

\begin{remark}
  \label{rmk:Keller_announcement}
  In relation to \Cref{thm:recognition_cluster_cat_self-inj_QP}, we note that
  Keller and Liu have recently announced~\cite{KL23a} a proof of the following (modified)
  conjecture of Amiot~\cite{Ami08}: Suppose that the ground field $\kk$ is
  algebraically closed and of characteristic $0$.
  Given a Karoubian pre-triangulated DG category $\A$ (enriched in cochain
  complexes of pseudo-compact
  vector spaces) equipped with a right
  $2$-Calabi--Yau structure in the sense of~\cite{KS06} which admits a basic
  $2$-cluster tilting object $c\in\dgH[0]{\A}$ with finite-dimensional
  endomorphism algebra, there exists a quiver with potential $(Q,W)$ and an
  equivalence of triangulated categories
  \[
    \C(Q,W)\stackrel{\simeq}{\longrightarrow}\dgH[0]{\A},\qquad
    \Gamma(Q,W)\longmapsto c.
  \]
  In particular, $\dgH[0]{\A(c,c)}\cong\dgH[0]{\Gamma(Q,W)}=J(Q,W)$. Notice,
  however, the explicit assumption on the enhancement. On the other hand,
  Keller's theorem is more general than \Cref{thm:AGK_uniqueness} in that the
  object $c$ need not be $2\ZZ$-cluster tilting and its endomorphism algebra is
  known to be isomorphic to a completed Jacobian algebra only \emph{a
    posteriori}, and this is one place where the existence of a right
  $2$-Calabi--Yau structure plays an important role in Keller's proof (compare
  with~\cite{VdBer15}).
\end{remark}

\subsection{Cohen--Macaulay modules for one-dimensional hypersurface
  singularities}

Let $(R,\mathfrak{m})$ be a local complete $d$-dimensional commutative
noetherian Gorenstein isolated singularity. We are interested in the category
\[
  \CM*{R}\coloneqq\set{M\in\mmod*{R}}[\operatorname{depth}M=\dim R]
\]
of \emph{maximal Cohen--Macaulay $R$-modules}; this is a Frobenius exact
category whose projective objects are the projective $R$-modules. Moreover, it
is well known that there is an equivalence of triangulated
categories~\cite{Buc21}
\[
  \stableCM*{R}\stackrel{\sim}{\longrightarrow}\SingCat{R}\coloneqq\DerCat[b]{R}/\DerCat[c]{R}
\]
between the stable category of maximal Cohen--Macaulay $R$-modules and the
singularity category of $R$. The following result of Burban, Iyama, Keller and
Reiten gives a large class of one-dimensional hypersurface singularities whose
stable category of maximal Cohen--Macaulay modules admits a $2\ZZ$-cluster
tilting object.

\begin{theorem}[{\cite[Thm.~1.5]{BIKR08}}]
  \label{thm:BIKR}
  Let $\kk$ be an algebraically closed
  field of characteristic zero. Let $R=\kk\llbracket x,y\rrbracket/(f)$ be a
  one-dimensional reduced hypersurface singularity. The following statements are
  equivalent:
  \begin{enumerate}
  \item The stable category $\stableCM*{R}$ of Cohen--Macaulay $R$-modules
    admits a $2\ZZ$-cluster tilting object.
  \item $f$ is a product $f=f_1f_2\cdots f_n$ with $f_i\not\in (x,y)^2$.
  \end{enumerate}
\end{theorem}

\begin{remark}
  In~\cite{BIKR08}, the authors only discuss $2$-cluster tilting objects.
  However, for hypersurface singularities the stable category of Cohen--Macaulay
  modules is $2$-periodic, that is the shift functor squares to the identity up
  to natural isomorphism~\cite{Eis80} (see also~\cite[Sec.~1]{BIKR08}). In particular, all $2$-cluster tilting
  objects are also $2\ZZ$-cluster tilting.
\end{remark}

\begin{theorem}
  \label{thm:application_BIKR}
  Let $\kk$ be an algebraically closed
  field of characteristic zero. Let $R=\kk\llbracket x,y\rrbracket/(f)$ be a
  one-dimensional reduced hypersurface singularity such that $f$ is a product
  $f=f_1f_2\cdots f_n$ with $f_i\not\in (x,y)^2$. Then, the stable category
  $\stableCM*{R}$ of Cohen--Macaulay $R$-modules admits a unique enhancement.
\end{theorem}
\begin{proof}
  In view of \Cref{thm:BIKR}, the claim follows immediately from \Cref{thm:dZ-Auslander_correspondence}.
\end{proof}

\subsection{Stable categories of self-injective higher Nakayama algebras}

Given integers $n\geq1$ and $\ell\geq2$, using results of Darp{\"o} and Iyama
\cite{DI20}, K{\"u}lshammer and the first-named author introduced in
\cite{JK19a} a family of finite-dimensional algebras
\[
  \widetilde{A}_{n-1,\ell}^{(d)},\qquad n,\ell\geq1,
\]
called the \emph{self-injective $d$-Nakayama algebras}. For example, if $d=1$,
then $\widetilde{A}_{n-1,\ell}^{(1)}$ is the classical (connected) Nakayama
algebra with $n$-simples and indecomposable projective modules of Loewy length
$\ell$. On the other hand, if $d=2$ and $n=1$, then
$\widetilde{A}_{0,\ell}^{(2)}$ is isomorphic to the Gelfand--Ponomarev
preprojective algebra of Dynkin type $\AA_{\ell}$. More generally, if $d\geq2$
and $n=1$, then $\widetilde{A}_{0,\ell}^{(d)}$ is isomorphic to the
\emph{$d$-preprojective algebra of Dynkin type $\AA_\ell$} first investigated by
Iyama and Oppermann in~\cite[Sec.~5]{IO11} (see also~\cite[6.3 and 6.5]{GKO13}).

\begin{theorem}[{\cite[Thm.~4.10]{JK19a}}]
  \label{thm:Nakayama}
  Let $n\geq1$ and $\ell\geq2$ and $A=\widetilde{A}_{n-1,\ell}^{(d)}$ the
  corresponding self-injective $d$-Nakayama algebra. The following statements
  hold:
  \begin{enumerate}
  \item The algebra $A$ is indeed a finite-dimensional self-injective algebra.
  \item There exists a distinguished $d\ZZ$-cluster tilting $A$-module
    $M=M_{n-1,\ell}^{(d)}$.
  \item The triangulated category $\smmod*{A}$ admits a $d\ZZ$-cluster tilting
    object (namely, the image of $M$).
  \end{enumerate}
\end{theorem}

We obtain the following recognition theorem for the stable module category of a
self-injective $d$-Nakayama algebra.

\begin{theorem}
  Let $\kk$ be an arbitrary field. Let $n\geq1$ and $\ell\geq2$ and
  $A=\widetilde{A}_{n-1,\ell}^{(d)}$ the corresponding self-injective
  $d$-Nakayama algebra. The following statements hold:
  \begin{enumerate}
  \item The stable module category $\smmod*{A}$ admits a unique enhancement.
  \item Let $\T$ be an algebraic triangulated category with finite-dimensional
    morphism spaces and split idempotents. If there exists a $d\ZZ$-cluster
    tilting object $c\in\T$ such that the algebra $\T(c,c)$ is isomorphic to the
    stable endomorphism algebra $\underline{\operatorname{End}}_{A}(M)$ of the
    distinguished $d\ZZ$-cluster tilting $A$-module $M$. Then,
    \[
      \T\simeq\smmod*{A}
    \]
    as triangulated categories.
  \end{enumerate}
\end{theorem}
\begin{proof}
  In view of \Cref{thm:Nakayama}, both claims are immediate from
  \Cref{thm:dZ-Auslander_correspondence}.
  Indeed, using the explicit
  description of the endomorphism algebra of the distinguished $d\ZZ$-cluster
  tilting $A$-module $M$ given in~\cite[Def.~4.9]{JK19a}, it is easy to verify
  that there is an isomorphism of algebras
  \[
    \underline{\operatorname{End}}_{A}(M)\cong A_{n-1,\ell-1}^{(d+1)},
  \]
  that is the length parameter $\ell$ is reduced by $1$ while the dimension
  parameter $d$ increases by $1$. In particular,
  $\underline{\operatorname{End}}_{A}(M)$ is isomorphic to the bounded path
  algebra of a finite quiver.
\end{proof}

\subsection{Singularity categories of Iwanaga--Gorenstein algebras with
  a $d\ZZ$-cluster tilting module}

Recall that a finite-dimensional algebra $A$ is \emph{Iwanaga--Gorenstein} if
\[
  \operatorname{inj.dim} A_A<\infty\qquad\text{and}\qquad\operatorname{inj.dim} {}_AA<\infty.
\]
We are interested in the category
\[
  \Gproj*{A}\coloneqq\set{M\in\mmod*{A}}[\forall i>0,\ \Ext[A]{M}{A}[i]=0]
\]
of (finite-dimensional) \emph{Gorenstein projective $A$-modules}; this is a
Frobenius exact category whose projective objects are the projective
$A$-modules. Moreover, it is well known that there is an equivalence of
triangulated categories~\cite{Buc21}
\[
  \stableGproj*{A}\stackrel{\sim}{\longrightarrow}\SingCat{A}\coloneqq\DerCat[b]{A}/\DerCat[c]{A}
\]
between the stable category of Gorenstein projective
$A$-modules and the singularity category of $A$. We recall the following theorem
of Kvamme.

\begin{theorem}[{\cite[Cor.~1.3]{Kva21}}]
  \label{thm:Kvamme}
  Let $A$ be a finite-dimensional Iwanaga--Gorenstein algebra. Suppose that
  there exists $d\ZZ$-cluster tilting $A$-module $M\in\mmod*{A}$ and let
  $\overline{M}$ be the largest direct summand of $M$ that is Gorenstein
  projective. Then, $\overline{M}$ is a $d\ZZ$-cluster tilting object in
  $\stableGproj*{A}\simeq\SingCat{A}$.
\end{theorem}

We obtain the following recognition theorem for the singularity category of an
Iwanaga--Gorenstein algebra with a $d\ZZ$-cluster tilting module.

\begin{theorem}
  \label{thm:Kvamme-recognition}
  Let $\kk$ be an arbitrary field. Let $A$ be a finite-dimensional
  Iwanaga--Gorenstein algebra. Suppose that there exists a $d\ZZ$-cluster
  tilting $A$-module $M\in\mmod*{A}$ and let $\overline{M}$ be the largest
  direct summand of $M$ that is Gorenstein projective. The following
  statements hold:
  \begin{enumerate}
  \item The singularity category $\SingCat{A}$ admits a unique enhancement.
  \item Suppose, moreover, that the
    algebra $\underline{\operatorname{End}}_{A}(\overline{M})$ is
    connected and non-semisimple. Let $\T$ be an algebraic triangulated category with finite-dimensional
    morphism spaces and split idempotents. If there exists a $d\ZZ$-cluster
    tilting object $c\in\T$ such that the algebra $\T(c,c)$ is isomorphic to the
    stable endomorphism algebra $\underline{\operatorname{End}}_{A}(\overline{M})$, then
    \[
      \T\simeq\SingCat{A}
    \]
    as triangulated categories.
  \end{enumerate}
\end{theorem}
\begin{proof}
  In view of \Cref{thm:Kvamme}, both claims are immediate from
  \Cref{thm:dZ-Auslander_correspondence}.
\end{proof}

For instances where \Cref{thm:Kvamme-recognition} can be applied
see~\cite{McM20,Xin23} for example.

\appendix

\section{The Donovan--Wemyss conjecture, by Bernhard Keller}
We work over the field of complex numbers $\kk=\CC$.
A {\em compound Du Val (=cDV) singularity} is a singularity of the form
\[
R=\CC[[u,v,x,y]]/(f(u,v,x)+y g(u,v,x,y))
\]
such that $\CC[[u,v,x]]/(f(u,v,x))$ is a Kleinian surface singularity,
cf.~Definition~4.2 in \cite{Aug20a}. These singularities were introduced by
Miles Reid \cite{Reid81} at the beginning of the eighties and play an important
role in the minimal model program in birational geometry. We refer to the
introduction of \cite{August19} as well as to \cite{Wem23} for excellent
introductory surveys on this subject.

From now on, we make the blanket assumption that all the cDV singularities
we consider are {\em isolated} and admit a \emph{(smooth) resolution} because the conjecture only applies to these.

Let us fix a cDV singularity $R$. Before stating the conjecture, let us list
the most important properties of its singularity category
\[
\sg(R)=\SingCat{R}=\DerCat[b]{R}/\DerCat[c]{R}.
\]
Notice first that since $R$ is a hypersurface, it is Gorenstein and the canonical functor 
\[
\CM(R) \to\mmod R \to  \DerCat[b]{R}
\]
induces an equivalence from the stable category of Cohen-Macauley modules
$\stableCM(R)$ to $\sg(R)$, cf. \cite{Buc21}. Since the singularity is isolated, the
category $\sg(R)$ is  $\mathrm{Hom}$-finite. Moreover,
it is Krull--Schmidt with split idempotents since $R$ is complete. By definition, it is algebraic.
Since $R$ is a hypersurface, we can describe $\sg(R)$ using matrix factorisations
and thus, the square $\Sigma^2$ of the suspension functor is isomorphic
to the identity functor, i.e. the category $\sg(R)$ is $2$-periodic.
Finally, it is $2$-Calabi--Yau since it is the
stable category of Cohen--Macauley modules over a Gorenstein
ring of Krull dimension $3$, cf. Prop.~1.3 in Ch.~III of \cite{Auslander76}. 

A {\em contraction algebra} for $R$ is \cite{DW16, Wem18} the stable
endomorphism algebra of a $2$-cluster-tilting object $T$ of the singularity category  $\sg(R)$.
Now let $R_1$ and $R_2$ be two cDV singularities and let $A_i$ be
a contraction algebra for $R_i$, $i=1,2$. In its original form
the Donovan--Wemyss conjecture states

\begin{conjecture}[\cite{DW16}]
Suppose that $A_1$ and $A_2$ are local.  Then $A_1$ is isomorphic to $A_2$ if and only if 
$R_1$ is isomorphic to $R_2$.
\end{conjecture}
The sufficiency was shown in \cite{DW16}. In type A, the conjecture follows
from the work of Reid \cite{Reid81}. In type D, evidence is given in
\cite{BrownWemyss18, vanGarderen20, Kaw23}. Further evidence
comes from \cite{Hua18, HuaToda18, HK24, Booth21}, where the authors
use enhancements of the contraction algebra.

Later Donovan--Wemyss generalized
their conjecture as follows for not necessarily local contraction algebras:
\begin{conjecture} The contraction algebras $A_1$ and $A_2$ are derived
equivalent if and only if the singularities $R_1$ and $R_2$ are
isomorphic. 
\end{conjecture}
In this form, the conjecture appears as Conjecture~1.3 in \cite{Aug20a}.
The  sufficiency follows by combining results of \cite{Wem18} with \cite{Dugas15}.
We will now deduce the necessity from  \Cref{thm:dZ-Auslander_correspondence}.
Let us emphasize that our strategy of proof has been known to the experts
since the appearance of \cite{HK24} on the arXiv in October 2018. The 
missing puzzle piece was precisely (a special case of)
\Cref{thm:dZ-Auslander_correspondence}.

Let $R_i$ be cDV singularities with contraction algebra $A_i$, $i=1,2$, and
suppose that $A_1$ is derived equivalent to $A_2$. Then, by Theorem~1.6
of \cite{Aug20a}, there is a contraction algebra $A'_2$ for $R_2$ which
is isomophic to $A_1$. Thus, we may and will assume that
$A_1$ and $A_2$ are isomorphic. By definition,
the algebras $A_i$ are endomorphism algebras of $2$-cluster-tilting
objects $T_i$ in $\sg(R_i)$. By $2$-periodicity, these are in
fact $2\ZZ$-cluster-tilting objects. The $2$-periodicity also
yields the commutative square
\[
\begin{tikzcd}
\sg(R_i)  \ar[d, "{\Sigma^2}"] & \add(T_i) \ar[l,hook'] \ar[rr, "{\mathrm{Hom}(T_i, ?)}"] &  & 
									\proj(A_i)  \ar[d, "{?\otimes_{A_i}  \mbox{}_{\sigma_i} A_i}"]\\
\sg(R_i) & \add(T_i) \ar[l,hook'] \ar[rr, "{\mathrm{Hom}(T_i, ?)}"] &  & \proj(A_i) 
\end{tikzcd}
\]
where $\sigma_i$ is the identity automorphism of $A_i$. Now clearly the given isomorphism 
$ A_1 \stackrel{_\sim}{\rightarrow} A_2$
yields an equivalence between the pairs $(A_1, \sigma_1)$ and $(A_2, \sigma_2)$. Thus, by 
\Cref{thm:dZ-Auslander_correspondence} applied in dimension $d=2$, 
we obtain a triangle equivalence
\[
\begin{tikzcd} 
\sg(R_1) \ar[r, "_\sim"] & \sg(R_2).
\end{tikzcd}
\]
Moreover, by the uniqueness of the dg enhancement in \Cref{thm:dZ-Auslander_correspondence} , 
we even obtain an isomorphism in the homotopy category of dg categories
\[
\begin{tikzcd} 
\sg_{dg}(R_1) \ar[r,"_\sim"] & \sg_{dg}(R_2) \, ,
\end{tikzcd}
\]
where $\sg_{dg}(R_i)$ denotes the canonically dg enhanced singularity category.
Using Theorem~5.9 of \cite{HK24}, we deduce that there is an isomorphism
$R_1\cong R_2$.

\bibliographystyle{halpha}%
\bibliography{library}

\end{document}

%% file: abstract.tex
We work over a perfect field. In this article, given an integer $d\geq1$, we establish a bijection between twisted $(d+2)$-periodic algebras and suitable algebraic
triangulated categories with a $d\ZZ$-cluster tilting object. Furthermore, we
shows that these triangulated categories admit a unique differential graded
enhancement. Our result yields recognition theorems for interesting
algebraic triangulated categories, such as the Amiot cluster category of a
self-injective quiver with potential in the sense of Herschend and Iyama and,
more generally, the Amiot--Guo--Keller cluster category associated with a
$d$-representation finite algebra in the sense of Iyama and Oppermann. As an application of our result, we obtain infinitely many triangulated categories with a unique differential graded enhancement that is not strongly unique. In the
appendix, B.~Keller explains how---combined with crucial results of August and
Hua--Keller---our main result yields the last key ingredient to prove the
Donovan--Wemyss Conjecture in the context of the
Homological Minimal Model Program for threefolds.